\documentclass[12pt, english, twoside, openright]{scrbook}

\usepackage{currvita} 

\usepackage{latexsym,amsfonts,epsfig,tabularx,bm,upgreek,enumitem}

\usepackage[disable]{todonotes}
  
\usepackage{amsmath, amsthm, eucal}

\usepackage{amssymb}

\usepackage{graphicx}
\usepackage{ragged2e}
\usepackage{titlesec}
\usepackage[font=small,format=plain,labelfont=bf,up,textfont=it,up]{caption}
\usepackage{pst-plot}

\titleclass{\chapter}{top}

\usepackage{xcolor} 
\definecolor{darkblue}{HTML}{00008B} 
\definecolor{forestgreen}{HTML}{006600} 
\definecolor{halfgray}{gray}{0.55} 
\definecolor{webgreen}{rgb}{0,.5,0}
\definecolor{webbrown}{rgb}{.6,0,0}
\definecolor{Maroon}{cmyk}{0, 0.87, 0.68, 0.32}
\definecolor{RoyalBlue}{cmyk}{1, 0.50, 0, 0}
\usepackage[plainpages=false,bookmarksnumbered, bookmarksopen,backref=page]{hyperref}
\hypersetup{%
    colorlinks=true, linktocpage=true, pdfborder={0 0 0},
    urlcolor=darkblue, linkcolor=darkblue, citecolor=forestgreen, 
}

\providecommand\phantomsection{}

\usepackage{backref}
\renewcommand*{\backrefalt}[4]{%
\ifcase #1 %
[No citations.]%
\or
[$\uparrow$~#2]%
\else
[$\uparrow$~#2]%
\fi
}

\usepackage[numbers,square]{natbib}
\usepackage{microtype} 
\titleformat{\chapter}[display] 
  {\normalfont\Large \color{darkblue}} 
  {                       
  \LARGE\MakeUppercase{\chaptertitlename} \Huge \thechapter \filright%
  }
  {1pt}                   
  {\titlerule \vspace{0.9pc} \filright \color{darkblue}}   
  [\color{darkblue} \vspace{0.9pc} \filright {\titlerule}]
  
\usepackage[headsepline,automark]{scrpage2}

\newcommand*{\kopffusskomplett}{
	\pagestyle{scrheadings}
	\automark[section]{chapter}
	\clearscrheadfoot
	\lehead{\leftmark}
	\rohead{\rightmark}
	\lefoot[\pagemark]{\pagemark}
	\rofoot[\pagemark]{\pagemark}
}

\newcommand*{\kopffussTOC}{
	\pagestyle{scrheadings}
	\clearscrheadfoot
	\lehead{Table of contents}
	\lefoot[\pagemark]{\pagemark}
	\rofoot[\pagemark]{\pagemark}
}

\newcommand*{\kopffussEN}{
	\pagestyle{scrheadings}
	\automark[section]{chapter}
	\clearscrheadfoot
	\lehead{Preface}
	\lefoot[\pagemark]{\pagemark}
	\rofoot[\pagemark]{\pagemark}
}

\newcommand*{\kopffussBib}{
	\pagestyle{scrheadings}
	\automark[section]{chapter}
	\clearscrheadfoot
	\lehead{Bibliography}
	\lefoot[\pagemark]{\pagemark}
	\rofoot[\pagemark]{\pagemark}
}

\newcommand*{\kopffussleer}{
	\clearscrheadfoot
	\pagestyle{empty}
}

\usepackage{chngcntr}	
\counterwithout{footnote}{chapter}

\usepackage{aliascnt}	

\theoremstyle{definition}    
   
\newtheorem{thm}{Theorem} 

\newtheorem{theorem}{Theorem}[chapter]

\newaliascnt{defi}{theorem}  
\aliascntresetthe{defi}  		 

\newaliascnt{lemma}{theorem}  
\newtheorem{lemma}[lemma]{Lemma}   
    
\aliascntresetthe{lemma} 

\newaliascnt{prop}{theorem}  
\newtheorem{prop}[prop]{Proposition}
 
\aliascntresetthe{prop} 

\newaliascnt{example}{theorem}  
\newtheorem{example}[example]{Example} 
 
\aliascntresetthe{example} 

\newaliascnt{rem}{theorem}  
\newtheorem{rem}[rem]{Remark}   
 
\aliascntresetthe{rem} 

\newaliascnt{cor}{theorem}  
\newtheorem{cor}[cor]{Corollary}   
      
\aliascntresetthe{cor} 

\renewcommand*{\epsilon}{\varepsilon}                                   
\renewcommand*{\rho}{\varrho}                                   				
\newcommand*{\nach}{\rightarrow}                                        
\newcommand*{\sep}{\; \vrule \;}                                        
\newcommand*{\N}{\mathbb{N}}                                            
\newcommand*{\R}{\mathbb{R}}                                            
\newcommand*{\C}{\mathbb{C}}                                            
\newcommand*{\E}{\mathbb{E}}                                            
\renewcommand*{\Pr}{\mathbb{P}}                                         
\newcommand*{\Hi}{\mathcal{H}}                                          
\newcommand*{\B}{\mathcal{B}}                                           
\newcommand*{\F}{\mathcal{F}}                                           
\newcommand*{\K}{\mathcal{K}}                                           
\newcommand*{\G}{\mathcal{G}}                                           
\newcommand*{\A}{\mathcal{A}}                                           
\renewcommand*{\S}{\mathcal{S}}                                         
\newcommand*{\I}{\mathcal{I}}                                           
\newcommand*{\AI}{\mathfrak{A}}                                         
\newcommand*{\SI}{\mathfrak{S}}                                         
\newcommand*{\M}{\mathcal{M}}                                           
\renewcommand*{\P}{\mathcal{P}}                                         
\newcommand*{\LO}{\mathcal{L}}                                          
\renewcommand*{\L}{\mathrm{L}}                                          
\renewcommand*{\a}{\mathfrak{a}}																				
\newcommand*{\fu}{\mathfrak{u}}																					

\renewcommand*{\l}{\ell}
\newcommand*{\leer}{\emptyset}                                          
\newcommand*{\0}{\mathcal{O}}                                           
\newcommand{\dlambda}{\,\mathrm{d}\uplambda}														
\newcommand{\wor}{\mathrm{wor}}
\newcommand{\lin}{\mathrm{lin}}
\newcommand{\all}{\mathrm{all}}
\newcommand{\init}{\mathrm{init}}
\newcommand{\eps}{\epsilon}
\newcommand{\ab}{\mathrm{abs}}
\newcommand{\no}{\mathrm{norm}}

\newcommand*{\norm}[1]{\left\| #1 \right\|}                             
\newcommand*{\abs}[1]{\left| #1 \right|}                                
\newcommand*{\floor}[1]{\left\lfloor #1 \right\rfloor}                  
\newcommand*{\ceil}[1]{\left\lceil #1 \right\rceil}                     
\newcommand*{\link}[1]{(\ref{#1})}                                      
\newcommand*{\distr}[2]{\left\langle #1, #2 \right\rangle}                         
\renewcommand{\max}[1]{ \mathop{\mathrm{max}}\left\{#1\right\} }        
\renewcommand{\min}[1]{ \mathop{\mathrm{min}}\left\{#1\right\} }
\newcommand*{\maxx}{\mathop{\mathrm{max}}\displaylimits}                
\newcommand*{\spann}[1]{\mathop{\mathrm{span}}\left\{#1\right\}}        
\renewcommand{\exp}[1]{ \mathop{\mathrm{exp}}\left( #1 \right) }        
\newcommand{\BIGOP}[1]{\mathop  
 {\mathchoice 
        {\raise-0.22em\hbox{\huge $#1$}} 
        {\raise-0.05em\hbox{\Large $#1$}}{\hbox{\large $#1$}}{#1}
 }}
\newcommand{\bigtimes}{\BIGOP{\times}}                                  
\DeclareMathOperator*{\esssup}{ess-sup}																	
\DeclareMathOperator{\rank}{rank}																	
\newcommand*{\id}{\mathrm{id}}                                          
\DeclareMathOperator{\inner}{int}																	
\DeclareMathOperator{\trace}{trace}																	

\newcommand{\wrt}{w.r.t.\ }
\newcommand{\ie}{i.e.\ }
\newcommand{\eg}{e.g.}

\makeatletter
\newcommand\ackname{Acknowledgements}
\if@titlepage
  \newenvironment{acknowledgements}{%
      \titlepage
      \null\vfil
      \@beginparpenalty\@lowpenalty
      \begin{center}%
        \bfseries \ackname
        \@endparpenalty\@M
      \end{center}}%
     {\par\vfil\null\endtitlepage}
\else
  \newenvironment{acknowledgements}{%
      \if@twocolumn
        \section*{\abstractname}%
      \else
        \small
        \begin{center}%
          {\bfseries \ackname\vspace{-.5em}\vspace{\z@}}%
        \end{center}%
        \quotation
      \fi}
      {\if@twocolumn\else\endquotation\fi}
\fi
\makeatother

\begin{document}
	\kopffussleer
\titlehead{
		\begin{center}
				\includegraphics[width=4.7cm, height=4.7cm,clip]{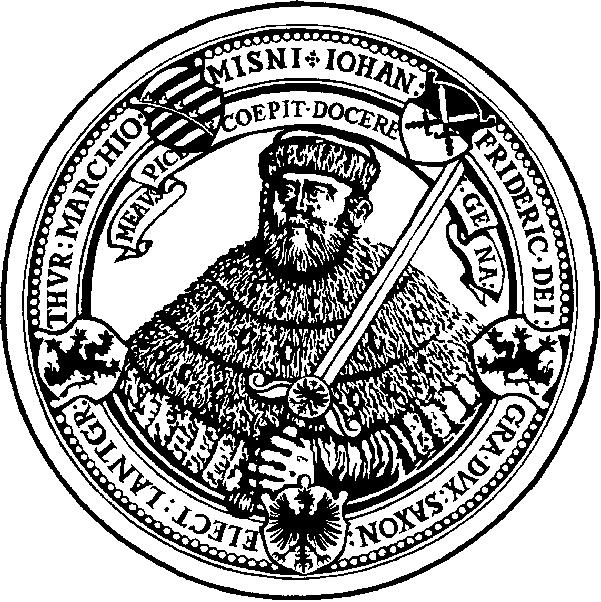}
		\end{center}
}

\title{
		Several Approaches to Break the\\
		Curse of Dimensionality
}

\author{
		\textbf{Dissertation}\\
		\normalsize{\textsl{zur Erlangung des akademischen Grades}}\\
		\normalsize{\textsl{doctor rerum naturalium (Dr. rer. nat.)}}\\
		\\[2.2 cm]
		\normalsize{vorgelegt dem Rat der}\\
		\normalsize{Fakult\"at f\"ur Mathematik und Informatik}\\
		\normalsize{der Friedrich-Schiller-Universit\"at Jena}\\
		\\[1 cm]
}

\date{
	\normalsize{von Dipl.-Math. Markus Weimar}\\
	\normalsize{geboren am 28. Februar 1986 in Weimar}
}

\lowertitleback{
	\textbf{Gutachter:}
	\vspace{5 mm}
	\begin{enumerate}[label=\arabic*.), leftmargin=*]
		\item Prof. Dr. Erich Novak (Jena) \, -- \, \emph{summa cum laude}
		\item Prof. Dr. Aicke Hinrichs (Rostock) \, -- \, \emph{summa cum laude}
		\item Prof. Dr. Henryk Wo\'zniakowski (New York, Warschau) \, -- \, \emph{summa cum laude}
	\end{enumerate}
	\vspace{8 mm}
	\textbf{Tag der \"offentlichen Verteidigung:} 13.05.2013
}

\pagenumbering{alph}
\maketitle

	\begin{acknowledgements}
I would like to express my deepest gratitude to my supervisor Professor Dr. Erich Novak for numerous hints, 
suggestions and remarks during the preparation of this work. 
Furthermore, I would like to thank all the members of the research groups ``Theoretical numerics'' 
and ``Function spaces'' in Jena for supporting me during the times of my Diploma thesis and my Ph.D. studies.
Finally, I like to thank our friends from the IBC community for many fruitful discussions at several conferences during the last years.
\end{acknowledgements}

	\frontmatter
	\kopffussTOC
	\phantomsection
	\pdfbookmark[1]{Table of contents}{table}
	\tableofcontents
	\cleardoubleplainpage

	\kopffussEN
	\phantomsection
\chapter*{Preface}
\addcontentsline{toc}{chapter}{Preface}
In modern science the efficient numerical treatment of high-dimensional problems becomes more and more important.
A fundamental insight of the theory of \emph{information-based complexity} (\emph{IBC} for short) is that the computational hardness of a problem can not be described properly only by the rate of convergence.
An impressive example that illustrates this fact was given recently by Novak and Wo\'{z}niakowski~\cite{NW09}. 
They studied a problem for which an exponential number of information operations is needed in order to reduce the initial error, although there exist algorithms which provide an arbitrary large rate of convergence.
Problems that yield this exponential dependence are said to suffer from the \emph{curse of dimensionality}.
While analyzing numerical problems it turns out that we can often vanquish this curse by exploiting additional structural properties.
The aim of this thesis is to present several approaches of this type.

A numerical problem $S$ is given by a sequence of compact linear operators~$S_d$
acting between normed spaces $\F_d$ and $\G_d$, where $d\in\N$.
In general we seek for algorithms~$A_{n,d}$ that approximate $S_d$ while using at most $n\in\N_0$ pieces of information on the input elements $f\in\F_d$.
The quality of this approximation is measured by the so-called \emph{worst case error}
\begin{gather*}
			\Delta^\wor(A_{n,d}; S_d) = \sup_{\norm{f\sep \F_d}\leq 1} \norm{S_d(f)-A_{n,d}(f) \sep \G_d}
\end{gather*}
which we try to minimize.
Problems based on tensor product structures, as well as linear algorithms that are easy to implement,
are of particular interest.
The minimal number of information operations needed to solve a given problem $S$ to within a threshold~$\eps>0$
is called \emph{information complexity}:
\begin{gather*}
		n(\eps,d;S_d) = \min{n\in\N_0 \sep \exists A_{n,d}\colon \Delta^\wor(A_{n,d}; S_d) \leq \eps }, 
		\quad \eps>0, d\in\N.
\end{gather*}
If this quantity grows exponentially fast with the dimension~$d$ then $S$ suffers from the curse of dimensionality. In the case where $n(\eps,d;S_d)$ is neither exponential in $d$, nor in $\eps^{-1}$, the problem $S$ is said to be \emph{weakly tractable}.
A special case is described by the notion of \emph{polynomial tractability} for which the information complexity needs to be bounded from above by a polynomial in $d$ and $\eps^{-1}$, \ie
\begin{gather*}
		n(\eps,d;S_d) \leq C\, \eps^{-p} \, d^q \quad \text{for some} \quad C,p>0,\, q\geq 0 \quad \text{and all} \quad \eps\in(0,1], d\in\N.
\end{gather*}
If the latter inequality is valid even for $q=0$ then $S$ is called \emph{strongly polynomially tractable}.\\

Next we present the three approaches to exploit structural properties we study in this thesis and
we briefly summarize our main complexity results.

A rather simple class of problems $S$ is given by the set of all compact linear operators between tensor products of Hilbert spaces.
Especially the complexity of tensor product problems $S_d=\bigotimes_{k=1}^d S_1 \colon H_d \nach \G_d$, induced by some operator $S_1\colon H_1\nach\G_1$, is well-understood.
It depends on the non-increasingly ordered sequence $\lambda=(\lambda_m)_{m\in\N}$ of the squares of the singular values of the underlying operator~$S_1$.
In particular, it is well-known that $S=(S_d)_{d\in\N}$ is not polynomially tractable if we have $\lambda_1 \geq 1$ and $\lambda_2>0$. Actually, we are faced with the curse of dimensionality if $\lambda_1$ is strictly larger than $1$ and $\lambda_2>0$, or if $\lambda_1\geq\lambda_2=1$; cf.~\autoref{thm:unweightedtensor_abs}.\\
A first approach to modify such a problem is to scale the inner products of the source spaces~$H_d$, $d\in\N$.
We set
\begin{gather*}
		\distr{\cdot}{\cdot}_{\F_d} = \frac{1}{s_d} \, \distr{\cdot}{\cdot}_{H_d} 
		\quad \text{for some} \quad s_d>0 \quad \text{and all} \quad d\in\N
\end{gather*}
and investigate the complexity of the problem operators $S_d$ interpreted as mappings between the Hilbert spaces~$\F_d$ and $\G_d$, $d\in\N$.
The resulting problem, scaled by factors from the sequence $s=(s_d)_{d\in\N}$, then is denoted by $S_{(s)}=(S_{d,s_d}\colon \F_d \nach \G_d)_{d\in\N}$.
We study the worst case setting with respect to the absolute error criterion and prove
\begin{thm}\label{Z_en_thm:scaled_PT}
		Using the introduced notation and assuming that $\lambda_2>0$ the following assertions are equivalent:
		\begin{enumerate}[label=(\Roman{*}), ref=\Roman{*}]
				\item \label{Cond_SPT_z_en} $S_{(s)}$ is strongly polynomially tractable.
				\item \label{Cond_PT_z_en} $S_{(s)}$ is polynomially tractable.
				\item \label{Cond_sup_z_en} There exists $\tau\in(0,\infty)$ such that $\lambda\in\l_\tau$ and $\sup_{d\in\N} s_d \norm{\lambda \sep \l_\tau}^d < \infty$.
				\item \label{Cond_limsup_z_en} There exists $\rho\in(0,\infty)$ such that $\lambda\in\l_\rho$ and $\limsup_{d\nach\infty} s_d^{1/d} < \frac{1}{\lambda_1}$.
		\end{enumerate}
		If one of these (and hence all) conditions applies then the \emph{exponent of strong
polynomial tractability} is given by
				$p^* = \inf\{2\tau \sep \tau \text{ fulfills condition \link{Cond_sup_z_en}} \}$.
\end{thm}
We refer to \autoref{thm:scaled_PT} in \autoref{sect:ScaledProbs_pt}.
It is remarkable that similar to unscaled problems polynomial tractability of the problem $S_{(s)}$ already implies 
strong polynomial tractability, despite the fact that we can choose the sequence of scaling factors $(s_d)_{d\in\N}$ completely arbitrary.\\
The less restrictive property weak tractability and the curse of dimensionality can be characterized, provided that we additionally assume a certain asymptotic behavior of the initial error $\eps_d^{\init}=\sqrt{s_d\cdot \lambda_1^d}$; see \autoref{thm:scaled_wt} in \autoref{sect:ScaledProbs_wt}.
\begin{thm}\label{Z_en_thm:scaled_WT}
		We study the scaled tensor product problem $S_{(s)} = (S_{d,s_d})_{d\in\N}$ in the worst case setting 
		\wrt the absolute error criterion and assume $\lambda_2>0$.
		Moreover, 
		\begin{itemize}
				\item let $\ln \!\left(\epsilon_d^{\mathrm{init}}\right)\notin o(d)$, as $d\nach\infty$. 
				Then we have the curse of dimensionality.
				\item let $\epsilon_d^{\mathrm{init}} \in \Theta(d^\alpha)$, as $d\nach\infty$, for some $\alpha\geq 0$.
							\begin{itemize}
									\item If $\lambda_1=\lambda_2$ then $S_{(s)}$ suffers from the curse of dimensionality.
									\item In the case $\lambda_1>\lambda_2$ the problem $S_{(s)}$ is weakly tractable if and only if
											  $\lambda_n \in o\!\left(\ln^{-2(1+\alpha)} n \right)$, as $n\nach\infty$.
							\end{itemize}
				\item let $\epsilon_d^{\mathrm{init}} \nach 0$, as $d$ approaches infinity.
				Then we are never faced with the curse of dimensionality.
				Furthermore, $S_{(s)}$ is weakly tractable if and only if
							\begin{enumerate}[label=(\roman{*}), ref=\roman{*}]
										\item $\lambda_1=\lambda_2$ and $\lambda_n \in o\!\left(\ln^{-2} n\right)$, as $n\nach\infty$, and $\epsilon_d^{\mathrm{init}}\in o(1/d)$, as $d\nach\infty$, or
										\item $\lambda_1>\lambda_2$ and $\lambda_n \in o\!\left(\ln^{-2} n\right)$, as $n\nach\infty$.
							\end{enumerate}
		\end{itemize}
\end{thm}
Here the parameter $\alpha$ that controls the polynomial growth of the initial error is of particular interest.
In the case where $\lambda_1>\lambda_2$ it directly enters the condition for the characterization of weak tractability.
Moreover, the condition $\epsilon_d^{\mathrm{init}}\in o(1/d)$, as $d\nach\infty$, in the third part of the theorem is quite surprising.
Since for unscaled problems the initial error only can grow or decline exponentially, or it equals one in any dimension, these phenomena can not occur in the classical theory, \ie in the case where $s_d=1$ for all $d\in\N$.

Another approach to overcome the curse of dimensionality is related to problems defined between function spaces.
Here we can make use of some a priori given knowledge about the influence of certain (groups of) variables on the functions in the source space, in order to approximate them efficiently.
To this end, we endow these spaces with weighted norms.
During the last years especially problems on function spaces that yield a Hilbert space structure, equipped with so-called product weights, attracted a lot of attention.
Problems where the source and/or target spaces are allowed to be more general Banach spaces
were studied less frequently within the IBC community.\\
Among other things, in this thesis we consider the uniform approximation problem
\begin{gather*}
		\mathrm{App} = \left(\mathrm{App}_d \colon F_d^\gamma \nach \L_\infty([0,1]^d)\right)_{d\in\N} 
		\quad \text{with} \quad \mathrm{App}_d(f)=f
		\quad \text{for} \quad d\in\N
\end{gather*}
defined on certain classes of smooth functions 
\begin{gather*}
		F_d^\gamma = \left\{f\colon[0,1]^d\nach\R \sep f\in C^\infty([0,1]^d) \text{ with } \norm{f \sep F_d^\gamma}<\infty\right\}
\end{gather*}
which are endowed with the weighted norms
\begin{gather*}
		\norm{f \sep F_d^\gamma} = \sup_{\bm{\alpha}\in\N_0^d} \frac{1}{\gamma_{\bm{\alpha}}} \norm{D^{\bm{\alpha}}f\sep \L_\infty([0,1]^d)}.
\end{gather*}
Here for every $\bm{\alpha}\in\N_0^d$, $d\in\N$, the \emph{product weights} $\gamma_{\bm{\alpha}}=\prod_{j=1}^d (\gamma_{d,j})^{\alpha_j}$ are constructed out of a uniformly bounded sequence $C_\gamma\geq \gamma_{d,1}\geq \ldots \geq \gamma_{d,d}>0$ of so-called generator weights.
It turns out that the complexity of the approximation problem depends on certain summability properties of these generators which also play an important role when dealing with problems on product-weighted Hilbert spaces.
We define the quantities
\begin{align*}
		&p(\gamma)=\inf\left\{\kappa>0\sep \limsup_{d\nach\infty} \sum_{j=1}^d (\gamma_{d,j})^\kappa<\infty\right\}, \quad \text{as well as} \\
		&q(\gamma)=\inf\left\{\kappa>0\sep \limsup_{d\nach\infty} \sum_{j=1}^d (\gamma_{d,j})^\kappa / \ln(d+1)<\infty\right\},
\end{align*}
and prove the following
\begin{thm}\label{Z_en_thm:weighted_PT}
		For the worst case setting \wrt the absolute error criterion we have:
		\begin{itemize}
				\item If the problem $\mathrm{App}$ is polynomially tractable then $q(\gamma) \leq 1$.
				Moreover strong polynomial tractability implies the condition $p(\gamma) \leq 1$.
				\item If $q(\gamma)<1$ or even $p(\gamma)<1$ then $\mathrm{App}$ is polynomially tractable
				or even strongly polynomially tractable, respectively.
		\end{itemize}
\end{thm}
In fact, we show these necessary and sufficient criteria for a whole scale of weighted Banach spaces
that fulfill certain embedding conditions; see \autoref{Thm_Necessary} and \autoref{Theorem_Sufficient} 
for details.
The source space~$F_d^\gamma$ as defined above appears as a special case within this scale.
On the other hand, it generalizes a space considered by Novak und Wo\'{z}niakowski \cite{NW09}.
In addition, we prove that the sufficient conditions $q(\gamma)<1$ and $p(\gamma)<1$
are also necessary for (strong) polynomial tractability of the $\L_\infty$-approximation problem 
defined on a certain unanchored Sobolev space~$\Hi_d^\gamma$; cf. \autoref{thm:tract_Sobolev}.\\
Weak tractability and the curse of dimensionality can be characterized as follows.
\begin{thm}\label{Z_en_thm:weighted_WT}
				For $\mathrm{App}=(\mathrm{App}_d)_{d\in\N}$ the following assertions are equivalent:
        \begin{enumerate}[label=(\roman*), ref=\roman{*}]
                \item The problem is weakly tractable. \label{Cond_WT_en}
                \item The curse of dimensionality is not present.\label{Cond_Curse_en}
                \item For all $\kappa > 0$ we have
                      $\lim_{d\nach \infty}\limits \frac{1}{d} 
                      \sum_{j=1}^d \left( \gamma_{d,j} \right)^\kappa = 0$.
                \item There exists $\kappa \in (0,1)$ such that
											$\lim_{d\nach \infty}\limits \frac{1}{d} 
                      \sum_{j=1}^d \left( \gamma_{d,j} \right)^\kappa = 0$.\label{Cond_Limit_en}
				\end{enumerate}
\end{thm}
This immediately follows from our \autoref{Theorem_Equivalence} in which we discuss a more general situation.
Note that the implication \link{Cond_Curse_en} $\Rightarrow$ \link{Cond_WT_en} is not trivial.
Moreover, the condition~\link{Cond_Limit_en} is typical for problems defined on Hilbert spaces equipped with product weights.

Finally, our third approach to vanquish the curse is based on exploiting certain symmetry properties of the elements in the source space.
For this purpose we again consider tensor product problems $S=(S_d\colon H_d \nach \G_d)_{d\in\N}$ 
between Hilbert spaces. 
But now we restrict them to suitable subspaces which solely consist of (anti)symmetric elements.
We illustrate this concept by considering the special case of problems defined between function spaces.\\
For $d\in\N$ and $I\subseteq\{1,\ldots,d\}$ let $\S_I$ denote the collection of all permutations~$\pi$ of the coordinate set $\{1,\ldots,d\}$ that leave the complement $I^c=\{1,\ldots,d\}\setminus I$ of~$I$ fixed.
Then a real-valued function $f\in H_d=H_1\otimes\ldots\otimes H_1$ on $[0,1]^d$ is called \emph{$I$-symmetric} if
\begin{gather*}
				f(\bm{x}) = f(\bm{\pi(x)})
				\quad \text{for every} \quad \bm{x}\in [0,1]^d \quad \text{and all} \quad \pi\in\S_I.
\end{gather*}
In contrast, $f$ is called \emph{$I$-antisymmetric} if the equality $f(\bm{x}) = (-1)^{\abs{\pi}} f(\bm{\pi(x)})$ holds true for every $\bm{x}$ and $\pi$.
In what follows we denote the corresponding linear subspaces of~$H_d$ that exclusively contain symmetric or antisymmetric functions by $\SI_I(H_d)$ 
and $\AI_I(H_d)$, respectively.
Particularly antisymmetric functions, \ie functions that change their sign when we exchange the variables $x_i$ and $x_j$, $i,j\in I$,
turned out be of some practical interest; see, e.g., \autoref{sect:wave}.
For the restriction of a given tensor product problem~$S=(S_d\colon H_d\nach\G_d)_{d\in\N}$ to the subspaces~$P_{I_d}(H_d)$, $d\in\N$, 
we write~$S_{I}=(S_{d,I_d})_{d\in\N}$.
Here the kind of symmetry $P\in\{\SI,\AI\}$, as well as a sequence $(I_d)_{d\in\N}$ of subsets of the coordinates, is assumed to be fixed.\\
Since for $d\in\N$ the operators $S_{d,I_d}$ can be interpreted as a composition of~$S_d$ with suitable orthogonal projections, there exists a close relation of the singular values of~$S_d$ with the corresponding singular values of the restricted operators~$S_{d,I_d}$.
These numbers essentially determine the minimal worst case error of the problem~$S_I$.
This knowledge furthermore allows the construction of an optimal (linear) algorithm that realizes this error; cf. \autoref{theo:opt_sym_algo}.\\
Consequently, we can conclude assertions that relate the information complexity of~$S_I$ to the squares of the singular values of~$S_1$ and to the number of (anti)symmetry conditions we impose.
For the sake of simplicity we restrict ourselves again to the absolute error criterion and start by discussing the case of symmetric problems; 
see \autoref{thm:tract_sym_abs}.
\begin{thm}[Polynomial tractability, $P=\SI$]\label{Z_en_thm:sym_PT}
			Let $S_1 \colon H_1 \nach \G_1$ denote a compact linear operator between Hilbert spaces
			and let $\lambda=(\lambda_m)_{m\in \N}$ be the sequence of eigenvalues of $W_1={S_1}^{\!\dagger} S_1$
			\wrt a non-increasing ordering.
			Assume $\lambda_2>0$ and for $d>1$ let $\leer \neq I_d \subseteq\{1,\ldots,d\}$ be fixed.
			We consider the restriction $S_I=(S_{d,I_d})_{d\in\N}$ of the tensor product problem
			$S=(S_d\colon H_d\nach\G_d)_{d\in\N}$ to the $I_d$-symmetric subspaces 
			$\SI_{I_d}(H_d)\subset H_d$, $d\in\N$.
			Then $S_I$ is strongly polynomially tractable if and only if $\lambda \in \l_\tau$ 
			for some $\tau\in(0,\infty)$ and
			\begin{itemize}
						\item $\lambda_1<1$, or
						\item $1=\lambda_1>\lambda_2$ and $(d-\#I_d) \in \0(1)$, as $d\nach\infty$.
			\end{itemize}
			Moreover, provided that $\lambda_1 \leq 1$ the problem is polynomially tractable if and only
			if $\lambda \in \l_\tau$ for some $\tau\in(0,\infty)$ and
			\begin{itemize}
						\item $\lambda_1<1$, or
						\item $\lambda_1=1$ and $(d-\#I_d) \in \0(\ln d)$, as $d\nach\infty$.
			\end{itemize}
\end{thm}
It remains the open problem to find sufficient conditions for polynomial tractability 
in the case $\lambda_1>1$.
However, our results show that the conditions $\lambda \in \l_\tau$ and $(d-\#I_d) \in \0(\ln d)$
are necessary in this situation, too.
In conclusion, we see that imposing sufficiently many additional symmetry assumptions, we can avoid the curse of dimensionality which we are faced with \eg\, in the case $\lambda_1=\lambda_2=1$; see also \autoref{thm:unweightedtensor_abs}.\\
The complexity analysis of antisymmetric problems is more demanding.
On the other hand, it turns out that here even weaker conditions are sufficient to conclude 
polynomial tractability and thus to vanquish the curse.
One of the reasons is the structure of the initial error which is more complicated in this case.
Similar to \autoref{thm_asy_abs} in \autoref{sect:complexity_antisym} we can
summarize the main results on the complexity as follows:
\begin{thm}[Polynomial tractability, $P=\AI$]\label{Z_en_thm:asy_PT}
			Let $S_1 \colon H_1 \nach \G_1$ denote a compact linear operator between Hilbert spaces
			and let $\lambda=(\lambda_m)_{m\in \N}$ be the sequence of eigenvalues of $W_1={S_1}^{\!\dagger} S_1$
			\wrt a non-increasing ordering.
			Assume $\lambda_2>0$ and for $d>1$ let $\leer \neq I_d \subseteq\{1,\ldots,d\}$ be fixed.
			We consider the restriction $S_I=(S_{d,I_d})_{d\in\N}$ of the tensor product problem
			$S=(S_d\colon H_d\nach\G_d)_{d\in\N}$ to the $I_d$-antisymmetric subspaces 
			$\AI_{I_d}(H_d)\subset H_d$, $d\in\N$.
			Then for the case $\lambda_1 < 1$ the following statements are equivalent:
			\begin{itemize}
						\item $S_I$ is strongly polynomially tractable.
						\item $S_I$ is polynomially tractable.
						\item There exists a constant $\tau \in (0,\infty)$ such that $\lambda \in \l_\tau$.
			\end{itemize}
			Moreover, the same equivalences hold true if $\lambda_1\geq 1$ and 
			the number of antisymmetric coordinates~$\#I_d$ grows linearly with the dimension~$d$.
\end{thm}
Clearly, these assertions show that antisymmetric tensor product problems are significantly easier than their
symmetric counterparts which on their part possess a lower information complexity than entire tensor product problems, as long as we impose enough (anti)symmetry conditions.
On the other hand, there exist quite natural examples which show that even fully antisymmetric problems
are not necessarily trivial or polynomially tractable, in general. For details we refer to \autoref{sect:toy_ex}.

Let us briefly explain the structure of the present thesis.
In the first chapter we settle some notational conventions and we define the abstract problem we are faced with in IBC.
Furthermore, here we introduce the used cost model and recall the formal definitions of several complexity categories.

In \autoref{chapt:prop} we discuss special classes of numerical problems, as well as elementary tools
that we need to handle them.
In particular, here we give a detailed introduction to the singular value decomposition (SVD) of
compact operators between Hilbert spaces.
In many cases it builds the basis for the construction of optimal algorithms. 
Hence it is of fundamental importance for the rest of our work.
In addition, we discuss tensor product structures in Hilbert spaces and recall some well-known complexity assertions for problems related to this concept.
Finally, we briefly introduce so-called reproducing kernel Hilbert spaces (RKHSs)
and collect some of their properties.

In the first two sections of the third chapter we derive the characterizations of the different types
of tractability of scaled tensor product problems between Hilbert spaces we presented in \autoref{Z_en_thm:scaled_PT} and \autoref{Z_en_thm:scaled_WT} above.
Moreover, from them we conclude a complete characterization for the normalized error criterion in \autoref{sect:ScaledProbs_normed}.
It turns out that here the scaling factors become irrelevant.
Apart from formulas of the optimal algorithm and its worst case error, we additionally show that these new assertions generalize the known theory in a quite natural way.
We conclude this chapter by the application of the obtained results to two simple examples.

\autoref{chapt:weighted} then deals with problems on function spaces endowed with weighted norms.
Here we explain the concept of weighted spaces in full detail and illustrate it using the example of some unanchored Sobolev~$\Hi_d^\gamma$ space equipped with product weights.
For the uniform approximation problem on this space we present an algorithm~$A_{n,d}^*$ that satisfies suitable upper error bounds.
Together with corresponding lower bounds, which we prove for spaces of low-degree polynomials,
the application of simple embedding arguments then leads us to complexity assertions for a whole scale of product-weighted Banach spaces.
In particular, these assertions cover the results for the space $F_d^\gamma$ stated in \autoref{Z_en_thm:weighted_PT} and \autoref{Z_en_thm:weighted_WT}.
Finally, the last section within this chapter, \autoref{sect:final_remarks}, presents some generalizations of the techniques developed before.
Among other things, here we show how to handle $\L_p$-approximation problems, where $1\leq p<\infty$, defined on suitable spaces.
Moreover, we show that the algorithm~$A_{n,d}^*$ is essentially optimal for $\L_\infty$-approximation
on $\Hi_d^\gamma$.
For the proof we make use of arguments due to Kuo, Wasilkowski and Wo\'{z}niakowski~\cite{KWW08} that relate
the uniform approximation problem in the worst case setting, defined on quite general reproducing kernel Hilbert spaces, to a certain average case $\L_2$-approximation problem.\\
Some of the results presented in this chapter were already published in~\cite{W12}.
However, we were able to partially improve these assertions.
We will explicitly emphasize generalizations and new results at the appropriate points.

Finally, \autoref{chapt:antisym} is devoted to problems with (anti)symmetry conditions.
We start with the definition of (anti)symmetry in Hilbert function spaces.
In particular, we focus our attention to tensor product structures and conclude fundamental properties
of the respective projections and subspaces.
At the end of \autoref{sect:basic_antisym_def} we use these properties in order to generalize
the notion of (anti)symmetry to tensor products of abstract Hilbert spaces.
Afterwards we define (anti)symmetric numerical problems $S_{I}=(S_{d,I_d})_{d\in\N}$ by the restriction of a given tensor product problem $S=(S_d)_{d\in\N}$ to the subspaces of (anti)symmetric elements in the source spaces.
We prove the commutativity of the operators $S_d$ with certain projections and conclude formulas for optimal algorithms and their worst case errors.
This in hand, in \autoref{sect:complexity_antisym} we then discuss the complexity of (anti)symmetric numerical problems.
We distinguish between symmetric and antisymmetric problems, as well as between the absolute and the normalized error criterion.
Here we particularly derive the proofs of \autoref{Z_en_thm:sym_PT} and \autoref{Z_en_thm:asy_PT}.
The chapter is concluded by a section which is devoted to several applications.
On the one hand, we use simple examples to show that the additional knowledge
about (anti)symmetry conditions can dramatically reduce the information complexity.
On the other hand, we also discuss more advanced problems that play a role in computational practice.
To this end, we illustrate the application of this new theory to the approximation problem of so-called wavefunctions that arise in certain models of quantum mechanics and theoretical chemistry.\\
A major part of the results proven in this chapter was published in~\cite{W12b}.
However, at some points we use different proof techniques that allow slight generalizations.

Within every chapter formulas are numbered consecutively.
Moreover, we use a sequential numbering for lemmata, remarks, examples, propositions, and theorems;
\eg\, \autoref{prop:scaled_tensor_eigenpairs}	is followed by \autoref{thm:scaled_PT} and \autoref{lemma:n}.
The symbols $\square$ and $\blacksquare$ are used to indicate the end of remarks and examples, as well as of proofs, respectively.
	\cleardoubleplainpage

	\mainmatter
  \kopffusskomplett
\chapter{Preliminaries}\label{chapt:1}
Apart from introducing some notational conventions, the aim of this first chapter is to define the general objects of interest in \emph{information based complexity} (IBC). 
We give an abstract formulation of the general problem in \autoref{sect:General}.
Afterwards we introduce some classes of algorithms and discuss the used cost model in \autoref{sect:Algos}.
Finally, in \autoref{sect:TractDef}, we recall the notions of tractability, as well as the definition of the curse of dimensionality.

\section{Basic notation}
As usual we denote by $\N$ the natural numbers and $\N_0=\N \cup \{0\}$ are all non-negative integers.
Moreover, $\R$ denotes the real line and $\R^d$ ($d\in\N$) is the collection of all points $\bm{x}=(x_1,\ldots,x_d)$ in the $d$-dimensional Euclidean space.
Given a real number $y>0$ the symbol $\floor{y}$ means the largest $n\in\N_0$ such that $n\leq y$ and 
we define~$\ceil{y}$ to be the smallest number $m\in\N$ with $y\leq m$.
The value of the Riemann zeta function at some $z>1$ is denoted by $\zeta(z)=\sum_{n=1}^\infty n^{-z}$.

If $\bm{k}=(k_1,\ldots,k_d)\in\N_0^d$ is a multi-index then $\abs{\bm{k}}=\sum_{i=1}^d k_i$ stands for its length. 
Furthermore, we use the common notation $\bm{x}^{\bm{k}}=x_1^{k_1}\cdot\ldots\cdot x_d^{k_d}$.
For $\bm\alpha\in\N_0^d$ partial derivatives of $d$-variate functions are denoted by $D^{\bm\alpha}$, \ie
\begin{gather*}
D^{\bm\alpha}f = \frac{\partial^{\abs{\bm\alpha}}f}{\partial x_1^{\alpha_1}\ldots\partial x_d^{\alpha_d}}.
\end{gather*}
Derivatives of univariate functions $g$ are indicated as $g'$, $g''$, \ldots, $g^{(n)}$.
For real numbers $a<b$ half-open intervals are symbolized by $[a,b)$, and $[a,b]^d$ stands for the Cartesian product $\bigtimes_{i=1}^d [a,b]=[a,b]\times\ldots\times [a,b]$.
If $\bm{x}\in\R^d$ belongs to $[a,b]^d$ then the value of the characteristic (or indicator) function $\chi_{[a,b]^d}(\bm{x})$ of this set equals~$1$. Otherwise we define $\chi_{[a,b]^d}(\bm{x})=0$.
Similarly the Kronecker delta function~$\delta_{i,j}$ is one if the two objects $i$ and $j$ coincide
and $\delta_{i,j}=0$ when they differ from each other.

In what follows we assume that the reader has a fundamental 
knowledge in measure theory and probability theory as it can be found, \eg, in the textbooks of Bauer~\cite{B01,B96}.
We write $\uplambda^d$ for the Lebesgue measure in $\R^d$ and use the symbols $\Pr$ and $\E$ for probabilities and expectations, respectively.
We use $\# I$ to denote the cardinality of a finite set~$I$.
As usual the sum over an empty index set~$I$ is to be interpreted as zero whereas empty products equal $1$ by definition.

Throughout the whole thesis we assume that 
the reader is familiar with the basic concepts in 
functional analysis such as, \eg, 
complete normed spaces (Banach spaces), 
weak derivatives or tensor products.
For a comprehensive introduction we refer to 
the textbooks of Triebel~\cite{T92} and Yosida~\cite{Y80}.
The norm in some space~$F$ is denoted by $\norm{\cdot\sep F}$.
We write $B_r(F)=\{f\in F\sep \norm{f\sep F} \leq r\}$ for 
centered, closed balls of radius $r \geq 0$ in normed spaces~$F$.
Moreover, we use $\partial M$ for the boundary and $\inner(M)$ for the interior of a set~$M$.
Consequently $\B(F) = B_1(F) = \inner(B_1(F)) \cup \partial B_1(F)$ denotes the unit ball in~$F$.
For the class of all bounded linear operators between 
normed spaces~$F$ and~$G$ we write $\LO(F,G)$.
The subset of all compact operators is denoted by $\K(F,G)$.
We say a space $F$ is (continuously) embedded into another space $G$ 
with norm $C$ if the operator norm of $\id \colon F \nach G$, $f\mapsto \id(f)=f$, equals $C \in[0,\infty)$.
In this case we write $F\hookrightarrow G$ and $\norm{\id \sep \LO(F,G)}=C$.
We use the symbol $\distr{\cdot}{\cdot}_H$ for the inner product 
in the case of a Hilbert spaces~$H$.
Moreover, we write $M^\bot$ for the orthogonal complement 
of some linear subspace $M\subset H$
and we use $\oplus$ to denote the orthogonal sum 
with respect to $\distr{\cdot}{\cdot}_H$.

If $(\mathcal{X}, \a,\mu)$ is an arbitrary measure space and
$0< p \leq \infty$ then we use the symbol $\L_p(\mathcal{X},\a,\mu)$ for the classical Lebesgue spaces.
Hence, if $p<\infty$ then we deal with the set of (equivalence classes of) $\mu$-measurable functions 
$f\colon \mathcal{X}\nach\R$ for which the norm\footnote{Actually, in the case $0<p<1$ the given formula only provides
a quasi-norm, \ie then we need an additional constant $k>1$ for the triangle inequality.
Since this does not play any role in our applications we do not emphasize this difference in what follows.}
\begin{gather*}
			\norm{f\sep \L_{p}(\mathcal{X},\a,\mu)} =	\left( \int_{\mathcal{X}} \abs{f(x)}^p \,\mathrm{d}\mu(x) \right)^{1/p}
\end{gather*}
is finite. 
Moreover, $\L_\infty(\mathcal{X},\a,\mu)$ is the space (of classes) of $\mu$-essentially bounded functions
on $\mathcal{X}$, equipped with the norm
\begin{gather*}
			\norm{f\sep \L_{\infty}(\mathcal{X},\a,\mu)} =	\esssup_{x\in \mathcal{X}} \abs{f(x)}.
\end{gather*}
As usual two functions are identified if they coincide $\mu$-almost everywhere on~$\mathcal{X}$ 
and we do not distinguish between functions and their equivalence classes.
The following special cases are of particular interest for us.

For a Borel measurable subset $\mathcal{X}=\Omega\subset \R^d$, the Borel sigma algebra $\a=\Sigma=\Sigma(\Omega)$ 
and $\mu=\uplambda^d$ we use the shorthand $\L_p(\Omega) = \L_p(\Omega,\Sigma,\uplambda^d)$.
If in this definition $\mu$ does not equal the Lebesgue measure, but is absolute continuous \wrt
$\uplambda^d$, and if $\rho = \mathrm{d}\mu/\mathrm{d}\uplambda^d$ 
describes a probability density function
that is strictly positive ($\uplambda^d$-{a.e.}) on $\Omega$, then we write $\L_p^\rho(\Omega)$.
On the other hand, for a discrete measure space $(\Gamma,\mathfrak{b},\nu)$ on some set $\Gamma$ 
with $\nu(\{i\})=1$ for each $i\in\Gamma$ we write 
$\l_p(\Gamma) = \L_p(\Gamma,\mathfrak{b},\nu)$ and we abbreviate the notation to $\l_p$ if $\Gamma=\N$.
Keep in mind that in this case the norms simplify to
\begin{gather*}
		\norm{\lambda \sep \l_p} = \begin{cases}
					\left( \sum_{m=1}^\infty\limits \abs{\lambda_m}^p \right)^{1/p}, & \text{if } 0< p < \infty,\\
					\sup_{m\in\N}\limits \abs{\lambda_m}, &\text{if } p=\infty,
		\end{cases}
\end{gather*}
where $\lambda=(\lambda_m)_{m\in\N}$ is any real-valued sequence such that the above norm is finite.

Finally, we make use of the Bachmann-Landau notation of asymptotic growth rates.
That is, for real-valued functions $f$ and $g$ defined on some subset of the real line we write $f(x)\in\0(g(x))$, as $x\nach a$,
if there exists a universal constant $M>0$ such that the estimate
\begin{gather*}
		\abs{f(x)} \leq M \, \abs{g(x)}
\end{gather*}
holds for all $x$ sufficiently close to the point $a$.
If $g$ is non-zero (at least in the neighborhood of $a$) then this definition equivalently reads
\begin{gather*}
		\limsup_{x\nach a} \abs{\frac{f(x)}{g(x)}}<\infty.
\end{gather*}
If we have $f(x)\in\0(g(x))$ and simultaneously $g(x)\in\0(f(x))$, as $x\nach a$, then we write 
$f(x)\in\Theta(g(x))$, $x\nach a$.
Moreover, we say that $f(x)\in o(g(x))$, as $x\nach a$, if for any $\delta>0$ there exists a neighborhood $U$ of $a$ such that
\begin{gather*}
		\abs{f(x)} \leq \delta \, \abs{g(x)}
\end{gather*}
for all $x \in U$.
Again this property can be reformulated for non-vanishing $g$.
In this case we have
\begin{gather*}
		\lim_{x\nach a} \abs{\frac{f(x)}{g(x)}}=0.
\end{gather*}
All these three notations will be used especially for sequences $(f_d)_{d\in\N}$
(interpreted as special classes of functions), where we have $a=\infty$.

\section{General problem}\label{sect:General}
In numerous applications from physics, chemistry, finance, economics, and computer
science we are faced with very high dimensional continuous problems 
which can almost never be solved analytically.
Therefore we search for algorithms which approximate the unknown solutions numerically
to within a threshold $\epsilon>0$.

In general, such a problem is given by a non-trivial \emph{solution operator}
\begin{gather}\label{problemS}
		S \colon \widetilde{\F} \nach \G,
\end{gather}
mapping a \emph{problem element} $f$ 
out of a subset~$\widetilde{\F}$ of some normed space~$\F$ 
onto its \emph{solution}~$S(f)$ in some (other) \emph{target space}~$\G$.
Often, but not always, $\widetilde{\F}$ is assumed to be the 
unit ball~$\B(\F)$ in some Banach space~$\F$
of multivariate functions $f\colon \Omega_d\subset\R^d\nach\R$. 
For the domain of definition $\Omega_d$ usually the unit cube $[0,1]^d$ is taken.
Since the dependence on $d$ will play a crucial role in this thesis
we concentrate on whole sequences $S=(S_d)_{d\in\N}$ of solution operators,
where every
\begin{gather}\label{problemSd}
		S_d \colon \widetilde{\F}_d \nach \G_d, \quad d\in\N,
\end{gather}
is of the form \link{problemS}.

Typically, $\widetilde{\F}_d$ is an infinite dimensional 
subset of the \emph{source space}~$\F_d$ 
and thus we cannot input $f\in\widetilde{\F}_d$ directly into the computer. 
Instead we assume that the input for our algorithms $A$
consists of finitely many cleverly chosen pieces of
information which hopefully describe $f$ as well as possible.
In \autoref{sect:Algos} we define different kinds of 
information operations which lead us to
different classes $\A_d$ of algorithms.
For now assume $A_d\colon \widetilde{\F}_d \nach \G_d$
to be a fixed element in some class $\A_d$.

The \emph{local error} $\Delta_{\mathrm{loc}}(f;A_d,S_d)$ of a given algorithm~$A_d\in\A_d$ applied to 
a problem element $f\in\widetilde{\F}_d$ is defined as the difference
of the exact solution~$S_d(f)$ and the approximate solution~$A_d(f)$, 
measured in the norm of the target space $\G_d$, \ie
\begin{gather*}
			\Delta_{\mathrm{loc}}(f;A_d,S_d) = \norm{S_d(f) - A_d(f) \sep \G_d}.
\end{gather*}
The latter definition in hand, there are several ways to quantify the quality of $A_d$.

In the \emph{worst case setting} this is done in terms of the 
maximal local error of the algorithm among all possible inputs $f\in\widetilde{\F}_d$.
Hence, by 
\begin{gather*}
			\Delta^{\mathrm{wor}} \!\left(A_d; S_d\colon \widetilde{\F}_d\nach\G_d \right) 
			= \sup_{f \in \widetilde{\F}_d}	\Delta_{\mathrm{loc}}(f;A_d,S_d)
\end{gather*}
we define the \emph{worst case error} of the algorithm~$A_d$
for the problem $S_d\colon \widetilde{\F}_d\nach\G_d$.
On the other hand, sometimes it is useful to measure the average performance of
a given algorithm on the input set $\widetilde{\F}_d$.
This corresponds to the so-called \emph{average case setting}.
Here we need to assume in addition that $\widetilde{\F}_d$ is equipped with a probability
measure $\mu_d$. 
The term
\begin{gather*}
			\Delta^{\mathrm{avg}} \!\left(A_d; S_d\colon \widetilde{\F}_d\nach\G_d \right) 
			= \left( \int_{\widetilde{\F}_d} \Delta_{\mathrm{loc}}(f;A_d,S_d)^2 \,\mathrm{d}\mu_d(f) \right)^{1/2}
\end{gather*}
then denotes the \emph{average case error} of $A_d$.\footnote{In fact, $\mu_d$ is defined on the Borel sets of $\widetilde{\F}_d$ and we need to claim $\Delta_{\mathrm{loc}}(\,\cdot\,;A_d,S_d)$ to be a measurable function, but these are only formal issues. See, \eg, \cite[p. 129]{NW08} for further details.}
Since the worst case setting seems to be much more important 
we will mainly deal with worst case errors in what follows.
However, for some problems there exist close relations to the average case setting.
One such example will be presented in \autoref{sect:opt_uniform_algo}.
For the sake of completeness we stress the point that there exist even more settings which are subject to current research. 
To this end, we mention the \emph{probabilistic} and the \emph{randomized setting} 
and refer to \cite[Chapter 3.2]{NW08} for an extensive discussion.

In numerical analysis one major assumption states that information is expensive.
Therefore we are interested in algorithms which solve a given problem within a tolerance~$\epsilon$ while
using as few as possible pieces of information on the inputs.
This property can be captured by the concept of the \emph{$n$th minimal error}
\begin{gather*}
			e^{\mathrm{sett}}\!\left(n,d; S_d\colon \widetilde{\F}_d\nach\G_d \right) 
			= \inf_{A_{n,d}\in\A_d^n} \Delta^{\mathrm{sett}}\!\left(A_{n,d}; S_d\colon \widetilde{\F}_d\nach\G_d \right)
\end{gather*}
for $\mathrm{sett}\in\{\mathrm{wor},\mathrm{avg}\}$, $d\in\N$, and $n\in\N_0$, where the infimum is taken over all algorithms in the class 
\begin{gather*}
			\A_{d}^n=\{A \in \A_d \sep A \text{ uses at most } n \text{ information operations on the input} \}.
\end{gather*}
Consequently, the \emph{initial error}
\begin{gather*}
			\epsilon_d^{\rm{init},\rm{sett}} 
			= e^{\mathrm{sett}}\!\left(0,d; S_d\colon \widetilde{\F}_d\nach\G_d\right),
			\quad d\in\N,
\end{gather*}
describes the smallest error we can achieve 
without using any information on the input in a given setting $\mathrm{sett}\in\{\mathrm{wor},\mathrm{avg}\}$.
We will see in \autoref{sect:Linearity} that 
under mild assumptions this initial error 
can be attained by the zero algorithm, 
\ie by $A_{0,d}\equiv 0 \in \G_d$.

If there is no danger of confusion we abbreviate the above notations and simply write
$\Delta^{\mathrm{sett}}(A_d; S_d)$ and $e^{\mathrm{sett}}(n,d;S_d)$, where $\mathrm{sett}$ is an element of $\{\mathrm{wor},\mathrm{avg}\}$, or even only $\Delta(A_d)$ and $e(n,d)$, respectively.
Moreover, in \autoref{chapt:weighted} and \autoref{chapt:antisym} it seems to be useful to stress especially the source spaces $\F_d$
the problem elements come from rather than the operator $S_d$.
There we slightly abuse notation and write $e^{\mathrm{wor}}(n,d;\F_d)$ instead of
$e^{\mathrm{wor}}(n,d;S_d\colon \B(\F_d)\nach\G_d)$.

The main goal in the classical theory is to find sharp bounds on the $n$th minimal error in terms
of the amount of information operations. 
In fact, there is a huge literature where the existence of constants 
$c_d,C_d>0$ and $p_{d},P_{d}>0$ was proven such that
estimates of the type
\begin{gather*}
			\frac{1}{c_d} \cdot n^{-p_d} \leq e(n,d) \leq C_d \cdot n^{-P_d} \quad \text{for all} \quad n\in\N
\end{gather*}
hold for certain problems $S$ in a given setting.\footnote{Actually, in many cases these estimates hold modulo $\log n$ to some power which usually depends linearly on $d$. For simplicity we omit these factors because they are not crucial for the following argument.}
Back then, the respective researchers did not pay much attention to the involved constants $c_d$ and $C_d$. 
These numbers can be arbitrary large and in some cases their dependence on~$d$ 
is completely unknown.
Instead the attention was focused on the so-called \emph{rate} (or \emph{order}) 
\emph{of convergence}, \ie on proofs which yield $p_d=P_d$.
Often this rate tends to zero as $d$ approaches infinity.
Therefore these bounds are not meaningful at all for large $d$.
Thus, usually the parameter $d$ was assumed to be a fixed (and reasonably small) constant in this approach.
Since we also want to work in huge dimensions a more careful error analysis is needed.

\section{Algorithms and cost model}\label{sect:Algos}
For fixed $n\in\N$ and $d\in\N$ an algorithm $A_{n,d}\in\A_d^n$ is 
modeled as a mapping $\varphi_n \colon \R^n \nach \G_d$ 
and a function $N_n \colon \widetilde{\F}_d \nach \R^n$ such that 
$A_{n,d} = \varphi_n \circ N_n$. 
For the sake of completeness in the case $n=0$ we simply assign a constant value 
$c\in\G_d$ to every element $f\in\widetilde{\F}$, \ie $A_{0,d}\equiv c$,
in order to model an algorithm that does not depend on the input at all.
If $n>0$ then the \emph{information map}~$N_n$ is given by
\begin{gather}\label{non-adapt}
        N_n(f)=\left( L_1(f), L_2(f), \ldots, L_n(f) \right), \qquad f\in \widetilde{\F}_d,
\end{gather}
where $L_j \in \Lambda$. 
Here we distinguish certain classes of 
information operations $\Lambda$. 
In one case we assume that we are allowed to 
compute arbitrary continuous linear functionals on the inputs $f$. 
Then $\Lambda=\Lambda^{\rm all}$ coincides with $\F_d^*$, the dual space of~$\F_d$. 
If we deal with problem operators $S_d$ defined on function spaces $\widetilde{\F}_d$
then often only function evaluations are permitted, \ie
$L_j(f)=f (t^{(j)})$ for a certain fixed 
$t^{(j)}\in\Omega_d$ in the domain of definition of $f$. 
In this case $\Lambda = \Lambda^{\rm std}$ is 
called \textit{standard information}. 
If function evaluation is continuous for all $t\in\Omega_d$ we have 
$\Lambda^{\rm std}\subset \Lambda^{\rm all}$.
In particular this is the case when dealing with problems defined on
reproducing kernel Hilbert spaces; see \autoref{sect:RKHS}. 
If $L_j$ depends continuously on $f$ 
but is not necessarily linear then the respective class is denoted 
by $\Lambda^{\rm cont}$. 
Note that in this case also $N_n$ is continuous 
and we obviously have $\Lambda^{\rm all} \subset \Lambda^{\rm cont}$. 

Furthermore, we distinguish between \textit{adaptive} and \textit{non-adaptive} 
algorithms. 
The latter case is described above in formula \link{non-adapt}, 
where $L_j$ does not depend on the previously computed values 
$L_1(f),\ldots,L_{j-1}(f)$. 
In contrast, we also discuss algorithms of 
the form $A_{n,d}=\varphi_n \circ N_n$ with
\begin{gather}\label{adapt}
        N_n(f)= \left( L_1(f), L_2(f;y_1), \ldots, L_n(f;y_1,\ldots,y_{n-1})\right), \qquad f\in\widetilde{\F}_d,
\end{gather}
where $y_1=L_1(f)$ and $y_j=L_j(f;y_1,\ldots,y_{j-1})$ for $j=2,3,\ldots,n$. 
If $N_n$ is adaptive we restrict ourselves to the case where $L_j$ depends 
linearly on $f$, \eg\, $L_j(\,\cdot\,; y_1,\ldots,y_{j-1}) \in \Lambda^{\rm all}$.
Note that in any case $N_n$ is either continuous, or it is constructed out of linear information operations (which may be combined adaptively).
Moreover, in all cases of information maps, 
the mapping $\varphi_n$ can be chosen 
arbitrarily.

For upper error bounds small classes of algorithms are most important.
The smallest such class under consideration is the family
of linear, non-adaptive algorithms of the form
\begin{gather}\label{LinAlg}
        A_{n,d}(f) = \sum_{j=1}^n L_j(f) \cdot g_j
\end{gather}
with some $g_j \in \G_d$ and $L_j \in \Lambda^{\rm all}$
or even $L_j\in\Lambda^{\rm std}$. 
We denote this set of algorithms by $\A_{d}^{n,\rm lin}(\Lambda)$,
where $\Lambda=\Lambda^{\rm all}$, or $\Lambda=\Lambda^{\rm std}$, respectively.
On the other hand, it is reasonable to prove lower error bounds 
for preferably large classes of algorithms.
The most general families consist of algorithms 
$A_{n,d}=\varphi_n \circ N_n$, where $\varphi_n$ is completely arbitrary and $N_n$ either uses 
non-adaptive continuous or adaptive linear information. 
We denote the respective classes by 
$\A_d^{n, \rm cont}$ and $\A_d^{n,\rm adapt}$.

One of the most fundamental assumptions in IBC is that
we can perform (exact) basic arithmetic operations 
on elements of the target space~$\G_d$, as well as on real numbers, with unit cost.
Formally this means that we work with the \emph{real number model}
in contrast to the \emph{bit number model} 
which is used in some other fields of computational science; see, \eg, \cite[Section 4.1.3]{NW08}.
Moreover, we assume that information operations on the input
are given by certain black box computations which are sometimes called
\emph{oracle calls}.
Typically the computational costs for information operations are much higher
than for simple arithmetic operations since 
the computation of a function value or a linear functional may require billions of 
such operations.
If we assume that every oracle call has a fixed cost $C\gg 1$
then the total cost of computing the output of an algorithm is proportional to
the number of needed information operations.\footnote{There also exist approaches in which the cost of an oracle call depends on the parameters of the problem. These attempts stress the point that the computational effort for function evaluations increases with the numbers of (active) variables. 
See, \eg, \cite{KSWW10} for details.}
Therefore it is reasonable to study not only the 
$n$th minimal error of a given problem but also the inverse quantity which we call
\emph{information complexity}
\begin{align*}
			n^{\mathrm{sett}}_{\mathrm{abs}}\!\left(\epsilon,d; S_d\colon\widetilde{\F}_d\nach\G_d\right) 
			&= \min{ n\in\N_0 \sep \exists A\in\A_d^n \text{ such that } \Delta^{\mathrm{sett}}(A) \leq \epsilon }\\
			&= \min{ n\in\N_0 \sep e^{\rm sett}(n,d) \leq \epsilon },
\end{align*}
where $d\in\N$, $\epsilon>0$ and $\mathrm{sett} \in \{\mathrm{wor},\mathrm{avg}\}$.
That is, we look at the amount of oracle calls needed to compute an $\epsilon$-approximation in dimension $d$.
Hence, due to our assumptions this information complexity 
roughly equals the \emph{total complexity} 
of a given problem and therefore describes its computational hardness.
For a detailed discussion of algorithms and their costs, as well as on the relations of
information complexity and total complexity we refer the reader to Section 4.1 in \cite{NW08}.

Finally we want to mention that the above definition addresses the
\emph{absolute error criterion}.
In contrast we will also consider the \emph{normalized error criterion} where
we search for the minimal number of information operations 
needed to improve the initial error
by some factor $\epsilon'>0$.
We denote the corresponding information complexity by
\begin{gather*}
			n^{\rm sett}_{\mathrm{norm}}\!\left(\epsilon',d; S_d\colon\widetilde{\F}_d\nach\G_d\right) 
			= n^{\rm sett}_{\mathrm{abs}}\!\left(\epsilon' \cdot \epsilon_d^{\rm init},d; S_d\colon\widetilde{\F}_d\nach\G_d\right) ,
			\quad \mathrm{sett} \in \{\mathrm{wor},\mathrm{avg}\}.
\end{gather*}
Obviously both the notions coincide if the problem 
under consideration is well-scaled. 
That is, if $\epsilon_d^{\rm init}=1$.
Otherwise the problem may be significantly harder with respect to the normalized error criterion, \eg\, if $\epsilon_d^{\rm init}$ is exponentially small in $d$.
Of course also the converse situation is conceivable.
However, note that both the information complexities are always non-increasing
in the first argument and we have
\begin{equation}\label{information_complexity_zero}
			n^{\rm sett}_\no\!\left(1,d; S_d\colon\widetilde{\F}_d\nach\G_d\right)
			= n^{\rm sett}_\ab\!\left(\eps_d^{\mathrm{init}},d; S_d\colon\widetilde{\F}_d\nach\G_d\right)
			= 0 
\end{equation}
for all $d\in\N$ due to the definition of the initial error.

Again we will use shorthands such as $n_{\mathrm{abs}}(\eps,d;S_d)$ or even $n(\eps,d)$ to simplify notation.

\section{Notions of tractability}\label{sect:TractDef}
As already indicated we strongly believe that it is not sufficient 
only to study the rate of convergence, \ie the dependence of $n(\epsilon,d)$ on $\epsilon$,
to properly describe the computational hardness of a given problem.
We also need to incorporate the dependence on the parameter~$d$.
Keep in mind that the following definitions equally refer to 
both, the absolute and the normalized, error criteria. 
Therefore we simply write $n(\epsilon,d)$ instead of
$n_\mathrm{abs}(\epsilon,d)$ or $n_\mathrm{norm}(\epsilon,d)$
for the information complexity.

When dealing with multivariate problems we often observe the
so-called \emph{curse of dimensionality} 
which goes back to Bellman in the late 1950s; cf.~\cite{B57}.
Given a concrete setting a problem is said to suffer 
from the curse of dimensionality if the corresponding 
information complexity~$n(\epsilon,d)$ increases 
exponentially with the dimension~$d$.
That is, for at least one $\epsilon > 0$ there exist positive constants 
$C$ and $\gamma$ which are independent of the dimension such that we have
\begin{gather*}
        n(\epsilon,d) \geq C \cdot (1+\gamma)^d
\end{gather*}
for infinitely many $d\in\N$.
More generally, if $n(\epsilon,d)$
depends exponentially on $d$ or~$\epsilon^{-1}$ then we call the problem 
\emph{intractable}\footnote{Formally that means that there exist universal constants $\gamma,C>0$, as well as sequences $(\eps_k)_{k\in\N}$ and $(d_k)_{k\in\N}$ with $\eps_k\in(0,1]$ and $d_k\in\N$ for all $k\in\N$, such that $\eps_k^{-1}+d_k\nach\infty$, as $k\nach\infty$, and $n(\eps_k,d_k)\geq C \cdot (1+\gamma)^{\eps_k^{-1}+d_k}$ for every $k\in\N$. Note that this definition includes the curse as a special case, where $\eps_k \equiv \eps_0$.}. 
Otherwise we have \emph{tractability} which goes back 
to Wo{\'z}niakowski in the early 1990s; see \cite{W94a,W94b}.
At this time a problem was called tractable 
if its complexity depends at most polynomially on~$\epsilon^{-1}$ and~$d$.
Today this is only one case in a whole hierarchy of notions of tractability.
We describe these classes starting with the weakest notion.

If a problem is not intractable then we have \emph{weak tractability}
which can be equivalently expressed by
\begin{gather*}
        \lim_{\epsilon^{-1}+d \nach \infty} 
        \frac{\ln \!\left( n(\epsilon,d) \right)}{\epsilon^{-1}+d} = 0,
\end{gather*}
see \cite{GW08, NW08}.
Here the limit is taken with respect to all two-dimensional sequences 
$((\eps_k, d_k))_{k\in\N} \subset (0,1] \times \N$ such that $\eps_k < \eps_{d_k}^\mathrm{init}$
and $\eps_k^{-1}+d_k\nach\infty$, as $k$ approaches infinity.
In particular, the latter restriction ensures that $n(\eps,d)\geq 1$.
Furthermore, we want to stress the point that weak tractability implies the 
absence of the curse of dimensionality, but in general the converse 
is not true.
Recently a slightly stronger notion called \emph{uniform weak tractability} has been suggested.
We will not follow this line of research and refer to \cite{P13}.

Since there are many ways to measure the lack of 
exponential dependence the abstract notion of 
\emph{generalized} (or $T$-) \emph{tractability}
was introduced; see \cite{GW07,GW09} and \cite[Chapter 8]{NW08}.
Here the essence is to describe the behavior of the information complexity
in terms of a multiple of some power of a so-called 
\emph{tractability function} 
$T$ depending on $\epsilon^{-1}$ and $d$. 
Without going into details we mention that the following 
classes can be seen as special cases in this general framework.

For the sake of completeness we also introduce the quite recently developed notion of
\emph{quasi-polynomial tractability}.
A problem is called quasi-polynomially tractable 
if there are universal constants $C,t>0$ such that
\begin{gather*}
			n(\epsilon,d) \leq C \exp{ t (1 + \ln \epsilon^{-1} ) (1 + \ln d) }
\end{gather*}
for every $\epsilon \in(0,1]$ and $d\in\N$.
Note that for fixed $\epsilon$ or $d$ this upper bound behaves polynomially in the second argument what somehow justifies the name of this class of problems.
For details see \cite{GW11} and \cite{NW12}.

Finally, the most important and until now most studied type of tractability is called
\emph{polynomial tractability}.
We say that a problem is polynomially tractable if there exist absolute
constants $C,p>0$ and $q \geq 0$ such that we can bound the information complexity by
\begin{gather}\label{def_poltract}
			n(\epsilon,d) \leq C \cdot \epsilon^{-p} \cdot d^{\,q} \quad \text{for all} \quad d\in\N, \quad \eps \in(0,1].
\end{gather}
If this last inequality holds with $q=0$, 
\ie if we have no dependence on the dimension at all, 
then the problem is called \textit{strongly polynomially tractable}.
In this case the smallest possible constant $p$ in \link{def_poltract} is denoted by
$p^*$. It is called the \emph{exponent of strong polynomial tractability}.\\
If, in contrast, there do not exist constants $C,p$ and $q$ which fulfill \link{def_poltract}
then the problem is said to be \emph{polynomially intractable}.

Observe that \link{information_complexity_zero} shows that, as long as the absolute error criterion is concerned, 
it is enough to consider $\eps \in \left(0,\min{\eps_d^{\mathrm{init}},1}\right]$
instead of $\eps \in (0,1]$ in all the above definitions.
  \cleardoubleplainpage
\chapter{Properties and tools for special problem classes}\label{chapt:prop}
This chapter deals with basic properties
of certain classes of problems and algorithms.
We state simple consequences obtained from fundamental assumptions 
on the operators under consideration.
Furthermore, we present more or less classical tools used in
the framework of information-based complexity to acquire tractability results
in a quite general context.

In detail, we begin with a simple lower error bound in a very general setting which will be used on several 
occasions later on.
In \autoref{sect:Linearity} we then show that for our purposes it is reasonable to
concentrate mainly on compact problems and linear, non-adaptive algorithms.
Moreover, there we derive a formula for the initial error of the problems we are interested in.
Afterwards, in \autoref{sect:General_HSP}, we turn to the important class of problems defined between Hilbert spaces. 
We recall well-known tools such as the singular value decomposition, conclude optimal algorithms and characterize several types of tractabilities of such problems.
In \autoref{sect:TensorBasics} we restrict ourselves further and assume an additional tensor product
structure which will play an important role throughout the rest of this thesis.
Finally we conclude this chapter with the discussion of so-called reproducing kernel Hilbert spaces.

The main references for the functional analytic background needed in this part,
as well as on the theory of $s$-numbers (or $n$-widths, respectively) are
the monographs of Pinkus~\cite{P85} and Pietsch~\cite{P87,P07}.
For a detailed discussion of applications to tractability questions we refer
again to Novak and Wo{\'z}niakowski~\cite{NW08,NW10,NW12} and to Math{\'e}~\cite{M90}.

\section{Lower bounds on linear subspaces}
For the purpose of this chapter it is enough to study the worst case setting.
In addition, we will only focus on the case where all the problem elements lie in
some centered ball of the respective source space.
In this first section we present a quite general method to obtain
lower bounds on the $n$th minimal error
with respect to a wide class of algorithms.
In contrast to the rest of this thesis (where we will restrict ourselves basically to linear and compact problems)
we present a result that holds for any 
homogeneous operator $S$ between linear normed spaces $\F$ and $\G$ over the field of real numbers.
That is, we first only assume that $S(\alpha\cdot f)=\alpha\cdot S(f)$ 
for every $f\in \F$ and all $\alpha\in\R$.

We start by proving the following (modified) assertion of 
Borsuk and Ulam for linear normed spaces:

\begin{lemma}[Borsuk-Ulam]\label{prop:borsuk}
			Let $V$ be a linear normed space over $\R$ with $0<\dim{V}=s<\infty$ and, 
			moreover, let $N \colon V \nach \R^n$ 
			be a continuous mapping for some $0 \leq n<s$. 
			Then for all $r\geq 0$ there exists an element 
			$f^{*} \in V$ with $\norm{f^{*} \sep V}=r$, such that $N(f^{*})=N(-f^{*})$.
\end{lemma}

\begin{proof}
Obviously, the cases $n=0$, \ie $N\equiv 0$, and $r=0$ are trivial. 
Hence, let $n\in\N$ and $r>0$.
Since $\dim V = s$ we find an isomorphism $T\colon V\nach\R^s$ 
such that $T$ and $T^{-1}$ are linear and bounded.
Hence, for every $r>0$ the set $\Omega_r=T(\inner(B_r(V)))$ is an open, 
bounded and symmetric subset of $\R^s$ which contains zero.
Moreover, the function $g=N\circ T^{-1}\colon\partial \Omega_r \nach \R^n$ is continuous.
From the theorem of Borsuk-Ulam (cf. Deimling~\cite[Corollary 4.2]{D85}) 
we conclude the existence of some $\bm{x^*}\in\partial\Omega_r$ with $g(\bm{x^*})=g(-\bm{x^*})$.
The claim now follows by taking $f^{*}=T^{-1} \bm{x^{*}}$.
\end{proof}

This result in hand, we can prove a generalization of \cite[Lemma~1]{W12}.
\begin{prop}\label{needed_lemma}
				Suppose $S$ to be a homogeneous operator between linear normed spaces $\F$ and $\G$.
				Further assume that $V \subset \F$ is a linear subspace with dimension $s\in\N$ and 
				that there exists a constant $a \geq 0$ such that
				\begin{gather}\label{eq:NormCondition}
								a \cdot \norm{f \sep \F} \leq \norm{S(f) \sep \G} \quad \text{for all} \quad f\in V.
				\end{gather}
				Then for every $0\leq n<s$, any algorithm $A_n \in \A^{n,\rm cont} \cup \A^{n,\rm adapt}$, and all $r\geq 0$
				\begin{gather}\label{eq:LowerBound_r}
								\Delta^{\wor}(A_n; S\colon B_r(\F)\nach \G) 
								= \sup_{f\in B_r(\F)} \norm{S(f) - A_n(f) \sep \G} \geq a \cdot r.
				\end{gather}
				In particular, the $n$th minimal worst case error 
				(among the unit ball $\B(\F)$ of $\F$) satisfies 
				$e^{\wor}(n; S\colon \B(\F)\nach\G) \geq a$ 
				for all $n<s$.
\end{prop}

\begin{proof}
It is well-known that for $A_n=\varphi_n \circ N_n \in \A^{n,\rm cont} \cup \A^{n,\rm adapt}$ 
with $n<s$ there exists $f^* \in V$ 
such that $N_n(f^*)=N_n(-f^*)$ and $\norm{f^*\sep \F}=r$. 

Without loss of generality let us again assume $n\in\N$ and $r>0$ to avoid triviality.
Then, for $A_n \in \A^{n,\rm cont}$, the existence of $f^*$ is a simple conclusion of \autoref{prop:borsuk}
since in this case $N_n$ is continuous by definition.
On the other hand, if $A_n\in \A^{n,\rm adapt}$ then the proof can be obtained by arguments from linear algebra.
We follow the lines of the proof of Werschulz and Wo\'zniakowski~\cite[Theorem~3.1]{WW09} and search for a nonzero $g \in V$
such that $N_n(g)=0$, \ie
\begin{align}
				L_1(g) =& 0, \nonumber\\
				L_2(g; 0) =& 0, \nonumber\\
				\vdots& \label{eq:HomSystem} \\
				L_n(g; 0,\ldots, 0) =& 0. \nonumber
\end{align}
Since $\dim V=s$ every $g\in V$ can be represented uniquely 
as a linear combination $g = \sum_{m=1}^s c_m b_m$ of at most $s$ linearly independent 
basis functions $b_m$ of $V$.
Due to the imposed linearity of $L_j(\,\cdot\,;0,\ldots,0)$, $j=1,\ldots,n$, the system \link{eq:HomSystem}
can be reformulated as a system of $n$ homogeneous linear equations in the $s>n$ unknowns $\bm{c}=(c_m)_{m=1}^s\in\R^s$.
Consequently, it possesses a non-trivial solution $\bm{c^*}=(c^*_m)_{m=1}^s$ which implies
the existence of some $g^* \in V\setminus \{0\}$ with $N_n(g^*)=0$.
Since with $L_j$ also $N_n$ is linear, we can easily construct $f^*$ out of $g^*$.

Anyway, every such $f^*$ satisfies $A_n(f^*)=A_n(-f^*)$.
Using the norm properties in the target space $\G$ 
and the homogeneity of $S$ we obtain \link{eq:LowerBound_r}:
\begin{align*}
				\Delta^{\wor}(A_n; S\colon B_r(\F)\nach \G) 
				&\geq \max{\norm{S(\pm f^*)-A_n(\pm f^*) \sep \G}} \\
				&= \max{\norm{S(f^*) \pm A_n(f^*) \sep \G}} \\
				&\geq \frac{1}{2}\, (\norm{ S(f^*) + A_n(f^*) \sep \G} + \norm{ S(f^*) - A_n(f^*) \sep \G}) \\
				&\geq \frac{1}{2}\,  \norm{2 \cdot S(f^*) \sep \G} \geq a \norm{f^* \sep \F} = a \cdot r.
\end{align*}
The remaining implication for the $n$th minimal error finally
follows from the case $r=1$ by taking the infimum over all $A_n\in\A^{n,\rm cont} \cup \A^{n,\rm adapt}$.
\end{proof}

At this point we stress that the case $r \neq 1$ in \link{eq:LowerBound_r} might be useful
only if we deal with non-homogeneous (and thus non-linear) algorithms $A_n$.
Otherwise we clearly have
\begin{gather*}
				\Delta^{\wor}(A_n; S\colon B_r(\F)\nach \G) 
				= r \cdot \Delta^{\wor}(A_n; S\colon \B(\F)\nach \G) 
\end{gather*}
for all $r\geq 0$ provided that $S\colon \F\nach \G$ is homogeneous.
The importance of \link{eq:LowerBound_r} for $r\neq 1$ will be made clear
in \autoref{sect:embeddings} when we deal with embeddings $\P\hookrightarrow\F$.
There we conclude a lower bound for the worst case error of $S$ 
on $\B(\F)$ out of a lower bound on $B_r(\P)$
using $r=\norm{\id \sep \LO(\P,\F)}^{-1}$.

\section{Linearity and compactness}\label{sect:Linearity}
In what follows we will exclusively consider linear continuous problems $S=(S_d)_{d\in\N}$.
That is, we assume every solution operator $S_d$ given by \link{problemSd}
to be the restriction of a bounded linear mapping between some Banach spaces defined over the field of real numbers.\footnote{In fact, for most of the following results completeness is not needed.
Many of them even remain valid (at least up to constants) using only quasi-norms or $p$-norms, 
but for simplicity we restrict ourselves to the case of Banach spaces.
Finally, for the ease of notation, we only consider spaces over $\R$.}
If we assume the set of problem elements $\widetilde{\F}_d$ to be some centered ball $B_r(\F_d)$, $r>0$,
in the source space then conversely every bounded mapping~$S_d$
that acts linearly on this set\footnote{That means, $S_d(\alpha\cdot f + \beta\cdot g)$ equals $\alpha \cdot S_d(f)+\beta \cdot S_d(g)$ for every convex combination $\alpha\cdot f + \beta\cdot g$ of elements $f,g\in\widetilde{\F}_d$.}
can be uniquely extended to a continuous linear operator $\widetilde{S}_d$ on the whole space $\F_d$, \ie
$\widetilde{S}_d\in\LO(\F_d,\G_d)$.
From this point of view $S_d$ and $\widetilde{S}_d$ can be identified with each other
and thus we use the symbol $S_d$ for both of them.

At the first glance the linearity assumption seems to be very restrictive.
On the other hand, both the most important problems, 
namely approximation and integration, are indeed of this type.
Moreover, the linear case is much better understood than the non-linear
such that an overwhelming percentage of work on IBC was done in this setting.
For the sake of completeness we also mention so-called quasilinear problems
and refer to \cite{WW07} and \cite[Chapter~28]{NW12}.

Since we are interested in algorithms which are easy (and cheap) to implement
we pay special attention to the family of linear and 
non-adaptive algorithms $\A_d^{n,\mathrm{lin}}(\Lambda)$; see \link{LinAlg}.
It is well-known that this choice is reasonable for many classes of problems, since 
it can be shown that under mild assumptions optimal algorithms are indeed linear and non-adaptive.
General assertions of this type can be found in 
Traub, Wasilkowski and Wo\'zniakowski~\cite{TWW88}, as well as in Novak and Wo\'zniakowski~\cite[Section~4.2]{NW08}.
We do not present these results here explicitly. 
The reason is that for the problems we are interested in, our assertions already imply the mentioned optimality statements.

Furthermore, we focus on information maps which are linear and continuous, 
\ie $\Lambda \subseteq \Lambda^{\mathrm{all}}$.
Observe that then $A_{n,d}\in\LO(\F_d,\G_d)$ and $\rank(A_{n,d})\leq n$. 
Moreover, for $\widetilde{\F}_d = \B(\F_d)$ we obtain
\begin{gather*}
		\Delta^{\mathrm{wor}}(A_{n,d};S_d)
		= \Delta^{\mathrm{wor}}\!\left(A_{n,d};S_d\colon \widetilde{\F}_d\nach\G_d\right) 
		= \norm{S_d - A_{n,d} \sep \LO(\F_d,\G_d)}.
\end{gather*}

It seems natural to ask when problems of this type are solvable at all.
We say a problem $S=(S_d)_{d\in \N}$ is \emph{solvable} if for any fixed $d\in\N$ there 
exists a sequence of algorithms 
$A_{n,d}\in\A_d^{n,\mathrm{lin}}(\Lambda^{\mathrm{all}})$ such that their 
worst case errors $\Delta^{\mathrm{wor}}(A_{n,d};S_d)$ tend to zero as $n$ approaches infinity.
Hence, $S_d$ needs to be an element of $\overline{\F\mathcal{R}(\F_d, \G_d)}$, the closure of the finite rank operators in $\LO(\F_d,\G_d)$, which
is a subset of $\K(\F_d,\G_d)$.
Therefore solvable problems are necessarily compact such that we can restrict 
ourselves in the following to $S_d\in \K(\F_d,\G_d)$. 
Due to the celebrated result of Enflo~\cite{E73} it is known that the converse is not true in this generality.
Indeed, there are compact problems which are not solvable since there exist 
Banach spaces~$\G_d$ which do not satisfy the so-called approximation property.
However, the following (incomplete) list shows that in the cases we are interested in every compact problem is solvable:
\begin{prop}\label{Prop:Solveable}
		Let $S=(S_d)_{d\in\N}$ be given such that $S_d\in \K(\F_d,\G_d)$ for all $d\in\N$.
		Then $S$ is solvable if for every $d\in\N$ one of the following conditions applies:
		\begin{itemize}
				\item The source space $\F_d$ is a Hilbert space, or
				\item The target space $\G_d$ is a Hilbert space, or
				\item The target space $\G_d$ is $\L_\infty(\mathcal{X},\a,\mu)$ for an arbitrary measure space $(\mathcal{X},\a,\mu)$.
		\end{itemize}
\end{prop}
\begin{proof}
Let $d\in\N$.
Given all the above restrictions we note that if we consider the class $\A_d^{n,\lin}(\Lambda^{\mathrm{all}})$ then the numbers
$e^{\mathrm{wor}}(n,d; S_d)$, $n\in\N_0$, per definition equal the \emph{linear $n$-widths} (or \emph{approximation numbers}) $\delta_n(S_d)$ as defined in \cite[Definition~7.3]{P85}.
Up to an index shift these numbers form an $s$-scale\footnote{Note that due to historical reasons there is some notational danger 
concerning $s$-numbers versus $n$-widths. See, \eg, \cite[p. 336]{P07} for details.} in the sense of Pietsch~\cite[Section 6.2]{P07}.
Other important $s$-scales are the \emph{Gelfand numbers}~$c_n(S_d)$ 
and the \emph{Kolmogorov numbers}~$d_n(S_d)$.
Without going into details we mention that for any compact operator~$S_d\in \K(\F_d,\G_d)$ 
both these numbers tend to zero as $n\nach\infty$; see Propositions 7.4 and 7.1 in \cite{P85}.
Hence, to prove solvability it suffices to show that $\delta_n(S_d) \leq \max{c_n(S_d),d_n(S_d)}$ for all $n\in\N$.
Indeed, if $\F_d$ is a Hilbert space then we have $\delta_n(S_d) = c_n(S_d)$.
Furthermore, $\delta_n(S_d) = d_n(S_d)$ if $\G_d$ is a Hilbert space; see, \eg, \cite[p. 33]{P85}.
Finally Proposition~8.13 in~\cite{P85} shows that the second last equality remains valid if the target space~$\G_d$ enjoys the so-called \emph{(metric) extension property}.
It is known that in particular $\L_\infty(\mathcal{X},\a,\mu)$ has this property; see, \eg, K{\"o}nig~\cite[1.c.2]{K86}.
\end{proof}

We want to stress that in \autoref{Prop:Solveable} we do not need to assume the Hilbert spaces to be separable.

Let us conclude this section with a proposition which shows that the zero algorithm 
$A_0 \equiv 0$ is the optimal choice among all
approximations to a given operator $S\in\LO(\F,\G)$ 
that do not use any information on the input $f\in\F$.
Here $\F$ and $\G$ can be arbitrary normed spaces.

\begin{prop}\label{prop:init_error}
			For $S\in\LO(\F,\G)$ and $\A_0 = 0 \in \LO(\F,\G)$ we have
			\begin{gather*}
					e^\wor(0;S\colon \B(\F)\nach\G)
					=	\Delta^\wor(A_0;S\colon\B(\F)\nach\G)
					= \norm{S\sep \LO(\F,\G)}.
			\end{gather*}
			Consequently, the zero algorithm is optimal for $S$
			within the class $\A^{0,\mathrm{cont}} \cup \A^{0,\mathrm{adapt}}$ and
			the initial worst case error $\eps^{\init,\wor}$ of $S$ is given by the its operator norm.
\end{prop}
\begin{proof}
Obviously the second equality is true by the definition of $\Delta^\wor$.
Moreover, the linear algorithm $A_0 \equiv 0$
is included in every class of algorithms we defined in \autoref{sect:Algos}.
This particularly implies $e^\wor(0;S) \leq	\Delta^\wor(A_0)= \norm{S\sep \LO(\F,\G)}$.

To show the converse inequality, 
recall that every algorithm $A$ that does not use any
information on the input necessarily takes the form $A(f) \equiv g$
for some element $g\in\G$.
A calculation similar to that in the proof of \autoref{needed_lemma} yields that
\begin{gather*}
		\norm{S(f) \sep \G} \leq \max{\norm{S(f)-g \sep \G}, \norm{S(-f)-g \sep \G}}
\end{gather*}
holds for every $f\in\F$.
Taking the supremum over $f\in\B(\F)$ now shows that $\norm{S\sep\LO(\F,\G)} \leq \Delta^\wor(A)$ 
which implies the desired result since $A$ was chosen arbitrary.
\end{proof}

We note in passing that the last step in the latter proof
crucially depends on the fact that the unit ball $\widetilde{\F}=\B(\F)$
of the source space $\F$ is symmetric in the sense that $f\in\widetilde{\F}$ implies $-f\in\widetilde{\F}$.


\section{General Hilbert space problems}\label{sect:General_HSP}
In this section we describe the \emph{singular value decomposition} (SVD)
which turns out to be the main tool when dealing with problems where
both the source and the target spaces are Hilbert spaces.
We prove well-known formulas for optimal linear algorithms 
using continuous linear functionals and
calculate their worst case errors.
Afterwards, we use the obtained assertions to give characterizations 
for (strong) polynomial tractability for these problems.
\subsection{Singular value decomposition}\label{sect:SVD}
Given any compact operator $T\in \K(\F,\G)$ acting between two 
arbitrary real Hilbert spaces~$\F$ and $\G$ we define its \emph{adjoint operator} 
$T^\dagger \colon \G \nach \F$ in the usual way by
\begin{equation}\label{def_adj}
			\distr{Tf}{g}_\G = \distr{f}{T^\dagger g}_\F, \quad \text{for all} \quad f\in \F, g\in \G.
\end{equation}
Of course, $T^\dagger$ is always unique and well-defined.
For details we refer the reader to Yosida~\cite[VII.2]{Y80}.
If $\F=\G$ and $T^\dagger = T$, then we say that $T$ is \emph{self-adjoint}.
Due to Schauder's Theorem we know that $T^\dagger \in \K(\G,\F)$ 
if and only if $T \in \K(\F,\G)$; see, e.g. \cite[p. 31]{P85}.
Hence, it is easily seen that also
\begin{equation*}
			W = T^\dagger T \colon \F \nach \F
\end{equation*}
defines a compact operator.
Moreover, $W$ is obviously self-adjoint and positive, \ie 
$\distr{Wf}{f}_{\F} \geq 0$ for every $f\in \F$.
It is a well-known fact that therefore all the eigenvalues 
$\lambda_m=\lambda_m(W)$ of $W$
are necessarily real and furthermore non-negative.
Following Pinkus~\cite[p. 64]{P85} we denote the
sum of the algebraic multiplicities of the non-zero eigenvalues of $W$
by $v=v(W)$.
Note that the theory of Riesz-Schauder 
provides that there are at most countably many non-zero eigenvalues. 
They are uniformly bounded, each of them has a finite multiplicity and there
are no accumulation points but (possibly) zero. 
See, e.g., Theorem 2 in \cite[X.5]{Y80}.
Observe further that in any case $v \leq \dim \F \in \N \cup\{\infty\}$.
Let us denote these eigenvalues in a non-increasing ordering 
subscripted by indices from the set $\M = \{m \in \N \sep m < v+1\}$,
\begin{equation}\label{univariate_eigenvalues}
			\lambda_1 \geq \lambda_2 \geq \ldots \geq \lambda_m \geq \ldots > 0.
\end{equation}
Note that without loss of generality we will always assume the existence of at least
one non-trivial eigenvalue, \ie we explicitly exclude the operator $T\equiv 0$ which ensures that $\M \neq \leer$.
We denote the corresponding (mutually orthonormal) eigenvectors 
of $W$ by $\phi_m$, $m\in\M$, and refer to $\{(\lambda_m,\phi_m) \sep m\in\M\}$
as the set of \emph{non-trivial eigenpairs} of $W$.
Consequently, for $i,j \in\M$ we have by \link{def_adj}
\begin{equation}\label{T_orth}
			\distr{T\phi_i}{T\phi_j}_\G 
			= \distr{\phi_i}{T^\dagger T\phi_j}_\F 
			= \distr{\phi_i}{W\phi_j}_\F 
			= \distr{\phi_i}{\lambda_j \,\phi_j}_\F 
			= \delta_{i,j}\cdot \lambda_j. 
\end{equation}
If we extend the possibly finite eigenvalue sequence $(\lambda_m)_{m=1}^v$ by taking 
$\lambda_m=0$ for all $m>v$
then, clearly, $\lambda=(\lambda_m)_{m\in\N}$ forms a null sequence.
Again following Pinkus, we call the square root $\sigma=\sigma(T)$ of $\lambda=\lambda(W)$,
\begin{equation*}
			\sigma_m = \sqrt{\lambda_m}, \quad m\in\N,
\end{equation*}
sequence of the \emph{singular values} of $T$.
The importance of this bunch of definitions comes from the following assertion.

\begin{theorem}[Singular value decomposition]
		Let $\F$ and $\G$ be arbitrary Hilbert spaces and $T\in \K(\F,\G)$. 
		Then, with the above notations,
		\begin{equation}\label{SVD}
				T = \sum_{m=1}^{v} \distr{\,\cdot\,}{\phi_m}_\F \, T\phi_m.
		\end{equation}
\end{theorem}

\begin{proof}
A detailed proof can be found in the monograph of K\"{o}nig~\cite[1.b.3]{K86}.
It is mainly based on the so-called \emph{polar decomposition} 
of linear continuous operators
and the theory of Riesz-Schauder.
Actually, the proof deals with complex Hilbert 
spaces but it literally transfers to the real case.
Moreover, only the existence of an orthonormal sequence 
$(\psi_m)_{m=1}^v$ in $\G$ is shown such 
that the pointwise equality
\begin{equation*}
			Tf = \sum_{m=1}^{v} \sigma_m \, \distr{f}{\phi_m}_\F \, \psi_m, \quad f\in \F,
\end{equation*}
holds true.
However, setting $f=\phi_k$ for $k\in\M$ 
together with the mutual orthonormality of $(\phi_m)_{m=1}^v$ 
immediately implies $\sigma_k \psi_k= T\phi_k$ for any $k$.
The claimed identity in $\K(\F,\G)$ finally follows 
from Bessel's inequality.
\end{proof}

\begin{rem}\label{rem:sep}
Note that again the Hilbert spaces $\F$ and $\G$ do not need to be separable.
Nevertheless the image of $\F$ under $T$ is indeed separable, because it is spanned
by at most countable many elements $T\phi_m\in\G$.
Since the elements of the set $\Phi=\{\phi_m \in \F \sep m\in\M\}$ 
are mutually orthonormal we can extend $\Phi$ to an orthonormal basis (ONB) $E$ of $\F$.
Then \link{SVD} shows that $\ker T = \Phi^\bot$.
Remember that we are only interested in the approximation of the image of $T$.
Hence, we can without loss of generality assume that $E=\Phi$.
In other words, even though $\F$ may be non-separable in general
we can restrict ourselves to the separable case in what follows.
We only need to replace $\F$ by $\overline{\Phi}$,
the closure of the orthonormal eigenelements of 
$W=T^\dagger T$ under $\distr{\cdot}{\cdot}_\F$.
\hfill$\square$
\end{rem}

\subsection{Optimal algorithm}\label{sect:opt_Hilbert_Algo}
Observe that by \link{SVD} we obtained a representation of any operator $T\in \K(\F,\G)$
as the limit of related finite rank operators.
Therefore we are able to construct $n$th optimal linear algorithms which only use information
from~$\Lambda^{\mathrm{all}}$. This is stated in the following corollary 
which can be found (slightly modified) as Corollary 4.12 in~\cite{NW08}.

\begin{cor}\label{Cor:OptAlgo}
		For $d\in\N$ assume $\F_d$ and $\G_d$ to be arbitrary Hilbert spaces.
		Further let $S=(S_d)_{d\in\N}$ denote a compact problem 
		acting between these spaces, \ie $S_d \in \K(\F_d,\G_d)$ for every $d$.
		Then for all $d\in \N$ and $n\in\N_0$ the algorithm 
		$A_{n,d}^{*}\in \A_d^{n,\mathrm{lin}}(\Lambda^{\mathrm{all}})$ given by
		\begin{equation*}
					A_{n,d}^* \colon \F_d \nach \G_d, \qquad f \mapsto A_{n,d}^*(f)=\sum_{m=1}^{\min{n,v(W_d)}} \distr{f}{\phi_{d,m}}_{\F_d} \cdot S_d \phi_{d,m},
		\end{equation*}
		for $S_d$ is optimal in the class $\A_d^{n,\rm cont} \cup \A_d^{n,\rm adapt}$
		and we have
		\begin{equation}\label{OptWorstCaseError}
					e^{\mathrm{wor}}(n,d;S_d) 
					= \Delta^{\mathrm{wor}}(A_{n,d}^*;S_d) = \sigma_{d,n+1} = \sqrt{\lambda_{d,n+1}}.
		\end{equation}
		Here for every $d\in\N$ the singular values $(\sigma_{d,m})_{m\in\N}$, 
		as well as the eigenvectors $(\phi_{d,m})_{m=1}^{v(W_d)}$, 
		are constructed out of $W_d={S_d}^{\!\dagger} S_d$ as explained above.
\end{cor}

\begin{proof}
Recall that $\Delta^{\mathrm{wor}}(A_{n,d}^*;S_d)$ equals 
$\norm{S_d-A_{n,d}^* \sep \LO(\F_d,\G_d)}$ for any fixed $d\in\N$ and $n\in\N_0$. 
Without loss of generality we can assume $n < v=v(W_d)$ since
otherwise $A_{n,d}^*=S_d$ due to \link{SVD}. 
This would imply \link{OptWorstCaseError} because of $\sigma_{d,m}=0$ for all $m>v$.

Let $M\in\N$ with $n+1\leq M \leq v$ and $f\in\B(\F_d)$ be arbitrarily fixed. 
Then, due to \link{T_orth}, the non-increasing ordering of $(\lambda_{m})_{m=1}^v$ and Bessel's inequality,
\begin{align*}
		\norm{\sum_{m=n+1}^M \distr{f}{\phi_{d,m}}_{\F_d} \, S_d \phi_{d,m}\sep \G_d}^2 
		&= \sum_{m=n+1}^M \distr{f}{\phi_{d,m}}_{\F_d}^2 \, \lambda_{d,m} 
		\leq \lambda_{d,n+1} \sum_{m=1}^v \distr{f}{\phi_{d,m}}_{\F_d}^2 \\
		&\leq \lambda_{d,n+1} \norm{f \sep \F_d}^2 \leq \sigma_{d,n+1}^2.
\end{align*}
In particular, the choice $f=\phi_{d,n+1}$ shows that the latter estimates are sharp.
Anyway, we obtain $\norm{S_d - A_{n,d}^* \sep \LO(\F_d,\G_d)} \leq \sigma_{d,n+1}$ which proves
\begin{equation*}
		e^{\mathrm{wor}}(n,d; S_d) \leq \Delta^{\mathrm{wor}}(A_{n,d}^*; S_d) \leq \sigma_{d,n+1}.
\end{equation*}

To show the converse, \ie $e^{\mathrm{wor}}(n,d; S_d) \geq \sigma_{d,n+1}$ for $n\in\N_0$ and $d\in\N$, 
we use Parseval's identity on $V=\spann{\phi_{d,m} \sep m \leq n+1} \subset \F_d$
together with the linearity of $S_d$ to obtain 
\begin{gather*}
		\norm{S_d f \sep \G_d}^2 = \sum_{m=1}^{n+1} \distr{f}{\phi_{d,m}}_{\F_d}^2 \lambda_{d,m}
												\geq \lambda_{d,n+1} \sum_{m=1}^{n+1} \distr{f}{\phi_{d,m}}_{\F_d}^2
												= \sigma_{d,n+1}^2 \norm{f \sep \F_d}^2
\end{gather*}
for all $f\in V$.
The claim now follows from the application of \autoref{needed_lemma} with $a=\sigma_{d,n+1}$.
Moreover, \autoref{needed_lemma} also shows that we cannot reduce the error by taking algorithms
$A_{n,d}\in (\A_d^{n,\rm cont} \cup \A_d^{n,\rm adapt}) \setminus \A_d^{n,\lin}$.
%
\end{proof}

Note that \link{OptWorstCaseError} 
together with \autoref{prop:init_error}
particularly implies that for $d\in\N$ 
the initial worst case error of $S_d$ is given by
\begin{gather*}
		\eps_d^{\init,\wor} = \norm{S_d \sep \LO(\F_d,\G_d)} = \sigma_{d,1} = \sqrt{\lambda_{d,1}}.
\end{gather*}

\subsection{Polynomial tractability}
As an immediate consequence of \link{OptWorstCaseError} we can calculate 
the information complexity of Hilbert space problems 
in the worst case setting 
(with respect to the class $\A_d^{n,\lin}(\Lambda^{\mathrm{all}})$) 
for every $d\in\N$ and $\eps>0$ by
\begin{equation}\label{n_wor}
		n^\wor_\ab(\eps,d)
		= \min{n\in\N_0\sep \sigma_{d,n+1}\leq \eps}
		= \# \left\{ n \in \N \sep \lambda_{d,n} > \eps^2 \right\}
\end{equation}
for the absolute and by
\begin{equation}\label{n_wor_norm}
		n^\wor_\no(\eps,d) = \# \left\{ n \in \N \sep \lambda_{d,n}/\lambda_{d,1} > \eps^2 \right\},
\end{equation}
for the normalized error criterion, respectively.
This observation leads to the following refinement of Theorem 5.1 in
Novak and Wo{\'z}niakowski~\cite{NW08}
which also can be found in~\cite{W11}.
It gives necessary and sufficient conditions 
for (strong) polynomial tractability in terms of summability
properties of the sequences $(\lambda_{d,m})_{m\in\N}$.

\begin{theorem}\label{Thm:General_Tract_abs}
			Assume $S$ to be a problem as in \autoref{Cor:OptAlgo} 
			and consider the absolute error criterion in the worst case setting.
			\begin{itemize}
						\item If $S$ is polynomially tractable 
								with the constants $C,p>0$ and $q\geq 0$ 
								then for all $\tau > p/2$ we have
								\begin{equation}\label{sup_condition}
										C_\tau=\sup_{d\in\N} \frac{1}{d^{r}} \left( \sum_{i=f(d)}^\infty \lambda_{d,i}^\tau \right)^{1/\tau} < \infty,
								\end{equation}
								where $r=2q/p$ and $f\colon \N \nach \N$ with 
								$f(d)=\ceil{(1+C) \, d^{q}}$.
								In this case $C_\tau \leq C^{2/p}\cdot \zeta(2\tau/p)^{1/\tau}$.
						\item If \link{sup_condition} is satisfied for some parameters $r \geq 0$, $\tau >0$ 
								and a	function $f\colon\N \nach \N$ such that 
								$f(d)=\ceil{C\,\left(\min{\eps_d^{\rm init},1}\right)^{-p}\, d^{q}}$, 
								where $C>0$ and $p,q \geq 0$, then the problem $S$ is polynomially tractable.
								In detail, we have the bound 
								$n_\ab^\wor(\eps,d)\leq (C+C_\tau^\tau)\, \eps^{-\max{p,2\tau}}\, d^{\max{q,r\tau}}$
								for any $\eps\in\left(0,1\right]$ 
								and every $d\in\N$.								
			\end{itemize}
\end{theorem}

\begin{proof}
If the problem is polynomially tractable 
then there exist constants $C,p>0$ and $q\geq0$ such that
for all $d\in \N$ and $\eps \in \left( 0,1 \right]$
\begin{equation*}
				n(\eps,d) = n_\ab^\wor(\eps,d) \leq C \cdot \epsilon^{-p} \cdot d^q.
\end{equation*}
Formula \link{n_wor} and the non-increasing ordering of $(\lambda_{d,i})_{i\in\N}$
therefore imply
\begin{equation*}
				\lambda_{d, \floor{C\epsilon^{-p}d^q}+1} \leq \lambda_{d, n(\eps,d)+1} \leq \epsilon^2, \quad \eps \in (0,1].
\end{equation*}
If we set $i=\floor{C\, \epsilon^{-p} \, d^q}+1$ and vary 
$\epsilon \in \left(0,1\right]$ 
then $i$ takes the values 
$\floor{C\, d^q}+1$, 
$\floor{C\, d^q}+2$, and so forth. 
On the other hand, we have $i \leq C\epsilon^{-p}d^q+1$ 
which is equivalent to $\epsilon^2 \leq (Cd^q/(i-1))^{2/p}$ if $i\geq 2$.
For all $i \geq f(d)=\ceil{(1+C)\,d^q}$ we indeed have $i \geq 2$ and, consequently,
\begin{equation*}
				\lambda_{d,i} \leq \lambda_{d, n(\epsilon,d)+1} \leq \epsilon^2 \leq \left( \frac{Cd^q}{i-1} \right)^{2/p}.
\end{equation*}
Choosing $\tau > p/2 > 0$ we conclude
\begin{equation*}
				\sum_{i = f(d)}^\infty \lambda_{d,i}^\tau 
				\leq \sum_{i = f(d)}^\infty \left( \frac{Cd^q}{i-1} \right)^{2\tau/p} 
				= (Cd^q)^{2\tau/p} \sum_{i = f(d)-1}^\infty \frac{1}{i^{2\tau/p}} 
				\leq \left(C^{2/p} d^{2q/p}\right)^\tau \cdot \zeta\left(\frac{2\tau}{p}\right).
\end{equation*}
for every $d\in\N$.
In other words, we have shown \link{sup_condition} with $r = 2q/p$, as well as the estimate on $C_\tau$.

Conversely, assume now that for some $r \geq 0$ 
and $\tau>0$ estimate~\link{sup_condition} holds true with
\begin{equation*}
				f(d)
				=\ceil{C \left(\min{\eps_d^{\mathrm{init}},1}\right)^{-p} d^{q}},
				\quad \text{where} \quad C>0 \quad \text{and} \quad p,q \geq 0.
\end{equation*}
That is, we assume $0<C_\tau<\infty$.
For $n \geq f(d)$ the ordering of $(\lambda_{d,i})_{i\in\N}$ implies 
$\sum_{i=f(d)}^n \lambda_{d,i}^\tau \geq \lambda_{d,n}^\tau \cdot (n - f(d)+1)$. 
Hence, for every $d\in\N$ and $n\geq f(d)$
\begin{equation*}
				\lambda_{d,n} \cdot (n - f(d)+1)^{1/\tau} \leq \left( \sum_{i=f(d)}^n \lambda_{d,i}^\tau \right)^{1/\tau} \leq \left( \sum_{i=f(d)}^\infty \lambda_{d,i}^\tau \right)^{1/\tau} \leq C_\tau \, d^r,
\end{equation*}
or, respectively, $\lambda_{d,n+1} \leq C_\tau \, d^r \cdot ((n+1) - f(d)+1)^{-1/\tau}$, for all $n\geq f(d)-1$. 
Note that for $\epsilon \in (0, \min{\epsilon_d^{\rm init},1}]$ we have 
$C_\tau \, d^r \cdot ((n+1) - f(d)+1)^{-1/\tau} \leq \epsilon^2$
if and only if
\begin{equation*}
				n \geq n^* = \ceil{\left( \frac{C_\tau \, d^r}{\eps^2} \right)^\tau }+ f(d)-2.
\end{equation*}
In particular, it is $\lambda_{d,n+1} \leq \eps^2$ at least for $n\geq \max{n^*,f(d)-1}$.
In other words, for every $d\in\N$ and all $\eps \in (0, \min{\eps_d^{\mathrm{init}},1}]$ it is
\begin{align*}
				n_\ab^\wor(\eps,d) 
				&\leq \max{n^*,f(d)-1} 
				\leq f(d)-1 + \left( \frac{C_\tau \, d^r}{\epsilon^2} \right)^\tau \\
				&\leq C \, \left(\min{\eps_d^{\mathrm{init}},1}\right)^{-p}\, d^{q} + C_\tau^\tau\, \eps^{-2\tau} \, d^{r\tau}\\
				&\leq (C+C_\tau^\tau) \, \eps^{-\max{p,2\tau}} \, d^{\max{q, r\tau}}.
\end{align*} 
Thus, the problem is polynomially tractable since $n_\ab^\wor(\eps,d) = 0$ for $\eps \geq \eps_d^{\mathrm{init}}$.
\end{proof}

Let us add some comments on this result.
\autoref{Thm:General_Tract_abs} clearly provides a characterization
for (strong) polynomial tractability.
In comparison to Theorem~5.1 in~\cite{NW08} 
our result yields the essential advantage that the given estimates 
incorporate the initial error $\eps_d^{\mathrm{init}}$.
Hence if~$\eps_d^{\mathrm{init}}$ is sufficiently small 
then we can conclude polynomial tractability 
while ignoring a larger set of eigenvalues in the
summation \link{sup_condition}. 

Observe that the first statement does not cover 
any assertion about the initial error itself, since $f(d)\geq2$. 
Thus it might happen that we have (strong) polynomial tractability 
\wrt the absolute error criterion, though the largest eigenvalue 
$\lambda_{d,1}=(\eps_d^{\rm init})^2$ 
tends faster to infinity than any polynomial.
To give an example, for $d\in\N$ we consider the sequences 
$(\lambda_{d,m})_{m\in\N}$ defined by
\begin{equation*}
				\lambda_{d,1}=e^{2d} 
				\quad \text{and} \quad \lambda_{d,m}=\frac{1}{m} 
				\quad \text{for} \quad m\geq 2.
\end{equation*}
Here, obviously, the initial error grows 
exponentially fast to infinity, but nevertheless the
second point of \autoref{Thm:General_Tract_abs} shows that 
$S$ is strongly polynomially tractable
since \link{sup_condition} holds with $r=p=q=0$, 
and $C=\tau=2$.\\

Next we present an analogue of \autoref{Thm:General_Tract_abs}
for the normalized error criterion.
Again a slightly modified statement can be found in \cite[Theorem 5.2]{NW08}.
\begin{theorem}\label{Thm:General_Tract_norm}
			Assume $S$ to be a problem as in \autoref{Cor:OptAlgo} 
			and consider the normalized error criterion in the worst case setting.
			\begin{itemize}
						\item If $S$ is polynomially tractable 
								with the constants $C,p>0$ and $q\geq 0$ 
								then for all $\tau > p/2$ we have
								\begin{equation}\label{sup_condition_norm}
										C_\tau = \sup_{d\in\N} \frac{1}{d^{r}} 
														   \left( \sum_{i=f(d)}^\infty \left( \frac{\lambda_{d,i}}{\lambda_{d,1}} \right)^\tau \right)^{1/\tau} 
													 < \infty,
								\end{equation}
								where $r=2q/p$ and $f\colon \N \nach \N$ with $f(d)\equiv 1$.
								In this case the bound $C_\tau \leq 2^{1/\tau} (1+C)^{2/p}\, \zeta(2\tau/p)^{1/\tau}$
								holds for any such $\tau$.
						\item If \link{sup_condition} is satisfied for some parameters $r \geq 0$, $\tau >0$ 
								and a	function $f\colon\N \nach \N$ such that 
								$f(d)=\ceil{C\,d^{q}}$, 
								where $C>0$ and $q \geq 0$, then the problem $S$ is polynomially tractable.
								If so, then 
								$n_\no^\wor(\eps,d)\leq (C+C_\tau^\tau)\, \eps^{-2\tau}\, d^{\max{q,r\tau}}$
								for any $\eps\in\left(0,1\right]$ 
								and every $d\in\N$.								
			\end{itemize}
\end{theorem}

\begin{proof}
Due to the strong relation between the absolute and the normalized error criterion,
\ie $n_\no^\wor(\eps,d)=n_\ab^\wor(\eps\cdot\sqrt{\lambda_{d,1}},d)$ for $\eps\in (0,1]$ and $d\in\N$, 
we note that \autoref{Thm:General_Tract_norm} can be shown using
essentially the same arguments an in the proof for \autoref{Thm:General_Tract_abs}.
Indeed, if we replace $\lambda_{d,i}$ by $\lambda_{d,i}/\lambda_{d,1}$ for $i\in\N$
we obtain a scaled problem $T$ with initial error $\eps_d^{\mathrm{init}}=1$.
Now the information complexity of $T$ (\wrt the absolute error criterion) 
equals the information complexity of $S$ \wrt normalized errors.\footnote{For details we refer to the proof of \autoref{thm:unweightedtensor_norm}.}
Following the lines of the proof of \autoref{Thm:General_Tract_abs} 
this shows the second point of \autoref{Thm:General_Tract_norm}, where we set $p=0$.
Moreover, we conclude for any $\tau > p/2$ and $d\in\N$
\begin{equation*}
		\sum_{i=\ceil{(1+C)d^q}}^\infty \left( \frac{\lambda_{d,i}}{\lambda_{d,1}} \right)^\tau
		\leq \left( C^{2/p} d^{2q/p} \right)^\tau \, \zeta\left(\frac{2\tau}{p}\right) 
		\leq (1+C)^{2\tau/p} \, \zeta\left(\frac{2\tau}{p}\right) \, d^{q\cdot 2\tau/p},
\end{equation*}
provided that $S$ is polynomially tractable with the constants $C,p>0$ and $q\geq 0$.
Furthermore, for any $d\in\N$ we have 
\begin{equation*}
		\sum_{i=1}^{\ceil{(1+C)d^q}-1} \left( \frac{\lambda_{d,i}}{\lambda_{d,1}} \right)^\tau 
		\leq \ceil{(1+C)d^q}-1 
		\leq (1+C)\, d^q
		\leq (1+C)^{2\tau/p} \, \zeta\left(\frac{2\tau}{p}\right) \, d^{q\cdot 2\tau/p}
\end{equation*}
since $\lambda_{d,i} \leq \lambda_{d,1}$, $2\tau/p>1$ and $\zeta(2\tau/p)>1$.
Consequently, setting $r=2q/p$ and combining both the previous estimates leads to
\begin{equation*}
		\frac{1}{d^r} \left( \sum_{i=1}^{\infty} \left( \frac{\lambda_{d,i}}{\lambda_{d,1}} \right)^\tau \right)^{1/\tau}
		\leq 2^{1/\tau} \, (1+C)^{2/p} \, \zeta\left(\frac{2\tau}{p}\right)^{1/\tau} 
		\quad \text{for} \quad d\in\N
\end{equation*}
which shows \link{sup_condition_norm}, as well as the claimed bound on $C_\tau$.
\end{proof}

Obviously \autoref{Thm:General_Tract_norm} again provides a 
characterization of (strong) polynomially tractability 
of a given compact Hilbert space problem $S=(S_d)_{d\in\N}$ in terms
of summability properties of the eigenvalue sequence 
$(\lambda_{d,i})_{i\in\N}$ of $W_d={S_d}^{\!\dagger} S_d$.

\section{Tensor product problems}\label{sect:TensorBasics}
In the former section we investigated tractability
properties of compact Hilbert space problems $S=(S_d)_{d\in\N}$
without assuming any relation between subsequent problem operators $S_d$.
Next we want to consider problems $S$ where every 
$S_d$ is generated out of one single (univariate)
operator $S_1$ via a $d$-fold tensor product construction.

\subsection{Definition and simple properties}\label{subsect:def_tensor_prob}
We need to recall the concept of tensor product Hilbert spaces first.
To this end, we use the approach given in Chapter 2.6 of Kadison and Ringrose~\cite{KR83}.
For a comprehensive introduction to more general tensor products in functional analysis 
we refer to the first chapter of Light and Cheney~\cite{LC85} and to Section 1.3 in Hansen~\cite{H10}.

Without going too much into details, we note that given a finite number of 
arbitrary Hilbert spaces~$H^{(k)}$ with inner products $\distr{\cdot}{\cdot}_{H^{(k)}}$, $k=1,\ldots,d$, 
the tensor product space 
\begin{equation*}
			H_d = \bigotimes_{k=1}^d H^{(k)} = H^{(1)} \otimes \ldots \otimes H^{(d)}
\end{equation*}
can be identified\footnote{Note that this association is unique up to some isometric isomorphism.} 
with the closure of the algebraic tensor product $H_{d,0}$, 
with respect to a (reasonable cross) norm
which is induced by a certain inner product $\distr{\cdot}{\cdot}_{H_{d,0}}$.
Keep in mind that the algebraic tensor product is defined 
as the quotient of the \emph{free vector space}, 
\ie the set of all finite linear combinations of formal
objects $f=\bigotimes_{k=1}^d f_k$ with $f_k\in H^{(k)}$, which we call
\emph{simple} (or \emph{pure}) \emph{tensors}, by a suitable linear subspace.\footnote{To abbreviate the notation we do not distinguish between simple tensors and their equivalence classes in what follows.}
Moreover, the mentioned inner product on the algebraic tensor product $H_{d,0}$ is defined by
\begin{equation*}
			\distr{\bigotimes_{k=1}^d f_k}{\bigotimes_{k=1}^d g_k}_{H_{d,0}} = \prod_{k=1}^d \distr{f_k}{g_k}_{H^{(k)}} \quad \text{for} \quad f_k, g_k\in H^{(k)}.
\end{equation*}
By means of continuous (multi-) linear extension this functional 
uniquely determines the inner product 
$\distr{\cdot}{\cdot}_{H_d}$ on $H_d$.
As usual we denote the corresponding norm by~$\norm{ \cdot \sep H_d }$.

Due to the tensor product structure, many useful properties 
such as completeness and separability 
of the underlying spaces $H^{(k)}$ are transferred directly to $H_d$
provided that all the $H^{(k)}$ share them.
In particular, it is well-known how to construct an orthonormal basis 
of the tensor product space given an ONB 
\begin{gather*}
			E^{(k)} = \left\{e_i^{(k)} \in H^{(k)} \sep i \in \I^{(k)} \right\}
\end{gather*}
in each $H^{(k)}$, $k=1,\ldots,d$.
Here every $\I^{(k)}$ denotes a (possibly non-countable) abstract index set.
Then the set of all $d$-fold simple tensors given by
\begin{gather*}
		E_d = \left\{ e_{d,\bm{j}} 
		= \bigotimes_{k=1}^d e_{j_k}^{(k)} \sep \bm{j}=(j_1,\ldots,j_d)\in\I_d 
		= \bigtimes_{k=1}^d \I^{(k)} \right\}
\end{gather*}
builds the desired ONB in $H_d$; see~\cite[Theorem 2.6.4]{KR83}.

For the applications we have in mind we will focus our attention on the special case
where all the building blocks $H^{(k)}$, $k=1,\ldots,d$, of $H_d$
coincide. 
In what follows we therefore assume that 
$H^{(k)} \equiv H_1$ for some Hilbert space $H_1$.
The respective ONB of $H_1$ will be denoted by 
$E_1=\left\{e_i \in H_1 \sep i \in \I_1\right\}$.
Then the latter formula for~$E_d$ simplifies to
\begin{gather}\label{tensor_ONB}
		E_d = \left\{ e_{d,\bm{j}} = \bigotimes_{k=1}^d e_{j_k} \sep \bm{j}=(j_1,\ldots,j_d) \in \I_d = (\I_1)^d \right\}.
\end{gather}

We are ready to introduce the tensor product problem operators $S_d$, $d\geq 1$, 
we are interested in.
Thus let $S_1 \colon \F_1 \nach \G_1$ 
be a compact linear operator between arbitrary Hilbert spaces $\F_1$ and $\G_1$.
For $d\geq 2$ we assume $\F_d=H_d$ to be the $d$-fold tensor product space 
of $H^{(k)} = H_1 = \F_1$, $k=1,\ldots,d$, as explained above.
Analogously, we construct the space $\G_d=\bigotimes_{k=1}^d \G_1$ out of $d$ copies of $\G_1$.
Now Proposition~2.6.12 of~\cite{KR83} yields that there exists 
a uniquely defined linear operator $S_d =\bigotimes_{k=1}^d S_1 \colon \F_d \nach \G_d$ such that
\begin{gather*}
		S_d\left(\bigotimes_{k=1}^d f_k\right) = \bigotimes_{k=1}^d S_1 f_k, \quad f_k \in \F_1,
\end{gather*}
and we have $\norm{ S_d \sep \LO(\F_d,\G_d)}=\norm{S_1 \sep \LO(\F_1,\G_1)}^d<\infty$ for any fixed $d\in\N$.
In detail, we define the bounded linear operator 
$\widetilde{S}_d\colon E_d\nach\G_d$ such that for all $\bm{j}\in\I_d$ we
have $\widetilde{S}_d (e_{d,\bm{j}}) = \widetilde{S}_d (\bigotimes_{k=1}^d e_{j_k}) = \bigotimes_{k=1}^d S_1 (e_{j_k})\in\G_d$. 
Then $S_d$ is assumed to be the uniquely
defined linear, continuous extension of $\widetilde{S}_d$ from $E_d$ to $\F_d$.
Due to the compactness of~$S_1$ it is easy to check that 
the problem operator~$S_d$ is not only bounded but even compact.
Moreover, a linear extension argument shows that the adjoint 
operator~${S_d}^{\!\dagger}$ is given by the $d$-fold tensor product of ${S_1}^{\!\dagger}$, 
\ie ${S_d}^{\!\dagger} = \bigotimes_{k=1}^d \left( {S_1}^{\!\dagger} \right)$, and hence
\begin{gather}\label{eq:tensor_W}
		W_d = {S_d}^{\!\dagger} S_d 
		= \left( \bigotimes_{k=1}^d \left({S_1}^{\!\dagger}\right) \right) \left( \bigotimes_{k=1}^d {S_1} \right) 
		= \bigotimes_{k=1}^d \left( {S_1}^{\!\dagger} {S_1}\right)
		= \bigotimes_{k=1}^d W_1;
\end{gather}
cf. \cite[p. 146]{KR83}.

\subsection{Eigenpairs and the optimal algorithm}\label{sect:Eigenpairs}
From \autoref{sect:General_HSP} we know that for $d\in\N$ the optimal algorithm, 
as well as the (information) complexity,
crucially depends on the singular value decomposition of $S_d$. 
Hence, we have to calculate the eigenpairs 
$(\lambda_{d,i}, \phi_{d,i})$ of the tensor product operator $W_d$ obtained in \link{eq:tensor_W}.
We follow the arguments presented in \cite[Section 5.2]{NW08} 
and claim that these eigenpairs are given by (tensor) products 
of the non-trivial eigenpairs $(\lambda_m,\phi_m)$, $m\in\M_1=\{m\in\N\sep m<v(W_1)+1\}$, 
of the univariate operator $W_1={S_1}^{\!\dagger} S_1$; see~\link{univariate_eigenvalues}.
This is the subject of the following assertion.

\begin{prop}\label{prop:tensor_eigenpairs}
For $d\in\N$ the non-trivial eigenpairs of the operator $W_d={S_d}^{\!\dagger} {S_d}$ are given by
$\left\{\left(\widetilde{\lambda}_{d,\bm{m}}, \widetilde{\phi}_{d,\bm{m}}\right) \sep \bm{m}=(m_1,\ldots,m_d) \in \M_d=(\M_1)^d\right\}$,
where
\begin{equation}\label{tensor_eigenpairs}
		\widetilde{\lambda}_{d,\bm{m}} = \prod_{k=1}^d \lambda_{m_k} 
		\quad \text{and} \quad \widetilde{\phi}_{d,\bm{m}} = \bigotimes_{k=1}^d \phi_{m_k}.
\end{equation}
\end{prop}

\begin{proof}
Obviously, all the $\widetilde{\phi}_{d,\bm{m}}$'s are mutually orthonormal in $\F_d$, \ie
\begin{gather*}
		\distr{\widetilde{\phi}_{d,\bm{i}}}{\widetilde{\phi}_{d,\bm{j}}}_{\F_d} 
		= \prod_{k=1}^d \distr{\phi_{i_k}}{\phi_{j_k}}_{\F_1} 
		= \prod_{k=1}^d \delta_{i_k,j_k}
		= \delta_{\bm{i},\bm{j}}, \quad \bm{i},\bm{j}\in\M_d.
\end{gather*}
Furthermore,
\begin{align*}
		W_d \widetilde{\phi}_{d,\bm{m}}
		&= \left(\bigotimes_{k=1}^d W_1 \right) \left( \bigotimes_{k=1}^d \phi_{m_k} \right)
		= \bigotimes_{k=1}^d \left( W_1 \phi_{m_k} \right) \\
		&= \bigotimes_{k=1}^d \left( \lambda_{m_k} \cdot \phi_{m_k} \right)
		= \prod_{k=1}^d \lambda_{m_k} \cdot \bigotimes_{k=1}^d \phi_{m_k}
		= \widetilde{\lambda}_{d,\bm{m}} \cdot \widetilde{\phi}_{d,\bm{m}}
\end{align*}
shows that $\widetilde{\phi}_{d,\bm{m}}$, $\bm{m}\in\M_d$, 
is indeed an eigenelement with respect to the strictly positive
eigenvalue~$\widetilde{\lambda}_{d,\bm{m}}$ of $W_d$.

Assume for a moment there exists an eigenpair $(\mu,\eta)$ of $W_d$ with $\mu \neq 0$
which cannot be represented by \link{tensor_eigenpairs}.
Then, due to the assertions in the former section, 
$\eta$ is orthogonal to every other eigenelement $\widetilde{\phi}_{d,\bm{m}}$, $\bm{m}\in\M_d$.
Remember that $\Phi_1=\{\phi_m \in\F_1 \sep m\in\M_1\}$ can be extended to 
an orthonormal basis $E_1=\{e_m \sep m \in \I_1 \}$ of $\F_1$ (see \autoref{rem:sep})
which can be used to construct an ONB
$E_d=\{e_{d,\bm{j}} \sep \bm{j} \in \I_d=(\I_1)^d \}$ of $\F_d$ given by \link{tensor_ONB}.
Therefore $\eta$ can be represented as
\begin{align*}
		\eta 
		&= \sum_{\bm{j}\in\I_d} \distr{\eta}{e_{d,\bm{j}}}_{\F_d} e_{d,\bm{j}} 
		= \sum_{\bm{j}\in\M_d} \distr{\eta}{\widetilde{\phi}_{d,\bm{j}}}_{\F_d} \widetilde{\phi}_{d,\bm{j}} + \sum_{\bm{j}\in \I_d\setminus \M_d} \distr{\eta}{e_{d,\bm{j}}}_{\F_d} e_{d,\bm{j}} \\
		&= \sum_{\bm{j}\in \I_d\setminus \M_d} \distr{\eta}{e_{d,\bm{j}}}_{\F_d} e_{d,\bm{j}},
\end{align*}
where each of these sums consists of at most countably many non-vanishing summands and converges unconditionally.
Now the boundedness of $S_d$ implies
\begin{gather*}
		S_d \eta 
		= \sum_{\bm{j}\in \I_d\setminus \M_d} \distr{\eta}{e_{d,\bm{j}}}_{\F_d} S_d e_{d,\bm{j}} 
		= 0,
\end{gather*}
since each of the tensor products 
$S_d e_{d,\bm{j}} = \bigotimes_{k=1}^d (S_1 e_{j_k})$, $\bm{j}\in\I_d \setminus \M_d$, 
includes at least one factor $S_1 e_{j_k}$ with $j_k \notin \M_1$.
These factors need to vanish because the set $\{S_1 e_m = S_1 \phi_m \sep m\in\M_1\}$ 
builds an ONB of the image of $S_1$ in $\G_1$.
Hence, $W_d \eta = {S_d}^{\!\dagger} (S_d \eta) = 0$ which contradicts our assumption.
In other words, \link{tensor_eigenpairs} completely describes the eigenpairs of $W_d$ as claimed.
\end{proof}

Again the latter proof justifies 
the restriction to separable spaces $\F_1$ (and hence also~$\F_d$) in what follows,
see \autoref{rem:sep}.
Thus we can assume that the set of univariate eigenelements $\Phi_1$ already builds an ONB in $\F_1$,
\ie that $\Phi_1=E_1$, and
consequently $\Phi_d=\{\widetilde{\phi}_{d,\bm{m}} \sep \bm{m}\in\M_d\}$ builds an ONB in $\F_d$.

To unify our notation we rearrange the obtained eigenpairs 
according to a non-increasing ordering of the eigenvalues.
To this end, note that $\# \M_d = (\# \M_1)^d$, \ie we have $v(W_d)=v(W_1)^d$
strictly positive eigenvalues in dimension $d$.
Therefore we define a sequence of bijections $\psi=\psi_d \colon \{ i \in \N \sep i < v(W_1)^d+1\} \nach \M_d $
such that
\begin{gather*}
		\lambda_{d,i} = \widetilde{\lambda}_{d,\psi(i)} \geq \widetilde{\lambda}_{d,\psi(i+1)} \quad \text{for all} \quad 1\leq i < v(W_1)^d+1.
\end{gather*}
Consequently the corresponding eigenelements are denoted by $\phi_{d,i}=\widetilde{\phi}_{d,\psi(i)}$.
Similar to the definitions in \autoref{sect:SVD} we extend the (possibly finite) sequence of eigenvalues by
$\lambda_{d,i}=0$ for $i>v(W_1)^d$.
Observe that the largest eigenvalue in dimension $d$ is given by
\begin{gather*}
		\lambda_{d,1}=\lambda_1^d 
\end{gather*}
and thus the initial error is $\eps_d^\mathrm{init} = \lambda_1^{d/2}$.

\autoref{prop:tensor_eigenpairs} in hand, the optimal algorithm $A_{n,d}^*$ for
linear tensor product problems $S=(S_d)_{d\in\N}$ is stated in \autoref{Cor:OptAlgo}.
For $d\in\N$ and $n\in\N_0$ it reads
\begin{equation}\label{eq:opt_tensor_algo}
		A_{n,d}^* \colon \F_d \nach \G_d, 
		\qquad f \mapsto A_{n,d}^*(f)=\sum_{i=1}^{\min{n,v(W_d)}} \distr{f}{\phi_{d,i}}_{\F_d} \cdot S_d \phi_{d,i},
\end{equation}
and its worst case error can be expressed in terms of the sequence $\lambda=(\lambda_m)_{m\in\N}$.
More precisely, we have $e^{\mathrm{wor}}(n,d;S_d) = \Delta^{\mathrm{wor}}(A_{n,d}^*;S_d) = \sqrt{\lambda_{d,n+1}}$.

We are ready to characterize tractability of such problems in the next subsection.

\subsection{Complexity}\label{sect:tensor_complexity}
We begin by analyzing the information complexity with respect to
the absolute error criterion.
Let $S_1\colon \F_1 \nach \G_1$ 
denote a compact linear operator between
arbitrary Hilbert spaces $\F_1$ and $\G_1$ and let $S=(S_d)_{d\in\N}$
be the sequence of $d$-fold tensor product problems 
defined in \autoref{subsect:def_tensor_prob}.
As before the non-increasing sequence of 
non-negative eigenvalues of the univariate operator
$W_1={S_1}^{\!\dagger} S_1$ is denoted by $\lambda=(\lambda_m)_{m\in\N}$.
At this point we stress that it is reasonable to assume that $\lambda_2>0$.
Otherwise for every $d\in\N$ there would be only at most one 
non-vanishing $d$-dimensional eigenvalue of
$W_d={S_d}^{\!\dagger} S_d$. 
Hence the problem $S_d$ would be trivial since then 
$n_{\mathrm{abs}}^{\mathrm{wor}}(\epsilon,d) \leq 1$ for all $\epsilon>0$.
Note that $\lambda_2>0$ also implies $\lambda_1>0$ such that
$S_1$ and $S_d$ are not the zero operator.

We proceed by presenting an assertion which is mainly based on 
Theorem 5.5 in Novak and Wo{\'z}niakowski~\cite{NW08}.
The sufficient condition for weak tractability later was given by
Papageorgiou and Petras~\cite{PP09}.
Although the results of these authors only refer to linear tensor product problems
defined between Hilbert \emph{function} spaces they remain valid even in our more general setting.

\begin{theorem}\label{thm:unweightedtensor_abs}
Consider the problem $S=(S_d)_{d\in\N}$ as described before. 
We study the absolute error criterion in the worst case setting.
\begin{itemize}
		\item Let $\lambda_1>1$. Then $S$ suffers from the curse of dimensionality.
		\item Let $\lambda_1=1$. Then
				\begin{itemize}
						\item $S$ is polynomially intractable. 
									In particular, if $\lambda_2=1$ then $S$ suffers from the curse of dimensionality.
						\item $S$ is weakly tractable if and only if 
											$\lambda_2<1$ and $\lambda_n \in o(\ln^{-2} n)$, 
											as $n\nach \infty$.\footnote{To avoid possible confusions, here and in what follows, $\ln^{\alpha} n$ means $[\ln(n)]^{\alpha}$ where $\alpha\in\R$.}
				\end{itemize}
		\item Let $\lambda_1<1$. Then
				\begin{itemize}
						\item $S$ never suffers from the curse.
						\item $S$ is weakly tractable if and only if $\lambda_n \in o(\ln^{-2} n)$, as $n\nach \infty$.
						\item	$S$ is polynomially tractable if and only if it is strongly polynomially tractable.
									Moreover, this holds if and only if there exists some $\tau\in(0,\infty)$ such that
									$\lambda \in \l_{\tau}$ and the exponent of strong polynomial tractability is given by
									\begin{gather*}
												p^* = \inf\left\{2\tau \sep \sum_{m=1}^\infty \lambda_m^{\tau}\leq 1\right\}.
									\end{gather*}
				\end{itemize}
\end{itemize}
\end{theorem}

For the sake of completeness we mention that Theorem 5.5 in \cite{NW08} includes some additional
lower bounds on the information complexity in the case $\lambda_1\geq 1$.
For polynomial (in)tractability the main idea of the proof is to apply \autoref{Thm:General_Tract_abs}
and to use the product structure of the involved sequences $(\lambda_{d,i})_{i\in\N}$ which are essentially given by
\autoref{prop:tensor_eigenpairs}.
We will not provide an explicit proof here.
Instead the interested reader is referred to \autoref{ex:unweighted_tensorproducts} 
in \autoref{chapt:ScaledNorms} where we conclude 
all assertions stated in \autoref{thm:unweightedtensor_abs} 
out of a generalized result for \emph{scaled} tensor product problems.
To conclude these more general assertions we will exactly follow the mentioned proof sketch.

Many authors in IBC use phrases like ``(unweighted) tensor product problems are intractable".
In this regard they refer to the following Theorem for the normalized error criterion
which is essentially based on Theorem 5.6 of~\cite{NW08}, as well as on~\cite{PP09}.
From our point of view it is not more than a simple consequence of the assertions
for absolute errors.

\begin{theorem}\label{thm:unweightedtensor_norm}
Consider the problem $S=(S_d)_{d\in\N}$ as described above. 
We study the normalized error criterion in the worst case setting.
\begin{itemize}
		\item Let $\lambda_1=\lambda_2$. Then $S$ suffers from the curse of dimensionality.
		\item Let $\lambda_1>\lambda_2$. Then
				\begin{itemize}
						\item $S$ is weakly tractable if and only if 
											$\lambda_n \in o(\ln^{-2} n)$, as $n\nach \infty$.
						\item $S$ is polynomially intractable.
				\end{itemize}
\end{itemize}
\end{theorem}

Since the subsequent proof technique is typical in this field of research, 
we include the proof of \autoref{thm:unweightedtensor_norm} in full detail.
\begin{proof}
Assume we had already proven \autoref{thm:unweightedtensor_abs}.
Given the problem $S=(S_d)_{d\in\N}$, constructed out of $S_1\colon\F_1\nach\G_1$,
as well as the associated sequence $(\lambda_m)_{m\in\N}$,
we define a new operator $T_1\colon\F_1\nach\G_1$ by 
$f\mapsto T_1 f = 1/\sqrt{\lambda_1} \cdot S_1 f$.
Clearly, $T_1$ is a linear and compact mapping between Hilbert spaces
and 
\begin{gather*}
		\distr{T_1 f}{g}_{\G_1} 
		= \frac{1}{\sqrt{\lambda_1}} \distr{S_1 f}{g}_{\G_1}
		= \frac{1}{\sqrt{\lambda_1}} \distr{f}{{S_1}^{\!\dagger} g}_{\F_1}
		= \distr{f}{\left( \frac{1}{\sqrt{\lambda_1}} {S_1}^{\!\dagger}\right) g}_{\F_1}
\end{gather*}
for $f\in\F_1$ and $g\in\G_1$.
Hence, ${T_1}^{\!\dagger} = 1/\sqrt{\lambda_1} \cdot {S_1}^{\!\dagger}$ and
the (extended) eigenvalue sequence of $V_1=W_1(T)={T_1}^{\!\dagger} T_1 = 1/\lambda_1 \cdot {S_1}^{\!\dagger} S_1 = 1/\lambda_1 \cdot W_1(S)$
is given by $\mu=(\mu_m)_{m\in\N}$, where $\mu_m = \lambda_m / \lambda_1$ for $m\in\N$.
For details, see also the arguments used in \autoref{sect:Scaled_basics}.
Anyway, the mapping $T_1$ in hand, we can construct the tensor product problem
$T=(T_d)_{d\in\N}$ by the usual procedure.
Now \link{tensor_eigenpairs} in \autoref{prop:tensor_eigenpairs} shows that the corresponding
eigenvalues of $V_d=W_d(T)={T_d}^\dagger T_d$ are given by 
\begin{gather*}
		\widetilde{\mu}_{d,\bm{m}} 
		= \prod_{k=1}^d \mu_{m_k}  
		= \frac{1}{\lambda_1^d} \prod_{k=1}^d \mu_{m_k} 
		= \frac{1}{\lambda_{d,1}} \widetilde{\lambda}_{d,\bm{m}}, \quad \bm{m}=(m_1,\ldots,m_d)\in\M_d,
\end{gather*}
such that $\mu_{d,i} = 1/\lambda_{d,1} \cdot \lambda_{d,i}$ for $i\in\N$.
This yields that the information complexity of $S$ \wrt the normalized error criterion coincides with the
absolute information complexity of the (scaled) problem $T$,
\ie
\begin{align*}
		n_\no(\eps',d; S_d)
		&= \# \{n\in\N_0 \sep \lambda_{d,n}/\lambda_{d,1} > (\eps')^2\}
		= \# \{n\in\N_0 \sep \mu_{d,n} > (\eps')^2\}\\
		&= n_\ab(\epsilon',d; T_d)
\end{align*}
for all $\eps'\in(0,1]$ and each $d\in\N$.
Since $\mu_1=1\geq\mu_2>0$ we are allowed to apply \autoref{thm:unweightedtensor_abs} for $T$.
Finally the observations that $\mu_2 = 1$ if and only if $\lambda_1=\lambda_2$,
as well as that $\mu_n\in o(\ln^{-2} n)$ (as $n\nach\infty$) if and only if $\lambda=(\lambda_n)_{n\in\N}$ belongs to this class, 
complete the proof.
\end{proof}

\section{Reproducing kernel Hilbert spaces}\label{sect:RKHS}
When we deal with problems defined on Hilbert \emph{function} spaces $H$
a special kind of Hilbert spaces is of particular interest.
The reason is that in practice often only function evaluations
rather than information obtained by arbitrary linear functionals are permitted.
In order to compare the power of these classes of information operations 
($\Lambda^{\mathrm{std}}$ vs. $\Lambda^\all$) from a theoretical point of view,
it seems to be useful to investigate conditions which ensure
that point evaluation functionals
\begin{gather*}
		L_y \colon H \nach \R, \quad f\mapsto L_y(f)=f(y),
\end{gather*}
for all $y$ in the domain of definition $\Omega$ of $f\in H$, belong to the class $\Lambda^\all$.
Clearly~$L_y$ is always linear such that it is enough to ask whether it is also continuous (or bounded, respectively) in $f$.
It turns out that, as long as we restrict ourselves to Hilbert spaces,
this property can be characterized by the existence of a so-called
\emph{reproducing kernel} $K$. If so, then the space $H$ is referred to as a
\emph{reproducing kernel Hilbert space} (\emph{RKHS} for short) and we write $H = \Hi(K)$.
In the present section we collect some basic properties of this concept which we will need later on in \autoref{sect:weightedSobolev}.
The presentation given here is based on the famous paper of Aronszajn~\cite{A50},
as well as the textbook of Wahba~\cite{W90}.\footnote{For the ease of notation (and in contrast to the mentioned references) we restrict ourselves to spaces over $\R$. Once more the theory can be transferred almost literally to $\C$.}
Standard examples for RKHSs such as Korobov spaces 
and Sobolev spaces of dominating mixed smoothness
can be found in \cite[Appendix A]{NW08}.

\subsection{Definition and properties}
A (real) Hilbert space $H$ of functions $f\colon \Omega \nach \R$, 
equipped with inner product $\distr{\cdot}{\cdot}_{H}$, is said to be a reproducing kernel Hilbert space if there exists a function
\begin{gather*}
		K \colon \Omega \times \Omega \nach \R
\end{gather*}
such that
\begin{itemize}
		\item for all fixed $y\in\Omega$ the function $K_y = K(\cdot,y)$ belongs to $H$, and
		\item for every $f \in H$ and all $y\in\Omega$ it is
					\begin{gather}\label{eq:reproducing_prop}
								L_y(f) = f(y) = \distr{f}{K_y}_{H} = \distr{f}{K(\cdot,y)}_{H}.
					\end{gather}		
\end{itemize}
The second point \link{eq:reproducing_prop} is known as the \emph{reproducing property}.
Together with the first point it obviously implies the boundedness of point evaluations on $H=\Hi(K)$.
The converse, \ie the existence (and uniqueness) of the reproducing kernel $K$, 
is a simple consequence of the Riesz representation theorem; see \cite[p.~90]{T92} or \cite[III.6]{Y80}.
Unfortunately the proof of this theorem is non-constructive and therefore it does not
provide an explicit method to find the \emph{representer} $K_y = K(\cdot,y)$ of $L_y$.
In fact, given a specific RKHS $\Hi(K)$ it seems to be a challenging problem to deduce a closed form of
its reproducing kernel $K$. 
However, as long as we restrict ourselves to separable RKHSs, 
it is easy to prove that $K$ is given by
\begin{gather}\label{eq:kernel_sum_basis}
			K(x,y) = \sum_{m\in\I} e_m(x) \, e_m(y), 
			\quad x,y\in\Omega,
\end{gather}
where $\{e_m \colon \Omega \nach \R \sep m\in\I \}$ denotes an arbitrary 
orthonormal basis of $\Hi(K)$.
Furthermore we know that every reproducing kernel $K$ is positive definite.
That is, for all $n\in\N$ and any sequence $\bm{x}=(x_m)_{m=1}^n \in \Omega^n$ the quadratic form
\begin{gather}\label{eq:quad_form}
			Q_{K;\bm{x}}(\xi_1,\ldots,\xi_n) 
			= \sum_{i,j=1}^n K(x_i,x_j) \, \xi_i \, \xi_j, 
			\quad \bm{\xi}=(\xi_m)_{m=1}^n \in \R^n,
\end{gather}
is a non-negative function of $\bm{\xi}$.
In particular,
\begin{gather*}
			K(x,x) \geq 0
			\quad
			\text{and}
			\quad
			K(x,y) = K(y,x)
			\quad
			\text{for all}
			\quad x,y\in \Omega.
\end{gather*}
Conversely, Moore showed that every positive definite function $K$ in the above sense
uniquely determines a RKHS $H$ admitting $K$ as its reproducing kernel; see~\cite{A50}.
Again it turned out to be a hard problem to conclude a suitable representation
of~$H$ (and its inner product) for a given function $K$.

Besides further fascinating properties, we want to focus our attention on
products of kernel functions.
To this end, for $d\in\N$ let $K^{(k)}$, $k=1,\ldots,d$, denote a finite number of reproducing kernels defined on the sets $\Omega^{(k)} \times \Omega^{(k)}$, respectively.
Then we may consider the tensor product
\begin{gather*}
		K_d=\bigotimes_{k=1}^d K^{(k)} \colon \left(\bigtimes_{k=1}^d \Omega^{(k)} \right) \times \left(\bigtimes_{k=1}^d \Omega^{(k)} \right) \nach \R,
		\quad K_d(\bm{x},\bm{y}) = \prod_{k=1}^d K^{(k)}(x_k,y_k),
\end{gather*}
where we set $\bm{x}=(x_1,\ldots,x_d)$ and $\bm{y}=(y_1,\ldots,y_d)$ with $x_k,y_k\in\Omega^{(k)}$.
On the other hand, each kernel $K^{(k)}$ induces a uniquely defined RKHS $\Hi(K^{(k)})$ which in turn
implies the existence of one (and only one) tensor product space $H_d = \bigotimes_{k=1}^d \Hi(K^{(k)})$
using the arguments presented in \autoref{sect:TensorBasics}.
Now it can be checked that $H_d$ itself is a RKHS and its kernel is given by $K_d$, \ie 
\begin{gather}\label{eq:tensor_RKHS}
		H_d = \Hi(K_d) = \bigotimes_{k=1}^d \Hi(K^{(k)}).
\end{gather}
The proof of this assertion can be obtained inductively by adding one factor in every step.
Then it remains to show that the resulting quadratic forms \link{eq:quad_form} are non-negative again which
can be done using a classical result due to Schur.

Note that the whole theory works for arbitrary point sets $\Omega$
which turned out to be useful in the context of so-called \emph{support vector machines}
which are instances of the more general class of \emph{kernel methods}.
However, in IBC special choices such as $\Omega=\Omega_1=[0,1]$ (or $\Omega=\R$) are of particular interest.
For multivariate problems the standard choice is $\Omega=\Omega_d=\Omega_1^d$ which perfectly fits to the tensor product
construction explained before.
In this respect the univariate kernels $K^{(k)}$, $k=1,\ldots,d$, are often taken as weighted instances $K_1^{\gamma_{d,k}}$
of some underlying kernel~$K_1$. 
A prominent example is given by $K_1^{\gamma_{d,k}}(x,y)=1+\gamma_{d,k}\min{x,y}$ which leads to
an anchored Sobolev space $\widetilde{\Hi}_d^\gamma$ related to the Wiener sheet measure; see, e.g.,
\cite[Section~8]{KWW08} or \cite{W12}.
Another example of this type will be discussed in detail within \autoref{sect:weightedSobolev}.

Finally we mention that the concept of RKHSs was generalized recently to
the class of so-called \emph{reproducing kernel Banach spaces} (\emph{RKBSs}).
For a brief introduction to this topic we refer to Zhang and Zhang~\cite{ZZ13}.

\subsection{Examples: Integration and approximation problems}
Let us conclude the presentation with some examples which show that
the knowledge about the existence of a reproducing kernel can be exploited
to obtain complexity assertions for the classical problems of integration and approximation.
\begin{example}[Worst case error of QMC rules]
For $d\in\N$ suppose $\Hi(K_d)$ to be a RKHS of real-valued functions $f$ defined on 
some Borel measurable subset $\Omega_d$ of~$\R^d$.
Consider the solution operator of the integration problem
\begin{gather*}
			\mathrm{Int}_d^{\rho_d} \colon \B(\Hi(K_d)) \nach \R,
			\quad
			f\mapsto \mathrm{Int}_d^{\rho_d} f = \int_{\Omega_d} f(\bm{x}) \, \rho_d(\bm{x}) \dlambda^d(\bm{x}),
\end{gather*}
where $\rho_d$ denotes a probability density function on $\Omega_d$.
Let us additionally assume that the function
\begin{gather*}
			h_d = \int_{\Omega_d} K_d(\cdot,\bm{x}) \, \rho_d(\bm{x}) \dlambda^d(\bm{x})
\end{gather*}
is well-defined and belongs to $\Hi(K_d)$.
Then it is easy to see that $h_d$ is the representer of the linear functional $\mathrm{Int}_d^{\rho_d}$,
\ie that $\mathrm{Int}_d^{\rho_d}f=\distr{f}{h_d}_{\Hi(K_d)}<\infty$ for all $f\in\Hi(K_d)$.
Since allowing arbitrary linear functionals to approximate
the value of the integral would make the problem trivial
we consider cubature rules of the form
\begin{gather*}
			A_{n,d}f = \sum_{i=1}^n a_i \, f\!\left(\bm{x}^{(i)}\right), \quad n\in\N_0,
\end{gather*}
defined by a priori chosen sample points $\bm{x}^{(i)}\in\Omega_d$
and some weights $a_i\in\R$, $i=1,\ldots,n$.
Due to \link{eq:reproducing_prop} also the linear operator $A_{n,d}$ possesses a representer
in the space $\Hi(K_d)$.
Consequently its worst case error can be computed exactly in terms of the reproducing kernel
and the parameters $(a_i)_{i=1}^n$ and $(\bm{x}^{(i)})_{i=1}^n$:
\begin{align*}
			&\Delta^\wor(A_{n,d}; \mathrm{Int}_d^{\rho_d} \colon \B(\Hi(K_d))\nach\R)^2 \\
			&\quad = \sup_{f\in\B(\Hi(K_d))} \abs{\left(\mathrm{Int}_d^{\rho_d}-A_{n,d}\right)(f)}^2 
			= \norm{h_d - \sum_{i=1}^n a_i \, K_d\!\left(\cdot, \bm{x}^{(i)}\right) \sep \Hi(K_d)}^2\\
			&\quad = \int_{\Omega_d^2} K_d(\bm{x},\bm{y})\, \rho_d(\bm{x}) \, \rho_d(\bm{y}) \dlambda^{2d}(\bm{x},\bm{y})
			-2 \sum_{i=1}^n a_i\int_{\Omega_d} K_d\!\left(\bm{x},\bm{x}^{(i)}\right) \, \rho_d(\bm{x}) \dlambda^{d}(\bm{x})\\
			&\quad\qquad+ \sum_{i,j=1}^n a_i\, a_j \, K_d\!\left(\bm{x}^{(i)},\bm{x}^{(j)}\right).
\end{align*}
Choosing special weights $a_i$ (such as $a_i \equiv 1/n$), as well as
specific sample points~$\bm{x}^{(i)}$ (e.g. from a so-called \emph{integration lattice}),
we end up with well-studied classes of cubature rules which are known as
\emph{quasi-Monte Carlo} (QMC) methods and \emph{lattice rules}, respectively.
The common feature of these integration schemes is that their complexity analysis
is mainly based on the presented worst case error formula and thus on the properties of the reproducing kernel $K_d$.
Moreover, the latter expression for $\Delta^\wor(A_{n,d})$ plays an important role in 
\emph{discrepancy theory}.

Various kinds of integration problems are studied in Novak and Wo{\'z}niakowski \cite{NW10}.
For the recent state of the art in discrepancy theory and QMC methods we refer the reader to the monograph of Dick and Pillichshammer~\cite{DP10}, as well as to
the survey article of Dick, Kuo and Sloan~\cite{DKS13} and the references therein.
An introduction to lattice rules can also be found in the textbook of Sloan and Joe~\cite{SJ94}.~\hfill$\square$
\end{example}

Our second example shows the relation of reproducing kernels
and the singular values for certain approximation operators.
\begin{example}[Weighted $\L_2$-approximation]\label{ex:L_2-approx}
For $d\in\N$ let $\Hi(K_d)$ be a separable and infinite-dimensional 
RKHS which is compactly embedded into $\L_2^{\rho_d}(\Omega_d)$.
Here $\rho_d$ again denotes some probability density on $\Omega_d\subseteq\R^d$.
Then we may study the approximation problem
\begin{gather}\label{eq:def_APP}
			\mathrm{App}_d^{\rho_d} \colon \B(\Hi(K_d))\nach \L_2^{\rho_d}(\Omega_d), 
			\quad f\mapsto \mathrm{App}_d^{\rho_d}f = f,
\end{gather}
in the worst case setting.
Since both source and target space are Hilbert spaces 
we can use the theory developed in \autoref{sect:General_HSP}
to conclude complexity results with respect to the class $\Lambda^\all$.
Therefore we need to analyze the eigenvalues of the compact operator
$W_d^{\rho_d} = \left(\mathrm{App}_d^{\rho_d}\right)^\dagger \mathrm{App}_d^{\rho_d}$.
Using the reproducing property~\link{eq:reproducing_prop} and the symmetry of $K_d$ we conclude
\begin{align*}
			\left(W_d^{\rho_d} f\right)(\bm{x})
			&= \distr{\left(\mathrm{App}_d^{\rho_d}\right)^\dagger \mathrm{App}_d^{\rho_d} f}{K_d(\cdot,\bm{x})}_{\Hi(K_d)}
			= \distr{\mathrm{App}_d^{\rho_d} f}{\mathrm{App}_d^{\rho_d} K_d(\cdot,\bm{x})}_{\L_2^{\rho_d}(\Omega_d)} \\
			&= \int_{\Omega_d} f(\bm{y}) \, K_d(\bm{x},\bm{y}) \, \rho_d(\bm{y}) \dlambda^d(\bm{y})
\end{align*}
for all $f\in\Hi(K_d)$ and any $\bm{x}\in\Omega_d$.
Hence, $W_d^{\rho_d}$ takes the form of a weighted integral operator against the kernel $K_d(\cdot,\bm{y})$
and its non-trivial eigenpairs 
$\{(\lambda_{d,\rho_d,i},\phi_{d,\rho_d,i})\sep i\in \M_d\}$ can be found by solving integral equations.
Formula \link{eq:kernel_sum_basis} yields that
\begin{gather*}
			K_d(\bm{x},\bm{x}) = \sum_{i\in\M_d} \phi_{d,\rho_d,i}(\bm{x})^2 < \infty
\end{gather*}
for every $\bm{x}\in\Omega_d$ because we know that $\{\phi_{d,\rho_d,i}\sep i\in \M_d\}$
forms an ONB in $\Hi(K_d)$.
Since $\phi_{d,\rho_d,i}\in \L_2^{\rho_d}(\Omega_d)$ and $\norm{\phi_{d,\rho_d,i} \sep \L_2^{\rho_d}(\Omega_d)}^2 =\lambda_{d,\rho_d,i}$ for all $i\in\M_d$, 
it easily follows that
\begin{gather}\label{eq:trace}
			\trace{W_d^{\rho_d}} 
			= \sum_{i\in\M_d} \lambda_{d,\rho_d,i}
			= \sum_{i\in\M_d} \norm{\phi_{d,\rho_d,i} \sep \L_2^{\rho_d}(\Omega_d)}^2
			= \int_{\Omega_d} K_d(\bm{x},\bm{x}) \, \rho_d(\bm{x}) \dlambda^d(\bm{x}).
\end{gather}
Note that this trace may be finite or infinite 
depending on the values of $K_d$ on the diagonal
$\{(\bm{x},\bm{x}) \sep \bm{x}\in\Omega_d\}$.
It turns out that an infinite trace implies
that there is, in general, no (non-trivial) relation of the power of $\Lambda^{\all}$ and $\Lambda^{\mathrm{std}}$
for the given approximation problem.
In contrast, it is known that for finite traces there exist close relations
of these classes of information operations.
In particular, it is possible to conclude bounds on the rate of convergence
for $\Lambda^{\mathrm{std}}$ out of corresponding bounds for $\Lambda^{\all}$.
For details we refer to \cite[Chapter 26]{NW12}.

Finally we note that the finite trace property of $W_d^{\rho_d}$
immediately implies $\lambda_{d,\rho_d,i}\in \0(i^{-1})$, as $i\nach\infty$.
Hence, if we deal with linear information then we can conclude 
$e^\wor(n,d; \mathrm{App}_d^{\rho_d}) \in \0(n^{-1/2})$, $n\nach\infty$,
directly out of an integrability property of the kernel $K_d$.
\hfill$\square$
\end{example}

In the last example we present a useful relation of
reproducing kernels and average case approximation problems.
\begin{example}[Average case approximation]\label{ex:average_case_approx}
For $d\in\N$ assume $\rho_d$ to be some probability density function on 
$\Omega_d=[0,1]^d$ and let 
$K_d\colon \Omega_d\times\Omega_d\nach\R$ denote a reproducing kernel
such that the mapping $\bm{x}\mapsto K_d(\bm{x},\bm{x})$ belongs to $\L_1^{\rho_d}(\Omega_d)$.
That is, suppose \link{eq:trace} to be finite.
Furthermore, let $\F_d$ denote a separable Banach space of real-valued
functions on $\Omega_d$ which is continuously embedded into $\L_2^{\rho_d}(\Omega_d)$
and for which function evaluations are continuous.
We equip $\F_d$ with a zero-mean Gaussian measure~$\mu_d$ 
such that its correlation operator $C_{\mu_d} \colon \F_d^* \nach \F_d$ 
applied to point evaluation functionals $L_{\bm{x}}$ can be expressed in terms of $K_d$:
\begin{gather*}
			K_d(\bm{x},\bm{y}) 
			= L_{\bm{x}} (C_{\mu_d}L_{\bm{y}})
			= \int_{\F_d} f(\bm{x})\, f(\bm{y}) \,\mathrm{d} \mu_d(f)
			\quad \text{for all} \quad \bm{x},\bm{y}\in\Omega_d.
\end{gather*}
We stress the point that this is always possible for a suitable choice of $\F_d$ and that our
assumptions imply a continuous embedding of the RKHS $\Hi(K_d)$ (induced by $K_d$) into $\F_d$. Consequently also $\mathrm{App}_d^{\rho_d}\colon \Hi(K_d)\nach \L_2^{\rho_d}(\Omega_d)$ is bounded, \ie continuous; see \link{eq:def_APP}.
For details and concrete examples 
the reader is referred to \cite[Appendix B]{NW08}, \cite[Section 13.2]{NW10}, and \cite[Section 24.1]{NW12}.

As in the previous example we want to look for good approximations 
$A_{n,d}f$ to input functions $f$ in the norm of $\L_2^{\rho_d}(\Omega_d)$.
This time we measure the average performance of the algorithm $A_{n,d}$ with respect to $\mu_d$, \ie
we try to minimize
\begin{gather*}
			\Delta^{\mathrm{avg}} \left( A_{n,d}; \id_d^{\rho_d} \colon \F_d \nach \L_2^{\rho_d}(\Omega_d) \right)
			= \left( \int_{\F_d} \norm{f-A_{n,d}f \sep \L_2^{\rho_d}(\Omega_d)}^2 \,\mathrm{d}\mu_d(f) \right)^{1/2}.
\end{gather*}
Observe that $\nu_d=\mu_d \circ \left(\id_d^{\rho_d}\right)^{-1}$ defines a
Gaussian measure on the subset $\id_d^{\rho_d}(\F_d)$ of $\L_2^{\rho_d}(\Omega_d)$.
Now it can be checked that the corresponding covariance operator 
$C_{\nu_d}^{\rho_d} \colon \id_d^{\rho_d}(\F_d) \nach \L_2^{\rho_d}(\Omega_d)$ of the measure~$\nu_d$ is given by
\begin{gather*}
			f \mapsto (C_{\nu_d}^{\rho_d} f)(\cdot) 
			= \int_{\Omega_d} f(\bm{y}) \, K_d(\cdot,\bm{y}) \, \rho_d(\bm{y}) \dlambda^{d}(\bm{y}).
\end{gather*}
This operator is self-adjoint, compact and has a finite trace due to the integrability assumption on $K_d$.
Consequently, there exists a countable set of non-trivial eigenpairs $(\lambda_{d,\rho_d,i}, \eta_{d,\rho_d,i})$
where the eigenfunctions $\eta_{d,\rho_d,i}$ are mutually orthogonal (and normalized) with respect to the
$\L_2^{\rho_d}(\Omega_d)$-norm; see also Hickernell and Wo{\'z}niakowski~\cite{HW00}.

Once more it turns out that the optimal algorithm $A_{n,d}^*$ in this setting is
given by the orthogonal projection of the input function onto the subspace
spanned by the eigenfunctions $\eta_{d,\rho_d,i}$ which correspond to the $n$ largest
eigenvalues $\lambda_{d,\rho_d,i}$.
In contrast to the worst case setting the $n$th minimal average case error is
\begin{gather*}
			e^{\mathrm{avg}}(n,d; \id_d^{\rho_d} \colon \F_d \nach \L_2^{\rho_d}(\Omega_d)) = \left( \sum_{i=n+1}^\infty \lambda_{d,\rho_d,i} \right)^{1/2}, 
			\quad n\in\N_0,\quad d\in\N,
\end{gather*}
if we assume a non-increasing ordering of the sequence $(\lambda_{d,\rho_d,i})_{i=1}^\infty$.\footnote{For the ease of notation we moreover assumed here that all the eigenvalues are strictly positive.}
Based on the latter error formula it is possible to obtain characterizations of several types
of tractability similar to the assertions given in \autoref{sect:tensor_complexity}; 
see, e.g., \cite[Chapter 6]{NW08}.

We complete the discussion with the observation that the sets of
(non-trivial) eigenpairs $(\lambda_{d,\rho_d,i},\eta_{d,\rho_d,i})$ of the operators $C_{\nu_d}^{\rho_d}$ as defined above and $W_d^{\rho_d}$ from \autoref{ex:L_2-approx}
coincide, since $C_{\nu_d}^{\rho_d}$ only takes values in $\Hi(K_d)$. 
To be precise, we note that
$K_d(\bm{x},\bm{y}) = K_d(\bm{y},\bm{x})$ equals $(\mathrm{App}_{d}^{\rho_d} K_d(\cdot,\bm{x}))(\bm{y})$ for each fixed $\bm{x}$ and $\uplambda^d$-almost every $\bm{y}\in\Omega_d$. Hence the chain of equations
\begin{align*}
		(C_{\nu_d}^{\rho_d}f)(\bm{x}) 
		&= \int_{\Omega_d} f(\bm{y}) \, (\mathrm{App}_{d}^{\rho_d} K_d(\cdot,\bm{x}))(\bm{y})\,  \rho_d(\bm{y}) \dlambda^d(\bm{y})
		= \distr{f}{\mathrm{App}_{d}^{\rho_d} K_d(\cdot,\bm{x})}_{\L_2^{\rho_d}(\Omega_d)}\\
		&= \distr{(\mathrm{App}_{d}^{\rho_d})^\dagger f}{K_d(\cdot,\bm{x})}_{\Hi(K_d)}
		= \left((\mathrm{App}_{d}^{\rho_d})^\dagger f\right)\!(\bm{x})
\end{align*}
holds true for every $f\in\id_{d}^{\rho_d}(\F_d)\subset \L_2^{\rho_d}(\Omega_d)$ and $\uplambda^d$-almost all $\bm{x}\in\Omega_d$.\footnote{Observe that $C_{\nu_d}^{\rho_d}f\in \L_2^{\rho_d}(\Omega_d)$, \ie it is uniquely defined on $\Omega_d$ up to a set of measure zero.}
Clearly $(\mathrm{App}_{d}^{\rho_d})^\dagger$ maps into $\Hi(K_d)$ per definition.
Thus, for every eigenfunction $\eta \in \id_d^{\rho_d}(\F)$ of $C_{\nu_d}^{\rho_d}$, \ie 
\begin{gather}\label{eq:ef_C}
			\lambda \, \eta 
			= C_{\nu_d}^{\rho_d} \eta 
			= (\mathrm{App}_{d}^{\rho_d})^\dagger \eta,
			\quad \uplambda^d\text{-a.e. on } \Omega_d,
\end{gather}
we can find a representer $\overline{\eta}\in\Hi(K_d)$ with $\overline{\eta}=\eta$ in the sense of $\L_2^{\rho_d}(\Omega_d)$,
such that the equalities in~\link{eq:ef_C} hold pointwise on the whole set $\Omega_d$
and therefore also in the norm of $\Hi(K_d)$.
Now it is easy to check that $(\lambda,\overline{\eta})$ indeed is an eigenpair of~$W_d^{\rho_d}=(\mathrm{App}_{d}^{\rho_d})^\dagger \mathrm{App}_{d}^{\rho_d}$, normalized \wrt the $\L_2^{\rho_d}(\Omega_d)$-norm.
Conversely every eigenpair $(\lambda,\overline{\eta})$ of the operator~$W_d^{\rho_d}$
obviously fulfills $\lambda \, \overline{\eta} = C_{\nu_d}^{\rho_d} \overline{\eta}$
interpreted in~$\L_2^{\rho_d}(\Omega_d)$.

In conclusion we see that the knowledge of these eigenpairs implies complexity assertions for both approximation problems in the respective (quite different) settings.
\hfill$\square$
\end{example}

  \cleardoubleplainpage
\chapter{Problems on Hilbert spaces with scaled norms}\label{chapt:ScaledNorms}
The present chapter deals with a generalization of tensor
product problems $S=(S_d)_{d\in\N}$ between Hilbert spaces
in the sense of \autoref{sect:TensorBasics}.
We introduce additional scaling factors $s_d$ 
to the norm of the source spaces $\F_d$
and analyze their influence on the squared singular values $\lambda_{d,s_d,i}$ 
of the new problem operators $S_{d,s_d}$.
Using the techniques from \autoref{sect:Eigenpairs} we conclude 
optimal algorithms for these modified problems at the end of \autoref{sect:Scaled_basics}.
Afterwards, in \autoref{sect:ScaledComplexity}, we investigate tractability properties
of this class of problems \wrt the worst case setting.
Finally we present some applications of the obtained results in \autoref{sect:ScaledProbs_Ex}.

\section{Definitions, eigenpairs and the optimal algorithm}\label{sect:Scaled_basics}
Let $H_1$ and $\G_1$ be arbitrary Hilbert spaces with inner products 
$\distr{\cdot}{\cdot}_{H_1}$ and $\distr{\cdot}{\cdot}_{\G_1}$, respectively.
Further assume $S_1 \in \K(H_1,\G_1)$ to be a compact linear operator between these spaces.
Following the constructions given in \autoref{subsect:def_tensor_prob}
for any $d\in\N$ there exist uniquely defined $d$-fold tensor product spaces of $H_1$
and $\G_1$.
Let us denote these spaces by $H_d = H_1 \otimes\ldots\otimes H_1$ and $\G_d=\G_1\otimes\ldots\otimes\G_1$,
respectively.
Finally we define $S=(S_d)_{d\in\N}$ to be the sequence of multivariate tensor product operators
constructed out of $S_1$.

In contrast to \autoref{sect:TensorBasics} we now adapt the source spaces
of our multivariate problem
by introducing an additional positive \emph{sequence of scaling factors} $s=(s_d)_{d\in\N}$.
That is, for every $d\in\N$ we define $\F_d$ 
to be Hilbert space $H_d$ equipped with the inner product
\begin{equation}\label{eq:scaled_innerprod}
		\distr{\cdot}{\cdot}_{\F_d} 
		= \frac{1}{s_d} \distr{\cdot}{\cdot}_{H_d}, 
		\quad \text{where} \quad s_d>0.
\end{equation}
Obviously $\F_d$ algebraically coincides with $H_d$ whereas the norms (induced by the respective inner products) are equivalent.
Accordingly, the operators $S_d$ are still well-defined for any $d\in\N$ when we replace
$H_d$ by $\F_d$.
On the other hand the approximability properties of $S$ crucially depend on the used norms
since we need to consider the whole unit ball $\widetilde{\F}_d=\B(\F_d)$ when dealing with the worst case setting.
So let us denote the modified problem by $S_{(s)}=(S_{d,s_d}\colon \F_d \nach\G_d)_{d\in\N}$.

From \autoref{sect:General_HSP} we know that for the $n$th optimal
algorithms for $S_{(s)}$ we need to study the eigenpairs of $W_{d,s_d} = {S_{d,s_d}}^{\!\!\dagger} S_{d,s_d}$.
Although $S_{d,s_d}$ equals $S_d$ (as a mapping) we can not claim
that $W_{d,s_d} = W_d$ since ${S_{d,s_d}}^{\!\!\dagger}$ does not necessarily
coincide with ${S_d}^\dagger$.
Nevertheless, there exists a strong relation.
The following proposition extends \autoref{prop:tensor_eigenpairs} to the case
of scaled problems in the mentioned sense.
Keep in mind that the eigenpairs of the univariate (unscaled) operator
$W_1={S_1}^{\!\dagger} S_1$
are given by 
$\{(\lambda_m,e_m) \sep m\in\M_1\}$, where $\M_1=\{m\in\N \sep m<v(W_1)+1\}$ and $0 < \lambda_{m+1}\leq \lambda_m$ for all $m<v(W_1)$.
\begin{prop}\label{prop:scaled_tensor_eigenpairs}
				For $d\in\N$ the non-trivial eigenpairs of the operator 
				$W_{d,s_d} = {S_{d,s_d}}^{\!\!\dagger} {S_{d,s_d}}$ are given by
				$\left\{\left(\widetilde{\lambda}_{d,s_d,\bm{m}}, \widetilde{\phi}_{d,s_d,\bm{m}}\right) \sep \bm{m} \in \M_d = (\M_1)^d\right\}$, where
				\begin{equation}\label{scaled_tensor_eigenpairs}
							\widetilde{\lambda}_{d,s_d,\bm{m}} 
							= s_d \, \widetilde{\lambda}_{d,\bm{m}}
							= s_d \, \prod_{k=1}^d \lambda_{m_k} 
							\quad \text{and} \quad 
							\widetilde{\phi}_{d,s_d,\bm{m}} 
							= \sqrt{s_d} \, \widetilde{\phi}_{d,\bm{m}}
							= \sqrt{s_d} \, \bigotimes_{k=1}^d \phi_{m_k}.
				\end{equation}
\end{prop}
\begin{proof}
Since $S_{d,s_d}f = S_d f$ for every $f\in \F_d$ (or $H_d$, respectively)
we have
\begin{equation*}
			\distr{S_{d,s_d}f}{g}_{\G_d}
			= \distr{S_d f}{g}_{\G_d}
			= \distr{f}{{S_d}^{\!\dagger} g}_{H_d}
			= s_d \cdot \distr{f}{{S_d}^{\!\dagger} g}_{\F_d}
			= \distr{f}{\left( s_d \cdot {S_d}^{\!\dagger} \right) g}_{\F_d}
\end{equation*}
for all $f\in\F_d$ and $g\in\G_d$.
Thus, \link{def_adj} and the uniqueness of the adjoint operator\footnote{Note that, clearly, $S_{d,s_d}$ is compact if and only if $S_{d}$ is compact.} 
yield that ${S_{d,s_d}}^{\!\!\dagger} = s_d \cdot {S_d}^{\!\dagger}$ holds pointwise
and, consequently, $W_{d,s_d}$ equals $s_d \cdot W_d$ as a mapping.
Hence, from \autoref{prop:tensor_eigenpairs} and the linearity of $W_d$ 
we conclude that \link{scaled_tensor_eigenpairs}
indeed are eigenpairs of $W_{d,s_d}$.
Due to the factor $\sqrt{s_d}$ and the relation~\link{eq:scaled_innerprod} the eigenelements $\widetilde{\phi}_{d,s_d,\bm{m}}$ are properly normalized in $\F_d$.

It remains to show that there cannot exist eigenpairs other than \link{scaled_tensor_eigenpairs}.
This can be seen using arguments similar to them in the second part of the proof of \autoref{prop:tensor_eigenpairs}.
To this end, note that due to \link{eq:scaled_innerprod} the inner product in $\F_d$ 
equals zero if and only if the elements under consideration are orthogonal in $H_d$.
\end{proof}

\autoref{prop:scaled_tensor_eigenpairs} in hand, the rest of this section is straightforward.
Namely, we can use the bijections $\psi=\psi_d$ from \autoref{sect:Eigenpairs}
to define the non-increasing sequences $(\lambda_{d,s_d,i})_{i\in\N}$ by
\begin{equation*}
		\lambda_{d,s_d,i} 
		= \begin{cases}
					\widetilde{\lambda}_{d,s_d,\psi(i)}, & 1 \leq i < v(W_1)^d + 1,\\
					0, & \text{otherwise},
			\end{cases}
\end{equation*}
for every $d\in\N$.
The corresponding reordered eigenelements are denoted by~$\phi_{d,s_d,i}$, $i< v(W_1)^d+1$.
Moreover, we again use \autoref{Cor:OptAlgo} to see that for any $d\in\N$ 
the $n$th optimal algorithm for $S_{d,s_d}$, $n\in\N_0$, is given by
\begin{equation*}
			A_{n,d,s_d}^* \colon \F_d \nach \G_d, 
			\qquad f \mapsto A_{n,d,s_d}^*(f) 
					= \sum_{i=1}^{\min{n,v(W_1)^d}} \distr{f}{\phi_{d,s_d,i}}_{\F_d} \cdot S_{d,s_d} \phi_{d,s_d,i}. 
\end{equation*}
It realizes the $n$th minimal worst case error in dimension $d$ which equals
\begin{equation}\label{eq:scaled_worstcaseerror}
			e^{\mathrm{wor}}(n,d;S_{d,s_d}) 
			= \Delta^{\mathrm{wor}}(A_{n,d,s_d}^*;S_{d,s_d}) 
			= \sqrt{\lambda_{d,s_d,n+1}}.
\end{equation}
In particular, the case $n=0$, \ie the initial error
\begin{gather*}
			\epsilon_d^{\mathrm{init}} = \sqrt{\lambda_{d,s_d,1}} = \sqrt{s_d \cdot \lambda_1^d},
\end{gather*}
will play an important role in what follows.

\section{Complexity}\label{sect:ScaledComplexity}
Similar to \autoref{sect:TensorBasics} we proceed with the
analysis of the information complexity of scaled tensor product problems
$S_{(s)}=(S_{d,s_d})_{d\in\N}$ in the worst case setting.
We first take a look at necessary and sufficient conditions for 
(strong) polynomial tractability with respect to absolute errors.
Afterwards, in \autoref{sect:ScaledProbs_wt}, we complete these assertions
and investigate respective conditions for weak tractability and the curse of dimensionality.
Finally we will see in \autoref{sect:ScaledProbs_normed}
that the obtained improvements due to scaling are completely ruled out 
when we turn to the normalized error criterion.

As usual $\lambda=(\lambda_m)_{m\in\N}$ denotes the (extended)
sequence of squared singular values of the underlying operator $S_1 \colon H_1\nach \G_1$.
To avoid triviality we assume that $\lambda_2>0$ throughout the rest of this section.
The reason for this assumption is explicitly stated in \autoref{sect:tensor_complexity}.

\subsection{Polynomial tractability}\label{sect:ScaledProbs_pt}
The next statement is originally based on Theorem~3.1 of Wo{\'z}niakowski~\cite{W94b}
which provided the underlying idea for~\cite[Theorem 5.5]{NW08}.
We extend the results stated there to the case of scaled problems.
\begin{theorem}\label{thm:scaled_PT}
		Let $S_{(s)}=(S_{d,s_d})_{d\in\N}$ denote a scaled tensor product problem in the sense of \autoref{sect:Scaled_basics}. 
		Assume $\lambda_2>0$ and consider the worst case setting \wrt the absolute error criterion.
		Then the following assertions are equivalent:
		\begin{enumerate}[label=(\Roman{*}), ref=\Roman{*}]
				\item \label{Cond_SPT} $S_{(s)}$ is strongly polynomially tractable.
				\item \label{Cond_PT} $S_{(s)}$ is polynomially tractable.
				\item \label{Cond_sup} There exists $\tau\in(0,\infty)$ such that 
							$\lambda\in\l_\tau$ and $\sup_{d\in\N} s_d \norm{\lambda \sep \l_\tau}^d < \infty.$
				\item \label{Cond_limsup} There exists $\rho\in(0,\infty)$ such that 
							$\lambda\in\l_\rho$ and $\limsup_{d\nach\infty} s_d^{1/d} < \frac{1}{\lambda_1}.$
		\end{enumerate}
		If one of these (and hence all) conditions applies then the exponent of strong polynomial tractability is given by
		\begin{gather*}
				p^* = \inf\{2\tau \sep \tau \text{ fulfills condition \link{Cond_sup}} \}.
		\end{gather*}
\end{theorem}
\begin{proof}
\emph{Step 1.}
Since \link{Cond_SPT} clearly implies \link{Cond_PT} we start by proving ``\link{Cond_PT} $\Rightarrow$ \link{Cond_sup}''.
Therefore let $S_{(s)}$ be polynomially tractable with the constants $C,p>0$ and $q\geq 0$.
Then \autoref{Thm:General_Tract_abs} yields that for all $\rho>p/2$,
\begin{gather*}
				0 < C_\rho 
				= \sup_{d\in\N} \frac{1}{d^{2q/p}} \left(\sum_{i=\ceil{(1+C)\,d^q}}^\infty \left( \lambda_{d,s_d,i} \right)^\rho \right)^{1/\rho} < \infty.
\end{gather*}
Because of $\lambda_{d,s_d,i}=s_d \,\lambda_{d,i}$ for any $d,i\in\N$ due to \link{scaled_tensor_eigenpairs}, this particularly implies 
that $s_1 \left( \sum_{m=\ceil{1+C}}^\infty \lambda_m^\rho \right)^{1/\rho}\leq C_\rho$ is finite and hence $\lambda=(\lambda_m)_{m\in\N} \in \l_\rho$.
Moreover, we have
\begin{gather*}
				\sum_{i=1}^\infty \left( \lambda_{d,i} \right)^\rho 
				= \norm{\lambda \sep \l_\rho}^{\rho\, d}, 
				\quad \text{as well as} \quad 
				\sum_{i=1}^{\ceil{(1+C)\,d^q}-1} \left( \lambda_{d,i}\right)^\rho 
				\leq \lambda_1^{\rho \,d}\, (1+C) \,d^q,
\end{gather*}
and therefore
\begin{gather}\label{scaled_est}
			  \norm{\lambda \sep \l_\rho}^{\rho \, d} - \lambda_1^{\rho\, d} \,(1+C)\, d^q
			  \leq \left( \frac{C_\rho\, d^{2q/p}}{s_d} \right)^\rho
				\quad \text{for all} \quad d\in\N.
\end{gather}
Now let $\tau > \rho$ and assume that \link{Cond_sup} is violated for this $\tau$. 
Then $\sup_{d\in\N} s_d \norm{\lambda \sep \l_\tau}^d$ is infinite 
since $\lambda \in \l_\rho$ and $\l_\rho\hookrightarrow \l_\tau$ with 
$\norm{\lambda \sep \l_\tau}<\norm{\lambda \sep \l_\rho}$.
That means, for any $C_0\in(0,\infty)$ there necessarily exists 
a sequence $(d_k)_{k\in\N}\subset \N$ such that for every $k\in\N$
\begin{gather*}
				C_0 \leq s_{d_k} \norm{\lambda \sep \l_\tau}^{d_k} = s_{d_k} \norm{\lambda \sep \l_\rho}^{d_k} / (t_{\rho,\tau})^{d_k},
\end{gather*}
where we set $t_{\rho,\tau}=\norm{\lambda \sep \l_\rho}/\norm{\lambda \sep \l_\tau}>1$.
Hence, at least for all $k$ larger than a certain $k_0\in\N$,
we conclude that
$C_0 \leq s_{d_k} \norm{\lambda \sep \l_\rho}^{d_k} / d_k^{2q/p}$.
In particular, we can choose $C_0>C_\rho$ such that $0<C_1=C_\rho/C_0<1$.
Therefore \link{scaled_est} implies
\begin{gather*}
				\norm{\lambda \sep \l_\rho}^{\rho \,d_k} - \lambda_1^{\rho \,d_k} \, (1+C) \, d_k^q 
				\leq C_1^\rho \, \norm{\lambda \sep \l_\rho}^{\rho \,d_k}, \quad k\geq k_0,
\end{gather*}
which leads to
\begin{gather*}
				\left( 1+ \left(\frac{\lambda_2}{\lambda_1}\right)^{\rho} \right)^{d_k} 
				\leq \left( \sum_{m=1}^\infty \left(\frac{\lambda_m}{\lambda_1}\right)^{\rho} \right)^{d_k} 
				= \frac{\norm{\lambda \sep \l_\rho}^{\rho \,d_k}}{\lambda_1^{\rho \,d_k}} \leq C_2\, d_k^q
\end{gather*}
for all $k\geq k_0$ and some $C_2=(1+C)/(1-C_1^\rho)>0$. 
Since $\lambda_2>0$ and $\rho>0$ this is a contradiction and thus we have condition \link{Cond_sup} for every $\tau > p/2$.
Note that this also shows that
\begin{gather*}
				\inf\{ 2\tau \sep \tau \text{ fulfills condition \link{Cond_sup}}\} 
				\leq p^*.
\end{gather*}
		
\emph{Step 2.}
Next we show that \link{Cond_SPT} follows from \link{Cond_sup}.
Let $\tau>0$ be given such that \link{Cond_sup} holds true and set $p=q=r=0$, as well as $C=1$.
Then, with $f(d) = \ceil{C \left( \min{s_d\,\lambda_1^d,1} \right)^{-p/2} d^q}=1$, we have
\begin{gather*}
				C_\tau = \sup_{d\in\N} \frac{1}{d^r} \left( \sum_{i=f(d)}^\infty \left(\lambda_{d,s_d,i}\right)^\tau \right)^{1/\tau} = \sup_{d\in\N} s_d \norm{\lambda \sep \l_\tau}^d < \infty.
\end{gather*}
Once more we apply \autoref{Thm:General_Tract_abs} to obtain
\begin{gather*}
				n_{\mathrm{abs}}^{\mathrm{wor}}(\epsilon,d; S_{d,s_d}) 
				\leq (1+C_\tau^\tau)\, \epsilon^{-2\tau} 
				\quad \text{for all} \quad \epsilon(0,1] \quad \text{and every} \quad d\in\N.
\end{gather*}
Thus $S_{(s)}$ is strongly polynomially tractable and
\begin{gather*}
				p^* \leq \inf\left\{ 2\tau \sep \tau \text{ fulfills condition \link{Cond_sup}}\right\}.
\end{gather*}
		
\emph{Step 3.}
The implication ``\link{Cond_sup} $\Rightarrow$ \link{Cond_limsup}'' can be seen as follows.
Assume \link{Cond_sup} to be valid for some $0<\tau<\infty$ and set $\rho=\tau$. 
Then, clearly, $\lambda\in\l_\rho$.
If we now assume \link{Cond_limsup} to be violated then 
for any $\delta>0$ there needs to exist
a sequence $(d_k)_{k\in\N}\subset\N$ such that for all $k$
\begin{gather*}
				s_{d_k}^{1/d_k} \geq 1/ (\lambda_1+\delta).
\end{gather*}
Hence, 
$s_{d_k} \norm{\lambda \sep \l_\tau}^{d_k} \geq \left( \norm{\lambda \sep \l_\tau}/ (\lambda_1+\delta) \right)^{d_k}$
tends to infinity (as $k\nach\infty$) if we take~$\delta$ 
small enough so that $\lambda_1+\delta < \norm{\lambda \sep \l_\tau}$.
Since $\tau \in (0,\infty)$ and 
\begin{gather*}
			\lambda_1 < (\lambda_1^\tau + \lambda_2^\tau)^{1/\tau}\leq\norm{\lambda \sep \l_\tau}<\infty
\end{gather*}
there needs to be some $\alpha_\tau\in(0,\infty)$ 
such that $\norm{\lambda \sep \l_\tau}=\lambda_1+\alpha_\tau$.
Choosing e.g. $\delta = \alpha_\tau / 2$ gives the needed contradiction.
		
\emph{Step 4.}
Finally we have to show that conversely \link{Cond_limsup} also implies \link{Cond_sup}.
Therefore assume that we have \link{Cond_limsup} for some $\rho\in(0,\infty)$. 
Then there exist constants $d_0\in\N$ and $\delta>0$ such that
\begin{gather*}
				s_d^{1/d} \leq 1 / (\lambda_1+\delta) 
				\quad \text{for all} \quad d\geq d_0.
\end{gather*}
Furthermore note that the function $N(\tau)=\norm{\lambda \sep \l_\tau}$ 
is strictly decreasing and continuous on the interval $[\rho,\infty]$ 
and that $N(\rho)>\lambda_1=N(\infty)$ because of the ordering of $\lambda=(\lambda_m)_{m\in\N}$.
Hence there necessarily exists some $\tau\in[\rho,\infty)$ 
such that $N(\tau)\leq \lambda_1 + \delta/2$, say.
Thus, $\lambda\in\l_\tau$ and for every $d\geq d_0$ we obtain
\begin{gather*}
				s_d \norm{\lambda \sep \l_\tau}^d 
				\leq \left( \frac{\lambda_1 + \delta/2}{\lambda_1 + \delta} \right)^d 
				\leq 1.
\end{gather*}
Since the term on the left is also finite for any $d=1,\ldots,d_0$ this completes the proof.
\end{proof}

Observe that \autoref{thm:scaled_PT} is not very surprising.
Indeed, the second assertion in condition \link{Cond_limsup}
is equivalent to the fact that the $d$th root of the initial error $\epsilon_d^{\mathrm{init}}$
is asymptotically strictly less than $1$.
Hence if $S_1\colon H_1\nach\G_1$ (and thus also the sequence $\lambda$) is given
then we need to select scaling factors $s_d$ such that $\epsilon_d^{\mathrm{init}}\nach 0$, $d\nach\infty$,
in order to obtain polynomial tractability.
More advanced illustrations will be given in \autoref{sect:ScaledProbs_Ex}.

\subsection{Weak tractability and the curse}\label{sect:ScaledProbs_wt}
To formulate necessary and sufficient conditions for weak tractability 
\wrt the worst case setting and the absolute error criterion
we need some additional notation.
Therefore let $S_{(s)}=(S_{d,s_d})_{d\in\N}$ denote a scaled tensor product problem between Hilbert spaces 
as explained in \autoref{sect:Scaled_basics} and assume $\lambda_2 > 0$.
Then for fixed $d \in \N$ and $0<\epsilon<\epsilon_{d}^{\mathrm{init}}=s_d^{1/2}\lambda_1^{d/2}$ 
formula \link{eq:scaled_worstcaseerror} implies
\begin{align}
		n(\epsilon,d)
		&= \min{n\in\N_0 \sep \lambda_{d,s_d,n+1} \leq \epsilon^2} 
		=\#\left\{\bm{j}\in\N^d \sep s_d \cdot \lambda_{j_1} \cdot \ldots \cdot \lambda_{j_d}>\epsilon^2\right\} \nonumber\\
		&=\# \left\{ \bm{j}\in\N^d \sep \frac{\lambda_{j_1}}{\lambda_1} \cdot \ldots \cdot \frac{\lambda_{j_d}}{\lambda_1} > \left( \frac{\epsilon}{\epsilon_d^{\mathrm{init}}} \right)^2 \right\}.\label{def_n}
\end{align}
By counting the number of indices equal to one we conclude that
\begin{gather*}
		n(\epsilon,d)
		= 1+ \sum_{k=1}^d \binom{d}{k} \cdot 
							\# \left\{\bm{j}=(j_1,\ldots,j_k)\in (\N\setminus\{1\})^k \sep \prod_{l=1}^k \frac{\lambda_{j_l}}{\lambda_1} > \left( \frac{\epsilon}{\epsilon_d^{\mathrm{init}}}\right)^2 \right\}.
\end{gather*}
Now we distinguish two cases.
First assume that $\lambda_1=\lambda_2$.
Then, obviously, each the sets in the latter equality contains at least one element.
Otherwise, in the case $\lambda_1>\lambda_2$, some of these $k$-dimensional sets might be empty if $k$
is larger than some~$k_d(\epsilon)$. 
The reason is that $\lambda_1>\lambda_2\geq\lambda_m$ for $m \geq 2$ implies that 
every factor in $\prod_{l=1}^k \lambda_{j_l}/\lambda_1$ is strictly smaller than $1$.
In detail,
$\left( \epsilon / \epsilon_d^{\mathrm{init}}\right)^2 \geq \left(\lambda_2/\lambda_1 \right)^k$ is equivalent to
\begin{gather*}
			k > k_d(\epsilon) = \ceil{ \frac{1}{\ln(\lambda_1/\lambda_2)} \cdot \ln \!\left( \epsilon_d^{\mathrm{init}} / \epsilon\right)^2}-1.
\end{gather*}
Hence, denoting $a_d(\epsilon)=\min{d,k_d(\epsilon)}$ we have
\begin{gather}\label{rep_n}
			n(\epsilon,d) = 1+ \sum_{k=1}^{a_d(\epsilon)} \binom{d}{k} \cdot 
							\# \left\{\bm{j}\in (\N\setminus\{1\})^k \sep \prod_{l=1}^k \frac{\lambda_{j_l}}{\lambda_1} > \left( \frac{\epsilon}{\epsilon_d^{\mathrm{init}}}\right)^2 \right\}
\end{gather}
for $d\in\N$ and $0<\epsilon<\epsilon_d^{\mathrm{init}}$.
If $\lambda_1=\lambda_2$ then the same equality remains true 
when we formally set $k_d(\epsilon)=\infty$, \ie $a_d(\epsilon)=d$.
Moreover, for $d\in\N$ we have $a_d(\epsilon)=0$ if and only if $\epsilon \geq \left(\lambda_2/\lambda_1 \right)^{1/2} \epsilon_d^{\mathrm{init}}$.
If so, then we obtain $n(\epsilon,d)=1$ as long as $\epsilon<\epsilon_d^{\mathrm{init}}$ and $n(\epsilon,d)=0$ otherwise.

Finally the following statement relates 
the decay properties of the univariate sequence of the squared singular values $\lambda$ 
with the growth behavior of the information complexity $n(\eps,d)=n^\wor_{\ab}(\eps,d;S_{d,s_d})$. 
It generalizes an assertion given in
Novak and Wo{\'z}niakowski~\cite[p. 178]{NW08}.
\begin{lemma}\label{lemma:n}
			Let $S_{(s)}$ and $\lambda=(\lambda_m)_{m\in\N}$ be given as before.
			Then, for all $\beta\geq 1$,
			\begin{gather*}
						\lambda_n\in o \!\left( \ln^{-2\beta} n \right), \text{ as } n\nach\infty, 
						\quad \text{if and only if} \quad 
						\ln n(t^{\beta},1) \in o \left( 1/t \right), \text{ as } t \nach 0.
		\end{gather*}
\end{lemma}
\begin{proof}
Assume $\beta\geq1$ to be fixed and let $t \in (0, (s_1 \lambda_2)^{1/(2\beta)})$.
Then \link{def_n} yields that for $d=1$
\begin{gather*}
				n = n(t^{\beta},1) 
				= \min{n\in\N_0 \sep s_1 \cdot \lambda_{n+1} \leq t^{2\beta}} \geq 2.
\end{gather*}
Thus we have 
$s_1 \cdot \lambda_{n(t^{\beta},1)+1} \leq t^{2\beta} < s_1 \cdot \lambda_{n(t^{\beta},1)}$
and $\ln^{2\beta} n \geq 1/4^{\beta} \cdot \ln^{2\beta}(n+1)$.
Combining both these estimates we conclude
\begin{gather*}
				\frac{s_1}{4^\beta} \cdot \frac{\lambda_{n(t^{\beta},1)+1}}{\ln^{-2\beta}(n(t^\beta,1)+1)} 
				\leq \left( \frac{\ln n(t^\beta,1)}{t^{-1}} \right)^{2\beta} 
				< s_1 \cdot \frac{\lambda_{n(t^{\beta},1)}}{\ln^{-2\beta} n(t^\beta,1) }.
\end{gather*}
Since the one-dimensional information complexity $n(\epsilon,1)$ is an increasing function in $1/\epsilon$
taking the limit for $t\nach0$ proves the claim.
\end{proof}

Now we are well-prepared to present necessary conditions for weak tractability
based on the representation of the information complexity given in \link{rep_n}.
\begin{prop}\label{prop:nescond_scaled}
		Weak tractability of $S_{(s)}$ implies 
		\begin{gather}\label{cond1}
				\lim_{\epsilon^{-1}+d\nach\infty} \frac{\ln \sum_{k=0}^{a_d(\epsilon)} \binom{d}{k}}{\epsilon^{-1}+d}=0
				\qquad \text{and} \qquad
				\lim_{\epsilon^{-1}+d\nach\infty} \frac{\ln n\!\left( \epsilon_{1}^{\mathrm{init}} \cdot \epsilon / \epsilon_{d}^{\mathrm{init}}, 1 \right)}{\epsilon^{-1}+d}=0.
		\end{gather}
		If so, then $\lambda_n \in o(\ln^{-2} n )$, as $n\nach\infty$.
		Furthermore, we have 
		\begin{gather}\label{InitErrorBound}
					\ln\!\left( \epsilon_d^{\mathrm{init}} \right) \in o(d), \quad \text{as} \quad d\nach\infty,
		\end{gather}
		since otherwise $S_{(s)}$ suffers from the curse of dimensionality.
		If, in addition, $\lambda_1=\lambda_2$ then we need to claim 
		$\lim_{d\nach\infty} \epsilon_d^{\mathrm{init}}=0$ to avoid the curse.
		Moreover, in this case weak tractability even yields
		\begin{gather}\label{cond2}
				\epsilon_d^{\mathrm{init}} \in o(1/d), \quad \text{as} \quad d\nach\infty.
		\end{gather}
\end{prop}
\begin{proof}
\emph{Step 1}.
We start by proving the necessity of the first limit condition in \link{cond1}
and study its consequences.
To this end, recall that due to the definition of $a_d(\epsilon)$ 
we know that all the sets in \link{rep_n} contain at least one element.
Consequently, for general $\lambda_1 \geq \lambda_2$ we have
\begin{gather}\label{eq:lowerbound_n}
				n(\epsilon,d) \geq 1+ \sum_{k=1}^{a_d(\epsilon)} \binom{d}{k} = \sum_{k=0}^{a_d(\epsilon)} \binom{d}{k}
				\quad \text{for} \quad d\in\N \quad \text{and} \quad 0<\epsilon<\epsilon_d^{\mathrm{init}}.
\end{gather}
Now assume the existence of some subsequence $(d_l)_{l\in\N}\subset\N$ 
such that the initial error~$\epsilon_{d_l}^{\mathrm{init}}$ grows
at least exponentially in $d_l$ for $l$ tending to infinity.
That is, we assume the condition~\link{InitErrorBound} to be violated.
Moreover, consider 
$\epsilon=\epsilon_0 \in (0, \inf\{\epsilon_{d_l}^{\mathrm{init}} \sep l\in\N\})$
to be fixed.
Then for any $l\in\N$ and some $\alpha\in(0,1/2)$ the term 
$a_{d_l}(\eps_0)$ is bounded from below by $\floor{\alpha\, d_l}$.
Accordingly, \link{eq:lowerbound_n} implies $n(\epsilon_0,d_l)\geq \binom{d_l}{\floor{\alpha\,d_l}}$
for all $l\in\N$.
Using similar calculations as in~\cite[p. 178]{NW08}
we see that this lower bound grows exponentially in $d_l$.
This proves the curse of dimensionality for the scaled problem~$S_{(s)}$
and thus it contradicts weak tractability.

If we assume in addition that $\lambda_1=\lambda_2$ then, as already noticed,
$a_d(\epsilon)$ equals $d$ because of $k_d(\epsilon)=\infty$. Thus we obtain
$\sum_{k=0}^{a_d(\epsilon)} \binom{d}{k}= 2^d$ in this case.
Therefore the existence of a sequence $(d_l)_{l\in\N}$ such that
$\epsilon_{d_l}^{\mathrm{init}}$ is larger than some $C>0$ for all $l\in\N$
would again imply the curse of dimensionality 
since then we could fix $\epsilon=\epsilon_0=C/2$, say.
Moreover $a_d(\epsilon)=d$ shows that the first part of \link{cond1} 
equivalently reads
\begin{gather}\label{lim_cond}
				\lim_{\epsilon^{-1}+d\nach\infty} \frac{d}{\epsilon^{-1}+d}=0.
\end{gather}
Observe that in any case the term $d/(\epsilon^{-1}+d)$ is equivalent to 
the minimum of $1$ and  $\epsilon\,d$ (up to some absolute constants).
Hence, \link{lim_cond} holds true if and only if
\begin{gather*}
				\lim_{\epsilon^{-1}+d\nach\infty} \epsilon\,d = 0
\end{gather*}
which in turn is equivalent to $\epsilon_d^{\mathrm{init}} \in o(1/d)$ for $d\nach\infty$.
To see this last equivalence, remember that due to \autoref{sect:TractDef} 
the domain of the sequences $((\epsilon_k,d_k))_{k\in\N}$ for the limit $\epsilon^{-1}+d\nach\infty$
is restricted per definition to those for which $\epsilon_k<\epsilon_{d_k}^{\mathrm{init}}$.

\emph{Step 2}.
We turn to the proof of the second point in \link{cond1}.
Again we distinguish the cases $\lambda_1=\lambda_2$ and $\lambda_1>\lambda_2$.
For the latter case keep in mind that
$a_d(\epsilon) \geq 1$ if and only if $\epsilon < \epsilon_d^{\mathrm{init}}(\lambda_2/\lambda_1)^{1/2}$.
If so, then \link{rep_n} shows that
\begin{align*}
				n(\epsilon,d) 
				&\geq 1 + \binom{d}{1} \cdot \# \left\{ j \geq 2 \sep \frac{\lambda_j}{\lambda_1} > \left(\frac{\epsilon}{\epsilon_d^{\mathrm{init}}}\right)^2 \right\} \\
				&= 1 + d \cdot \# \left\{ j \geq 2 \sep s_1 \lambda_j > \left(\epsilon_1^{\mathrm{init}} \cdot \epsilon / \epsilon_d^{\mathrm{init}}\right)^2 \right\}
				\geq n(\epsilon_1^{\mathrm{init}} \cdot \epsilon / \epsilon_d^{\mathrm{init}},1).
\end{align*}
On the other hand, if 
$\epsilon \in \left[\epsilon_d^{\mathrm{init}}(\lambda_2/\lambda_1)^{1/2}, \epsilon_d^{\mathrm{init}}\right)$ then
$n(\epsilon_1^{\mathrm{init}} \cdot \epsilon / \epsilon_d^{\mathrm{init}},1)$
is no larger than $n(\epsilon_1^{\mathrm{init}} (\lambda_2/\lambda_1)^{1/2},1)$ 
which is an absolute, positive constant.
Thus, as claimed in \link{cond1}, we conclude
\begin{equation}\label{est_lnn}
				0 \leq
				\frac{\ln n(\epsilon_1^{\mathrm{init}} \cdot \epsilon / \epsilon_d^{\mathrm{init}},1)}{\epsilon^{-1}+d}
				\leq \max{\frac{\ln n(\epsilon ,d)}{\epsilon^{-1}+d}, \frac{\ln n(\epsilon_1^{\mathrm{init}} (\lambda_2/\lambda_1)^{1/2},1)}{\epsilon^{-1}+d}} \nach 0 
\end{equation}
for $\epsilon^{-1}+d$ tending to infinity in the above sense.
In the case $\lambda_1=\lambda_2$ we have $a_d(\epsilon)=d$ which is trivially bounded from below by $1$
for any $\epsilon \in (0,\epsilon_d^{\mathrm{init}})$.
The assertion now follows using the same arguments as in the first part of the previous case.
To complete the proof it finally remains to show that $\lambda_n \in o(\ln^{-2} n)$, as $n\nach\infty$.
Let us consider the case $d=d_k \equiv 1$ in \link{est_lnn}.
Then we obtain
\begin{gather*}
		0 
		\leq \frac{\ln n(\epsilon,1)}{\epsilon^{-1}} 
		\leq 2 \cdot \frac{\ln n(\epsilon_1^{\mathrm{init}} \cdot \epsilon / \epsilon_1^{\mathrm{init}},1)}{\epsilon^{-1}+1} \nach 0, 
		\quad \text{as} \quad \epsilon^{-1}\nach\infty.
\end{gather*}
In other words, weak tractability yields $\ln n(\epsilon,1) \in o(\epsilon^{-1})$
which is equivalent to the claimed assertion due to \autoref{lemma:n}.
\end{proof}

Let us add some comments on the latter necessary conditions.
\begin{rem}\label{rem:neccond_wt}
First of all note that from \link{cond1} we concluded \link{InitErrorBound} 
which is equivalent to the fact that
$\limsup_{d\nach \infty} s_d^{1/d} \leq 1/\lambda_1$.
Aside from that \link{cond1} also implies another condition which we will need later on;
namely
\begin{gather*}
			\lim_{\epsilon^{-1}+d\nach\infty} \frac{\ln n\!\left( \epsilon_{1}^{\mathrm{init}} \cdot \left(\epsilon / \epsilon_{d}^{\mathrm{init}}\right)^{1/2}, 1 \right)}{\epsilon^{-1}+d}=0.
\end{gather*}
Conclusively we stress that the second point of \link{cond1} 
already indicates a certain trade-off between the decay of 
the sequence $(\lambda_n)_{n\in\N}$ and the growth of the initial error~$\epsilon_d^{\mathrm{init}}$.
Indeed, if $(\lambda_n)_{n\in\N}$ decreases almost logarithmically then $n(t,1)$ increases subexponentially as $t$ tends to zero.
Consequently \link{cond1} can be fulfilled only if $\epsilon_d^{\mathrm{init}}$ is polynomially bounded in $d$.
On the other hand, if the (squares of the) singular values tend to zero like the inverse of some polynomial, say, then 
$n(t,1)$ grows polynomially in $1/t$ and hence it is enough to assume that the initial error is subexponentially bounded in $d$ to fulfill \link{cond1}.
\hfill$\square$
\end{rem}

We complement the necessary conditions in \autoref{prop:nescond_scaled} 
by the following sufficient conditions for weak tractability
of scaled tensor product problems $S_{(s)} = (S_{d,s_d})_{d\in\N}$.
For the proof we essentially follow the arguments of 
Papageorgiou and Petras~\cite{PP09} for the unscaled case
which are based on estimates from Wo{\'z}niakowski~\cite{W94b}.
\begin{prop}\label{prop:suffcond_scaled}
		Let $S_{(s)}$ and $a_d(\epsilon)$ be defined as before and assume that $\lambda_2 > 0$.
		If the condition \link{cond1} from \autoref{prop:nescond_scaled} holds true and if we have
		\begin{gather}\label{cond3}
				\lim_{\epsilon^{-1}+d\nach\infty} \frac{a_d(\epsilon) \cdot \ln n\!\left( \epsilon_{1}^{\mathrm{init}} \cdot \left(\epsilon / \epsilon_{d}^{\mathrm{init}}\right)^{1/2}, 1 \right)}{\epsilon^{-1}+d}=0
		\end{gather}
		then $S_{(s)}$ is weakly tractable.
\end{prop}
\begin{proof}
Given $d\in\N$ and 
$\epsilon \in (0, \left( \lambda_2/\lambda_1\right)^{1/2}\epsilon_d^{\mathrm{init}})$ 
consider the representation~\link{rep_n} and keep in mind that for larger $\epsilon$ 
the information complexity $n(\epsilon,d)$ is trivially bounded
by $1$ because then $a_d(\epsilon)=0$.
For every $k\in\{1,\ldots,a_d(\epsilon)\}$ we have
\begin{gather*}
				\#\left\{ \bm{j}\in\left(\N\setminus\{1\}\right)^k \sep \prod_{l=1}^k \frac{\lambda_{j_l}}{\lambda_{1}} > \left( \frac{\epsilon}{\epsilon_d^{\mathrm{init}}} \right)^2 \right\}
				\leq \# \left\{ \bm{j} \in \N^{a_d(\epsilon)} \sep \prod_{l=1}^{a_d(\epsilon)} \frac{\lambda_{j_l}}{\lambda_{1}} > \left( \frac{\epsilon}{\epsilon_d^{\mathrm{init}}} \right)^2 \right\}
\end{gather*}
since $\lambda_m/\lambda_1\leq 1$	for all $m\in\N$.
Hence we concentrate on all the multi-indices 
$\bm{j}=(j_1,\ldots,j_{a_d(\epsilon)})$ that fulfill
\begin{gather}\label{eq1}
				\prod_{l=1}^{a_d(\epsilon)} \frac{\lambda_{j_l}}{\lambda_{1}} > \left( \frac{\epsilon}{\epsilon_d^{\mathrm{init}}} \right)^2.
\end{gather}
Clearly the largest possible index $j_{\mathrm{max}}^{(1)}$ which can appear in those $\bm{j}\in\N^{a_d(\epsilon)}$
is bounded because the sequence $(\lambda_{n})_{n\in\N}$ tends to zero as $n\nach\infty$.
Indeed, using the arguments given in \cite{PP09} we conclude that 
\begin{gather*}
				j^{(1)}_{\mathrm{max}} 
				\leq \min{n\in\N_0 \sep s_1 \lambda_{n+1} \leq \left( \epsilon_1^{\mathrm{init}} \frac{\epsilon}{\epsilon_d^{\mathrm{init}}} \right)^2 } 
				= n\! \left( \epsilon_1^{\mathrm{init}} \frac{\epsilon}{\epsilon_d^{\mathrm{init}}}, 1 \right).
\end{gather*}
More generally, in \cite{PP09} it was noticed that, using the same reasoning, 
we can bound the $i$th largest index $j^{(i)}_{\mathrm{max}}$ in \link{eq1} by
\begin{gather*}
				j^{(i)}_{\mathrm{max}} 
				\leq n \!\left( \epsilon_1^{\mathrm{init}} \cdot \left( \epsilon / \epsilon_d^{\mathrm{init}}\right)^{1/i}, 1 \right).
\end{gather*}
We use this estimate for $i=1$ and $2$ to conclude the upper bound
\begin{gather*}
				a_d(\epsilon) \cdot n \!\left( \epsilon_1^{\mathrm{init}} \cdot \epsilon / \epsilon_d^{\mathrm{init}}, 1 \right) 
								\cdot n \!\left( \epsilon_1^{\mathrm{init}} \cdot \left( \epsilon / \epsilon_d^{\mathrm{init}}\right)^{1/2}, 1 \right)^{a_d(\epsilon)-1}
\end{gather*}
for $\# \left\{ \bm{j} \in \N^{a_d(\epsilon)} \sep \bm{j} \text{ fulfills } \link{eq1}	\right\}$.
Note that due to $\epsilon<\epsilon_d^{\mathrm{init}}$ both the univariate complexities 
in the latter bound need to be at least $1$.
Therefore we can extend the estimate by adding an additional factor
$n \!\left( \epsilon_1^{\mathrm{init}} \cdot \left( \epsilon / \epsilon_d^{\mathrm{init}}\right)^{1/2}, 1 \right)$
and replacing~$a_d(\epsilon)$ by~$d$.
In summary we have 
\begin{gather}\label{bound_n}
				n(\epsilon,d) 
				\leq d \cdot n \!\left( \epsilon_1^{\mathrm{init}} \cdot \epsilon / \epsilon_d^{\mathrm{init}}, 1 \right) 
										\cdot n \!\left( \epsilon_1^{\mathrm{init}} \cdot \left( \epsilon / \epsilon_d^{\mathrm{init}}\right)^{1/2}, 1 \right)^{a_d(\epsilon)} \cdot \sum_{k=0}^{a_d(\epsilon)} \binom{d}{k}
\end{gather}
for each $d\in\N$ and all 
$\epsilon\in(0,\left(\lambda_2/\lambda_1\right)^{1/2} \epsilon_d^{\mathrm{init}})$.
Because of 
$n(\epsilon,d)=1$ if $\epsilon$ belongs to $[ \left(\lambda_2/\lambda_1\right)^{1/2} \epsilon_d^{\mathrm{init}}, \epsilon_d^{\mathrm{init}} )$,
the estimate~\link{bound_n} remains valid for every $\epsilon\in(0,\epsilon_d^{\mathrm{init}})$.
Proceeding as in \cite{PP09} we take the logarithm 
and divide by $\epsilon^{-1}+d$	to conclude
\begin{align*}
				\frac{\ln n(\epsilon,d)}{\epsilon^{-1}+d} 
				&\leq \frac{\ln(d)}{\epsilon^{-1}+d} 
								+ \frac{\ln \!\left[ n \left( \epsilon_1^{\mathrm{init}} \cdot \epsilon / \epsilon_d^{\mathrm{init}} , 1 \right) \right]}{\epsilon^{-1}+d} \\
				& \qquad \qquad	+ \frac{a_d(\epsilon) \cdot \ln\!\left[ n \left( \epsilon_1^{\mathrm{init}} \cdot \left( \epsilon/ \epsilon_d^{\mathrm{init}}\right)^{1/2},1 \right) \right] }{\epsilon^{-1}+d}
								+ \frac{\sum_{k=0}^{a_d(\epsilon)} \binom{d}{k}}{\epsilon^{-1}+d}.
\end{align*}
For weak tractability it suffices to show that each of these fractions
tends to zero as $\epsilon^{-1}+d$ approaches infinity.
Obviously, for the first one this is true without any further conditions.
For the second and fourth fraction the assertion follows from~\link{cond1}.
Finally the third fraction tends to zero due to 
the additional condition~\link{cond3} we imposed for this proposition.
\end{proof}

To illustrate the obtained results the following theorem considers several cases
for the behavior of the initial error $\epsilon_d^{\mathrm{init}}$.
\begin{theorem}\label{thm:scaled_wt}
		Let $S_{(s)} = (S_{d,s_d})_{d\in\N}$ denote a scaled tensor product problem
		in the sense of \autoref{sect:Scaled_basics}. 
		Assume that $\lambda_2>0$ and consider the worst case setting 
		\wrt the absolute error criterion.
		\begin{itemize}
				\item Let $\ln \!\left(\epsilon_d^{\mathrm{init}}\right)\notin o(d)$, as $d\nach\infty$.\\ 
							Then $S_{(s)}$ suffers from the curse of dimensionality.
				\item Let $\epsilon_d^{\mathrm{init}} \in \Theta(d^\alpha)$, as $d\nach\infty$, 
							for some $\alpha\geq 0$.
							\begin{itemize}
									\item If $\lambda_1=\lambda_2$ then $S_{(s)}$ suffers from the curse of dimensionality.
									\item In the case $\lambda_1>\lambda_2$ the problem $S_{(s)}$ is weakly tractable if and only if 
									\begin{gather}\label{eq:wt3}
												\lambda_n \in o\!\left(\ln^{-2(1+\alpha)} n \right), \text{ as } n\nach\infty. 
									\end{gather}
							\end{itemize}
				\item Let $\epsilon_d^{\mathrm{init}} \nach 0$, as $d$ approaches infinity.\\ 
							Then we never have the curse of dimensionality.
							Moreover, $S_{(s)}$ is weakly tractable if and only if 
							\begin{enumerate}[label=(\roman{*}), ref=\roman{*}]
										\item $\lambda_1=\lambda_2$ and $\lambda_n \in o\!\left(\ln^{-2} n\right)$, as $n\nach\infty$,
												  and $\epsilon_d^{\mathrm{init}}\in o(1/d)$, as $d\nach\infty$, or
										\item $\lambda_1>\lambda_2$ and $\lambda_n \in o\!\left(\ln^{-2} n\right)$, as $n\nach\infty$.
							\end{enumerate}
		\end{itemize}
\end{theorem}
\begin{proof}
\emph{Step 1}.
In this first step we handle the assertions concerning the curse of dimensionality.
From the proof of \autoref{prop:nescond_scaled} we know that $S_{(s)}$ suffers from the curse
if either $\ln (\epsilon_d^{\mathrm{init}}) \notin o(d)$, or if $\lambda_1=\lambda_2$
and $\lim_{d\nach \infty} \epsilon_d^\mathrm{init} \neq 0$.
Of course the latter condition is fulfilled particularly if the initial error grows polynomially with the dimension $d$, \ie if $\epsilon_d^{\mathrm{init}} \in \Theta(d^\alpha)$ for some $\alpha \geq 0$.
Furthermore the fact that we cannot have the curse of dimensionality 
as long as $\epsilon_d^{\mathrm{init}}$
tends to zero is clear from the definition.
		
\emph{Step 2.}
Next we show that weak tractability implies \link{eq:wt3}.
Therefore we note that $\epsilon_d^{\mathrm{init}}\in\Theta(d^\alpha)$ implies the existence of some $c>0$
such that we have $\epsilon_d^{\mathrm{init}} \geq c \, d^{\alpha}$ for all $d\in\N$.
Moreover we see that there is some $d_0\in\N$ such that $1/c < d^{1+\alpha}$ for every $d$ larger than $d_0$.
Setting $\epsilon = 1/d$ now yields
\begin{gather*}
				\frac{\epsilon}{\epsilon_d^{\mathrm{init}}} 
				\leq \frac{1}{c} \cdot \frac{1}{d^{1+\alpha}} 
				< 1 
				\quad \text{for all} \quad d\geq d_0
\end{gather*}
and $\epsilon^{-1}+d = 2d \nach \infty$, as $d\nach \infty$.
Hence, the sequence $((1/d,d))_{d\geq d_0}$ is admissible 
for the second limit condition of \link{cond1} in \autoref{prop:nescond_scaled}.
On the other hand, we have
\begin{gather*}
				\frac{\ln n(\epsilon_1^{\mathrm{init}} \cdot \epsilon / \epsilon_d^{\mathrm{init}}, 1)}{\epsilon^{-1}+d} 
				\geq \frac{c'}{2} \cdot \frac{\ln n( \left( c'/d \right)^{1+\alpha},1)}{(c'/d)^{-1}} 
				\geq 0
\end{gather*}
where we set $c' = (\epsilon_1 / c)^{1/(1+\alpha)}$.
Thus weak tractability implies $\ln n(t^{1+\alpha},1) \in o \left( 1/t \right)$ for $t \nach 0$.
Now the assertion follows from \autoref{lemma:n}.
		
\emph{Step 3.}
For the case of polynomial initial errors it remains to prove the converse implication, 
namely that \link{eq:wt3} is also sufficient for weak tractability 
provided that $\lambda_1>\lambda_2$.
To this end, we first show that for all $d\in\N$, 
every $\epsilon \in(0,\epsilon_d^{\mathrm{init}})$ and for some $C>0$,
\begin{gather}\label{eq:young}
				\epsilon_d^{\mathrm{init}} / \epsilon \leq C \cdot (\epsilon^{-1} + d)^{1+\alpha}.
\end{gather}
To see this, we notice that $\epsilon_d^{\mathrm{init}}\in\Theta(d^\alpha)$ 
implies the existence of some $C>0$ such that
$\epsilon_d^{\mathrm{init}} \leq C\, d^{\alpha}$ for all $d\in\N$.
If $\alpha=0$ then \link{eq:young} is obvious.
For the case $\alpha > 0$ we apply Young's inequality\footnote{Recall that Young's inequality 
states that $a,b\geq 0$ and $p,q>1$ with $\frac{1}{p}+\frac{1}{q}=1$ yields that 
$ab\leq \frac{1}{p}\, a^p + \frac{1}{q} \, b^q$. 
We use this assertion for $a=d^{\alpha}$, $b=\epsilon^{-1}$ and $p=1+1/\alpha$, $q=1+\alpha$.} 
and obtain 
$\epsilon_d^{\mathrm{init}} / \epsilon \leq C ( d^{1+\alpha} + (\epsilon^{-1})^{1+\alpha})$.
Now the inequality~\link{eq:young} follows from the relation 
$\norm{\cdot \sep \l_{1+\alpha}} \leq \norm{\cdot \sep \l_{1}}$, $\alpha\geq 0$, 
for (two-dimensional) sequence spaces.

We want to conclude weak tractability from \autoref{prop:suffcond_scaled}.
Hence we have to check the limit conditions stated in \link{cond1} and \link{cond3}.
In what follows we abbreviate the notation and set 
\begin{gather*}
				t=t(\epsilon,d)=\epsilon^{-1}+d.
\end{gather*}
Given \link{eq:young}, as well as the definition of $a_d(\epsilon)$ 
in front of formula \link{rep_n}, it is easy to see that 
\begin{gather*}
				a_d(\epsilon) \in \0(\ln(t)), \text{ as } t\nach \infty.
\end{gather*}
In particular, we have $a_d(\epsilon) < \floor{\floor{t}/2}$, if $t$ is sufficiently large.
Moreover note that $d<t$ implies $\binom{d}{k} \leq \binom{\floor{t}}{k}$ 
for all $k\in\{0,1,\ldots,a_d(\epsilon)\}$ such that 
$\binom{\floor{t}}{a_d(\epsilon)}$ is an upper bound 
for each of those binomial coefficients $\binom{d}{k}$.
Consequently,
\begin{align*}
				\ln \sum_{k=0}^{a_d(\epsilon)} \binom{d}{k} 
				&\leq \ln \!\left( (a_d+1)\cdot\binom{\floor{t}}{a_d(\epsilon)}\right) 
				\leq \ln \!\left( 2\, a_d(\epsilon) \, (e \floor{t})^{a_d(\epsilon)} \right)\\
				&\leq \ln 2 + \ln(a_d(\epsilon)) + a_d(\epsilon) \cdot (1+\ln(t)) \\
				& \in \0(\ln^2(t)) \subseteq o(t),
		\end{align*}
for $t\nach\infty$.
In other words, the first part of condition \link{cond1} is fulfilled.
Also the second limit condition in \link{cond1} can be shown easily using \link{eq:young}.
Indeed, due to the assumption in~\link{eq:wt3} 
(or its equivalent reformulation due to \autoref{lemma:n}, respectively) 
we conclude that
\begin{gather*}
				\frac{\ln n(\epsilon_1^{\mathrm{init}} \cdot \epsilon/\epsilon_d^{\mathrm{init}},1)}{\epsilon^{-1}+d}
				\leq \frac{\ln n(\epsilon_1^{\mathrm{init}} \, C \, (\epsilon^{-1}+d)^{-(1+\alpha)},1)}{\epsilon^{-1}+d}
				= \frac{1}{C'} \cdot \frac{\ln n( (C'/ t)^{1+\alpha},1)}{(C'/t)^{-1}} 
\end{gather*}
tends to zero as $\epsilon^{-1}+d$ (and therefore also $t/C'$) approaches infinity.
Finally,
\begin{align*}
				a_d(\epsilon)\cdot \ln n\!\left(\epsilon_1^{\mathrm{init}} \cdot \left( \epsilon/\epsilon_d^{\mathrm{init}} \right)^{1/2},1\right)
				&\leq a_d(\epsilon) \cdot \ln n\!\left(\left( \frac{C''}{\epsilon^{-1}+d} \right)^{(1+\alpha)/2},1\right) \\
				&\in \0(\ln t) \cdot \0(t^{1/2}) \subseteq o(t)
\end{align*}
for $t=\epsilon^{-1}+d$ tending to infinity. 
Hence we have shown \link{cond3}.
Now the application of \autoref{prop:suffcond_scaled} 
completes the proof for the case of polynomial initial errors.
		
\emph{Step 4.}
In this last step we consider the case of initial errors 
which tend to zero for $d$ tending to infinity.
We already know from \autoref{prop:nescond_scaled} that 
$\lambda_n \in o(\ln^{-2} n )$, $n\nach\infty$, 
(or equivalently $\ln n(t,1)\in o(1/t)$, as $t\nach 0$) is necessary
for weak tractability, independent of the relation of the two largest 
(squares of the) singular values $\lambda_1$ and $\lambda_2$ to each other.
Moreover \autoref{prop:nescond_scaled} states that 
$\epsilon_d^{\mathrm{init}}\in o(1/d)$, $d\nach\infty$, is a necessary
condition	when we assume $\lambda_1=\lambda_2$, in addition.
It remains to show that these conditions are also sufficient 
for weak tractability in the particular situations.

If $\lambda_1>\lambda_2$ then we can exactly 
follow the lines of Step 3 with $\alpha=0$ in order to conclude the assertion.
Hence we are left with the case $\lambda_1=\lambda_2$.
Similar to the previous step we want to apply \autoref{prop:suffcond_scaled} 
and thus we need to check the conditions in \link{cond1} and \link{cond3}.
Setting
\begin{gather}\label{def_u}
				u = u(\epsilon,d) = \frac{d}{\epsilon^{-1}+d}
\end{gather}
we note that (due to \link{lim_cond} in the proof of \autoref{prop:nescond_scaled}) 
$u$ tends to zero if $\epsilon^{-1}+d\nach \infty$. 
This follows from $\epsilon_d^{\mathrm{init}}\in o(1/d)$, as $d\nach\infty$, and, on the other hand, it implies
the first condition in~\link{cond1} because
$\sum_{k=0}^{a_d(\epsilon)} \binom{d}{k}$ equals $2^d$.
Since, in particular, $\epsilon_d^{\mathrm{init}} \leq C$ for some $C>0$, we have $C'=C / \epsilon_1^{\mathrm{init}}>0$ 
and thus we obtain
\begin{gather*}
				\frac{\ln n(\epsilon_1^{\mathrm{init}} \cdot \epsilon/\epsilon_d^{\mathrm{init}},1)}{\epsilon^{-1}+d}
				\leq \frac{\ln n(\epsilon_1^{\mathrm{init}}/C \cdot \epsilon,1)}{\epsilon^{-1}+d}
				\leq C' \cdot \frac{\ln n( (C' \cdot (\epsilon^{-1}+d))^{-1},1)}{C' \cdot (\epsilon^{-1}+d)}
				\nach 0,
\end{gather*}
as $\epsilon^{-1}+d\nach\infty$ due to $\ln n(t,1)\in o(1/t)$, $t\nach 0$.
In other words, we have shown the second condition in~\link{cond1}.
To see that also \link{cond3} holds true we once more use 
$\epsilon_d^{\mathrm{init}}\in o(1/d) \subseteq \0(1/d)$ as well as Young's inequality to conclude
\begin{gather*}
				\left( \frac{\epsilon_d^{\mathrm{init}}}{\epsilon} \right)^{1/2} 
				\leq C_1 \cdot \frac{1}{d} \cdot d^{1/2} \cdot \left( \frac{1}{\epsilon} \right)^{1/2} 
				\leq \frac{C_1}{2} \cdot \frac{\epsilon^{-1}+d}{d}
\end{gather*}
with some $C_1>0$.
Hence, using \link{def_u} we have 
$\epsilon_1^{\mathrm{init}} \cdot (\epsilon / \epsilon_d^{\mathrm{init}})^{1/2} \geq C_2 \, u$
and therefore
\begin{gather*}
				\frac{d \cdot \ln n(\epsilon_1^{\mathrm{init}} \cdot (\epsilon/\epsilon_d^{\mathrm{init}})^{1/2},1)}{\epsilon^{-1}+d}
				\leq \frac{1}{C_2} \cdot \left(C_2\, u\right) \cdot \ln n(C_2\, u,1) \nach 0 \quad \text{if } \epsilon^{-1}+d\nach\infty,
\end{gather*}
because then $C_2\,u = C_2\, u(\epsilon,d)$ tends to zero.
Since $a_d(\epsilon)=d$ this yields \link{cond3} and we are allowed to 
conclude weak tractability from \autoref{prop:suffcond_scaled}.
\end{proof}

Before we turn to normalized errors we want to stress the point that
\autoref{thm:scaled_wt} contains at least two surprising results.
At first, we can have weak tractability even if the initial error
of $S_{(s)}$ grows with increasing dimension.
Hence, although the performance of the zero 
algorithm gets steadily worse for $d\nach\infty$ 
we are not necessarily faced with the curse of dimensionality.
In contrast, remember that we need decreasing initial errors in order to conclude
polynomial tractability.
Secondly, it seems to be quite surprising that also in the case $\lambda_1=\lambda_2$
we can break the curse by imposing only moderate additional conditions on the scaling sequence~$s$.
Indeed, it is enough to guarantee that $\epsilon_d^{\mathrm{init}}=\sqrt{s_d \lambda_1^d}\in o(1/d)$
for $d\nach\infty$.

\subsection{Normalized errors}\label{sect:ScaledProbs_normed}
We complete our studies of the complexity of scaled problems 
$S_{(s)}=(S_{d,s_d})_{d\in\N}$ by investigating tractability
properties with respect to the normalized error criterion.
This can be done by analyzing the information complexity
of a related problem \wrt absolute errors.

Let $\lambda=(\lambda_n)_{n\in\N}$ and $s=(s_d)_{d\in\N}$ be fixed
and define a tensor product problem
$T = \left(T_d \colon H_d \nach \G_d \right)_{d\in\N}$ 
out of the building blocks $T_1=(1/\sqrt{\lambda_1} \, S_1)\colon H_1\nach\G_1$
as described in the proof of \autoref{thm:unweightedtensor_norm}.
Then the extended sequence of squared singular values of $T_d$, 
based on the univariate sequence $\mu=(\mu_m)_{m\in\N}=(\lambda_m/\lambda_1)_{m\in\N}$, 
reads
\begin{gather*} 
			(\mu_{d,1,i})_{i\in\N} 
			= (\mu_{d,i})_{i\in\N} 
			= \left( \frac{\lambda_{d,i}}{\lambda_{d,1}} \right)_{i\in\N}
			= \left( \frac{\lambda_{d,s_d,i}}{s_d \, \lambda_{d,1}} \right)_{i\in\N}
			= \left( \frac{\lambda_{d,s_d,i}}{(\eps_{d}^{\mathrm{init}})^2} \right)_{i\in\N}.
\end{gather*}
Here the second subscript in $\mu_{d,1,i}$ indicates that $T$ 
can be seen as a trivially scaled tensor product problem.
Furthermore, $\eps_{d}^{\mathrm{init}}=\sqrt{s_d\,\lambda_1^d}$ denotes
the initial error of~$S_{d,s_d}$.
Thus, from \link{def_n} applied to $T$ and $S_{(s)}$ we conclude
\begin{align}
			n_{\mathrm{abs}}^\wor(\eps,d; T_d)
			&= \min{n\in\N_0 \sep \mu_{d,1,n+1} \leq \eps^2} \label{eq:scaled_norm}\\
			&= \min{n\in\N_0 \sep \lambda_{d,s_d,n+1} \leq (\eps \cdot \eps_{d}^{\mathrm{init}})^2}
			= n_{\mathrm{abs}}^\wor(\eps \cdot \eps_{d}^{\mathrm{init}},d; S_{d,s_d}). \nonumber
\end{align}
By definition this also equals $n_{\no}^\wor(\eps,d;S_{d,s_d})$, \ie
the information complexity of~$S_{d,s_d}$ \wrt to the normalized error criterion.
This relation in hand, we can use our results from the previous subsections
to prove the following assertion.
\begin{theorem}\label{thm:scaled_no}
		Let $S_{(s)} = (S_{d,s_d})_{d\in\N}$ denote a scaled tensor product problem
		in the sense of \autoref{sect:Scaled_basics}.
		Assume $\lambda_2>0$ and consider the worst case setting \wrt the normalized error criterion.
		\begin{itemize}
					\item Let $\lambda_1=\lambda_2$. Then $S_{(s)}$ suffers from the curse of dimensionality.
					\item Let $\lambda_1>\lambda_2$. Then $S_{(s)}$ is not polynomially tractable.
								Moreover, in this case~$S_{(s)}$ is weakly tractable if and only if $\lambda_n \in o(\ln^{-2} n)$, as $n\nach\infty$.
		\end{itemize}
\end{theorem}
\begin{proof}
Note that for all $d\in\N$ the initial error of $T_d$ is $1$
since $\mu_1$, as well as the scaling parameters, equal $1$.
Thus, obviously, condition~\link{Cond_limsup} in \autoref{thm:scaled_PT} is violated
and therefore~$T$ is polynomially intractable \wrt to the absolute error criterion.
Moreover, the second point of \autoref{thm:scaled_wt} with $\alpha=0$ shows 
that~$T$ suffers from the curse of dimensionality if $\mu_1=\mu_2$.
Otherwise, \ie if $\mu_1>\mu_2$, the problem $T$ is weakly tractable
if and only if $\mu_n \in o(\ln^{-2} n)$, as $n\nach\infty$.
Since we set $\mu_m=\lambda_m/\lambda_1$, $m\in\N$, 
all these conditions on $\mu=(\mu_m)_{m\in\N}$ are fulfilled
if and only if the corresponding assertions holds true for the sequence $\lambda=(\lambda_m)_{m\in\N}$.
Equation~\link{eq:scaled_norm} finally shows that every complexity assertion 
for $T$ \wrt absolute errors
is equivalent to the corresponding statement for $S_{(s)}$ 
and the normalized error criterion.
This simple observation completes the proof.
\end{proof}

In conclusion the scaling sequence $s=(s_d)_{d\in\N}$ 
does not have any influence on the complexity of $S_{(s)}$, as long as we consider normalized errors.
So the advantages of scaling are completely ruled out in this setting.

\section{Examples}\label{sect:ScaledProbs_Ex}
In this last part of \autoref{chapt:ScaledNorms} we briefly discuss
two applications of the complexity results obtained in
the previous section.
We start by proving that our assertions reproduce
the known facts for unscaled tensor product problems studied in
\autoref{thm:unweightedtensor_abs} and \autoref{thm:unweightedtensor_norm}.

\begin{example}[Unscaled problems]\label{ex:unweighted_tensorproducts}
Let $S_{(s)}=(S_{d,s_d})_{d\in\N}$ denote a tensor product problem
between Hilbert spaces in the sense of \autoref{sect:Scaled_basics}
where all the scaling factors $s_d$ equal $1$.
As usual we assume $\lambda_2>0$ and consider the worst case setting.
Then for every $d\in\N$ the operators $S_{d,s_d}$ coincide with $S_d$
as defined in \autoref{subsect:def_tensor_prob}.
Since we already saw that for the normalized error criterion the conditions stated in \autoref{thm:scaled_no}
exactly match the assertions of \autoref{thm:unweightedtensor_norm},
it remains to consider the absolute error criterion.
Here $\eps_d^{\mathrm{init}}$ is given by $\lambda_1^{d/2}$.
Hence there are three scenarios for the behavior of the initial error depending 
on the largest squared singular value $\lambda_1$ of the underlying operator $S_1$.

From \autoref{thm:scaled_PT} we know that 
strong polynomial tractability and polynomial tractability are equivalent; 
see \link{Cond_SPT} and \link{Cond_PT}, respectively.
Moreover, condition~\link{Cond_limsup} shows that this holds if and only if $\lambda_1<1$
and $\lambda=(\lambda_m)_{m\in\N}\in\l_\tau$ for some $\tau\in(0,\infty)$.
In this case the exponent of strong polynomial tractability is given by
\begin{gather*}
		p^* 
		= \inf \left\{2\tau \sep \sup_{d\in\N} \norm{\lambda \sep \l_\tau}^d < \infty \right\}
		= \inf \left\{2\tau \sep \left( \sum_{m=1}^\infty (\lambda_m)^\tau \right)^{1/\tau} \leq 1 \right\}.
\end{gather*}
In turn, $\lambda_1 \geq 1$ yields polynomial intractability.
More precisely, if $\lambda_1>1$ then the initial error grows exponentially in $d$ and $S=S_{(s)}$
suffers from the curse of dimensionality due to the first point of \autoref{thm:scaled_wt}.
Setting $\alpha=0$ the second point of the latter theorem describes 
the case of constant initial errors which corresponds to the case $\lambda_1=1$ in the unscaled situation.
In detail, if $\lambda_2=\lambda_1=1$ then we are faced with the curse again. 
In contrast, if $\lambda_2<1$ then we have weak tractability if and only if
\begin{gather}\label{eq:decay}
			\lambda_n \in o(\ln^{-2} n), \quad \text{as} \quad n\nach\infty.
\end{gather}
Finally the initial error tends to zero exponentially fast if $\lambda_1<1$. 
The last point of \autoref{thm:scaled_wt} thus shows that in this case
the curse of dimensionality is not possible and that
\link{eq:decay} is necessary and sufficient for weak tractability.

Altogether these results exactly match the conditions stated in \autoref{thm:unweightedtensor_abs}.
Hence, scaled tensor product problems indeed yield a generalization.
\hfill$\square$
\end{example}

So let us turn to a more advanced application.
To this end, recall the definition of 
$S_{d,s_d}\colon\F_d\nach\G_d$ in \autoref{sect:Scaled_basics}.
There we constructed the source spaces $\F_d$ by scaling the norm in the
tensor product space $H_d=H_1\otimes\ldots\otimes H_1$.
Alternatively we can think of $\F_d$ as the successively taken tensor product of
some building blocks~$H^{(k)}$, $k=1,\ldots,d$, in the sense of \autoref{subsect:def_tensor_prob}, 
where we define $H^{(k)}$ to be the univariate space~$H_1$ scaled by some factor $s^{(k)}>0$. 
That is, let
\begin{gather*}
			\distr{\cdot}{\cdot}_{H^{(k)}} = \frac{1}{s^{(k)}} \, \distr{\cdot}{\cdot}_{H_1}.
\end{gather*}
Then the scaling factor $s_d$ in dimension $d$ 
is given by $\prod_{k=1}^d s^{(k)}>0$.
The following example illustrates how the behavior of the 
\emph{generator sequence} $(s^{(k)})_{k\in\N}$
effects the complexity of $S_{(s)}$.
\begin{example}
Because scaling has no influence on assertions for 
normalized errors we restrict ourselves
to the absolute error criterion in what follows.
For simplicity we further assume that $\lambda_1>\lambda_2>0$ 
and that the generator sequence is non-increasing, \ie 
\begin{gather*}
		s^{(1)} \geq s^{(2)} \geq \ldots \geq s^{(k)} \geq \ldots > 0, \quad k\in\N.
\end{gather*}
Then \autoref{thm:scaled_PT} states that
$S_{(s)}$ is strongly polynomially tractable if and only if 
the geometric mean of the first $d$ elements $s^{(k)}$
is asymptotically strictly smaller than $1/\lambda_1$, 
provided that $\lambda\in\l_\tau$ for some $\tau>0$.
This holds iff at most finitely many of these generators are bounded from below by $1/\lambda_1$. 
Moreover, from \autoref{rem:neccond_wt} we know that we need
\begin{gather*}
				\limsup_{d\nach \infty} s_d^{1/d} 
				= \limsup_{d\nach \infty} \left(\prod_{k=1}^ds^{(k)} \right)^{1/d} \leq \frac{1}{\lambda_1}
\end{gather*}
in order to obtain weak tractability.
Therefore let the generators be given by
\begin{gather*}
				s^{(k)}=\frac{1}{\lambda_1} \cdot (1+\delta_k), \qquad k\in\N,
\end{gather*} 
with a non-increasing null sequence $(\delta_k)_{k\in\N}$ and note that then
some elementary calculations yield
\begin{gather*}
				\exp{\frac{c}{2} \cdot \sum_{k=1}^d \delta_k} 
				\leq \epsilon_d^{\mathrm{init}} 
				\leq \exp{ \frac{1}{2} \cdot  \sum_{k=1}^d \delta_k}, \qquad d\in\N,
\end{gather*}
where $c=\ln(1+\delta_1)/\delta_1 \leq 1$.
Furthermore this observation shows that 
$\epsilon_d^{\mathrm{init}}\in\Theta(d^\alpha)$, as $d\nach\infty$, for some $\alpha\geq 0$
implies that
\begin{gather*}
				L=\lim_{d\nach\infty} \frac{1}{\ln d} \sum_{k=1}^d \delta_k \in 2\, \alpha \cdot \left[ 1, \frac{1}{c}\right].
\end{gather*}
Conversely, from the existence of $L$ it follows 
that for any $\delta>0$ there is some $d_0=d_0(\delta)$ such that
\begin{gather*}
				\epsilon_d^{\mathrm{init}} \in [d^{\alpha_1},d^{\alpha_2}] \quad \text{for all} \quad d\geq d_0,
\end{gather*}
where $\alpha_1 = c \cdot L/2 - \delta$ and $\alpha_2 = L/2 + \delta$.
Hence, if $L$ is sufficiently small then the initial error 
$\epsilon_d^{\mathrm{init}}$ behaves like a polynomial
of small degree and thus a quite slow decay of the sequence 
$(\lambda_{n})_{n\in\N}$ is enough to conclude weak tractability
using \autoref{thm:scaled_wt}.
\hfill$\square$
\end{example}

  \cleardoubleplainpage
\chapter{Problems on function spaces with weighted norms}\label{chapt:weighted}
In \cite{NW09} it is shown that the approximation problem defined 
on $C^\infty([0,1]^d)$ is intractable. In fact, Novak and 
Wo\'zniakowski considered the linear space $F_d$ of all real-valued 
infinitely differentiable functions $f$ defined on the unit cube 
$[0,1]^d$ in $d$ dimensions for which the norm
\begin{gather}\label{normFd}
        \norm{f \sep F_d} = \sup_{\bm\alpha \in \N_0^d} \norm{D^{\bm\alpha} f \sep \L_\infty([0,1]^d)}
\end{gather}
of $f \in F_d$ is finite.
In this case the (uniform) approximation problem is given by the sequence of solution operators
$S=(S_d)_{d\in\N}$,
\begin{gather}\label{eq:uniform_app}
		S_d = \id_d \colon \widetilde{F}_d \nach \L_{\infty}([0,1]^d), 
		\quad
		f \mapsto \id_d(f)=f,
		\quad d\in\N,
\end{gather}
defined on the unit ball $\widetilde{F}_d=\B(F_d)$ of $F_d$.
The authors studied this problem in the worst case setting using algorithms from the classes
$\A_d^{n, \mathrm{cont}}$ and $\A_d^{n,\mathrm{adapt}}$ as defined in \autoref{sect:Algos}.

The \textit{initial error} of this problem is given by 
$\epsilon_d^{\init} = e^\wor(0,d; \id_d)=1$, the norm of the embedding 
$F_d \hookrightarrow \L_\infty$, since $A_{0,d}\equiv 0$ is a valid 
choice of an algorithm which does not use any information of $f$; see \autoref{prop:init_error}. 
This means that the problem is well-scaled such that there is no
difference in studying the absolute or the normalized error criterion.

Now \cite[Theorem 1]{NW09} yields that the $n$th minimal worst case error of 
$\L_\infty$-approximation defined on $F_d$
satisfies
\begin{gather}\label{eq:NW09_bound}
        e^{\wor}(n,d; \id_d)=1 \quad \text{for all} \quad n=0,1,\ldots, 2^{\floor{d/2}}-1.
\end{gather}
Therefore, for all $d\in \N$ and every $\epsilon \in (0,1)$, 
the information complexity is bounded from below by
\begin{gather*}
        n^{\wor}(\epsilon,d; \id_d) \geq 2^{\floor{d/2}}.
\end{gather*}
Hence the problem suffers from the curse of dimensionality; 
in particular it is intractable.
One possibility to avoid this exponential dependence on $d$, 
\ie to break the curse, is to shrink the function space $F_d$
by introducing \emph{weights}.

In the present chapter we follow this idea.
We show that turning to spaces equipped with \emph{product weights}
can dramatically improve the tractability behavior of certain problems such as uniform approximation.
In \autoref{sect:concept} we formally introduce the concept of weighted spaces
by considering the examples of weighted Banach spaces of smooth functions and of
weighted reproducing kernel Hilbert spaces.
Uniform approximation in the latter class of spaces then is studied in \autoref{sect:uniformRKHS}.
Afterwards, in \autoref{sect:LinftyApprox}, we show how to use the obtained upper error bounds
for the $\L_\infty$-approximation problem defined on scales of smooth functions.
Moreover we prove corresponding lower bounds on the information complexity which enable us
to give necessary and sufficient conditions for several kinds of tractabilities
in terms of the used weights.
Most of the results stated in this chapter are published in the article~\cite{W12}.

\section{The concept of weighted spaces}\label{sect:concept}
The idea to introduce weights directly into the norm of the 
function space appeared for the first time in a paper of 
Sloan and Wo\'zniakowski in 1998; see \cite{SW98}. They studied 
the integration problem defined over some Sobolev Hilbert space, 
equipped with so-called \textit{product weights}, to explain 
the overwhelming success of QMC integration rules. Thenceforth 
weighted problems attracted a lot of attention. 

For example it turned out that tractability of approximation of 
linear compact operators between Hilbert spaces can be fully characterized 
in terms of the weights and the singular values of the operators
if we use information operations from the class~$\Lambda^{\all}$.
The proof of this kind of assertions is once again based on the
singular value decomposition; see \autoref{sect:General_HSP}.
One such result is given in \autoref{sect:final_remarks} below.

But first let us illustrate the concept of weighted spaces
by modifying the space~$F_d$ we introduced before.

\subsection{Weighted Banach spaces of smooth functions}\label{sect:weightedSmooth}
A closer look at the norm given in \link{normFd} yields that for 
$f\in \B(F_d)$ we have
\begin{gather}\label{unitball}
        \norm{D^{\bm\alpha} f \sep \L_\infty([0,1]^d)} \leq 1 
        \quad \text{ for all } \quad {\bm\alpha} \in \N_0^d. 
\end{gather}
Hence every derivative is equally important. 
In order to shrink the space, for each ${\bm\alpha} \in \N_0^d$ 
we replace the right-hand side of inequality \link{unitball} 
by a non-negative weight~$\gamma_{\bm\alpha}$. 
For $\bm\alpha$ with $\abs{\bm\alpha}=1$ this means that we control the 
importance of every single variable.
So, the norm in the weighted space $F_d^\gamma$
is now given by
\begin{gather}\label{eq:weightednorm}
        \norm{f \sep F_d^\gamma} = \sup_{\bm\alpha \in \N_0} 
        \frac{1}{\gamma_{\bm\alpha}} \norm{D^{\bm\alpha} f \sep \L_\infty([0,1]^d)},
\end{gather}
where we demand $D^{\bm\alpha} f$ to be equal to zero if $\gamma_{\bm\alpha}=0$. 
It is clear from the construction that we indeed shrink the space if all $\gamma_{\bm\alpha}$
are chosen strictly less than one.

Since this approach is quite general we restrict ourselves 
to so-called \emph{product weights} (with uniformly bounded generators) in what follows.
Thus we assume that 
for every $d\in\N$ there exists an ordered and uniformly bounded sequence
\begin{gather*}
        C_\gamma \geq \gamma_{d,1} \geq \gamma_{d,2} \geq \ldots \geq \gamma_{d,d} \geq 0.
\end{gather*}
Then for $d\in \N$ the product weight sequence
$\gamma=\left( \gamma_{\bm\alpha} \right)_{\bm\alpha \in \N_0^d}$ is given by
\begin{gather}\label{ProdWeights}
        \gamma_{\bm\alpha} 
        = \prod_{j=1}^d \left( \gamma_{d,j} \right)^{\alpha_j}, 
        \quad \bm\alpha \in \N_0^d.
\end{gather}
Note that the dependence of $x_j$ on $f$ is now controlled by the 
so-called \textit{generator weight}~$\gamma_{d,j}$. 
Since $\gamma_{d,j}=0$ for some $j\in \{1,\ldots,d\}$ implies that $f$ 
does not depend on $x_j,\ldots,x_d$ we assume that $\gamma_{d,d}>0$ 
in the rest of this chapter.
Moreover observe that the ordering of $\gamma_{d,j}$ is without loss of generality. 
Later on we will see that tractability of our problem 
will only depend on summability properties of the generator weights.

Among other things, we show in \autoref{sect:conclusions} that for the $\L_\infty$-approximation 
problem defined on the Banach spaces $F_d^\gamma$ with the norm given above 
and generator weights $\gamma_{d,j}\equiv \gamma^{(j)}\in\Theta \left(j^{-\beta} \right)$ we have
\begin{itemize}
        \item intractability for $\beta = 0$,
        \item weak tractability but no polynomial tractability 
              for $0 < \beta < 1$,
        \item strong polynomial tractability if $1 < \beta$.
\end{itemize}
Furthermore, we prove that for $\beta = 1$ the problem is 
not strongly polynomially tractable.

\subsection{Weighted Hilbert spaces and weighted RKHS}\label{sect:weightedSobolev}
Let us briefly discuss the idea of weighted norms in the case
of Hilbert (function) spaces, before we turn to weighted RKHSs.
Our approach is based on a generalization of the so-called ANOVA\footnote{\textbf{an}alysis \textbf{o}f \textbf{va}riance.}
decomposition of $d$-variate functions $f$, where~$d$ is an arbitrary large integer.
For the ease of presentation we follow the lines of \cite[Section 5.3.1]{NW08}. 
Thus, we focus our attention on Hilbert function spaces constructed out of tensor products
and equipped with some assumptions that can be significantly relaxed.
For further information on more general settings the interested reader 
is referred to \cite{KSWW10b} and the references therein.

Given a $d$-fold tensor product space $H_d=H_1\otimes\ldots\otimes H_1$, $d\in\N$, as well as
an orthonormal basis $\{e_i \sep i\in\N\}$ of the underlying 
univariate Hilbert space\footnote{We assume $H_1$ to be separable and infinite-dimensional to keep the notation as short as possible.} $H_1$ that contains the constant function $e_1 \equiv 1$, it is easy to see that every $f\in H_d$
can be represented as
\begin{gather*}
		f = \sum_{\fu \subseteq \{1,\ldots,d\}} f_{\fu}.
\end{gather*}
In this decomposition the (formally $d$-variate) functions $f_\fu$ solely depend on the
variables $x_j$ with index $j\in\fu$.
The main advantage of this kind of representation is that for fixed $f$ the collection of all $f_\fu$, $\fu\subseteq\{1,\ldots,d\}$, can be taken mutually orthogonal \wrt the inner product $\distr{\cdot}{\cdot}_{H_d}$ in $H_d$.
Therefore the norm of $f\in H_d$ can be expressed by
\begin{gather*}
		\norm{f \sep H_d}^2 
		= \sum_{\fu\subseteq\{1,\ldots,d\}} \norm{f_{\fu} \sep H_d}^2 
		= \sum_{\fu\subseteq\{1,\ldots,d\}} \norm{f_{\fu,1} \sep H_{\abs{\fu}}}^2,
\end{gather*}
where $f_{\fu,1}$ equals $f_\fu$ interpreted as an element 
of the $\abs{\fu}$-fold tensor product space~$H_{\abs{\fu}}$ of the closed subspace
\begin{gather*}
		H_1' = \left\{h \in H_1 \sep \distr{h}{e_1}_{H_1}=0\right\} \subset H_1
\end{gather*}		
with itself.
That is, in the unweighted situation the contribution of each $f_\fu$ to the norm of $f\in H_d$ is the same.

Now suppose that we have some additional, a priori knowledge about the
importance of some (groups of) variables in dimension $d$.
This can be modeled by assigning positive\footnote{Also zero weights are possible but for reasons of simplification we do not discuss this more complicated situation in the present brief introduction to weighted Hilbert spaces.}
weights $\gamma_{d,\fu}$ to each of the $2^d$ subsets $\fu$ of $\{1,\ldots,d\}$.
We denote the collection of these weights by $\gamma_{(d)} = \{\gamma_{d,\fu} \sep \fu \subseteq \{1,\ldots,d\} \}$.
Then it can be verified that
\begin{gather}\label{eq:weighted_innerprod}
		\distr{f}{g}_{\gamma_{(d)}} 
		= \sum_{\fu \subseteq \{1,\ldots,d\}} \frac{1}{\gamma_{d,\fu}} \distr{f_{\fu}}{g_{\fu}}_{H_d}
		= \sum_{\fu \subseteq \{1,\ldots,d\}} \frac{1}{\gamma_{d,\fu}} \distr{f_{\fu,1}}{g_{\fu,1}}_{H_{\abs{\fu}}}
\end{gather}
defines an inner product on the tensor product space $H_d$ which implies an equivalent norm depending on $\gamma_{(d)}$.
The Hilbert space $H_d$ endowed with this new inner product will be denoted by $H_d^{\gamma_{(d)}}$.
At this point we need to stress the fact that for general weights $\gamma_{(d)}$ these spaces are no longer
tensor product spaces, although their construction is based on $H_d=H_1\otimes\ldots\otimes H_1$
and $H_{\abs{\fu}}$, respectively.
To overcome this problem we restrict ourselves to the case of \emph{product weights} in the following.
Thus we assume
\begin{gather}\label{eq:prod_weights_u}
		\gamma_{d,\fu} = \prod_{k\in\fu} \gamma_{d,k}
\end{gather}
for some positive $\gamma_{d,k}$, $k=1,\ldots,d$, and every $\fu\subseteq\{1,\ldots,d\}$.
Then it can be checked that indeed $H_d^{\gamma_{(d)}}$ is again a tensor product space.
For the study of other types of weights such as \emph{finite-order}, \emph{finite-diameter},
\emph{order-dependent} or the recently developed \emph{POD\footnote{\textbf{p}roduct and \textbf{o}rder-\textbf{d}ependent.} weights} we refer to Novak and Wo{\'z}niakowski~\cite[Section 5.3.2]{NW08} and to Kuo, Schwab and Sloan~\cite{KSS11}.

In the last decade it turned out that weighted norms provide a powerful tool to
vanquish the curse of dimensionality that we are often faced with.
Since the $H_d^{\gamma_{(d)}}$'s are still Hilbert spaces 
the complexity analysis of weighted problems $S^{\gamma_{(d)}}=(S_d^{\gamma_{(d)}} \colon \B(H_d^{\gamma_{(d)}}) \nach \G_d)_{d\in\N}$
again is based on the singular value decomposition presented in \autoref{sect:SVD}; at least in
the cases where the target spaces $\G_d$ are also Hilbert spaces.
Fortunately, the introduced weights enter the spectrum of the operator 
$W_d^{\gamma_{(d)}}=\left(S_d^{\gamma_{(d)}}\right)^\dagger S_d^{\gamma_{(d)}}$ in a straightforward way.
Therefore in many cases tractability properties of $S$ can be fully characterized in terms
of the singular values and the introduced weights.\\

For our purposes weighted Hilbert spaces that possess a reproducing kernel are of particular interest.
Typical examples of such weighted RKHSs are the following unanchored Sobolev spaces endowed with product weights 
which will play an important role in our further argumentation; 
see also Sloan and Wo{\'z}niakowski~\cite{SW02}.
Instead of applying the presented approach which is based on decompositions 
we use the common procedure and define them directly.

\begin{example}[Unanchored Sobolev spaces $\Hi_d^\gamma$]\label{ex:Unanchored_Sobolev}
As usual we start with the definition for $d=1$ and $\gamma>0$.
Then the space $\Hi_1^\gamma$ is nothing but the Sobolev space of all absolutely continuous
real-valued functions $f$ defined on the unit interval~$[0,1]$ whose 
first derivative\footnote{in the weak or distributional sense} 
$f'$ belongs to the space $\L_2([0,1])$.
The difference to the classical Sobolev space is the inner product which 
here depends on the parameter~$\gamma$:
\begin{align}
		\distr{f}{g}_{\Hi_1^\gamma} 
		&= \distr{f}{g}_{\L_2([0,1])} + \gamma^{-1} \distr{f'}{g'}_{\L_2([0,1])} \label{eq:uni_prod}\\
		&= \int_0^1 f(x) \, g(x) \, \dlambda^1(x) + \gamma^{-1} \int_0^1 f'(x) \, g'(x) \, \dlambda^1(x),
		\qquad f,g\in\Hi_1^\gamma.\nonumber
\end{align}
For the sake of completeness we define the space $\Hi_1^0$ as the limit of
$\Hi_1^\gamma$ for $\gamma \nach 0$.
Consequently the derivatives of $f\in\Hi_1^0$ need to vanish $\uplambda^1$-almost everywhere
on~$[0,1]$ which implies that the space $\Hi_1^0$ only consists of constant functions.
This coincides with the common convention $0/0=0$.

Note that the univariate space $\Hi_1^\gamma$ algebraically coincides with
its anchored analogue $\widetilde{\Hi}_1^\gamma$ 
where the term $\distr{f}{g}_{\L_2([0,1])}$ in \link{eq:uni_prod} 
is replaced by $f(a)\cdot g(a)$
for some \emph{anchor point} $a\in[0,1]$. 
For details we refer to~\cite{SW02} and~\cite{W12}.
Finally we mention that for positive parameters $\gamma$ all these definitions imply
equivalent norms on the classical Sobolev space $W_2^1([0,1])$.

Once more the $d$-variate spaces $\Hi_d^\gamma$ for $d>1$ are defined by a tensor 
product construction similar to \autoref{subsect:def_tensor_prob}.
We set $\Hi_d^\gamma = \bigotimes_{k=1}^d \Hi_1^{\gamma_{d,k}}$,
where now $\gamma$ denotes
a (subset of a) product weight sequence $(\gamma_{\bm{\alpha}})_{\bm{\alpha}\in\{0,1\}^d}$ 
induced by some generator weights $\gamma_{d,k}$, $k=1,\ldots,d$; see \link{ProdWeights}.
Remember that at the beginning of this chapter we assumed $\gamma_{d,d}>0$ for all $d\in\N$. 
That is, we avoid to take the trivial spaces $\Hi_1^0$ as factors in the definition of $\Hi_d^\gamma$.

How does the inner product of $\Hi_d^\gamma$ looks like?
Following the lines of \autoref{subsect:def_tensor_prob} it is uniquely determined by the coordinate-wise
inner products of the factors of simple tensors $f=\bigotimes_{k=1}^d f_k$ and $g=\bigotimes_{k=1}^d g_k$,
where $f_k,g_k\in \Hi_1^{\gamma_{d,k}}$ for $k=1,\ldots,d$.
Consequently,
\begin{align*}
			\distr{f}{g}_{\Hi_d^{\gamma}} 
			&= \prod_{k=1}^d \distr{f_k}{g_k}_{\Hi_1^{\gamma_{d,k}}}
			= \prod_{k=1}^d \left( \distr{f_k}{g_k}_{\L_2([0,1])} + \frac{1}{\gamma_{d,k}} \, \distr{f_k'}{g_k'}_{\L_2([0,1])} \right) \\
			&= \sum_{\fu \subseteq \{1,\ldots,d\}} \prod_{k\in \fu} \frac{1}{\gamma_{d,k}} \cdot \prod_{k\in \fu} \distr{f_k'}{g_k'}_{\L_2([0,1])} \cdot \prod_{j\in \{1,\ldots,d \}\setminus\fu} \distr{f_j}{g_j}_{\L_2([0,1])}\\
			&= \sum_{\fu \subseteq \{1,\ldots,d\}} \prod_{k\in \fu} \frac{1}{\gamma_{d,k}} \cdot 
					\int_{[0,1]^d} \prod_{k\in \fu} f_k'(x_k) \, g_k'(x_k) \prod_{j\in \{1,\ldots,d \}\setminus\fu} f_j(x_j) \, g_j(x_j) \dlambda^d(\bm{x}) \\
			&= \sum_{\fu \subseteq \{1,\ldots,d\}} \frac{1}{\gamma_{d,\fu}} \cdot 
					\int_{[0,1]^d} \frac{\partial^{\abs{\fu}} f}{\partial x_\fu}(\bm{x}) \, \frac{\partial^{\abs{\fu}} g}{\partial x_\fu}(\bm{x}) \dlambda^d(\bm{x}),
\end{align*}
where we used \link{eq:prod_weights_u} and the shorthand notation $\partial^{\abs{\fu}}/(\partial x_\fu)$ 
for $\prod_{k\in\fu} \partial/(\partial x_k)$.
Note that this representation resembles \link{eq:weighted_innerprod} from the general approach 
to weighted Hilbert spaces introduced at the beginning of this subsection.
For our purposes it is more convenient to rewrite the subsets 
$\fu \subseteq \{1,\ldots,d\}$ in terms of multi-indices $\bm{\alpha}=(\alpha_1,\ldots,\alpha_d)\in\{0,1\}^d$.
In detail, we set $\alpha_k=1$ if and only if $k\in\fu$ and $\alpha_k=0$ otherwise.
Then we can express the norm of any $f\in \Hi_d^\gamma$ by
\begin{gather}\label{eq:norm_Sobol}
		\norm{f \sep \Hi_d^\gamma}^2 
		= \sum_{\bm{\alpha} \in \{0,1\}^d} \frac{1}{\gamma_{\bm{\alpha}}} \cdot 
					\int_{[0,1]^d} \abs{D^{\bm{\alpha}}f(\bm{x})}^2 \dlambda^d(\bm{x})
\end{gather}
since then $\gamma_\fu=\gamma_{\bm{\alpha}}$.
The inner products of the multivariate anchored spaces, $\widetilde{\Hi}_d^\gamma$, can be found by a similar reasoning;
see \cite[p. 67]{W12} for the final result.

It is known (cf. Micchelli and Wahba~\cite{MW81}) that the univariate spaces 
$\Hi_1^\gamma$ are reproducing kernel Hilbert spaces for any $\gamma>0$.
Consequently, this property is transferred to the multivariate tensor product space.
To stress this fact we write $\Hi(K_d^\gamma)$ for $\Hi_d^\gamma$ in what follows.
Equation~(5) in~\cite{WW09} now states that the 
reproducing kernel $K_d^\gamma\colon [0,1]^d\times[0,1]^d\nach\R$ 
in dimension $d\geq 1$ is given by\footnote{Here $\sinh$ and $\cosh$ denote the hyperbolic sine and cosine functions, respectively.}
\begin{align*}
				&K_d^\gamma(\bm{x},\bm{y}) \\
				&\quad = \prod_{k=1}^d \frac{\sqrt{\gamma_{d,k}}}{\sinh\!\left(\sqrt{\gamma_{d,k}}\right)} \, \cosh\!\left(\sqrt{\gamma_{d,k}} \, (1-\max{x_k,y_k})\right) \, \cosh\!\left(\sqrt{\gamma_{d,k}}\, \min{x_k,y_k}\right),
\end{align*}
$\bm{x},\bm{y}\in[0,1]^d$.
For $d=1$ this kernel formula follows from Thomas-Agnan~\cite[Corollary~2]{TA96} whereas the higher-dimensional generalization
for product weights $\gamma$ results from the tensor product structure; see \link{eq:tensor_RKHS} in \autoref{sect:RKHS}.
In particular we note that $K_d^\gamma$ is continuous (and thus also bounded) 
along its diagonal 
\begin{gather*}
			\left\{(\bm{x},\bm{y})\in[0,1]^{2d} \sep \bm{x}=\bm{y}\right\}.
\end{gather*}

Moreover, from \cite[Lemma 4.1]{WW09} we know that for $\gamma>0$ the set
\begin{gather*}
				E_1(\gamma) = \left\{ e_{1,\gamma,i} \colon [0,1]\nach\R \sep i\in\N \right\}
\end{gather*}
with $e_{1,\gamma,1} \equiv 1$ and
\begin{gather*}
				e_{1,\gamma,i}(x) 
				= \cos(\pi (i-1) x) \cdot \sqrt{\frac{2\gamma}{\gamma + \pi^2(i-1)^2}}, 
				\qquad x\in[0,1], \quad i \geq 2,
\end{gather*}
builds an orthonormal basis in the univariate space $\Hi(K_1^\gamma)$.
Applying the arguments from \autoref{subsect:def_tensor_prob} this
leads to an ONB $E_d(\gamma)$ of $\Hi(K_d^\gamma)=\bigotimes_{k=1}^d \Hi(K_1^{\gamma_{d,k}})$ that consists of tensor product functions 
\begin{gather}\label{eq:basis_sobolev}
				\widetilde{e}_{d,\gamma,\bm{m}} 
				= \bigotimes_{k=1}^d e_{1,\gamma_{d,k},m_k}, \qquad \bm{m}=(m_1,\ldots,m_d)\in\N^d.
\end{gather}
For a direct proof of this result we refer to \cite[Appendix A.2.1]{NW08}\footnote{Note the missing factor $1/2$ in \cite[p. 351, line 5]{NW08}.}
and to~\cite[Lemma 4.2]{WW09}.
Actually, these proofs show a little bit more;
namely that the functions~$\widetilde{e}_{d,\gamma,\bm{m}}$ together with 
\begin{gather}\label{eq:L2_eigenvalues}
				\widetilde{\lambda}_{d,\gamma,\bm{m}} 
				= \prod_{k=1}^d \lambda_{1,\gamma_{d,k},m_k} 
				= \prod_{k=1}^d \frac{\gamma_{d,k}}{\gamma_{d,k} + \pi^2 (m_k-1)^2},
				\qquad \bm{m}\in\N^d,
\end{gather}
describe the full set of eigenpairs $\{(\widetilde{\lambda}_{d,\gamma,\bm{m}}, \widetilde{e}_{d,\gamma,\bm{m}}) \,|\, \bm{m}\in\N^d\}$ of the 
operator $W_d^\gamma = \left( S_d^\gamma \right)^\dagger S_d^\gamma$
where $S_d^\gamma \colon \Hi_d^\gamma \hookrightarrow \L_2([0,1]^d)$
denotes the solution operator of the $\L_2$-approximation problem on 
$\Hi_d^\gamma=\Hi(K_d^\gamma)$.
\hfill$\square$
\end{example}

\section{Uniform approximation in reproducing kernel Hilbert spaces}\label{sect:uniformRKHS}
The main result of this section is based on a paper of Kuo, Wasilkowski and Wo{\'z}niakowski~\cite{KWW08}.
In contrast to the presentation given in \cite{W12} we decided to apply this result
to the case of the unanchored Sobolev Space introduced in \autoref{sect:weightedSobolev}
instead of the anchored analogue studied in~\cite{KWW08}.
This opens up the opportunity to explain the underlying ideas 
without literally repeating the proof given in~\cite{KWW08} while obtaining a result
which is (according to our knowledge) not published elsewhere so far.

We start with an upper error bound which remains valid for 
any reproducing kernel Hilbert space $\Hi(K_d)$ of real-valued functions $f$ on $[0,1]^d$ with
\begin{gather}\label{eq:BoundedKernel}
		\esssup_{\bm{x}\in[0,1]^d} K_d(\bm{x},\bm{x})<\infty.
\end{gather}
This condition guarantees that $\Hi(K_d)$ is continuously embedded into $\L_\infty([0,1]^d)$
since the reproducing property \link{eq:reproducing_prop}, together with the Hahn-Banach theorem (cf.~\cite[IV.6 Cor.2]{Y80}), yields
that $\norm{\id_d \sep \LO(\Hi(K_d),\L_\infty([0,1]^d))}$ is given by
\begin{align*}
		\sup_{f\in\B(\Hi(K_d))} \norm{f \sep \L_\infty([0,1]^d)} 
		&= \esssup_{\bm{x}\in [0,1]^d} \sup_{f\in\B(\Hi(K_d))} \abs{f(\bm{x})} \\
		&= \esssup_{\bm{x}\in [0,1]^d} \sup_{f\in\B(\Hi(K_d))} \abs{\distr{f}{K_d(\cdot,\bm{x})}_{\Hi(K_d)}} \\
		&= \esssup_{\bm{x}\in [0,1]^d}  K_d(\bm{x},\bm{x})^{1/2}.
\end{align*}

Now the mentioned upper bound reads as follows:
\begin{prop}\label{prop:L_inftyAlgo}
			For $d\in\N$ consider a RKHS $\Hi(K_d)$, where $K_d$ fulfills~\link{eq:BoundedKernel}, \ie $\Hi(K_d)\hookrightarrow\L_\infty([0,1]^d)$.
			Furthermore, suppose $\Xi=\{ \xi_j \colon [0,1]^d\nach\R \sep j\in\N\}$ 
			to be some orthonormal basis of $\Hi(K_d)$ and let $n\in\N_0$.
			Then the algorithm $A_{n,d}^{\Xi}\in\A_d^{n,\lin}(\Lambda^{\all})$, given by
			\begin{gather*}
					f\mapsto A_{n,d}^{\Xi} f = \sum_{j=1}^n \distr{f}{\xi_j}_{\Hi(K_d)}\xi_j(\cdot),
			\end{gather*}
			for uniform approximation on $\Hi(K_d)$ fulfills
			\begin{gather}\label{eq:L_inftyErrorBound}
					\Delta^{\wor}(A^{\Xi}_{n,d}; \id_d \colon \B(\Hi(K_d)) \nach \L_\infty([0,1]^d)) 
					\leq \norm{ \sum_{j=n+1}^\infty \xi_j(\cdot)^2 \sep \L_{\infty}([0,1]^d) }^{1/2}.
			\end{gather}
\end{prop}
\begin{proof}
Since $\Xi$ builds an ONB we may represent any
$f\in\Hi(K_d)$ by its basis expansion, $f=\sum_{j=1}^\infty \distr{f}{\xi_j}_{\Hi(K_d)}\, \xi_j$.
Therefore Parseval's identity implies
\begin{align*}
		\abs{f(\bm{x}) - A^{\Xi}_{n,d}f(\bm{x})}
		&= \abs{(f - A^{\Xi}_{n,d}f)(\bm{x})}
		= \abs{\sum_{j=n+1}^\infty \distr{f}{\xi_j}_{\Hi(K_d)}\, \xi_j(\bm{x})}\\
		&=\abs{\distr{f}{\sum_{j=n+1}^\infty \xi_j(\bm{x})\,\xi_j}_{\Hi(K_d)}}
\end{align*}
which can be estimated from above using the inequality of Cauchy and Schwarz.
Thus we obtain
\begin{align}
		\abs{f(\bm{x}) - A^{\Xi}_{n,d}f(\bm{x})}
		&\leq \norm{f \sep \Hi(K_d)} \cdot \norm{\sum_{j=n+1}^\infty \xi_j(\bm{x})\,\xi_j \sep \Hi(K_d)} \label{pw_est}\\
		&= \norm{f \sep \Hi(K_d)} \cdot \left(\sum_{j=n+1}^\infty \xi_j(\bm{x})^2 \right)^{1/2} \nonumber
\end{align}
for every $f\in\Hi(K_d)$ and all fixed $\bm{x}\in[0,1]^d$.
Taking the (essential) supremum with respect to $\bm{x}$ in the $d$-dimensional unit cube and
the supremum over all $f\in\B(\Hi(K_d))$ gives the desired result.
\end{proof}

We note in passing that we can easily prove more than we stated in the latter assertion.
In what follows we only need the given upper error bound
such that we restrict ourselves to some brief comments on further results in the next remark.

\begin{rem}
For fixed $\bm{x}\in[0,1]^d$ we see that the function $f^*=C\cdot \sum_{j=n+1}^\infty \xi_j(\bm{x})\,\xi_j$ with $C>0$ gives equality in \link{pw_est}. 
Of course, we can choose the constant $C$ such that 
$\norm{f^* \sep\Hi(K_d)}=1$ provided that $\bm{x}$ is not a common root of $\xi_j$ for all $j > n$.
Hence, the upper bound in \link{eq:L_inftyErrorBound} is sharp.

Moreover, \cite[Theorem 2]{KWW08} shows that the $n$th minimal worst case error for $\L_\infty$-approximation on $\Hi(K_d)$ is given by
\begin{gather*}
			e^{\wor}(n,d; \id_d \colon \B(\Hi(K_d)) \nach \L_\infty([0,1]^d)) 
			= \inf_{\Xi=\{\xi_j \sep j\in\N \}} \norm{ \sum_{j=n+1}^\infty \xi_j(\cdot)^2 \sep \L_{\infty}([0,1]^d) }^{1/2},
\end{gather*}
where the infimum is taken \wrt all orthonormal bases $\Xi\subset\Hi(K_d)$.
Thus, any clever choice of the basis $\Xi$ in \autoref{prop:L_inftyAlgo} leads to algorithms $A^{\Xi}_{n,d}$ with almost optimal worst case errors.
\hfill$\square$
\end{rem}

Next we apply \autoref{prop:L_inftyAlgo} to the weighted unanchored Sobolev spaces $\Hi_d^\gamma$ 
introduced in  \autoref{sect:weightedSobolev} using the basis $\Xi=E_d(\gamma)$ given in \link{eq:basis_sobolev}.
Since the ordering of the basis functions $\xi \in \Xi$ is essential for our application 
we rearrange them non-increasingly with respect to their $\L_\infty$-norm:
\begin{gather}\label{eq:OrderedBasis}
		\norm{\xi_j \sep \L_{\infty}([0,1]^d)} 
		\geq \norm{\xi_{j+1} \sep \L_{\infty}([0,1]^d)}
		\quad \text{for all} \quad j\in\N.
\end{gather}
We obtain an estimate which resembles the corresponding result for the anchored case studied in \cite[Proposition~2]{W12}.
\begin{cor}\label{cor:optAlgoSobolev}
			For $n\in\N_0$ and $d\in\N$ there exists an algorithm $A_{n,d}^*\in\A_d^{n,\lin}(\Lambda^{\all})$ for uniform approximation on $\Hi^\gamma_d$ such that for every $\tau \in (1/2,1)$ 
			\begin{gather*}
						\Delta^\wor(A_{n,d}^*; \id_d\colon \B(\Hi_d^\gamma)\nach \L_{\infty}([0,1]^d) )
				< a_\tau \exp{b_\tau \, \sum_{k=1}^d (\gamma_{d,k})^\tau} \cdot n^{-(1-\tau)/(2\tau)},
			\end{gather*}
			where the constants $a_\tau, b_{\tau}>0$ are independent of $\gamma$, $n$, and $d$.
\end{cor}
\begin{proof}
To keep the notation as short as possible we abbreviate the $\L_\infty$-norm in $d$ dimensions,
$\norm{\cdot \sep \L_{\infty}([0,1]^d)}$, by $\norm{\cdot}_d$ within this proof.

Following our plan we fix $n\in\N_0$, as well as $d\in\N$, and take $A_{n,d}^*=A_{n,d}^{\Xi}$ 
defined in \autoref{prop:L_inftyAlgo} with $\Xi=E_d(\gamma)$ as above.
From \link{eq:basis_sobolev} we conclude for $d=1$ and any $\gamma>0$ that
\begin{gather*}
		\norm{e_{1,\gamma,1}^2}_1 = 1
		\quad \text{and} \quad
		\norm{e_{1,\gamma,i}^2}_1 = \frac{2\gamma}{\gamma + \pi^2 (i-1)^2} < \frac{2\,\gamma}{\pi^2} \cdot (i-1)^{-2}, \quad i\geq 2.
\end{gather*}
Moreover, for every simple tensor $f = \bigotimes_{k=1}^d f_k \in \Hi(K_d^\gamma)$ we clearly have 
\begin{gather*}
		\norm{f}_d = \prod_{k=1}^d \norm{f_k}_1
		\quad \text{and} \quad f(\bm{x})^2 = \prod_{k=1}^d f_k(x_k)^2, \quad \bm{x}\in[0,1]^d.
\end{gather*}
Consequently, for any $j\in\N$ and all $\tau \in(1/2,\infty)$ 
the ordering of $\Xi$ given in \link{eq:OrderedBasis} implies
\begin{align*}
			 j \cdot \norm{\xi_j^2}_d^\tau
			 &\leq \sum_{m=1}^\infty \norm{\xi_m^2}_d^\tau
			 = \sum_{\bm{m}\in\N^d} \norm{\widetilde{e}_{d,\gamma,\bm{m}}^2}_d^\tau
			 = \prod_{k=1}^d \sum_{i=1}^\infty \norm{e_{1,\gamma_{d,k},i}^2}_1^\tau \\
			 &=\prod_{k=1}^d \left( 1 + \sum_{i=2}^\infty \norm{e_{1,\gamma_{d,k},i}^2}_1^\tau \right)
			 <\prod_{k=1}^d \left( 1 + \left(\frac{2\, \gamma_{d,k}}{\pi^2}\right)^\tau \sum_{i=2}^\infty (i-1)^{-2\tau} \right) \\
			 &= \prod_{k=1}^d \left( 1 + c_\tau \gamma_{d,k}^\tau \right),
\end{align*}
where we set $c_\tau = (2/\pi^2)^\tau \, \zeta(2\tau)$.
Hence, if $\tau\in(1/2,1)$ then
\begin{align*}
			 \norm{\sum_{j=n+1}^\infty \xi_j^2}_d
			 \leq \sum_{j=n+1}^\infty \norm{\xi_j^2}_d
			 < \sum_{j=n+1}^\infty j^{-1/\tau} \cdot \left(\prod_{k=1}^d \left( 1 + c_\tau \gamma_{d,k}^\tau \right) \right)^{1/\tau}<\infty.
\end{align*}
Since the first factor is no larger than $\int_n^\infty x^{-1/\tau} \dlambda^1(x)= \tau/(1-\tau)\cdot n^{-(1-\tau)/\tau}$
and the second factor can be bounded by $\exp{c_\tau/\tau \cdot \sum_{k=1}^d (\gamma_{d,k})^\tau}$
we conclude
\begin{gather*}
				\norm{\sum_{j=n+1}^\infty \xi_j(\cdot)^2 \sep \L_{\infty}([0,1]^d)}^{1/2} 
				< a_\tau \exp{b_\tau \, \sum_{k=1}^d (\gamma_{d,k})^\tau} \cdot n^{-(1-\tau)/(2\tau)}
\end{gather*}
with $a_\tau = \sqrt{\tau/(1-\tau)}$ and $b_\tau = c_\tau/(2\tau)=(2/\pi^2)^\tau \, \zeta(2\tau) / (2\tau)$.
Now the claim follows from \link{eq:L_inftyErrorBound} in \autoref{prop:L_inftyAlgo}.
\end{proof}

\section{Uniform approximation in Banach spaces of smooth functions}\label{sect:LinftyApprox}
Our derivation of necessary and sufficient conditions for various kinds of tractability
for the $\L_\infty$-approximation problem defined on the weighted spaces $F_d^\gamma$
introduced in \autoref{sect:weightedSmooth} is based on simple embedding arguments.
To this end, we consider a whole scale of Banach spaces $\F_d^\gamma$ (where $F_d^\gamma$ is a special case of).
Then we first study lower bounds on the $n$th minimal error on 
a space $\P_d^\gamma \hookrightarrow \F_d^\gamma$ which consists of
$d$-variate polynomials of low degree. 
Afterwards, in \autoref{sect:embeddings}, we use the results 
for $\Hi_d^\gamma \hookleftarrow \F_d^\gamma$
from \autoref{sect:uniformRKHS} to conclude corresponding upper bounds.
Finally we discuss a couple of concrete examples 
in \autoref{sect:conclusions}.

\subsection{Lower bounds for spaces of low-degree polynomials}\label{sect:pol_bound}
Following the lines of 
\cite[Section 4]{W12} we use \autoref{needed_lemma} to obtain a lower bound 
for the $\L_\infty$-approximation error for the space
\begin{gather*}
        \P_d^\gamma
        = \spann{p_{\bm{i}} \colon [0,1]^d \nach \R, \, p_i(\bm{x}) = \bm{x}^{\bm{i}} = \prod_{j=1}^d \left( x_j\right)^{i_j} 
        		\sep \bm{i}=(i_1,\dots,i_d)\in\{0,1\}^d}
\end{gather*}
of all real-valued $d$-variate polynomials of degree at most one 
in each coordinate direction, defined on the unit cube $[0,1]^d$. 
We equip this linear space with the weighted norm
\begin{gather}\label{eq:norm_P}
        \norm{f \sep \P_d^\gamma} = \maxx_{\bm{\alpha} \in \{0,1\}^d} 
        \frac{1}{\gamma_{\bm{\alpha}}} \norm{D^{\bm{\alpha}} f \sep \L_\infty([0,1]^d)}, \qquad f \in \P_d^\gamma,
\end{gather}
similar to \link{eq:weightednorm}, where $\gamma$ is a product weight sequence as described 
in \link{ProdWeights}, and study the worst case setting.

\begin{theorem}\label{Theorem_Polynom}
				For $d\in\N$ and $n\in\N_0$ assume 
				$A_{n,d}\in\A_d^{n, \rm cont} \cup \A_d^{n, \rm adapt}$
				to be an arbitrary algorithm for the uniform approximation problem defined on $\P_d^\gamma$.
				Then we have 
        \begin{gather*}
                \Delta^\wor (A_{n,d}; \id_d \colon B_r(\P_d^\gamma)\nach \L_{\infty}([0,1]^d)) 
                \geq r \quad \text{for all} \quad r\geq 0
        \end{gather*}
        provided that $n<2^s$, where $s=s(\gamma,d)\in\{0,1,\ldots,d\}$ is some integer such that
				\begin{gather}\label{estimate_s}
                s > \frac{1}{2+C_\gamma} \cdot \left(\sum_{j=1}^d \gamma_{d,j} - 2 \right).
        \end{gather}
\end{theorem}

\begin{proof}
The proof of this lower error bound consists of several steps. 
First we fix $d\in\N$ and construct a partition of the set of coordinates $\{1,\ldots,d\}$ 
into $s+1$ parts which we will need later and with $s=s(\gamma,d)$ 
satisfying~\link{estimate_s}. 
In a second step we define a special 
linear subspace $V \subseteq \P_d^{\gamma}$ with $\dim V = 2^s$. 
Step~3 then shows that~$V$ satisfies the assumptions of 
\autoref{needed_lemma}. The proof is completed in Step~4.

\textit{Step 1}. 
For $k\in\{0,\ldots,d\}$ let us define inductively $m_0=0$ and
\begin{gather*}
        m_k = \inf \left\{t \in \N \sep m_{k-1} < t 
        \leq d, \, \text{ with } \, 2 \leq \sum_{j=m_{k-1}+1}^t \gamma_{d,j} \right\}
\end{gather*}
with the usual convention $\inf \leer = \infty$. 
Note that the infimum coincides with the minimum in the finite case, since 
then $m_k\in\N$. 
Moreover we set
\begin{gather*}
        s = \max{ k\in\{0,\ldots,d\} \sep m_k < \infty }.
\end{gather*}
We denote 
$I_k = \{ m_{k-1}+1, m_{k-1}+2, \ldots, m_k \}$ for $k=1,\ldots,s$. 
Thus, this gives a uniquely defined disjoint partition of the set 
\begin{gather*}
        \{1,\ldots,d\} = \left( \bigcup_{k=1}^s I_k \right) \cup \{m_s+1,\ldots,d\}, 
\end{gather*}
and $m_k$ denotes the last element of the block $I_k$. 
For all $k=1,\ldots,s$ we conclude
\begin{gather*}
				2 \leq \sum_{j\in I_k} \gamma_{d,j} < 2 + \gamma_{d,m_k} \leq 2 + C_\gamma,
\end{gather*}
where $C_\gamma$ is the uniform upper bound for $\gamma_{d,j}$; see \autoref{sect:weightedSmooth}.
Finally, summation of these inequalities gives 
\begin{gather*}
				\sum_{j=1}^{d} \gamma_{d,j} < \sum_{k=1}^s \sum_{j\in I_k} \gamma_{d,j} + 2 < (2+C_\gamma) s + 2,
\end{gather*}
and \link{estimate_s} follows immediately. 

If $s=0$ then we can stop at this point since the initial error is $1$ 
as the norm of the embedding $\P_d^\gamma\hookrightarrow \L_\infty$ (cf. \autoref{prop:init_error})
and the remaining assertion is trivial.
Hence, from now on we can assume that $s>0$ and thus $m_s \geq 1$.

\textit{Step 2}. 
To apply \autoref{needed_lemma} we have to construct a linear subspace 
$V$ of $\F = \P_d^\gamma$ such that the condition \link{eq:NormCondition} 
holds for the target space $\G = \L_\infty([0,1]^d)$, the embedding operator $S=\id_d$, and $a = 1$.
Note that we restrict ourselves to the set 
\begin{gather*}
				\widehat{\F} = \left\{ f\in \F \sep \ f \text{ depends only on } x_1,\ldots,x_{m_s} \right\},
\end{gather*}
since we can interpret $\widehat{\F}$ 
as the space $\P_{m_s}^\gamma$ by a simple isometric isomorphism. 

We are ready to construct a suitable space $V$ using the partition 
from Step~1. 
We define~$V$ as the span of all functions
$g_{\bm{i}} \colon [0,1]^{m_s} \nach \R$, $\bm{i}=(i_1,\dots,i_s) \in\{0,1\}^s$, 
of the form
\begin{gather*}
        g_{\bm{i}}(\bm{x}) 
        = \prod_{k=1}^s \left( \sum_{j\in I_k} \gamma_{d,j} \cdot x_j \right)^{i_k}, 
        \quad \bm{x} \in X = [0,1]^{m_s}.
\end{gather*}
Clearly, $V$ is a linear subspace of $\P_{m_s}^\gamma$ and 
with the interpretation above it is also a linear subspace 
of $\F$. 
Moreover it is easy to see that we have by construction 
\begin{gather*}
				\norm{g \sep \F} = \norm{g \sep \P_{m_s}^\gamma}
				\quad \text{and}\quad 
				\norm{g \sep \L_\infty(X)} = \norm{g \sep \L_\infty([0,1]^d)}
				\quad \text{for} \quad g\in V.
\end{gather*}
Finally we note that $\dim V = \# \{0,1\}^s = 2^s$. 
It remains to show 
that this subspace is the right choice to prove the claim using 
\autoref{needed_lemma}.

\textit{Step 3}. 
The proof of the needed condition \link{eq:NormCondition},
\begin{gather*}
        \norm{g \sep \P_{m_s}^\gamma} \leq 
        \norm{g \sep \L_\infty(X)} \quad \text{for all} \quad g\in V,
\end{gather*}
is a little bit technical. Due to the special structure of the functions 
$g\in V$, the left-hand side reduces to 
$\max{\gamma_{\bm{\alpha}}^{-1} \norm{D^{\bm{\alpha}} g \sep \L_\infty(X)} \sep \bm{\alpha}\in \mathbb{M}}$,
where the maximum is taken over all multi-indices $\bm{\alpha}$ in the set
\begin{gather*}
        \mathbb{M}= \left\{ {\bm{\alpha}} \in \{0,1\}^{m_s} 
        \sep \sum_{j\in I_k} \alpha_j \leq 1 \text{ for all } k=1,\ldots,s \right\}.
\end{gather*}
This is simply because for ${\bm{\alpha}} \notin \mathbb{M}$ we have $D^{\bm{\alpha}} g \equiv 0$ 
and then the inequality is trivial. 
To simplify the notation let us define
\begin{gather*}
        T \colon \{0,1\}^{m_s} \nach \N_0^s, \quad 
        {\bm{\alpha}} \mapsto T({\bm{\alpha}}) = \bm{\sigma} = (\sigma_1,\dots,\sigma_s),
\end{gather*} 
where
\begin{gather*}
        \sigma_k= \sum_{j\in I_k} \alpha_j \quad \text{for} \quad k=1,\ldots,s.
\end{gather*}
Note that $T(\mathbb{M})=\{0,1\}^s$.
Moreover, for every $g=\sum_{\bm{i}\in\{0,1\}^s} c_{\bm{i}} \, g_{\bm{i}}(\cdot) \in V$ 
we define a function
\begin{gather*}
        h_g \colon Z= \bigtimes_{k=1}^s \left[ 0, \sum_{j\in I_k} \gamma_{d,j} \right] 
        \nach \R, 
        \qquad \bm{z} \mapsto h_g(\bm{z}) = \sum_{\bm{i}\in \{0,1\}^s} c_{\bm{i}} \prod_{k=1}^s z_k^{i_k} = \sum_{\bm{i}\in \{0,1\}^s} c_{\bm{i}} \, \bm{z}^{\bm{i}}.
\end{gather*}
Hence, $h_g(\bm{z})=g(\bm{x})$ under the transformation $\bm{x}\mapsto \bm{z}$ such that 
\begin{gather*}
				z_k = \sum_{j\in I_k} \gamma_{d,j} x_j 
				\quad \text{for every}\quad
				k=1,\ldots,s
				\quad \text{and every}\quad 
				\bm{x}\in X. 
\end{gather*}
The span, $W$, of all functions $h\colon Z \nach \R$ with this structure is a linear space, too. 
Furthermore, easy calculus yields that
\begin{gather}\label{g_and_h}
        \left( D_{\bm{x}}^{\bm{\alpha}} g \right)(\bm{x}) 
        = \left( \prod_{j=1}^{m_s} \left( \gamma_{d,j} \right)^{\alpha_j} \right) \left( D_{\bm{z}}^{T({\bm{\alpha}})} h_g \right)(\bm{z}) 
\end{gather}
for all $g\in V$, $\bm{\alpha} \in \mathbb{M}$ and $\bm{x}\in X$.
Here the $\bm{x}$ and $\bm{z}$ in $D_{\bm{x}}^{\bm{\alpha}}$ and $D_{\bm{z}}^{T({\bm{\alpha}})}$ 
indicate differentiation with respect to $\bm{x}$ and $\bm{z}$, respectively.
Since the mapping $\bm{x} \mapsto \bm{z}$ is 
surjective we obtain 
$\norm{D^{\bm{\alpha}} g \sep \L_\infty(X)} = \gamma_{\bm{\alpha}} 
\norm{D^{T({\bm{\alpha}})} h_g \sep \L_\infty(Z)}$ by the form of $\gamma$ given by 
\link{ProdWeights}. 
Thus,
\begin{gather*}
        \maxx_{{\bm{\alpha}} \in \mathbb{M}} \frac{1}{\gamma_{\bm{\alpha}}} \norm{D^{\bm{\alpha}} g \sep \L_\infty(X)} 
        = \maxx_{\bm{\sigma} \in \{0,1\}^s} \norm{D^{\bm{\sigma}} h_g \sep \L_\infty(Z)}.
\end{gather*}
Observe that \link{g_and_h} with $\bm{\alpha}=0$ particularly yields that
$\norm{g \sep \L_\infty(X)} = \norm{h_g \sep \L_\infty(Z)}$. 
Therefore the claim reduces to 
\begin{gather*}
				\maxx_{\bm{\sigma} \in \{0,1\}^s} \norm{D^{\bm{\sigma}} h_g \sep \L_\infty(Z)} 
				\leq \norm{h_g \sep \L_\infty(Z)} 
				\quad \text{for every} \quad g\in V.
\end{gather*}
We show this estimate for every $h \in W$, i.e.,  
\begin{gather}\label{estimate_norm_h_1}
        \norm{D^{\bm{\sigma}} h \sep \L_\infty(Z)} \leq \norm{h \sep \L_\infty(Z)} 
        \quad \text{for all} \quad 
				\bm{\sigma} \in \{0,1\}^s.
\end{gather}
We start with the special case of one derivative. 
That is, we first consider $\bm{\sigma} = \bm{e_k}$ 
for a certain $k\in \{1,\ldots,s\}$. 
Since $h$ is affine 
in each coordinate we can represent it as
\begin{gather*}
			h(\bm{z}) = a({\bm{z_{(k)}}}) \cdot z_k + b({\bm{z_{(k)}}})
\end{gather*}
with functions $a$ and $b$ which only depend on 
${\bm{z_{(k)}}}=(z_1,\ldots,z_{k-1},z_{k+1},\ldots,z_s)$. 
Hence we have 
$(D^{\bm{e_k}} h)(\bm{z}) = a({\bm{z_{(k)}}})$ and we need to show that
\begin{gather}\label{estimate_norm_h_2}
        \abs{a({\bm{z_{(k)}}})} 
        \leq \max{ \abs{b({\bm{z_{(k)}}})}, \abs{a({\bm{z_{(k)}}}) \cdot 
        \sum_{j\in I_k} \gamma_{d,j} + b({\bm{z_{(k)}}})}}.
\end{gather}
This is obviously true for every $\bm{z} \in Z$ with $a({\bm{z_{(k)}}})=0$. 
For $a({\bm{z_{(k)}}}) \neq 0$ we can divide by $\abs{a({\bm{z_{(k)}}})}$ to get 
\begin{gather*}
				1 \leq \max{ \abs{t}, \abs{\sum_{j\in I_k} \gamma_{d,j} - t}}
\end{gather*}
if we set $t=-b({\bm{z_{(k)}}})/a({\bm{z_{(k)}}})$. 
The last maximum is minimal if both of its entries coincide. 
This is for $t=\frac{1}{2} \sum_{j\in I_k} \gamma_{d,j}$. 
Consequently, we need to ensure that 
\begin{gather*}
        2 \leq \sum_{j\in I_k} \gamma_{d,j}
\end{gather*}
to conclude \link{estimate_norm_h_2} for all admissible $\bm{z}\in Z$. 
But this is true for every $k\in \{1,\ldots,s\}$ by definition of 
the sets $I_k$ in Step 1.
Thus we have shown \link{estimate_norm_h_1} for the special case 
$\bm{\sigma}=\bm{e_k}$ for all $k\in \{1,\ldots,s\}$.
 
The inequality \link{estimate_norm_h_1} also holds true for every 
$\bm{\sigma} \in\{0,1\}^s$ by an easy inductive argument on the cardinality 
of $\abs{\bm{\sigma}}$. 
Indeed, if $\abs{\bm{\sigma}} \geq 2$ then $\bm{\sigma}=\bm{\sigma'}+\bm{e_k}$ 
with $\abs{\bm{\sigma'}} = \abs{\bm{\sigma}}-1$. 
We now need to estimate 
$\norm{D^{\bm{\sigma'}+\bm{e_k}}h \sep \L_\infty(Z)}$. 
Since $(D^{\bm{e_k}} h)(\bm{z}) = a({\bm{z_{(k)}}})$ has the same structure as the 
function~$h$ itself, we see that $\norm{D^{\bm{\sigma'}+\bm{e_k}} h \sep \L_\infty(Z)}$
equals $\norm{D^{\bm{\sigma'}} a({\bm{z_{(k)}}}) \sep \L_\infty(Z)}$ and 
the proof of \link{estimate_norm_h_1} then is completed by the inductive step. 

\textit{Step 4}. 
Collecting the previous equalities and estimates we obtain
\begin{align*}
        \norm{g \sep \P_d^\gamma} 
        &= \norm{g \sep \P_{m_s}^\gamma} 
        = \maxx_{\substack{{\bm{\alpha}}\in\{0,1\}^{m_s}\\ T({\bm{\alpha}}) \in \{0,1\}^s}} \frac{1}{\gamma_{{\bm{\alpha}}}} 
        \norm{D^{\bm{\alpha}} g \sep \L_\infty(X)} 
        = \maxx_{\bm{\sigma} \in \{0,1\}^s} \norm{D^{\bm{\sigma}} h_g \sep \L_\infty(Z)} \\
        &\leq \norm{h_g \sep \L_\infty(Z)} = \norm{g \sep \L_\infty(X)} 
        =  \norm{g \sep \L_\infty([0,1]^d)}
\end{align*}
for every $g \in V$, where $V$ is a linear subspace of $\F = \P_d^\gamma$ with $\dim V = 2^s$. 
Therefore \autoref{needed_lemma} with $a = 1$ yields that 
for $n < \dim V$
the worst case error
\begin{gather*}
				\Delta^\wor(A_{n,d}; \id_d \colon B_r(\P_d^\gamma)\nach \L_\infty([0,1]^d))
\end{gather*}
of any algorithm $A_{n,d}$ 
from the class $\A_d^{n,\rm cont} \cup \A_d^{n,\rm adapt}$ 
is lower bounded by $r$, the radius of the centered ball $B_r(\P_d^\gamma)$.
\end{proof}

\subsection{Complexity results via embeddings}\label{sect:embeddings}
Keeping in mind the assertions shown in the previous sections, we 
are ready to give conditions for tractability of the 
uniform approximation problem 
\begin{gather*}
			\mathrm{App} = (\mathrm{App}_d)_{d\in\N},
			\qquad \mathrm{App}_d \colon \B(\F_d^\gamma) \nach \L_\infty([0,1]^d), \quad \mathrm{App}_d(f) =\id_d(f)= f.
\end{gather*}
We suppose $(\F_d^\gamma)_{d\in\N}$ to be a sequence of
Banach spaces of real-valued functions~$f$ defined on the unit cube $[0,1]^d$.
We further assume that this sequence depends on product weights $\gamma=(\gamma_{\bm{\alpha}})_{\bm{\alpha}\in\N_0^d}$ and 
fulfills one of the following simple assumptions:
\begin{enumerate}[label=(A\thechapter.\arabic*), ref=A\thechapter.\arabic{*}, leftmargin=!, labelwidth=\widthof{(A4.2)}]
        \item \label{Assumption1}
        			$\P_d^{\gamma} \hookrightarrow \F_d^\gamma$ 
        			with norm 
        			\begin{gather*}
        					C_{1,d}\leq c \cdot d^{q_1}
        					\quad \text{for all}\quad  d\in\N 
        			\end{gather*}
        			and some absolute constants $c,q_1 \geq 0$,
				\item \label{Assumption2}
							$\F_d^\gamma \hookrightarrow \Hi_d^\gamma$ 
							with norm
							\begin{gather}\label{eq:bound_c2}
									C_{2,d} \leq a \cdot \exp{b\cdot \sum_{j=1}^d (\gamma_{d,j})^t}
									\quad \text{for all} \quad d\in\N
							\end{gather}
							and some absolute constants $a>0$, $b\geq 0$, as well as a parameter $t\in(0,1]$ 
							independent of $d$ and $\gamma$.
\end{enumerate}
Here the spaces $\P_d^{\gamma}$ and $\Hi_d^\gamma=\Hi(K_d^\gamma)$ are defined as in
\autoref{sect:pol_bound} and \autoref{sect:weightedSobolev}, respectively.

To simplify the notation we use the commonly known definitions of the so-called
\textit{sum exponents}\footnote{Note that some authors use the name \emph{decay} for $1/p(\cdot)$.} 
for the product weight sequence 
$\gamma=(\gamma_{\bm{\alpha}})_{\bm{\alpha}\in\N_0^d}$, $d\in\N$,
induced by uniformly bounded generator weights $0<\gamma_{d,j}\leq C_\gamma$, $j=1,\ldots,d$; see \link{ProdWeights}.
We set
\begin{gather*}
        p(\gamma) 
        = \inf\left\{ \kappa \geq 0 \sep P_\kappa(\gamma) 
        = \limsup_{d\nach \infty} \sum_{j=1}^d \left( \gamma_{d,j} \right)^\kappa < \infty \right\},
\end{gather*}
as well as
\begin{gather*}
        q(\gamma) 
        = \inf\left\{ \kappa \geq 0 \sep Q_\kappa(\gamma) 
        = \limsup_{d\nach \infty} \frac{\sum_{j=1}^d \left( \gamma_{d,j} \right)^\kappa}{\ln(d+1)} < \infty \right\},
\end{gather*}
with the usual convention that $\inf \leer = \infty$.

The following necessary conditions for (strong) polynomial tractability 
slightly generalize Theorem 2 of \cite{W12}.
\begin{prop}[Necessary conditions]\label{Thm_Necessary}
				Assume that \link{Assumption1} holds true with some $q_1\geq 0$. 
				Consider $\L_\infty$-approximation
				over $(\F_d^\gamma)_{d\in\N}$ in the worst case setting with respect to the class of algorithms  
        $\A_d^{n,\rm cont} \cup \A_d^{n,\rm adapt}$ and the absolute error criterion. Then
				\begin{gather}\label{LowerBound_n}
						n^\wor(\epsilon,d; \mathrm{App}_d) 
						> \frac{1}{2} \cdot 2^\wedge\!\!\left( \frac{1}{2+C_\gamma} \sum_{j=1}^d \gamma_{d,j} \right)
				\end{gather}
				for all $d\in \N$ and every $\epsilon \in (0,C_{1,d}^{-1})$.
        Hence,
				\begin{itemize}
								\item if the problem $\mathrm{App}$ is polynomially tractable then $q(\gamma) \leq 1$,
								\item if $q_1=0$ and the problem is strongly polynomially tractable then $p(\gamma) \leq 1$.
				\end{itemize}
\end{prop}

\begin{proof}
Let $d\in\N$. 
Due to \link{Assumption1}, every algorithm 
$A_{n,d}\in \A_d^{n,\rm cont} \cup \A_d^{n,\rm adapt}$ for 
$\L_\infty$-approximation defined on~$\F_d^\gamma$ also applies to 
the embedded space $\P_d^\gamma$.
Furthermore the embedding constant $C_{1,d}$ implies that
the ball $B_r(\P_d^\gamma)$ of radius $r=C_{1,d}^{-1}$ in $\P_d^\gamma$ 
is completely contained in the unit ball 
$\B(\F_d^\gamma)$ of $\F_d^\gamma$. 
Therefore, 
\begin{align*}
        &\Delta^{\wor}(A_{n,d}; \mathrm{App}_d\colon \B(\F_d^\gamma) \nach \L_\infty([0,1]^d)) \\
        &\qquad\qquad\geq \Delta^{\wor}\left(A_{n,d} 
        \big|_{\P_d^\gamma}; \id_d\colon B_r(\P_d^\gamma)\nach \L_\infty([0,1]^d)\right).
\end{align*}
From \autoref{Theorem_Polynom} we have that the latter quantity is lower bounded by $r=C_{1,d}^{-1}$
provided that $n<2^s$, where $s=s(\gamma,d)\in\{0,\ldots,d\}$ satisfies \link{estimate_s}.
Since this lower bound holds for any such $A_{n,d}$ it remains valid for the
$n$th minimal error, \ie
\begin{gather*}
				e^\wor(n,d;\mathrm{App}_d) 
				\geq C_{1,d}^{-1} \quad \text{for all} \quad n < 2^s.
\end{gather*}
Hence we obtain $n^\wor(\epsilon,d; \mathrm{App}_d) \geq 2^s$ for all $d\in\N$ 
and every $\epsilon \in (0,C_{1,d}^{-1})$ which implies \link{LowerBound_n} using \link{estimate_s}.

Now suppose the problem $\mathrm{App} = (\mathrm{App}_d)_{d\in\N}$ to be polynomially tractable.
Then there are constants $C,p>0$ and $q_2\geq0$ such that 
\begin{gather*}
				n^\wor(\epsilon,d; \mathrm{App}_d) 
				\leq C \, \epsilon^{-p} \, d^{q_2} 
				\quad	\text{for all} \quad 
				d\in\N \quad \text{and} \quad \epsilon \in (0,1].
\end{gather*}
For any given $d\in\N$ we can take, say,  
$\epsilon = \epsilon(d) = \frac{1}{2} \cdot \min{1, C_{1,d}^{-1}}$ to conclude
\begin{gather}\label{eq:strong_res}
		2^\wedge\!\!\left( \frac{1}{2+C_\gamma} \sum_{j=1}^d \gamma_{d,j} \right) 
		< C' \, \max{1,C_{1,d}^p}\, d^{q_2}
\end{gather}
for some $C'>0$ independent of $d$.
If we now assume that $C_{1,d}\in\0(d^{q_1})$ then the right-hand side of the last inequality
belongs to $\0(d^{p q_1+q_2})$, as $d\nach \infty$.
Provided that $\max{q_1,q_2}>0$ this is equivalent to the boundedness of 
$\sum_{j=1}^d \gamma_{d,j} / \ln(d+1)$ such that we arrive at $q(\gamma) \leq 1$, as claimed.

Finally, the case of strong polynomial tractability can be treated similarly
by setting $q_1=q_2=0$ in the latter bounds.
Then we obtain that $\sum_{j=1}^d \gamma_{d,j}$ is uniformly bounded in $d$ which implies
$p(\gamma) \leq 1$.
\end{proof}

Of course, the conditions $q(\gamma)\leq 1$ and $p(\gamma)\leq 1$ are also necessary for polynomial and strong polynomial tractability with respect to smaller classes of algorithms such as, e.g., $\A_d^{n,\mathrm{lin}}(\Lambda^{\all})$.

Observe that one of the improvements compared to \cite[Theorem 2]{W12} 
is the possibility to choose the uniform upper bound for the generator weights, $C_\gamma$, different than $1$.
Moreover, now we have weaker conditions on the embedding constant $C_{1,d}$.
For the application we have in mind we will see that there still $C_{1,d}=1$.
But we note in passing that the stated conclusions for (strong) polynomial tractability 
are only special instances of the more
general bound \link{eq:strong_res} obtained in the latter proof 
which we will not investigate further.

We next assume \link{Assumption2} and show that slightly stronger conditions on the product weights $\gamma$ than in
\autoref{Thm_Necessary} are sufficient for polynomial and strong polynomial tractability, respectively.
This is stated in the next assertion which can be found as Theorem~3 in \cite{W12}.

\begin{prop}[Sufficient conditions]\label{Theorem_Sufficient}
				Suppose that \link{Assumption2} holds true with some $t\in(0,1]$.
				Consider $\L_\infty$-approximation over $(\F_d^\gamma)_{d\in\N}$ in the worst case setting with respect 
				to the class of linear algorithms $\A_d^{n,\mathrm{lin}}(\Lambda^{\all})$
				and the absolute error criterion. 
				Then 
				\begin{itemize}
									\item $q(\gamma)<t$ implies polynomial tractability,
									\item $p(\gamma)<t$ implies strong polynomial tractability.
				\end{itemize}
\end{prop}

\begin{proof}
Due to \link{Assumption2}, the restriction of the algorithm 
$A_{n,d}^*$ in \autoref{cor:optAlgoSobolev} from $\Hi_d^\gamma$ to 
$\F_d^\gamma$ is admissible for 
$\L_\infty$-approximation over $\F_d^\gamma$.
Furthermore, due to the linearity of $A_{n,d}^*$, we have
\begin{align*}
				\norm{f-A^*_{n,d}f \sep \L_\infty([0,1]^d) } 
				&\leq \Delta^{\wor}(A_{n,d}^*; \id_d\colon\B(\Hi_d^\gamma)\nach \L_\infty([0,1]^d)) \cdot \norm{f\sep \Hi_d^\gamma} \\
				&\leq \Delta^{\wor}(A_{n,d}^*;\id_d\colon\B(\Hi_d^\gamma)\nach \L_\infty([0,1]^d)) \cdot C_{2,d} \cdot \norm{f \sep \F_d^\gamma}
\end{align*}
for all $f\in \F_d^\gamma$.
Therefore we can estimate the $n$th minimal error by
\begin{align*}
        e^\wor(n,d;\mathrm{App}_d) 
        &\leq \Delta^{\wor}\left(A_{n,d}^* \big|_{\F_d^\gamma}; \mathrm{App}_d\colon\B(\F_d^\gamma)\nach \L_\infty([0,1]^d) \right) \\
        &\leq C_{2,d} \cdot \Delta^{\wor}(A_{n,d}^*; \id_d\colon\B(\Hi_d^\gamma)\nach \L_\infty([0,1]^d) \\
        &\leq 
        a \cdot a_\tau \cdot \exp{b \sum_{j=1}^d  \left( \gamma_{d,j}\right)^t + b_\tau \sum_{j=1}^d  \left( \gamma_{d,j}\right)^\tau} 
        \cdot n^{-(1-\tau)/(2\tau)},
\end{align*}
where $\tau$ is an arbitrary number from $(1/2, 1)$. 
Choosing $n$ such that the right-hand side is not greater than a given $\epsilon\in(0,1]$,
we obtain an estimate for the information complexity with respect to 
the class of linear algorithms,
\begin{gather}\label{UpperBound_tau}
        n^\wor(\epsilon,d;\mathrm{App}_d) 
        \leq c_1 \cdot \epsilon^{-2\tau/(1-\tau)} 
        \cdot \exp{ c_2 \sum_{j=1}^d \left( \gamma_{d,j} \right)^t + c_3 
        \sum_{j=1}^d  \left( \gamma_{d,j} \right)^{\tau}},
\end{gather}
where the non-negative constants $c_1$, $c_2$ and $c_3$ only depend on 
$\tau$, $a$ and $b$.

Suppose that $q(\gamma)<t$. Then $Q_\kappa(\gamma)$ is finite 
for every $\kappa >q(\gamma)$. Taking $\kappa=t$ we obtain
\begin{gather*}
        \frac{\sum_{j=1}^d  \left( \gamma_{d,j} \right)^t}{\ln(d+1)} 
        \cdot \ln(d+1) \leq (Q_t(\gamma)+\delta) \cdot \ln(d+1) = 
        \ln (d+1)^{Q_t(\gamma) + \delta}
\end{gather*}
for every $\delta>0$ whenever $d$ is larger than a certain
$d_\delta\in\N$. This means that the factor 
$\exp{c_2\sum_{j=1}^d(\gamma_{d,j})^t}$ in~\link{UpperBound_tau} 
is polynomially dependent on $d$. 
On the other hand, we can choose $\tau \in (\max{q(\gamma),1/2},1)$ 
such that $Q_\tau(\gamma)$ is finite and thus the factor
$\exp{c_3\sum_{j=1}^d(\gamma_{d,j})^\tau}$ in~\link{UpperBound_tau} 
is also polynomially dependent on $d$. 
So, for this value of $\tau$ we can rewrite~\link{UpperBound_tau} as 
\begin{gather*}
        n^\wor(\epsilon,d;\mathrm{App}_d) \in 
        \0 \left(\epsilon^{-2\tau/(1-\tau)} \cdot (d+1)^{c_4} \right),
\end{gather*}
with $c_4$, as well as the implied factor in the $\0$-notation, 
independent of $d$ and $\epsilon$ which means that
the problem is polynomially tractable, as claimed.

Suppose finally that $p(\gamma)<t$.
Then the sums $\sum_{j=1}^d(\gamma_{d,j})^t$ and 
$\sum_{j=1}^d(\gamma_{d,j})^\tau$ for $\tau\in(\max{p(\gamma),1/2},1)$
are both uniformly bounded in $d$. 
Consequently \link{UpperBound_tau} yields strong polynomial tractability, 
and completes the proof. 
\end{proof}

The conditions in \autoref{Theorem_Sufficient} 
are obviously also sufficient if we consider larger classes of algorithms such as, \eg,
$\A_d^{n,\mathrm{cont}} \cup \A_d^{n,\mathrm{adapt}}$.
Moreover note that the given proof
also provides explicit upper bounds for the exponents of tractability.

Let us briefly discuss the different roles of the assumptions \link{Assumption1} and
\link{Assumption2} in the following remark.
\begin{rem}
Assumption \link{Assumption1} is used to find a lower bound
on the information complexity for the space~$\F_d^\gamma$ as 
long the space $\P_d^\gamma$ is continuously embedded in $\F_d^\gamma$ 
with an embedding constant which grows at most polynomially with the dimension~$d$. 
Such an embedding can be shown for several different 
classes of functions.

On the other hand, assumption \link{Assumption2} is used 
to find an upper bound on the
information complexity for the space~$\F_d^\gamma$ as long as it is
continuously embedded in the unanchored weighted Sobolev space $\Hi_d^\gamma=\Hi(K_d^\gamma)$ with an embedding
constant depending exponentially on the sum of some power of the
generators $\gamma_{d,j}$ of the product weights $\gamma$. 
This considerably restricts the choice of $\F_d^\gamma$.
We need this assumption in order to use the linear algorithm 
$A_{n,d}^*$ defined on the space $\Hi_d^\gamma$  
and the error bound given in \autoref{cor:optAlgoSobolev}.

Obviously, we can replace the space $\Hi_d^\gamma$ in \link{Assumption2} 
by any other space which contains at least $\P_d^\gamma$ and 
for which we know a linear algorithm using
$n$ linear functionals 
whose worst case error is polynomial in $n^{-1}$ with 
an explicit dependence on the product weights $\gamma$.
\hfill$\square$
\end{rem}

We now show that the assumptions \link{Assumption1} and
\link{Assumption2} allow us to characterize weak tractability 
and the curse of dimensionality.

\begin{theorem}[Weak tractability and the curse of dimensionality]\label{Theorem_Equivalence}
        Suppose that for a sequence of Banach spaces $(\F_d^\gamma)_{d\in\N}$ 
				equipped with product weights~$\gamma$ 
        the assumptions \link{Assumption1} and \link{Assumption2} hold true 
        with some parameter $t\in(0,1]$.
        Consider the $\L_\infty$-approximation problem $\mathrm{App}$
        in the worst case setting and with respect to the absolute error criterion.
        Then the following statements are equivalent:
        \begin{enumerate}[label=(\roman*), ref=\roman{*}]
                \item The problem is weakly tractable with respect to the class 
                      $\A_d^{n, \lin}(\Lambda^{\all})$. \label{Equi_1}
                \item The problem is weakly tractable with respect to the class 
                      $\A_d^{n,\mathrm{cont}} \cup \A_d^{n,\mathrm{adapt}}$. \label{Equi_2}
                \item There is no curse of dimensionality for the class 
                      $\A_d^{n, \lin}(\Lambda^{\all})$. \label{Equi_3}
                \item There is no curse of dimensionality for the class 
                      $\A_d^{n,\mathrm{cont}} \cup \A_d^{n,\mathrm{adapt}}$. \label{Equi_4}
                \item For all $\kappa > 0$ we have  
                      $\lim_{d\nach \infty}\limits \frac{1}{d} 
                      \sum_{j=1}^d \left( \gamma_{d,j} \right)^\kappa = 0$. \label{Equi_5}
                \item There exists $\kappa \in (0,t)$ such that
                      $\lim_{d\nach \infty}\limits \frac{1}{d} 
                      \sum_{j=1}^d \left( \gamma_{d,j} \right)^\kappa = 0$. \label{Equi_6}
        \end{enumerate}
\end{theorem}

\begin{proof}
We start by showing that \link{Equi_6} implies \link{Equi_1}, i.e.,  
\begin{gather*}
        \lim_{\epsilon^{-1}+d \nach \infty} 
        \frac{\ln \left(n^\wor(\epsilon,d; \mathrm{App}_d)\right)}{\epsilon^{-1}+d} = 0,
\end{gather*}
where the information complexity is taken with respect to the class 
$\A_d^{n, \lin}(\Lambda^{\all})$ of linear 
algorithms that use continuous linear functionals. 
By the arguments used in the proof 
of \autoref{Theorem_Sufficient} we obtain estimate \link{UpperBound_tau} 
for all $\epsilon$ in $(0,1]$, as well as for every $d\in \N$, and all 
$\tau \in (1/2,1)$, due to assumption \link{Assumption2}. 
Clearly, for 
$\kappa \in (0,t)$ as in the hypothesis and $t\in(0,1]$ as in the 
embedding condition, we find $\tau \in (1/2, 1)$ such that 
$\kappa < \min{t,\tau}$. 
So, since $\gamma_{d,j} \leq C_\gamma$, we can estimate 
\begin{gather*}
		\sum_{j=1}^d \left( \gamma_{d,j} \right)^s 
		= C_\gamma^s \cdot \sum_{j=1}^d \left(\frac{\gamma_{d,j}}{C_\gamma}\right)^s
		\leq C_\gamma^{s-\kappa} \cdot \sum_{j=1}^d \left( \gamma_{d,j} \right)^\kappa
		\leq C \cdot \sum_{j=1}^d \left( \gamma_{d,j} \right)^\kappa,
\end{gather*}
where $s$ either equals $t$ or $\tau$ and $C=\max{1,C_\gamma}$.
Therefore the right-hand side of \link{UpperBound_tau} can be estimated from above
and thus
\begin{gather*}
        \frac{\ln \left( n^\wor(\epsilon,d;\mathrm{App}_d) \right)}{\epsilon^{-1}+d} 
        \leq \frac{\ln(c_1)}{\epsilon^{-1}+d} + \frac{2\tau}{1-\tau} \cdot \frac{\ln \left( \epsilon^{-1} \right)}{\epsilon^{-1}+d} + C \cdot \max{c_2,c_3} \cdot \frac{\sum_{j=1}^d \left( \gamma_{d,j} \right)^\kappa}{\epsilon^{-1}+d}
\end{gather*}
tends to zero when $\epsilon^{-1}+d$ approaches infinity, as claimed.

Clearly, \link{Equi_1} $\Rightarrow$ \link{Equi_2} $\Rightarrow$ \link{Equi_4} 
and \link{Equi_1} $\Rightarrow$ \link{Equi_3} $\Rightarrow$ \link{Equi_4}.
Moreover the implication from~\link{Equi_5} to~\link{Equi_6} is obvious.
Hence, it only remains to show that \link{Equi_4}~$\Rightarrow$~\link{Equi_5}.

From \link{Assumption1} we have estimate \link{LowerBound_n}. 
Then the absence of the curse of dimensionality implies
\begin{gather*}
				\lim_{d\to\infty}\frac1d\,\sum_{j=1}^d\gamma_{d,j}=0.
\end{gather*}
Now Jensen's inequality yields that
\begin{gather*}
       \frac{1}{d} \,\sum_{j=1}^d \gamma_{d,j} \geq 
       \left( \frac{1}{d} \,\sum_{j=1}^d 
       \left( \gamma_{d,j} \right)^\kappa \right)^{1/\kappa} 
       \quad \text{for} \quad 0 < \kappa \leq 1,
\end{gather*}
because $f(y) = y^{\kappa}$ is a concave function for $y > 0$.
This shows
\begin{gather*}
				\lim_{d\to\infty}\frac{1}{d}\,\sum_{j=1}^d 
				\left(\gamma_{d,j} \right)^\kappa=0
				\quad \text{for all}\quad 0 < \kappa \le 1.
\end{gather*}
Finally, for every $\kappa \geq 1$ we can estimate 
$\gamma_{d,j} \geq C_\gamma^{1-\kappa} \left( \gamma_{d,j} \right)^\kappa$ 
since $\gamma_{d,j}\leq C_\gamma$ for $j=1,\ldots,d$. 
Therefore 
$\lim_{d\nach \infty} d^{-1} \sum_{j=1}^d (\gamma_{d,j})^\kappa = 0$
also holds true for $\kappa>1$, and the proof is complete.
\end{proof}

\subsection{Conclusions and applications}\label{sect:conclusions}
In this last part of the current section we give some examples to illustrate 
the obtained complexity results. 
To this end, we only have to prove the corresponding embeddings, 
\ie we need to verify assumption \link{Assumption1} and/or \link{Assumption2} 
from the beginning of \autoref{sect:embeddings}.

\begin{example}[Limiting cases $\P_d^\gamma$ and $\Hi_d^\gamma$]
To begin with, we check the case where $\F_d^\gamma = \P_d^\gamma$ for every $d\in\N$. 
Then~\link{Assumption1} obviously holds with $C_{1,d}=1$, \ie $c=1$ and $q_1=0$. 
To prove \link{Assumption2}, note that the algebraical inclusion 
$\F_d^\gamma \subset \Hi_d^\gamma$ is trivial by the definition of $\Hi_d^\gamma = \Hi(K_d^\gamma)$
given in \autoref{sect:weightedSobolev}.
For $f \in \F_d^\gamma = \P_d^\gamma$ we calculate
\begin{gather*}
        \norm{f \sep \Hi_d^\gamma}^2 
        \leq \sum_{\bm{\alpha} \in \{0,1\}^d} \frac{1}{\gamma_{\bm{\alpha}}} 
        			\int_{[0,1]^d} \norm{D^{\bm{\alpha}} f \sep \L_\infty([0,1]^d)}^2 \dlambda^d(\bm{x})
        \leq \norm{f \sep \F_d^\gamma}^2 \cdot \sum_{\bm{\alpha} \in \{0,1\}^d} \gamma_{\bm{\alpha}}
\end{gather*}
using \link{eq:norm_Sobol}, as well as \link{eq:norm_P}.
Hence the norm of the embedding $\F_d^\gamma \hookrightarrow \Hi_d^\gamma$ 
is bounded by 
\begin{gather*}
        \left( \sum_{\bm{\alpha} \in \{0,1\}^d} \gamma_{\bm{\alpha}} \right)^{1/2} 
        = \left( \prod_{j=1}^d (1+\gamma_{d,j}) \right)^{1/2} 
				\leq \exp{ \frac{1}{2} \sum_{j=1}^d \gamma_{d,j} }.
\end{gather*}
So, with $a=1$, $b=1/2$, and $t=1$ the assumption \link{Assumption2} 
is also fulfilled and we can apply the stated assertions from \autoref{sect:embeddings} for the spaces 
$\F_d^\gamma = \P_d^\gamma$, $d\in\N$.

We now turn to the case $\F_d^\gamma = \Hi_d^\gamma$.
Unfortunately, the estimate above indicates that~\link{Assumption1} 
may not hold for 
$\F_d^\gamma=\Hi_d^\gamma$ with $C_{1,d} \in \0(d^{q_1})$ 
without imposing additional conditions on the product weights~$\gamma$. 
Nevertheless, in this case assumption~\link{Assumption2} 
is trivially true with $C_{2,d}=1$, i.e., $a=1$, $b=0$, and $t=1$. 
Therefore we can apply \autoref{Theorem_Sufficient} for this space. 
Thus the problem is polynomially tractable if $q(\gamma)<1$ 
and we have strong polynomial tractability if $p(\gamma)<1$.
It can be shown that these conditions are also necessary; 
see \autoref{sect:final_remarks}.
\hfill$\square$
\end{example}

Next we discuss a more advanced sequence of Banach function spaces.
\begin{example}[$C^{(1,\ldots,1)}$]
For every $d\in\N$ consider the space 
\begin{gather*}
        \F_d^\gamma 
        = \left\{ f \colon [0,1]^d \nach \R \sep f \in C^{(1,\ldots,1)}([0,1]^d), \, \text{ where } \, 
        			\norm{f \sep \F_d^\gamma} < \infty \right\}
\end{gather*}
of functions which are once continuously differentiable in every coordinate direction,
where
\begin{gather*}        			
       \norm{f \sep \F_d^\gamma} 
       = \maxx_{\bm{\alpha} \in \{0,1\}^d} 
        			\frac{1}{\gamma_{\bm{\alpha}}} \norm{D^{\bm{\alpha}} f \sep \L_\infty([0,1]^d)}.
\end{gather*}
Since $\P_d^\gamma$ is a linear subset of 
$\F_d^\gamma$ and, due to \link{eq:norm_P}, the norm $\norm{\cdot \sep \P_d^\gamma}$ is simply the restriction 
of $\norm{\cdot \sep \F_d^\gamma}$ we have
$\P_d^\gamma \hookrightarrow \F_d^\gamma$ with an embedding factor 
$C_{1,d}=1$ 
and hence \link{Assumption1} holds true. 
For the norm $C_{2,d}$ of the embedding $\F_d^\gamma \hookrightarrow \Hi_d^\gamma$, 
the same estimates hold exactly as in the previous example and, 
moreover, the set inclusion is obvious. 
Therefore also assumption \link{Assumption2} is fulfilled and we can 
apply the propositions and theorems of \autoref{sect:embeddings}
to the sequence $(\F_d^\gamma)_{d\in\N}$.
\hfill$\square$
\end{example}

Our last example $\F_d^\gamma = F_d^\gamma$, for all $d\in\N$, finally
shows that even very high smoothness 
does not improve the conditions for tractability.

\begin{example}[$C^\infty$]
For $d\in\N$ and product weights $\gamma$ let 
\begin{gather*}
				\F_d^\gamma 
				= F_d^\gamma 
				= \left\{ f \colon [0,1]^d \nach \R \sep f \in C^{\infty}([0,1]^d) \, \text{ with } \, \norm{f \sep F_d^\gamma} < \infty\right\},
\end{gather*}
where the norm is given by \link{eq:weightednorm}.
Obviously, $\P_d^\gamma \subset C^\infty$, because functions 
from~$\P_d^\gamma$ are at most linear in each coordinate. 
This moreover implies that $D^{\bm{\alpha}} f \equiv 0$ for all 
$\bm{\alpha} \in \N_0^d \setminus \{0,1\}^d$. 
Therefore, once again we have
\begin{gather*}
        \norm{f \sep \P_d^\gamma} 
        = \maxx_{\bm{\alpha} \in \{0,1\}^d} \frac{1}{\gamma_{\bm{\alpha}}} \norm{D^{\bm{\alpha}} f \sep \L_\infty([0,1]^d) }
        = \norm{f \sep \F_d^\gamma} \quad \text{for all} \quad f \in \P_d^\gamma.
\end{gather*}
Together this yields $\P_d^\gamma \hookrightarrow \F_d^\gamma$ with 
an embedding constant $C_{1,d}=1$ for all $d\in\N$.
In addition, also \link{Assumption2} can be concluded as in the examples above. 
So, even infinite smoothness leads to the the same conditions for 
tractability and the curse of dimensionality as before.
\hfill$\square$
\end{example}

Note that in the latter example we do not need to claim 
a product structure for the weights according to multi-indices 
$\bm{\alpha} \in \N_0^d \setminus \{0,1\}^d$. 
Furthermore, this example 
is a generalization of the space $F_d$ studied in~\cite{NW09}. 
For $\gamma_{\bm{\alpha}} \equiv 1$ we reproduce the intractability result 
stated there because then $F_d^\gamma$ equals $F_d$ for each $d\in\N$.\\

In conclusion we discuss the tractability behavior of 
uniform approximation defined on one of the spaces $\F_d^\gamma$ 
above using a special class of product weights $\gamma$ 
which are \emph{independent of the dimension}~$d$.
That is, for the generator weights we claim that
\begin{gather}\label{eq:pol_generator}
				\gamma_{d,j} 
				\equiv \gamma^{(j)} 
				\in \Theta(j^{-\beta})
				\quad \text{for some} \quad
				\beta \geq 0,
\end{gather}
and all $j$ and $d\in\N$.
The imposed polynomial behavior of $\gamma^{(j)}$ 
is a typical example in the theory of product weights.
Clearly, $p(\gamma)$ is finite if and only if $\beta>0$,
and if so then $p(\gamma)=1/\beta$.
For details see \cite[Section~5.3.4]{NW08}.

If $\beta = 0$ then the $\L_\infty$-approximation problem
$\mathrm{App}=(\mathrm{App}_d)_{d\in\N}$ is intractable 
(more precisely it suffers from the curse of dimensionality) due to
\autoref{Theorem_Equivalence}, assertion \link{Equi_5}, 
since then $d^{-1} \sum_{j=1}^d \gamma_{d,j}$ does not tend to zero.
For $\beta \in (0,1)$, easy calculus yields $q(\gamma)>1$. 
So, using \autoref{Thm_Necessary} we conclude 
polynomial intractability in this case. 
On the other hand, 
for all $\delta$ and $\kappa$ with $0< \delta < \kappa \leq 1$, we have
\begin{gather*}
        \frac{\sum_{j=1}^d j^{-\kappa}}{d} = \frac{\sum_{j=1}^d j^{-\kappa} 
        d^{\kappa -(1+\delta)}}{ d^{\kappa - \delta}} \leq 
        \frac{\sum_{j=1}^d j^{-(1+\delta)}}{d^{\kappa - \delta}} 
        \nach 0 \quad \text{for} \quad d\nach \infty
\end{gather*}
and if $\kappa > 1$ then the most left fraction obviously tends to zero, too. 
Hence condition~\link{Equi_6} of \autoref{Theorem_Equivalence} holds and
the problem is weakly tractable for all $\beta>0$. 

For $\beta=1$ we use inequality \link{LowerBound_n} from 
\autoref{Thm_Necessary} and estimate
\begin{gather*}
				\sum_{j=1}^d \gamma_{d,j} \geq c \cdot \ln (d+1) 
\end{gather*}
for some positive $c$. 
Therefore, for every $\eps\in(0,1)$ the information complexity~$n(\epsilon, d; \mathrm{App}_d)$ 
is lower bounded polynomially in $d\in\N$.
This proves that strong polynomial tractability does not hold for $\beta=1$. 
Moreover, it is easy to show that in this case the sufficient condition 
$q(\gamma)<1$ for polynomial tractability is not fulfilled. 
So, we do not know whether polynomial tractability holds or not.

Finally, consider $\beta>1$ in \link{eq:pol_generator}. 
Then we easily see that $p(\gamma) = \frac{1}{\beta}<1=t$. 
Thus \autoref{Theorem_Sufficient} provides strong polynomial tractability 
in this situation.

In summary, we proved all the assertions we claimed at the end of \autoref{sect:weightedSmooth}.

\section{Possible extensions and further results}\label{sect:final_remarks}
Note that the main result of this chapter, the lower bound given in 
\autoref{Theorem_Polynom}, can be easily transferred from $[0,1]^d$ 
to more general domains $\Omega_d\subset\R^d$. 
Indeed, the case $\Omega_d = [c_1,c_2]^d$, where $c_1<c_2$, 
can be immediately obtained using the presented techniques.
It turns out that in this case we have to modify estimate \link{estimate_s} 
by a constant which depends only on the length of the interval 
$[c_1,c_2]$. 
Consequently, the general tractability behavior does not change.

Another extension of the obtained results is possible if we consider 
$\L_p$-norms ($1\leq p <\infty$) instead of the $\L_\infty$-norm. 
In \autoref{sect:Lp_approx} we briefly discuss these norms for the unweighted case. 
Then the modifications for the weighted case are obvious and
thus we leave it for the interested reader.
In passing we correct a small mistake stated in \cite{NW09}.

Finally, in \autoref{sect:opt_uniform_algo}, we show that
the algorithm studied in \autoref{cor:optAlgoSobolev}
is essentially optimal for the uniform approximation problem on
the unanchored weighted Sobolev space $\Hi(K_d^\gamma)$
defined in \autoref{ex:Unanchored_Sobolev}.

\subsection{\texorpdfstring{$\L_p$-approximation}{Lp-approximation}}\label{sect:Lp_approx}
As in \cite[Section~7]{W12} we follow Novak and Wo\'zniakowski~\cite{NW09} and define the spaces
\begin{gather*}
        F_{d,p} = \left\{ f \in C^{\infty}([c_1,c_2]^d) \sep  \norm{f \sep F_{d,p}} = 
        		\sup_{\bm{\alpha} \in \N_0^d} \norm{D^{\bm{\alpha}} f \sep \L_p([c_1,c_2]^d)} < \infty \right\}
\end{gather*}
for $1\leq p < \infty$ and $d\in\N$, where we assume that $l=c_2-c_1>0$. 
In what follows we want to approximate $f\in F_{d,p}$ in the norm of $\L_p$. 
That is, we modify \link{eq:uniform_app} and consider the problem $S^p=(S^p_d)_{d\in\N}$ given by
\begin{gather*}
		S_d^p = \id_{d}^p \colon \B(F_{d,p}) \nach \L_p([c_1,c_2]^d), \quad f\mapsto \id^p_{d}(f) = f.
\end{gather*}
Hence we try to minimize the $n$th minimal worst case error
\begin{gather*}
        e_p^\wor(n,d; \id_d^p)
        = \inf_{A_{n,d}} \sup_{f \in \B(F_{d,p})} \norm{f- A_{n,d}(f) \sep \L_p([c_1,c_2]^d)} 
\end{gather*}
which now depends on the additional integrability parameter~$p$.
Observe that, without loss of generality, we can restrict ourselves to the case $[c_1,c_2]=[0,l]$. 

In order to conclude a lower bound analogue to \link{eq:NW09_bound} and \autoref{Theorem_Polynom}, 
i.e., $e_p^\wor(n,d; \id_{d}^p) \geq 1$ for $n<2^{s}$, we once again use 
\autoref{needed_lemma} with $\F = F_{d,p}$ and $\G = \L_p([0,l]^d)$.\footnote{Note that
it is sufficient to restrict ourselves to the case $r=1$ since now we do not need to take care of embedding constants as in the proof of \autoref{Thm_Necessary}.} 
The authors of \cite{NW09} suggest to use 
the subspace $V_d^{(k)} \subset F_{d,p}$ defined as
\begin{gather*}
        V_d^{(k)}=\spann{ g_i \colon [0,l]^d \nach \R, \, \bm{x} 
        \mapsto g_{\bm{i}}(\bm{x})=\prod_{j=1}^s 
        \left( \sum_{m=(j-1)k+1}^{jk} x_m \right)^{i_j} \sep \bm{i}\in\{0,1\}^s },
\end{gather*}
where $s=\floor{d/k}$ and $k\in\N$ such that $kl \geq 2(p+1)^{1/p}$. 
Hence, if $l<2(p+1)^{1/p}$ then we have to use blocks of variables with 
size $k>1$, in order to guarantee \link{eq:NormCondition}. 
That is, to fulfill the condition
\begin{gather}\label{NormCondition_Lp}
        \norm{g \sep F_{d,p}} 
        \leq \norm{g \sep \L_p([0,l]^d)} 
        \quad \text{for all} \quad g\in V_d^{(k)}.
\end{gather}
Therefore Novak and Wo\'zniakowski defined $k=\ceil{2(p+1)^{1/p}/l}$, 
but this is too small as the following example shows. 

\begin{example}
For $d\geq 4$ take $l=1$, i.e. $[c_1,c_2]^d=[0,1]^d$, and $p=1$. 
Then $k=4$ should be a proper choice, but for $g^*(\bm{x})=(x_1+x_2+x_3+x_4)-2$ 
it can be checked (using a computer algebra system) that
\begin{gather*}
		\norm{g^* \sep \L_1([0,1]^d)} 
		= \frac{7}{15}
		< 1 
		= \norm{\frac{\partial g^*}{\partial x_1} \sep \L_1([0,1]^d)}.
\end{gather*}
This obviously contradicts \link{NormCondition_Lp}. 
\hfill$\square$
\end{example}

For an exhaustive proof of an assertion which states that a slightly larger choice of $k\in\N$ suffices to conclude
the desired intractability result we need to show the following technical lemma first.
Its proof is based on some well-known arguments from Banach space geometry.

\begin{lemma}\label{lemma:int_est}
		Let $p\in[1,\infty)$ and $k\in\N$. Then
		\begin{gather}\label{Estimate_Integrals}
        \mathfrak{I}_{k,p}
        = \int_{[-1/2,1/2]^k} \abs{\sum_{m=1}^k z_m}^p \dlambda^k(\bm{z})
        \geq C_p \cdot k^{p/2}
\end{gather}
		with some $C_p \geq 1/[(2\sqrt{2})^p(1+p)]$ independent of $k$.
\end{lemma}
\begin{proof}
For $k=1$ we easily calculate $\mathfrak{I}_{1,p} = 1 / [2^p(1+p)]$.
Hence, without loss of generality we can assume $k\geq 2$ in what follows.

To abbreviate the notation, let us define
\begin{gather}\label{eq:def_sum}
        f=f_k \colon \R^k \nach \R, 
        \qquad \bm{z}=(z_1,\ldots,z_k) \mapsto f(\bm{z})=\sum_{m=1}^k z_m
\end{gather}
for any fixed $k\geq 2$. 
Moreover, for given vectors $\bm{z},\bm{y} \in \R^k$, let $\distr{\bm{z}}{\bm{y}}$ 
denote the inner product $\sum_{m=1}^k z_m y_m$ in $\R^k$.
In the special case $\bm{y}=\bm{\xi} = 1/\sqrt{k} \cdot \bm{(1,\ldots,1)}\in \mathbb{S}^{k-1}$ 
it is $\distr{\bm{z}}{\bm{\xi}} = t$ for a given $t\in \R$ if and only if 
$f(\bm{z}) = t\sqrt{k}$. 
Furthermore note that every $\bm{y}$ in the 
$k$-dimensional unit sphere $\mathbb{S}^{k-1}\subset \R^k$ 
uniquely defines a hyperplane 
$\bm{y}^{\bot} = \left\{\bm{z}\in\R^k \sep \distr{\bm{z}}{\bm{y}}=0 \right\}$ 
perpendicular to $\bm{y}$ which contains zero. 
Therefore, for $\bm{y}=\bm{\xi}$ and every $t\in[0,\infty)$,
the set
\begin{gather*}
		\mathfrak{H}_t=\bm{\xi}^\bot + t \cdot \bm{\xi} 
		= \left\{\bm{z}\in\R^k \sep \distr{\bm{z}}{\bm{\xi}}=t \right\}
\end{gather*}
describes a parallel shifted hyperplane in $\R^k$ with distance $t$ to the origin. 
Using Fubini's theorem, this leads to the following representation:
\begin{align*}
        \mathfrak{I}_{k,p}
        &= \int_{[-1/2,1/2]^k} \abs{f(\bm{z})}^p \dlambda^k(\bm{z}) 
        = 2 \cdot \int_{\substack{[-1/2, 1/2]^k\\\distr{\bm{z}}{\bm{\xi}} \geq 0}} f(\bm{z})^p \dlambda^k(\bm{z}) \\
        &= 2 \cdot k^{p/2} \cdot \int_0^\infty t^p 
        \left( \int_{[-1/2, 1/2]^k\cap \mathfrak{H}_t} 1 \, \dlambda^k(\bm{z})  \right) \dlambda^1(t).
\end{align*}
Now we see that the inner integral describes the $(k-1)$-dimensional volume 
\begin{gather*}
        v(t) = \uplambda^{k-1}\!\left( [-1/2,1/2]^k \cap \mathfrak{H}_t \right)
\end{gather*}
of the parallel section of the unit cube with the hyperplanes defined above. 
Because of Ball's famous theorem we know that $v(0)\leq \sqrt{2}$ holds
independently of $k$; see, e.g., Chapter 7 in the monograph of Koldobsky~\cite{K05}. 
Moreover taking $\mathfrak{H}_0 = \bm{\xi}^\bot$ provides a central hyperplane section of the 
unit cube. 
From this observation we conclude that
\begin{gather*}
        \int_0^\infty v(t) \, \dlambda^1(t) 
        = \frac{1}{2} \cdot \uplambda^{k}([-1/2,1/2]^k) 
        = \frac{1}{2}
\end{gather*}
because of the symmetry of $[-1/2,1/2]^k$ \wrt $\mathfrak{H}_0$.
In addition, by Brunn's theorem (cf. \cite[Theorem 2.3]{K05}), 
the function $v$ is non-negative and
non-increasing on the interval $[0,\infty)$. 
Thus $v$ is related to the 
distribution function of a certain non-negative real-valued 
random variable~$X$, up to some normalizing factor, i.e. 
$v(t) = v(0) \cdot \Pr(\{X\geq t\})$. 
Using H\"{o}lder's inequality\footnote{See also \cite[Lemma 7.5]{K05}.} 
we obtain $\E (X^{1+p}) \geq ( \E X )^{1+p}$ and, respectively,
\begin{gather*}
        \mathfrak{I}_{k,p} 
				= k^{p/2}\cdot 2\int_0^{\infty} t^p \, v(t) \, \dlambda^1(t) 
        \geq k^{p/2}\cdot \frac{2}{v(0)^p \, (1+p)} \left( \int_0^{\infty} v(t) \, \dlambda^1(t) \right)^{1+p}
\end{gather*}
by integration by parts. 

In summary we have shown \link{Estimate_Integrals} 
and hence the proof is complete.
\end{proof}

Now the mentioned intractability result reads as follows:
\begin{prop}\label{prop:Lp_intract}
				Let $1\leq p < \infty$ and $l>0$. 
				Moreover, choose $k\in\N$ such that
				\begin{gather}\label{LowerBound_k}
        			k \geq \kappa_{p,l}=\ceil{8(p+1)^{2/p}/l^2}.
				\end{gather}
				Then condition \link{NormCondition_Lp} holds for $V_d^{(k)}\subset F_{d,p}$. 
				Hence the $\L_p$-approximation problem 
				$S^p=(\id_d^p\colon F_{d,p}\nach \L_p([0,l]^d))_{d\in\N}$ suffers from the curse of dimensionality since 
				\begin{gather*}
							e_p^\wor(n,d; \id_d^p) \geq 1 \quad \text{for all} \quad n<2^{\floor{d/k}}
				\end{gather*} 
				and every $d\in\N$.
\end{prop}
\begin{proof}
Due to the structure of the functions $g$ from $V_d^{(k)}$, it suffices to show that
\begin{gather*}
        \norm{D^{\bm{\alpha}} g \sep \L_p([0,l]^{ks})} \leq 
        \norm{g \sep \L_p([0,l]^{ks})} \quad \text{for all} 
        \quad g \in V_d^{(k)} \quad \text{and every} \quad \bm{\alpha} \in \mathbb{M}_d^{(k)},
\end{gather*}
where the set of multi-indices $\mathbb{M}_d^{(k)}$ is defined by
\begin{gather*}
        \mathbb{M}_d^{(k)} 
        = \left\{ \bm{\alpha}=(\alpha_1,\ldots,\alpha_{ks}) \in \{0,1\}^{ks} 
        \sep \sum_{m\in I_j} \alpha_m \leq 1 \, 
        \text{ for all } \, j=1,\ldots,s \right\}
\end{gather*}
and $I_j=\{(j-1)k+1, \ldots, jk\}$.
Observe that $\mathbb{M}_d^{(k)}$ depends on $d$ via $s=\floor{d/k}$.
Similar to the proof of \autoref{Theorem_Polynom}, we only need to consider the case 
$\bm{\alpha} = \bm{e_t} \in \{0,1\}^{ks}$ with $t \in I_j$. 
The rest then follows by induction.

Given $t\in I_j$ for some $j\in\{1,\ldots,s\}$ we can represent every fixed $g\in V_d^{(k)}$, 
as well as its partial derivative $D^{\bm{e_t}}g$, 
by some functions $a, b \colon [0,l]^{k(s-1)} \nach \R$ (depending on $g$ and $j$) such that
\begin{gather*}
        g(\bm{x}) = a(\bm{\widetilde{x}}) \sum_{m=1}^k y_m + b(\bm{\widetilde{x}}) 
        \quad \text{and} \quad 
				(D^{\bm{e_t}} g)(\bm{x}) = a(\bm{\widetilde{x}}), 
        \quad \bm{x}\in[0,l]^{ks}.
\end{gather*}
Here we split the $ks$-dimensional vector
$\bm{x} = (\bm{x_{I_1}}, \ldots, \bm{x_{I_{j-1}}}, \bm{y}, \bm{x_{I_{j+1}}}, \ldots, \bm{x_{I_s}})$  into
$\bm{\widetilde{x}} = (\bm{x_{I_1}}, \ldots, \bm{x_{I_{j-1}}}, \bm{x_{I_{j+1}}}, \ldots, \bm{x_{I_s}}) \in [0,l]^{k(s-1)}$ and $\bm{y}=(y_1,\ldots,y_k)\in [0,l]^k$, 
where $\bm{x_{I_j}}$ denotes the $k$-dimensional block of components 
$x_m$ in $\bm{x}$ with coordinates $m \in I_j$. 
Using this representation we can rewrite the inequality 
$\norm{D^{\bm{e_t}} g \sep \L_p([0,l]^{ks})} \leq \norm{g \sep \L_p([0,l]^{ks})}$ 
as
\begin{align*}
        &\int_{[0,l]^{k(s-1)}} \int_{[0,l]^k} \abs{a(\bm{\widetilde{x}})}^p \, \dlambda^k(\bm{y}) \, \dlambda^{k(s-1)}(\bm{\widetilde{x}}) \\
        &\qquad \leq \int_{[0,l]^{k(s-1)}} \int_{[0,l]^k} \abs{a(\bm{\widetilde{x}}) 
        \sum_{m=1}^k y_m + b(\bm{\widetilde{x}})}^p \, \dlambda^k(\bm{y}) \, \dlambda^{k(s-1)}(\bm{\widetilde{x}})
\end{align*}
such that it is enough to prove a pointwise estimate of the inner integrals 
for ($\uplambda^{k(s-1)}$-almost every) fixed $\bm{\widetilde{x}}\in [0,l]^{k(s-1)}$ with $a=a(\bm{\widetilde{x}}) \neq 0$. 
Easy calculus yields
\begin{gather*}
        \int_{[0,l]^k} \abs{a \sum_{m=1}^k y_m + b}^p \dlambda^k(\bm{y}) 
        = l^{p+k} \cdot \int_{[-1/2,1/2]^k} \abs{a \sum_{m=1}^k z_m + b'}^p \dlambda^k(\bm{z})
\end{gather*}
for some constant $b'\in\R$ that depends on $b=b(\bm{\widetilde{x}})$. 
Note that the right-hand side of the latter equality is minimized for $b'=0$. 
Therefore we can estimate the left-hand side from below by
\begin{align*}
        \int_{[0,l]^k} \abs{a \sum_{m=1}^k y_m + b}^p \dlambda^k(\bm{y}) 
        &\geq \abs{a}^p \cdot l^{p+k} \cdot \int_{[-1/2,1/2]^k} \abs{\sum_{m=1}^k z_m}^p \dlambda^k(\bm{z}) \\
        &= \int_{[0,l]^k} \abs{a}^p \dlambda^k(\bm{y}) \cdot l^p \cdot 
        \int_{[-1/2,1/2]^k} \abs{\sum_{m=1}^k z_m}^p \dlambda^k(\bm{z}).
\end{align*}
To complete the proof it remains to show that our choice of 
$k\geq \kappa_{p,l}$, with $\kappa_{p,l}$ given in \link{LowerBound_k}, implies that 
\begin{gather}\label{eq:int_bound}
				\int_{[-1/2,1/2]^k} \abs{\sum_{m=1}^k z_m}^p \dlambda^k(\bm{z}) \geq l^{-p}
\end{gather}
but this easily follows from \autoref{lemma:int_est} above.
\end{proof}

Actually, using other proof methods 
we can slightly improve the lower bound for 
$C_p$ in \autoref{lemma:int_est} and thus
also $\kappa_{p,l}$ in formula \link{LowerBound_k} of \autoref{prop:Lp_intract}.
This is the subject of our final remark within this subsection:

\begin{rem}
Let $\bm{Y}=(Y_1,\ldots,Y_k)$ denote a random vector of $k\in\N$ 
independent copies of some uniformly $[-1/2,1/2]$-distributed random variable $Y_0$.
Then $\mathfrak{I}_{k,p}$ can be interpreted as the $p$th absolute moment $\E(\abs{f_k(\bm{Y})}^p)$ of $f_k(\bm{Y})$, where
$f_k$ again is given by \link{eq:def_sum}.
In the case of even $p=2N$, $N\in\N$, this can be calculated exactly using the multinomial theorem.
For $k,N\in\N$ we obtain
\begin{gather*}
		\mathfrak{I}_{k,2N} 
		= \E\left( \abs{f_k(\bm{Y})}^{2N} \right)
		= 2^{-2N} \sum_{\substack{\bm{j}=(j_1,\ldots,j_k)\in\N_0^k\\j_1+\ldots+j_k=N}} \binom{2N}{2j_1,\ldots,2j_k} \prod_{m=1}^k \frac{1}{2j_m+1},
\end{gather*}
where we used the independence of the $Y_m$'s and fact that
\begin{gather*}
		\E(Y_0^n)
		= \int_{-1/2}^{1/2} y^n\, \dlambda^1(y)
		= \begin{cases}
				0	, &\text{if } n = 2j+1,\\
				(2j+1)^{-1} \cdot 2^{-2j}	, &\text{if } n = 2j,
		\end{cases}
\end{gather*}
and $j\in\N_0$.
In particular, we conclude
\begin{gather*}
		\mathfrak{I}_{k,2} = \frac{1}{2^2 \cdot 3} \cdot k
		\qquad \text{and} \qquad
		\mathfrak{I}_{k,4} = \frac{1}{48} \cdot k \left(k-\frac{2}{5} \right) \geq \frac{1}{2^4 \cdot 5} \cdot k^2.
\end{gather*}

Since $\mathfrak{I}_{k,p}=\norm{f_k\sep \L_p([-1/2,1/2]^k)}^p$ we can use the monotonicity of the Lebesgue spaces in
order to estimate $C_p$ for the remaining powers $p$.
For $k\in\N$ and $1 \leq q \leq p < \infty$ we obtain $\mathfrak{I}_{k,p} \geq (\mathfrak{I}_{k,q})^{p/q} \geq (C_q)^{p/q} \, k^{p/2}$, \ie $C_p \geq (C_q)^{p/q}$, provided that $\mathfrak{I}_{k,q} \geq C_q\, k^{q/2}$.
Consequently, we can take
\begin{gather*}
		k\geq \begin{cases}
				\ceil{12/l^2}, &\text{if } 2\leq p < 4,\\
				\ceil{4\sqrt{5}/l^2}, &\text{if } 4\leq p
		\end{cases}
\end{gather*}
to fulfill \link{eq:int_bound} in the proof of \autoref{prop:Lp_intract}.
This clearly improves the bound $k\geq \kappa_{p,l}$ in \link{LowerBound_k}.
\hfill$\square$
\end{rem}

Nevertheless, we want to stress the point that also with these improvements the lower bounds on $k$ 
are not sharp since we know from \cite{NW09} that in the limit case $p = \infty$ we can take $k = \ceil{2/l}$.
On the other hand, we note that Hoeffding's inequality implies the existence of some universal constants $C_p'$
such that $\mathfrak{I}_{k,p} \leq C_p'\, k^{p/2}$ for all $p\in[1,\infty)$ and every $k\in\N$. 
Thus the estimates on the integrals $\mathfrak{I}_{k,p}$ are of the right order in $k$
such that we need other proof techniques to obtain a better dependence of $k$ on $l=c_2-c_1$.

\subsection{Uniform approximation in the weighted Sobolev space}\label{sect:opt_uniform_algo}
To show that the linear algorithm $A_{n,d}^*$ studied in \autoref{cor:optAlgoSobolev} is
essentially optimal for $\L_\infty$-approximation on the unanchored Sobolev space $\Hi_d^\gamma=\Hi(K_d^\gamma)$ 
in the worst case setting we study (weighted) $\L_2$-approximation on a related Banach space $\F_d$ in the average case setting; 
see \autoref{ex:average_case_approx} for details.
The relation of these two problems is given by the assertion below
which follows from~\cite[Theorem 1]{KWW08}.

\begin{prop}
For $d\in\N$ let $\Hi(K_d)$ denote a RKHS induced by a kernel $K_d\colon[0,1]^d\times[0,1]^d\nach\R$ 
that satisfies \link{eq:BoundedKernel}.\footnote{Note that \link{eq:BoundedKernel} clearly implies that $\int_{[0,1]^d} K_d(\bm{x},\bm{x}) \, \rho(\bm{x})\, \dlambda^d(\bm{x})$ is finite for every probability density function $\rho$ on $[0,1]^d$.}
Moreover, define the set of non-vanishing probability density functions $\rho$ on the unit cube by 
\begin{gather*}
			\mathcal{D}_d 
			= \left\{ \rho \colon [0,1]^d \nach [0,\infty) \sep \int_{[0,1]^d} \rho(\bm{x}) \dlambda^d(\bm{x})=1
								\quad \text{and} \quad \rho>0 \text{ ($\uplambda^d$-a.e.)} \right\}.
\end{gather*}
Then, for every $n\in\N_0$ and all $d\in\N$,
\begin{align*}
			&e^\wor\! \left(n,d; \id_d \colon \B(\Hi(K_d))\nach \L_\infty([0,1]^d) \right) \\
			&\qquad\,\qquad \geq \sup_{\rho \in \mathcal{D}_d} \, e^{\mathrm{avg}} \!\left(n,d; \id_d^{\rho} \colon \F_d \nach \L_2^\rho([0,1]^d) \right). 
\end{align*}
Here the $n$th minimal errors are taken with respect to all algorithms from the class $\A_d^{n,\lin}(\Lambda^\all)$.
\end{prop}
In particular, it follows that the ($n$th minimal) worst case error for $\L_\infty$-approximation
on the unit ball of the \emph{Sobolev space} $\Hi_d^\gamma$ is
lower bounded by the average case error of \emph{unweighted} $\L_2$-approximation on the corresponding Banach space.
That is, we set $K_d=K_d^\gamma$ and $\rho = \chi_{[0,1]^d}\in\mathcal{D}_d$
in the following.

In turn we have (strong) polynomial tractability for the uniform approximation problem \wrt the worst case setting only if average case $\L_2$-approximation is polynomially tractable, as long as we consider the absolute error criterion.
Due to \cite[Theorem 6.1]{NW08} we know that the latter holds true if and only if
there exist a positive constant $c_1$, non-negative $q_1$, $q_2$ and $\tau\in(0,1)$ such that
\begin{gather*}
		c_2 = \sup_{d\in\N} \frac{1}{d^{q_2}} \left( \sum_{i=\ceil{c_1\, d^{q_1}}}^\infty (\lambda_{d,i})^\tau \right)^{1/\tau} < \infty,
\end{gather*}
where $(\lambda_{d,i})_{i=1}^\infty$ denotes the sequence of eigenvalues of the correlation operator 
$C_{\nu_d}$ with respect to a non-increasing ordering.
Moreover we have strong polynomial tractability if and only if this holds with $q_1=q_2=0$.

Because of the observation at the end of \autoref{ex:average_case_approx} it suffices to consider
the eigenvalues of $W_d^\gamma = \left( S_d^\gamma \right)^\dagger S_d^\gamma \colon \Hi_d^\gamma \nach \Hi_d^\gamma$, where $S_d^\gamma$ describes the embedding $\Hi_d^\gamma \hookrightarrow \L_2([0,1]^d)$.
Recall that these eigenvalues are given by 
\begin{gather*}
		\left\{ \widetilde{\lambda}_{d,\gamma,\bm{m}}=\prod_{k=1}^d \lambda_{1,\gamma_{d,k},m_k} 
		=\prod_{k=1}^d \frac{\gamma_{d,k}}{\gamma_{d,k}+\pi^2\,(m_k-1)^2} \sep \bm{m}=(m_1,\ldots,m_d)\in\N^d \right\};
\end{gather*}
see~\link{eq:L2_eigenvalues} at the end of \autoref{ex:Unanchored_Sobolev}.
Thus we only need to reorder this set appropriately using a rearrangement $\psi_d \colon \N \nach \N^d$
such that
\begin{gather*}
		\lambda_{d,i} = \widetilde{\lambda}_{d,\gamma,\psi_d(i)} \geq \widetilde{\lambda}_{d,\gamma,\psi_d(i+1)}
		\quad \text{for all} \quad i\in\N.
\end{gather*}

Given $d\in\N$, $\tau\in(0,1)$, as well as $c_1>0$, and $q_1\geq 0$ we estimate
\begin{align*}
		\sum_{i=\ceil{c_1\, d^{q_1}}}^\infty (\lambda_{d,i})^\tau
		&=\sum_{\bm{m}\in\N^d} \left( \widetilde{\lambda}_{d,\gamma,\bm{m}} \right)^\tau 
				- \sum_{i=1}^{\ceil{c_1 \, d^{q_1}}-1} (\lambda_{d,i})^\tau \\
		&\geq \prod_{k=1}^d \sum_{m\in\N} \left( \lambda_{1,\gamma_{d,k},m} \right)^\tau 
				- (\lambda_{d,1})^\tau \left( \ceil{c_1 \, d^{q_1}}-1 \right) \\
		&\geq \prod_{k=1}^d \left( 1+ \sum_{m=2}^\infty \left( \frac{\gamma_{d,k}}{\gamma_{d,k}+\pi^2(m-1)^2} \right)^\tau \right)
				- c_1 \, d^{q_1},
\end{align*}
since $\lambda_{d,1}=\widetilde{\lambda}_{d,\gamma,\bm{(1,\ldots,1)}}=\prod_{k=1}^d \lambda_{d,\gamma_{d,k},1} = 1$.
Due to the boundedness of the generator weights $\gamma_{d,k}\leq C_\gamma$
for every $k\in\{1,\ldots,d\}$, we can further estimate the sum by
\begin{gather*}
			\sum_{m=2}^\infty \left( \frac{\gamma_{d,k}}{\gamma_{d,k}+\pi^2(m-1)^2} \right)^\tau
			\geq \gamma_{d,k}^\tau \sum_{i=1}^\infty \frac{(C_\gamma')^\tau}{i^{2\tau}}
			= \gamma_{d,k}^\tau \, (C_\gamma')^\tau\, \zeta(2\tau),
\end{gather*}
where we set $C_\gamma' = (C_\gamma + \pi^2)^{-1}$.
Because of $\ln(1+y)\geq y/(1+y)$ for all $y \geq 0$ we conclude that
for $k=1,\ldots,d$ and some positive $C$ depending on $C_\gamma$ and $\tau$
\begin{gather*}
			\ln\left(1+\gamma_{d,k}^\tau \, (C_\gamma')^\tau\, \zeta(2\tau)\right) 
			\geq \frac{(C_\gamma')^\tau\, \zeta(2\tau)}{1+\gamma_{d,k}^\tau \, (C_\gamma')^\tau\, \zeta(2\tau)} \cdot \gamma_{d,k}^\tau
			\geq C \cdot \gamma_{d,k}^\tau.
\end{gather*}
Consequently, this yields
\begin{gather*}
		\sum_{i=\ceil{c_1\, d^{q_1}}}^\infty (\lambda_{d,i})^\tau
		\geq \prod_{k=1}^d \exp{C\cdot \gamma_{d,k}^\tau} - c_1 \, d^{q_1}
		= \exp{C \sum_{k=1}^d \gamma_{d,k}^\tau} - c_1 \, d^{q_1}.
\end{gather*}

Therefore, polynomial tractability implies $q(\gamma)<1$ and
strong polynomial tractability is possible only if $p(\gamma)<1$.
Here $p$ and $q$ describe the sum exponents of the product weight
sequence $\gamma=(\gamma_{\bm{\alpha}})_{\bm{\alpha}\in\N_0^d}$, $d\in\N$, defined at the beginning of
\autoref{sect:embeddings}.

Together with \autoref{Theorem_Sufficient} this finally proves
\begin{theorem}\label{thm:tract_Sobolev}
			Consider the uniform approximation problem defined on the sequence 
			of unanchored Sobolev spaces $(\Hi_d^{\gamma})_{d\in\N}$, where
			the product weight sequence~$\gamma$ is constructed out of a uniformly bounded
			generator sequence $C_\gamma \geq \gamma_{d,1}\geq\ldots\geq\gamma_{d,d}$, $d\in\N$.
			We study this problem in the worst case setting and with respect to the absolute error criterion.
			Then we have
			\begin{itemize}
					\item polynomial tractability if and only if $q(\gamma)<1$ and
					\item strong polynomial tractability if and only if $p(\gamma)<1$.
			\end{itemize}
\end{theorem}
  \cleardoubleplainpage
\chapter{Problems on Hilbert spaces with (anti)symmetry conditions}\label{chapt:antisym}
In this last chapter we describe an essentially new kind of a priori knowledge
which can help to overcome the curse of dimensionality.
As in \autoref{sect:TensorBasics}, we study compact linear problems $S=(S_d)_{d\in\N}$ 
defined between tensor products of Hilbert spaces but now we restrict our attention to problem 
elements which fulfill certain (anti)symmetry conditions.
After investigating some basic properties of the related subspaces of
(anti)symmetric problem elements in \autoref{sect:basic_antisym_def}
we construct a linear algorithm that uses finitely many continuous linear
functionals and show an explicit formula for its worst case error 
in terms of the eigenvalues $\lambda=(\lambda_m)_{m\in\N}$ 
of the operator $W_1 = {S_1}^{\!\dagger} S_1$.
Moreover, in \autoref{sect:opt_asym_algo} we show that this algorithm is optimal
\wrt a wide class of algorithms.
Next we clarify the influence of different
(anti)symmetry conditions on the complexity, 
compared to the case for the classical unrestricted problem studied in \autoref{sect:tensor_complexity}.
In particular, we give necessary and sufficient conditions for 
(strong) polynomial tractability of (anti)symmetric problems in \autoref{sect:complexity_antisym}.
Apart from the absolute error criterion we also deal with normalized errors.
Finally, in \autoref{sect:applications}, we discuss several applications.
\autoref{sect:wave} particularly indicates how to apply our
results to the approximation problem for wavefunctions.

Most of the results stated in this chapter are already published in the articles
\cite{W11} and \cite{W12b}.
At some points we improve the known results and/or proof techniques slightly.
In particular, the presented results also hold for problems defined on finite-dimensional or on non-separable source spaces.

\section{Basic definitions related to (anti)symmetry}\label{sect:basic_antisym_def}
The aim of this section is to introduce the notion of (anti)symmetry
in Hilbert spaces.
In order to illustrate this concept 
we mainly deal with \textit{function spaces}.
For this purpose in \autoref{sect:antisym_HFS} we start 
by defining (anti)symmetry properties 
for functions which will lead us to orthogonal projections, 
mapping the whole space onto its
subspace of (anti)symmetric functions.
In \autoref{sect:antisym_tensor} it will turn out that these projections 
applied to a given basis 
of a tensor product Hilbert function space lead us 
to handsome formulas for orthonormal bases of the subspaces.
Finally we generalize our approach
and define (anti)symmetry conditions 
for \textit{arbitrary} tensor product Hilbert spaces 
based on the deduced results for function spaces.
\autoref{sect:antisym_general} is devoted to this generalization.

\subsection{Hilbert function spaces}\label{sect:antisym_HFS}
Following Hamaekers~\cite[Section 2.5]{H09} we use a general approach
to (anti)sym\-metric functions
which also can be found in~\cite{W12b}.
Consider~$H$ to be a (possibly non-separable) Hilbert space of
real-valued multivariate functions~$f$ defined 
on some domain $\Omega$ in $\R^d$,
where we assume $d \geq 2$ to be fixed.
Furthermore, take an arbitrary non-empty 
subset of coordinates $I \subseteq \{1,\ldots,d\}$.
For every such subset we define the set
\begin{equation}
			\S_I = \left\{\pi \colon \{1,\ldots,d\} \nach \{1,\ldots,d\} 
								\sep \pi \text{ bijective and } \pi\big|_{\{1,\ldots,d\}\setminus I} = \id \right\}
\end{equation}
of all permutations on $\{1,\ldots,d\}$ that leave the complement of $I$ fixed.
To abbreviate the notation we identify $\pi \in \S_I$
with the corresponding permutation $\pi'$ on $\R^d$, 
\begin{equation*}
			\pi' \colon \R^d \nach \R^d, \qquad \bm{x}=(x_1,\ldots,x_d) \mapsto \bm{\pi'(x)}=(x_{\pi(1)}, \ldots, x_{\pi(d)}).
\end{equation*}

For an appropriate definition of partial (anti)symmetry of functions $f\in H$ 
we need the following simple assumptions. 
Given any $\pi \in \S_I$ we assume that
\begin{enumerate}[label=(A\thechapter.\arabic{*}), ref=A\thechapter.\arabic{*}, leftmargin=!, labelwidth=\widthof{(A5.2)}]
				\item \label{A1} $\bm{x}\in \Omega$ implies $\bm{\pi(x)}\in\Omega$,
				\item \label{A2} $f\in H$ implies $f(\bm{\pi(\cdot)}) \in H$ and
				\item \label{A3} there ex. $c_{\pi} \geq 0$ (independent of $f$) such that 
												 $\norm{f(\bm{\pi(\cdot)}) \sep H} \leq c_{\pi} \norm{f \sep H}$.
\end{enumerate}
A function $f \in H$ is called \emph{partially symmetric \wrt $I$} 
(or \emph{$I$-symmetric} for short) if any permutation $\pi\in\S_I$ 
applied to the argument $\bm{x}$ does not affect the value of~$f$. 
Hence,
\begin{gather}\label{sym}
				f(\bm{x})=f(\bm{\pi(x)}) \quad \text{for all} 
				\quad \bm{x}\in \Omega \quad \text{and every} \quad \pi \in \S_I.
\end{gather}
Moreover, we call a function $f \in H$ 
\emph{partially antisymmetric \wrt $I$}
(or \emph{$I$-antisymmetric}, respectively) 
if $f$ changes its sign by exchanging the variables~$x_i$ and $x_j$ with each other, where $i,j\in I$.
That is, we have 
\begin{gather}\label{antisym}
				f(\bm{x})=(-1)^{\abs \pi} f(\bm{\pi(x)}) \quad \text{for all} 
				\quad \bm{x}\in \Omega \quad \text{and every} \quad \pi \in \S_I,
\end{gather}
where $\abs \pi$ denotes the \emph{inversion number} 
of the permutation $\pi$. 
The term $(-1)^{\abs{\pi}}$ therefore coincides with the \emph{sign}, 
or \emph{parity of $\pi$} and is equal to the determinant 
of the associated permutation matrix.
In the case $\#I=1$ we do not claim any (anti)symmetry, 
since then the set $\S_I = \{ \id \}$ is trivial.
For $I=\{1,\ldots, d\}$ functions~$f$ which satisfy 
\link{sym} or \link{antisym}, respectively, are called 
\emph{fully (anti)symmetric}.

Note that, in particular, formula \link{antisym} yields that the value $f(\bm{x})$ of
$I$-antisym\-metric functions $f$ equals zero if $x_i=x_j$ with $i \neq j$ and $i,j\in I$. 
For $I$-symmetric functions such an implication does not hold.
Therefore the (partial) antisymmetry property is a somewhat more restrictive condition 
than the (partial) symmetry property with respect to the same subset $I$.
As we will see in \autoref{sect:complexity_antisym} 
this will also affect our complexity estimates.

Next we define the so-called \emph{symmetrizer} $\SI_I^H$ and 
\emph{antisymmetrizer $\AI_I^H$ on $H$ with respect to the subset $I$} by
\begin{gather*}
				\SI_I^H \colon H \nach H, \quad f\mapsto \SI_I^H (f) = \frac{1}{\#\S_I} \sum_{\pi \in \S_I} f(\bm{\pi(\cdot)})
\end{gather*}
and
\begin{gather*}
				\AI_I^H \colon H \nach H, \quad f\mapsto \AI_I^H (f) = \frac{1}{\#\S_I} \sum_{\pi \in \S_I} (-1)^{\abs{\pi}} f(\bm{\pi(\cdot)}).
\end{gather*}
If there is no danger of confusion we use the notation 
$\SI_I$ and $\AI_I$ instead of~$\SI_I^H$ and~$\AI_I^H$, respectively.
The following lemma collects some basic properties.
It generalizes Lemma~10.1 in Zeiser~\cite{Z10}.

\begin{lemma}\label{projection}
				For $\leer \neq I\subseteq\{1,\ldots,d\}$ both the mappings 
				$P_I\in\{\SI_I, \AI_I\}$ 
				define bounded linear operators on the Hilbert space $H$ with $P_I^2=P_I$. 
				Thus, $\SI_I$ and $\AI_I$ provide projections 
				of $H$ onto the closed linear subspaces
				\begin{gather}\label{antisymsubspace}
								\SI_I(H) = \{ f \in H \sep f \text{ satisfies } \link{sym} \}  
								\, \, \text{ and } \, \,
								\AI_I(H) = \{ f \in H \sep f \text{ satisfies } \link{antisym} \} 
				\end{gather}
				of all partially (anti)symmetric functions 
				\wrt $I$ in $H$, respectively. 
				If, in addition,
				\begin{gather}\label{eq:invariant}
						\distr{f(\bm{\pi(\cdot)})}{g(\bm{\pi(\cdot)})}_{H} = \distr{f}{g}_H
						\quad \text{for all} \quad f,g\in H \quad \text{and every} \quad \pi\in\S_I
				\end{gather}
				then the operators are self-adjoint and hence the projections are orthogonal.
				Consequently, 
				\begin{gather}\label{orth_decomp}
								H = \SI_I(H) \oplus (\SI_I(H))^\bot = \AI_I(H) \oplus (\AI_I(H))^\bot.
				\end{gather}				
\end{lemma}

\begin{proof}
Obviously $P_I \in \{\SI_I, \AI_I\}$ is well-defined due to the assumptions~\link{A1} and~\link{A2}. 
The linearity directly follows from the definition and, using \link{A3}, 
we see that the operator norm of $P_I$ is bounded by $\max{c_\pi \sep \pi \in \S_I}$.

To show that the operators are idempotent, \ie that $P_I^2=P_I$, 
we first prove that~$\AI_I(f)$ satisfies~\link{antisym} for every $f \in H$. 
Therefore, we use the representation
\begin{align*}
				(\AI_I(f))(\bm{\pi(\cdot)})
				&= \frac{1}{\#\S_I} \sum_{\sigma \in \S_I} (-1)^{\abs{\sigma}} f(\bm{\sigma(\pi(\cdot))}) 
				= \frac{1}{\#\S_I} \sum_{\lambda \in \S_I} (-1)^{\abs{\lambda} + \abs{\pi}} f(\bm{\lambda(\cdot)})\\
				&= (-1)^{\abs{\pi}} (\AI_I(f))(\cdot)
\end{align*}
for every fixed $\pi\in \S_I$. 
Here we imposed $\lambda = \sigma \circ \pi \in \S_I$ and used that
\begin{gather*}
				\abs{\lambda \circ \pi^{-1}} = \abs{\lambda} + \abs {\pi^{-1}} = \abs{\lambda} + \abs {\pi}.
\end{gather*}
Hence we have shown $\AI_I(H) \subseteq \{ f \in H \sep f \text{ satisfies } \link{antisym} \}$.
In a second step, it is easy to check that for every function $g\in H$ which satisfies \link{antisym}
it is $\AI_I(g)=g$. Thus, $\{ f \in H \sep f \text{ satisfies } \link{antisym} \} \subseteq \AI_I(H)$
and $\AI_I$ is a projector onto $\AI_I(H)$.
Since the same arguments also apply for the symmetrizer $\SI_I$ this shows \link{antisymsubspace}, 
as well as $P_I^2=P_I$ for $P_I\in\{\SI_I, \AI_I\}$.

To prove the self-adjointness of $P_I$ we need to show that for $f$ and $g$ in $H$ we have
$\distr{P_I f}{g}_H=\distr{f}{P_I g}_H$.
To this end, note that \link{eq:invariant} is equivalent to the fact that
\begin{gather*}
			\distr{f(\bm{\pi(\cdot)})}{g}_H
			= \distr{f}{g(\bm{\sigma(\cdot)})}_H, \quad f,g\in H, \quad \pi\in\S_I,
\end{gather*}
where we set $\sigma=\pi^{-1}$ and used \link{A2}.
Now the claimed assertion follows from the bilinearity of the inner product $\distr{\cdot}{\cdot}_H$.
Moreover, orthogonality and the 
decompositions stated in \link{orth_decomp} are simple consequences.
\end{proof}

We note in passing that \link{eq:invariant} already implies \link{A3}.
Furthermore, the notion of partially (anti)symmetric functions 
can be easily extended to more than one subset~$I$.
Therefore, consider two non-empty subsets of coordinates 
$I,J\subset \{1,\ldots,d\}$ with $I\cap J = \leer$.
Then we call a function $f\in H$ 
\emph{multiple partially (anti)symmetric \wrt $I$ and $J$} 
if $f$ satisfies \link{sym}, or \link{antisym}, 
respectively, for~$I$ and~$J$.
Since $I$ and~$J$ are disjoint we observe that 
$\pi \circ \sigma = \sigma \circ \pi$
for all $\pi\in \S_I$ and $\sigma\in \S_J$.
Hence the linear projections $P_I \in \{\SI_I, \AI_I\}$ and 
$P_J \in \{\SI_J, \AI_J\}$ commute on $H$. 
That is, we have 
$P_I \circ P_J = P_J \circ P_I$.
Further extensions to more than two disjoint subsets of coordinates are possible. 
We will restrict ourselves to the case of at most two coordinate subsets, 
because in particular wavefunctions can be modeled as functions
which are antisymmetric \wrt $I$ and $J=I^c$, 
where~$I^c$ denotes the complement of~$I$ in $\{1,\ldots,d\}$; see, 
e.g., \autoref{sect:wave}.

\subsection{Tensor products of Hilbert function spaces}\label{sect:antisym_tensor}
In the previous subsection the function space $H$ was a somewhat abstract
Hilbert space of \mbox{$d$-variate} real-valued functions.
Indeed, for the definition of (anti)symmetry 
we do not need to claim any product structure.
On the other hand, 
it is also motivated by applications to consider tensor product function spaces;
see, e.g., Section 3.6 in Yserentant~\cite{Y10}.
In detail, it is well-known that so-called spaces of dominated mixed smoothness, 
e.g. $W_2^{(1,\ldots,1)}(\R^{3d})$, can be represented as certain tensor products; 
see Section 1.4.2 in Hansen~\cite{H10}. 

Anyway, let us take into account such a structure, 
\ie let us assume that
\begin{gather*}
		H = H_d = H_1\otimes \ldots \otimes H_1 \quad \text{($d\geq 2$ times)},
\end{gather*}
where $H_1$ is a suitable Hilbert space of functions $f\colon D\nach\R$; 
see also the constructions given in 
\autoref{subsect:def_tensor_prob}.
There it is stated that we can construct an orthonormal basis $E_d$ of $H_d$
out of a given ONB $E_1$ of $H_1$; see \link{tensor_ONB}.
Since now we deal with function spaces, the $d$-fold simple tensors in $E_d$ 
are $d$-variate functions $e_{d,\bm{j}}\colon D^d \nach \R$. 
More precisely, they are given by
\begin{equation*}
		e_{d,\bm{j}}(\bm{x}) = \prod_{l=1}^d e_{j_l}(x_l), 
		\quad \text{where} 
		\quad \bm{x}=(x_1,\ldots,x_d)\in D^d 
		\quad \text{and} 
		\quad \bm{j}\in\I_d=(\I_1)^d,
\end{equation*}
provided that $E_1=\{e_m \colon D\nach\R \sep m \in \I_1 \}$ 
denotes the underlying ONB in $H_1$.
To exploit this representation 
we start with a simple observation.
 
Let $d\in\N$. 
Moreover assume $\bm{j}\in \I_d$ and $\bm{x}\in D^d$, as well as a 
non-empty subset~$I$ of $\{1,\ldots,d\}$, to be arbitrarily fixed.
If we define $\sigma = \pi^{-1} \in \S_I$ then
\begin{equation}\label{pi_inside}
				e_{d,\bm{j}}(\bm{\pi(x)}) 
				= \prod_{l=1}^d e_{j_l}(x_{\pi(l)})
				= \prod_{l=1}^d e_{j_{\sigma(l)}}(x_l)
				= e_{d,\bm{\sigma(j)}}(\bm{x}).
\end{equation}
For simplicity, once again we identified $\bm{\pi(j)}=\pi(j_1,\ldots,j_d)$ 
with $(j_{\pi(1)},\ldots,j_{\pi(d)})$ for $\bm{j}\in\I_d=(\I_1)^d$.
Since $\bm{x}\in D^d$ was arbitrary and $\abs{\pi}=\abs{\pi^{-1}}=\abs{\sigma}$ we obtain
\begin{gather}\label{antisym_basis}
				\SI_I e_{d,\bm{j}} = \frac{1}{\#\S_I} \sum_{\sigma \in \S_I} e_{d,\bm{\sigma(j)}}	
				\quad \text{and} \quad \AI_I e_{d,\bm{j}} = \frac{1}{\#\S_I} \sum_{\sigma \in \S_I} (-1)^{\abs{\sigma}} e_{d,\bm{\sigma(j)}}
\end{gather}
for all $\bm{j}\in \I_d$.
Besides this, \link{pi_inside} can be used to verify that 
\link{eq:invariant} in \autoref{projection} always holds true 
for (unweighted) tensor products of Hilbert function spaces.

Note that in general, \ie for arbitrary $\bm{j}\in\I_d$ and $\sigma\in\S_I$, 
the tensor products $e_{d,\bm{\sigma(j)}}$ and $e_{d,\bm{j}}$ do not coincide,
because taking the tensor product is not commutative in general.
Therefore $\SI_I$ is not simply the identity on the set of basis functions 
$E_d=\{ e_{d,\bm{j}} \sep \bm{j} \in \I_d \}$.
On the other hand, 
we see that for different $\bm{j}\in \I_d$ many of the functions 
$\SI_I e_{d,\bm{j}}$ coincide.
Of course the same holds true for $\AI_I e_{d,\bm{j}}$, 
at least up to a factor of $(-1)$.

We will see in the following that for $P_I\in\{\SI_I, \AI_I\}$ 
a linearly independent subset of all projections 
$\{ P_I e_{d,\bm{j}} \sep \bm{j}\in\I_d \}$ 
equipped with suitable normalizing constants can be used as
an ONB of the linear subspace $P_I(H_d)$ of $I$-(anti)symmetric
functions in $H_d$.
For the application we have in mind, we need this result only in
the case where the underlying space $H_1$ is separable.
Without loss of generality, we can thus assume that\footnote{Note that also 
the case of abstract, countable index sets~$\I_1$ can be reduced to this form by 
the application of some simple isomorphism.} 
\begin{gather*}
		\I_1 = \M_1 = \{ m \in\N \sep m < \dim H_1 + 1\}
\end{gather*}
and consequently $\I_d = \M_d=(\M_1)^d \subseteq \N^d$.
Clearly, in the most interesting case the set~$\I_d$ equals $\N^d$.

To state the claimed assertion, we need a further definition.
For fixed $d\geq 2$ and $I\subseteq\{1,\ldots,d\}$, let us introduce a function
\begin{gather*}
				M_I =M_{I,d} \colon \N^d \nach \{0,\ldots,\#I\}^{\#I}
\end{gather*} 
which counts how often different indices occur in a 
given multi-index $\bm{j}\in\N^d$ among the subset $I$ of coordinates,
ordered with respect to their rate.
To give an example let $d=7$ and $I=\{1,\ldots,6\}$. 
Then $M_{I,7}$ applied to $\bm{j}=(12, 4, 4, 12, 6, 4, 4) \in \N^7$ 
gives the $\#I=6$ dimensional vector $\bm{M_{I,7}(j)} = (3, 2, 1, 0, 0, 0)$, 
because $\bm{j}$ contains the number ``$4$'' three times
among the coordinates $j_1,\ldots,j_6$, ``$12$'' two times, and so on. 
Since in this example there are only three different numbers involved,
the fourth to sixth coordinates of $\bm{M_{I,7}(j)}$ equal zero.
Obviously, $M_{I}$ is invariant under 
all permutations $\pi\in\S_I$ of the argument.
Thus, 
\begin{gather*}
				\bm{M_{I}(j)} = \bm{M_{I}(\pi(j))} 
				\quad \text{for all} \quad 
				\bm{j}\in\N^d 
				\quad \text{and} \quad 
				\pi \in \S_I.
\end{gather*}
In addition, since $\bm{M_{I}(j)}$ again is a multi-index, we see that
$\abs{\bm{M_{I}(j)}}=\#I$ and $\bm{M_{I}(j)}!$ are well-defined for every $\bm{j}\in\N^d$.
Prepared with this tool, we are ready to prove the following

\begin{lemma}\label{lemma_basis}
				Assume $E_d=\{e_{d,\bm{j}} \sep \bm{j}\in \M_d \}$ to be a given orthonormal 
				tensor product basis in the space $H_d$ 
				and let $\leer \neq I=\{i_1,\ldots,i_{\# I}\} \subseteq\{1,\ldots,d\}$.
				Moreover, for $P_I\in\{\SI_I,\AI_I\}$ define the functions $\xi_{\bm{j}} \colon D^d \nach \R$ by
				\begin{gather*}
								\xi_{\bm{j}} = \sqrt{\frac{\#\S_I}{\bm{M_{I}(j)}!}} \cdot P_I(e_{d,\bm{j}}) 
								\quad \text{ for } \quad \bm{j} \in \M_d.
				\end{gather*}
				Then the set $\Xi_d=\{ \xi_{\bm{k}} \sep \bm{k} \in \nabla_d \}$ builds an 
				orthonormal basis of the partially (anti)sym\-metric subspace $P_I(H_d)$, 
				where $\nabla_d$ is given by
				\begin{gather}\label{def_nabla}
								\nabla_d = \begin{cases}
															\{ \bm{k} \in \M_d \sep k_{i_1} \leq k_{i_2} \leq \ldots \leq k_{i_{\#I}}\},& \text{ if } P_I=\SI_I,\\
															\{ \bm{k} \in \M_d \sep k_{i_1} < k_{i_2} < \ldots < k_{i_{\#I}}\},& \text{ if } P_I=\AI_I.
													\end{cases}
				\end{gather}
\end{lemma}
\begin{proof}
To abbreviate the notation, we suppress
the index $H_d$ at the inner products $\distr{\cdot}{\cdot}_{H_d}$ 
in this proof. 

\textit{Step 1}. 
We start by proving orthonormality.
Therefore let us recall \link{antisym_basis} and remember that now $\I_d=\M_d$. 
For $P_I=\AI_I$ and $\bm{j},\bm{k} \in \nabla_d$ easy calculations yield
\begin{align*}
				\distr{\xi_{\bm{j}}}{\xi_{\bm{k}}} 
				&= \frac{\# \S_I}{\sqrt{\bm{M_{I}(j)}! \cdot \bm{M_{I}(\bm{k})}!}} \distr{\AI_I (e_{d,\bm{j}})}{\AI_I (e_{d,\bm{k}})} \\
				&= \frac{1}{\# \S_I \sqrt{\bm{M_{I}(j)}! \cdot \bm{M_{I}(\bm{k})}!}} \sum_{\pi,\sigma \in \S_I} (-1)^{\abs{\pi}+\abs{\sigma}} \distr{e_{d,\bm{\pi(j)}}}{e_{d,\bm{\sigma(k)}}}.
\end{align*}
Of course, up to the factor controlling the sign, the same is true for the case $P_I=\SI_I$.
Now assume that there exists $l \in \{1,\ldots,d\}$ such that $j_l \neq k_l$. 
Then the ordering of $\bm{j},\bm{k}\in\nabla_d$ implies that
$\bm{\pi(j)}\neq\bm{\sigma(k)}$ for all $\sigma,\pi \in \S_I$,
since $\pi$ and $\sigma$ leave the coordinates $l \in I^c$ fixed.
Hence, we conclude that we have
$\bm{\pi(j)} = \bm{\sigma(k)}$ only if $\bm{j}=\bm{k}$.

At this point we have to distinguish the antisymmetric and the symmetric case.
For $P_I=\AI_I$ the only way to conclude $\bm{\pi(j)} = \bm{\sigma(k)}$ is to claim $\bm{j}=\bm{k}$ and $\pi=\sigma$.
Furthermore we see that in the antisymmetric case we have $\bm{M_{I}(j)}!=1$ for all $\bm{j}\in\nabla_d$, 
because then all coordinates~$j_l$, where $l\in I$, differ.
Therefore, in this case the last inner product coincides with 
$\delta_{\bm{j},\bm{k}} \cdot \delta_{\pi,\sigma}$ because of the mutual orthonormality
of the elements from $E_d=\{e_{d,\bm{j}} \sep \bm{j}\in \M_d \}$.
Hence we arrive at
\begin{gather*}
				\distr{\xi_{\bm{j}}}{\xi_{\bm{k}}} 
				= \frac{1}{\# \S_I} \sum_{\pi \in \S_I} (-1)^{2\abs{\pi}} \delta_{\bm{j},\bm{k}} 
				= \delta_{\bm{j},\bm{k}} 
				\quad \text{for all} \quad \bm{j},\bm{k}\in\nabla_d,
\end{gather*}
as claimed. 

So, let us consider the case $P_I=\SI_I$ and $\bm{j}=\bm{k}\in \nabla_d$, 
since we already saw that otherwise $\distr{\xi_{\bm{j}}}{\xi_{\bm{k}}}$ equals zero.
Then for fixed $\sigma \in \S_I$ there are $\bm{M_{I}(j)}!$ different permutations
$\pi \in \S_I$ such that $\bm{\pi(j)} = \bm{\sigma(j)}$.
This leads to
\begin{gather*}
				\distr{\xi_{\bm{j}}}{\xi_{\bm{j}}} 
				= \frac{1}{\# \S_I \cdot \bm{M_{I}(j)}!} \sum_{\sigma \in \S_I} \bm{M_{I}(j)}! 
				= 1
\end{gather*}
and completes the proof of orthonormality.

\textit{Step 2}. It remains to show that 
the span of $\Xi_d=\{\xi_{\bm{k}} \sep \bm{k}\in \nabla_d\}$ 
is dense in $P_I(H_d)$ for $P_I\in \{\SI_I,\AI_I\}$.
Note that every multi-index $\bm{j}\in\M_d$ 
can be represented by a uniquely defined multi-index $\bm{k}\in \nabla_d$
and exactly $\bm{M_I(k)}!$ different permutations $\pi\in\S_I$
such that $\bm{j}=\bm{\pi(k)}$.
Assume that $f \in \AI_I(H_d)$, \ie $f\in H_d$ satisfies \link{antisym}. 
Then \link{pi_inside} together with \link{eq:invariant} yields
\begin{gather}\label{formula_coeff}
				\distr{f}{e_{d,\bm{j}}} 
				= (-1)^{\abs{\pi}} \cdot \distr{f}{e_{d,\bm{\pi(j)}}} 
				\quad \text{for all} \quad 
				\bm{j}\in \M_d 
				\quad \text{and} \quad 
				\pi \in \S_I.
\end{gather}
Now expanding $f$ with respect to the basis functions 
in $E_d \subset H_d$ gives
\begin{align*}
				f 
				&= \sum_{\bm{j}\in\M_d} \distr{f}{e_{d,\bm{j}}} e_{d,\bm{j}} 
				= \sum_{\bm{k}\in \nabla_d} \sum_{\pi \in \S_I} \frac{ \distr{f}{e_{d,\bm{\pi(k)}}} e_{d,\bm{\pi(k)}} }{\bm{M_I(k)}!} \\
				&= \sum_{\bm{k}\in \nabla_d} \frac{1}{\bm{M_I(k)}!} \sum_{\pi \in \S_I} (-1)^{\abs{\pi}} \distr{f}{e_{d,\bm{k}}} e_{d,\bm{\pi(k)}} \\
				&= \sum_{\bm{k}\in \nabla_d} \sqrt{\frac{\# \S_I}{\bm{M_I(k)}!}} \cdot \distr{f}{e_{d,\bm{k}}} \cdot \sqrt{\frac{\# \S_I}{\bm{M_I(k)}!}} \cdot \AI_I(e_{d,\bm{k}}),
\end{align*}
where we used \link{antisym_basis} for the last equality.
Furthermore, due to the self-adjointness of $\AI_I$, 
we have 
$\distr{f}{e_{d,\bm{k}}} = \distr{\AI_I f}{e_{d,\bm{k}}} = \distr{f}{\AI_I e_{d,\bm{k}}}$,
such that finally $f\in \AI_I(H_d)$ possesses the representation
\begin{gather*}
				f = \sum_{\bm{k} \in \nabla_d} \distr{f}{\xi_{\bm{k}}} \cdot \xi_{\bm{k}}
\end{gather*}
since $\xi_{\bm{k}}=\sqrt{\# \S_I / \bm{M_I(k)}!} \cdot \AI_I(e_{d,\bm{k}})$ per definition. 
This proves the assertion for the case $P_I=\AI_I$.
The remaining case $P_I=\SI_I$ can be treated in the same way.
\end{proof}

Observe that in the antisymmetric case the definition of $\xi_{\bm{j}}$
for $\bm{j}\in\nabla_d$ simplifies, since then $\bm{M_I(j)}!=1$ for all $\bm{j}\in\nabla_d$.
Moreover we see that in this case $\nabla_d$ is trivial if $d>\#\M_1$. 
Hence we should assume that $\dim H_1$ is infinite in order to work with antisymmetric
tensor products for arbitrarily many building blocks.
We note in passing that the square of the normalizing factor,
$\#\S_I / \bm{M_{I}(j)}!$, coincides with the multinomial coefficient $\binom{\abs{\bm{M_{I}(j)}}}{\bm{M_{I}(j)}}$
which is quite natural due to combinatorial issues.
Furthermore, in the special case $I=\{1,2,\ldots,\#I \}$ 
we have 
\begin{gather*}
				P_I(H_d) = P_I \!\left( \bigotimes_{m\in I} H_1 \right) \otimes \left( \bigotimes_{m\notin I} H_1 \right).
\end{gather*}
That is, we can consider the subspace of $I$-(anti)symmetric functions 
$f\in H_d$ as the tensor product
of the set of all fully (anti)symmetric $\#I$-variate functions 
with the $(d-\#I)$-fold tensor product of $H_1$.
If $\# I = 1$, \ie if we do not claim any (anti)symmetry, then
$P_I(H_d)=H_d$ and thus we have $\nabla_d=\M_d$, as well as $\Xi_d=E_d$.
Modifications in connection with multiple partially (anti)symmetric functions are obvious.

\subsection{Arbitrary tensor product Hilbert spaces}\label{sect:antisym_general}
Up to now we exclusively dealt with Hilbert \textit{function} spaces.
However, the proofs of \autoref{projection} and \autoref{lemma_basis} yield 
that there are only a few key arguments in connection with 
(anti)symmetry such that we can cut out this restriction.
We briefly sketch the points which need to be changed.

Starting from the very beginning we have to adapt the definition 
of $I$-(anti)sym\-metry due to \link{sym} and \link{antisym} in \autoref{sect:antisym_HFS}.
Of course it is sufficient to define this property at first only for basis elements.
Therefore, if $E_d=\{e_{d,\bm{k}} \sep \bm{k} \in (\I_1)^d=\I_d \}$ 
denotes a tensor product ONB of $H_d$
and $\leer \neq I\subseteq\{1,\ldots,d\}$ is given then
we call an element $e_{d,\bm{k}} = \bigotimes_{l=1}^d e_{k_l}$
\emph{partially symmetric with respect to $I$} 
(\emph{$I$-symmetric}), if
\begin{gather*}
			e_{d,\bm{k}} = e_{d,\bm{\pi(k)}} \quad \text{for all} \quad \pi \in \S_I,
\end{gather*}
where $\S_I$ and $\bm{\pi(k)}=(k_{\pi(1)},\ldots,k_{\pi(d)})\in\I_d$ are defined as before.
Analogously, we define \emph{$I$-antisymmetry} 
with an additional factor $(-1)^{\abs{\pi}}$.
Finally, an arbitrary element in $H_d$ is called $I$-(anti)symmetric
if in its basis expansion every element 
with non-vanishing coefficient possesses this 
property.\footnote{Note that even in the non-separable case any such expansion only has countably many terms.}

Next, the \emph{antisymmetrizer} $\AI_I$ is given as the uniquely 
defined continuous extension of the linear mapping
\begin{gather}\label{eq:general_def_AI}
			\widetilde{\AI}_I \colon E_d \nach H_d, 
			\quad 
			e_{d,\bm{k}} \mapsto \frac{1}{\# \S_I} \sum_{\pi \in \S_I} (-1)^{\abs{\pi}} e_{d,\bm{\pi(k)}}
\end{gather}
from $E_d$ to $H_d$.
Again the \emph{symmetrizer} $\SI_I$ is given in a similar way.
Hence, in the general setting we define the mappings using 
formula~\link{antisym_basis} which we derived for the special case of function spaces.
Note that the triangle inequality yields $\norm{P_I} \leq 1$, for $P_I\in\{\SI_I,\AI_I\}$.

Once more we denote the sets of all $I$-(anti)symmetric elements of $H_d$
by~$P_I(H_d)$, where $P_I\in \{\SI_I,\AI_I\}$.
Observe that this can be justified
since the operators~$P_I$ again provide orthogonal projections onto closed linear subspaces.
That is, a generalization of \autoref{projection} remains valid
also in the more general case of tensor products of arbitrary Hilbert space
which we consider here. 
This can be shown using \link{eq:general_def_AI} and its analogue for $\SI_I$, as well as with the help of some simple extension arguments.
Moreover, also the proof of \autoref{lemma_basis} can be adapted to the
generalized setting.
Indeed, the only difference is the
conclusion of formula~\link{formula_coeff} in Step 2. 
Now, for arbitrary Hilbert spaces, this simply follows from our
definitions.
Then the rest of the proof transfers literally.

Finally and without going into details, we stress 
the point that further generalizations are possible.
Here we can think of tensor products of arbitrary Hilbert spaces 
with multiple partial (anti)symmetry conditions or of
scaled tensor products in the sense of \autoref{chapt:ScaledNorms}.
Since the corresponding calculations are straightforward we leave them to the reader.

\section{Optimal algorithms for (anti)symmetric problems}\label{sect:opt_asym_algo}
Keeping the definitions and assertions from the previous \autoref{sect:basic_antisym_def} in mind, 
we are ready to study algorithms for linear problems defined on (anti)symmetric subsets of tensor product Hilbert spaces.

Let $S_d\colon H_d \nach \G_d$ denote a tensor product problem
in the sense of \autoref{sect:TensorBasics}.
It is constructed out of a compact linear operator $S_1 \colon H_1 \nach \G_1$
between arbitrary Hilbert spaces $H_1$ and $\G_1$ via a tensor product construction;
see \autoref{subsect:def_tensor_prob}.
Hence, let $H_d = H_1 \otimes\ldots\otimes H_1$ in what follows and
refer to the problem of approximating $S=(S_d)_{d\in\N}$ 
as the \emph{entire $d$-variate problem}.
Note that we completely solved this problem in \autoref{sect:TensorBasics}.
In detail, the $n$th optimal algorithm $A_{n,d}^*$, given by \link{eq:opt_tensor_algo},
was related to a certain subset 
$\{e_{d,\bm{j}}=\widetilde{\phi}_{d,\bm{j}}\sep \bm{j}\in\M_d\}$ of a tensor product ONB.

In contrast, now we are interested in the approximation of the restriction 
\begin{gather*}
			S_{d,I_d}=S_d\big|_{P_{I_d}(H_d)} \colon P_{I_d}(H_d) \nach \G_d
\end{gather*} 
of $S_d$ to some (anti)symmetric subspace $P_{I_d}(H_d)$ 
as defined in \autoref{sect:antisym_general}, where $P_{I_d}\in\{\SI_{I_d},\AI_{I_d}\}$
and $\leer\neq I_d \subseteq\{1,\ldots,d\}$ for $d\in\N$.
We refer to $S_I = (S_{d,I_d})_{d\in\N}$ as the \emph{$I$-(anti)symmetric problem}.
Using the notation from \autoref{sect:General} we thus have $\F_d = P_{I_d}(H_d)$
and, consequently, $\widetilde{\F}_d=\B(P_{I_d}(H_d))$.

Due to \link{formula_coeff} it is quite clear that $A_{n,d}^*$ 
cannot be optimal in this restricted setting
since it calculates redundant pieces of information.
Hence we need to go beyond this naive attempt to
solve $I$-(anti)symmetric problems efficiently.
On the other hand, $P_{I_d}(H_d)$ equipped with the inner product of $H_d$, $\distr{\cdot}{\cdot}_{H_d}$,
again is a Hilbert space. 
Therefore we basically know how to construct an optimal algorithm; see \autoref{sect:opt_Hilbert_Algo}.
If $\# I_d=1$
then our new algorithm should resemble~$A_{n,d}^*$, because 
then we do not claim any (anti)symmetry and thus we deal with the
entire tensor product problem.

Before we state the main assertion of this section
we present an auxiliary result which shows that any optimal algorithm 
$A^*$ for $S_{d,I_d}$ needs
to preserve the (anti)symmetry properties of its domain of definition.
The following proposition generalizes Lemma 10.2 in Zeiser~\cite{Z10} where this
assertion was shown for the approximation problem, that is for $S_{d,I_d} = \id\colon P_{I_d}(H_d)\nach \G_d$.

\begin{prop}\label{prop:bestapprox}
				Let $d>1$ and $\leer \neq I \subseteq \{1,\ldots,d\}$ be arbitrarily fixed. 
				Furthermore, for $X\in\{H,\G\}$ let $P_I^X$ denote the (anti)symmetrizer 
				$P_I\in\{\SI_I,\AI_I\}$ on~$X_d$ 
				with respect to $I$.
				Then we have
				\begin{gather}\label{commute}
								(S_{d} \circ P_I^{H})(g) = (P_I^{\G} \circ S_{d})(g) \quad \text{for any} \quad g \in H_d.
				\end{gather}
				Moreover, for all $A \colon P_I^{H}(H_d) \nach \G_d$ and every $f \in P_I^{H}(H_d)$,
				\begin{gather}\label{bestapprox}
								\norm{S_{d,I} f - A f \sep \G_d}^2 
								= \norm{S_{d,I} f - P_I^{\G} (A f) \sep \G_d}^2 + \norm{Af - P_I^{\G} (A f) \sep \G_d}^2.
				\end{gather}
				Hence an optimal algorithm $A^*$ for $S_{d,I}$ preserves (anti)symmetry, \ie
				\begin{gather*}
								A^*f \in P_I^{\G}(\G_d) \quad \text{for all} \quad f \in P_I^{H}(H_d).
				\end{gather*}
\end{prop}
\begin{proof}
The proof is organized as follows. 
First we show that the tensor product operator $S_{d}$ and 
the (anti)symmetrizer $P_I$ commute on $H_d$, \ie it holds \link{commute}.
In a second step we conclude \link{bestapprox} out of this. 
The (anti)symmetry of $A^*f$ for an optimal algorithm $A^*$ then follows immediately.

\textit{Step 1}. 
Assume $E_d=\{e_{d,\bm{j}} \sep \bm{j}\in \I_d\}$ 
to be an arbitrary tensor product ONB of~$H_d$,
as defined in \link{tensor_ONB}.  
Then, for fixed $\bm{j}\in\I_d$, formula \link{eq:general_def_AI} 
and the structure of 
$S_{d}=S_1 \otimes \ldots \otimes S_1$ 
yields in the case $P_I=\AI_I$
\begin{align*}
				S_{d}(\AI_I^{H} (e_{d,\bm{j}})) 
				&= S_{d}\left(\frac{1}{\# \S_I} \sum_{\pi \in \S_I} (-1)^{\abs{\pi}} \bigotimes_{l=1}^d e_{j_{\pi(l)}}\right)\\ 
				&= \frac{1}{\# \S_I} \sum_{\pi \in \S_I} (-1)^{\abs{\pi}} \bigotimes_{l=1}^d S_1(e_{j_{\pi(l)}})
				= \AI_I^{\G} (S_{d} (e_{d,\bm{j}})).
\end{align*} 
Obviously the same is true for $P_I=\SI_I$.
Hence, \link{commute} holds at least on the set of basis elements $E_d$ of $H_d$.
Because of the representation $g = \sum_{\bm{j}\in\I_d} \distr{g}{e_{d,\bm{j}}}_{H_d} \cdot e_{d,\bm{j}}$ of $g \in H_d$,
as well as the linearity and boundedness of the operators 
$P_I^{H}, P_I^{\G}$ and~$S_d$, we can extend the relation \link{commute}
from $E_d$ to the whole space $H_d$.

\textit{Step 2}. Now let $f\in P_I^{H}(H_d)$ and let $A f$ denote 
an arbitrary approximation to~$S_{d,I} f$.
Then $S_{d,I} f = S_d(P_I^{H} f) = P_I^{\G}(S_d f)$, due to Step 1.
Using the fact that $P_I^{\G}$ provides an orthogonal projection 
onto $P_I^{\G}(\G_d)$, see \link{orth_decomp} in \autoref{projection}, 
we obtain \link{bestapprox}, \ie
\begin{eqnarray*}
				\norm{S_{d,I} f - A f \sep \G_d}^2 
				&=& \norm{P_I^{\G} (S_d f) - [P_I^{\G} (A f) + (\id^{\G} - P_I^{\G})(A f)] \sep \G_d}^2 \\
				&=& \norm{P_I^{\G} (S_d f - A f) \sep \G_d}^2 + \norm{(\id^{\G} - P_I^{\G})(A f) \sep \G_d}^2 \\
				&=& \norm{S_{d,I} f - P_I^{\G}(A f) \sep \G_d}^2 + \norm{Af - P_I^{\G}(A f) \sep \G_d}^2,
\end{eqnarray*}
as claimed.
\end{proof}

Apart from this qualitative assertion, we are interested in an explicit formula 
for the optimal algorithm, as well as in sharp error bounds. 
To this end, let $d\in\N$ and $\leer \neq I_d=\{i_1, \ldots, i_{\#I}\} \subseteq\{1,\ldots,d\}$,
as well as $P\in\{\SI,\AI\}$.
Furthermore, consider the singular value decomposition of $S_1\colon H_1\nach\G_1$.
That is, let $\{(\lambda_m,\phi_m) \sep m\in \M_1\}$ 
denote the non-trivial eigenpairs of $W_1={S_1}^{\!\dagger} S_1$; see \autoref{sect:SVD}.
Due to \autoref{prop:tensor_eigenpairs} in \autoref{sect:Eigenpairs} we know that for $d>1$
the (tensor) product eigenpairs 
$\{ (\widetilde{\lambda}_{d,\bm{m}}, \widetilde{\phi}_{d,\bm{m}} ) \sep \bm{m}\in\M_d \}$
of $W_d={S_d}^{\!\dagger} S_d$ are given by \link{tensor_eigenpairs}.
Moreover, $E_d = \Phi_d = \{ \widetilde{\phi}_{d,\bm{m}} \sep \bm{m}\in\M_d \}$ 
builds an tensor product ONB in $H_d$.
Hence, we can apply \autoref{lemma_basis} to 
$e_{d,\bm{j}}=\widetilde{\phi}_{d,\bm{j}}$, $\bm{j}\in\M_d$,
in order to obtain an orthonormal basis 
$\Xi_d=\{ \widetilde{\xi}_{\bm{k}} \sep \bm{k}\in\nabla_d \}$ 
of the partially (anti)symmetric subspaces $P_{I_d}(H_d)$.
More precisely, for $\bm{k}\in\nabla_d$ we define
\begin{gather}\label{eq:eigenpairs_sym}
			\widetilde{\xi}_{\bm{k}} 
			= \sqrt{\frac{\# S_I}{\bm{M_I(k)}!}} \cdot P_{I_d} \left( \bigotimes_{l=1}^d \phi_{k_l} \right) \in P_I(H_d)
			\quad \text{and} \quad
			\widetilde{\lambda}_{d,\bm{k}} = \prod_{l=1}^d \lambda_{k_l}>0,
\end{gather}
where $\nabla_d$ is given by \link{def_nabla}.
Similar to the approach in \autoref{sect:Eigenpairs}, let
\begin{gather*}
			\psi = \psi_d\colon\{i\in\N \sep i < \# \nabla_d + 1\}\nach \nabla_d
\end{gather*}
denote a bijection which provides a non-increasing ordering of 
$\{ \widetilde{\lambda}_{d,\bm{k}} \sep \bm{k}\in\nabla_d \}$
and set $\lambda_{d,i}=\widetilde{\lambda}_{d,\psi(i)}$, as well as 
$\xi_{d,i}=\widetilde{\xi}_{\psi(i)}$ for $i<\# \nabla_d + 1$.
Finally, if $\# \nabla_d$ is finite then we extend the sequence of $\lambda$'s by
setting $\lambda_{d,i}=0$ for $i > \# \nabla_d$.

Given this bunch of notations we are well-prepared to prove our main theorem of this section.
For every $d\in\N$ it provides a linear algorithm $A_{n,d}'$ which uses at most~$n$ continuous linear functionals on the input
to approximate the solution operator~$S_{d,I_d}$ of a given~$I_d$-(anti)symmetric tensor product problem between Hilbert spaces.
Since the worst case error of this algorithm coincides with the $n$th minimal error of the problem, $A_{n,d}'$ is optimal in this setting;
thus it cannot be improved by any other algorithm from the class $\A_d^{n,\rm cont} \cup \A_d^{n,\rm adapt}$; see \autoref{sect:Algos}.
The assertion reads as follows.

\begin{theorem}\label{theo:opt_sym_algo}
				Assume $S_I=(S_{d,I_d})_{d\in\N}$ to be the linear tensor product problem $S$ restricted 
				to the $I_d$-(anti)symmetric subspaces $P_{I_d}(H_d)$ 
				of the $d$-fold tensor product spaces $H_d$.
				Then for every $d\in\N$ the set
				\begin{gather}\label{eq:eigenpairs_sym_ordered}
								\{ ( \lambda_{d,i}, \xi_{d,i}) \sep 1 \leq i < \# \nabla_d + 1 \} 
								= \left\{ \left( \widetilde{\lambda}_{d,\bm{k}}, \widetilde{\xi}_{\bm{k}} \right) \sep \bm{k} \in \nabla_d \right\}
				\end{gather} 
				denotes the eigenpairs of 
				$W_{d,I_d}={S_{d,I_d}}^{\!\!\dagger} {S_{d,I_d}} \colon P_{I_d}(H_d)\nach P_{I_d}(H_d)$. 
				Thus, for every $n \in \N_0$ and all $d\in\N$ the linear algorithm 
				$A_{n,d}'\colon P_{I_d}(H_d) \nach P_{I_d}(\G_d)$ given by
				\begin{gather}\label{opt_sym_algo}
							A_{n,d}' f 
							= \sum_{i=1}^{\min{n,\#\nabla_d}} \distr{f}{\xi_{d,i}}_{H_d} \cdot S_d \xi_{d,i},
				\end{gather}
				is $n$th optimal for $S_{d,I_d}$ \wrt the worst case setting. 
				Furthermore we have
				\begin{gather}\label{nth_error}
								e^\wor(n,d; P_{I_d}(H_d)) 
								= \Delta^{\mathrm{wor}}(A_{n,d}'; P_{I_d}(H_d)) 
								= \sqrt{\lambda_{d, n+1}}.
				\end{gather}
\end{theorem}

\begin{proof}
Since $S_I$ is a compact problem between Hilbert spaces it is enough to prove that
for $d\in\N$ the eigenpairs of $W_{d,I_d} = {S_{d,I_d}}^{\!\!\dagger} S_{d,I_d}$ are given by~\link{eq:eigenpairs_sym_ordered}.
The remaining assertions then follow from \autoref{Cor:OptAlgo}.
Indeed, we only need to show that $W_{d,I_d} \widetilde{\xi}_{\bm{k}} = \widetilde{\lambda}_{d,\bm{k}} \cdot \widetilde{\xi}_{\bm{k}}$ 
for every $\bm{k}\in\nabla_d$ because we already know that the set $\Xi_d=\{\widetilde{\xi}_{\bm{k}}\sep \bm{k}\in\nabla_d\}$ builds an ONB in $P_{I_d}(H_d)$.
Hence there cannot be more than these eigenpairs.

To prove the claim, observe that from the first part of \autoref{prop:bestapprox} it follows
\begin{gather*}
			S_{d,I_d} = S_d \circ P_{I_d}^{H} = P_{I_d}^{\G} \circ S_d 
			\quad \text{which implies that} \quad 
			S_{d,I_d} \colon P_{I_d}^{H}(H_d)\nach P_{I_d}^{\G}(\G_d).
\end{gather*}
Moreover, due to the self-adjointness of the projectors (see \autoref{projection}), 
it is easily seen that this yields
\begin{gather*}
			{S_{d,I_d}}^{\!\!\dagger} = P_{I_d}^{H} \circ {S_d}^{\!\dagger} = {S_d}^{\!\dagger} \circ P_{I_d}^{\G} 
			\quad \text{such that} \quad 
			{S_{d,I_d}}^{\!\!\dagger} \colon P_{I_d}^{\G}(\G_d)\nach P_{I_d}^{H}(H_d).
\end{gather*}
Consequently, we have
\begin{align*}
			W_{d,I_d} \, P_{I_d}^{H} 
			= \left( P_{I_d}^{H} \, {S_d}^{\!\dagger} \right) \left(S_d \, P_{I_d}^{H} \right) P_{I_d}^{H}
			= P_{I_d}^{H} \, \left( {S_d}^{\!\dagger} P_{I_d}^{\G} \right) \, S_d
			= P_{I_d}^{H} \, \left( {S_d}^{\!\dagger} S_d \right)
			= P_{I_d}^{H} \, W_d,
\end{align*}
because of $(P_{I_d}^X)^2 = P_{I_d}^X$, where $X\in\{H,\G\}$.
Since for every $\bm{j}\in\M_d$ the simple tensor $\widetilde{\phi}_{d,\bm{j}}$ is an eigenelement of $W_d$ with respect to the eigenvalue $\widetilde{\lambda}_{d,\bm{j}}$, we conclude
\begin{gather*}
			W_{d,I_d} \left( P_{I_d}^{H} \, \widetilde{\phi}_{d,\bm{j}} \right) 
			= \widetilde{\lambda}_{d,\bm{j}} \cdot \left( P_{I_d}^{H} \,  \widetilde{\phi}_{d,\bm{j}} \right)
\end{gather*}
from the linearity of $P_{I_d}^{H}$.
In particular, this is true for every $\bm{j}=\bm{k}\in\nabla_d\subseteq \M_d$.
But now we note that $\widetilde{\xi}_{\bm{k}}$ equals $P_{I_d}^{H} \, \widetilde{\phi}_{d,\bm{k}}$, 
at least up to some normalizing constant.
Hence, using linearity once again, we have proven the claimed assertion.
\end{proof}

We conclude this section by adding some final remarks on the above theorem.

\begin{rem}
Obviously, our former result for the entire tensor product problem $S=(S_d\colon H_d\nach\G_d)_{d\in\N}$ in \autoref{sect:Eigenpairs} is also covered by \autoref{theo:opt_sym_algo}.
We simply have to choose $I_d$ such that $\# I_d = 1$ for every $d\in\N$ and obtain $A_{n,d}' = A_{n,d}^*$.
As in this case, the worst case error can be attained by the element $\xi_{d,n+1}$ provided that $n<\#\nabla_d$.
Otherwise it trivially equals zero.

It should be clear to the reader how to generalize the results of this section to
the case of multiple partially (anti)symmetric problems where we claim (anti)symmetry \wrt 
more than one subset of coordinates $I$. 
Recall that this definition is given at the end of \autoref{sect:antisym_HFS}.

Finally we want to mention that we decided to give a different proof of \autoref{theo:opt_sym_algo}
than in \cite{W11} and \cite{W12b}, respectively. 
The reason is that the usage of the self-adjointness of the projections $P_{I_d}$ seems to be more elegant
than again repeating the arguments used for \autoref{Cor:OptAlgo} in \autoref{sect:opt_Hilbert_Algo}.
Furthermore, now we can handle also problems defined on non-separable or on finite-dimensional 
source spaces~$H_d$.
Thus we slightly generalized our old results.
\hfill$\square$
\end{rem}

\section{Complexity of (anti)symmetric problems}\label{sect:complexity_antisym}
Encouraged by the exact formula for the
$n$th minimal worst case error in \autoref{theo:opt_sym_algo}
the intention of the present section is to investigate the information complexity
of (anti)symmetric tensor product problems.
We restrict our attention to the study of polynomial and strong polynomial tractability in
what follows.
The aim is to find necessary and sufficient conditions for these properties in terms of the univariate sequence $\lambda=(\lambda_m)_{m\in\N}$
and the number of (anti)symmetry conditions we impose.
From the definition of $\nabla_d$ in \link{def_nabla} it is quite clear that 
antisymmetric problems are significantly easier than their symmetric counterparts.
Therefore, after proving some general assertions, 
we handle these cases separately in order to
conclude sharp conditions.
Moreover, we distinguish between the absolute and the normalized error criterion.

Let us fix the basic notation for this section.
As before, assume $S_{I}=(S_{d,I_d})_{d\in\N}$ to denote
a tensor product problem $S=(S_d\colon H_d\nach\G_d)_{d\in\N}$, 
restricted to some sequence of (anti)symmetric subspaces $P_{I_d}(H_d)$,
where $P\in\{\SI,\AI\}$,
of the tensor product Hilbert spaces $H_d=H_1\otimes\ldots\otimes H_d$, $d\in\N$.
Here for every $d\in\N$ the elements are (anti)symmetric \wrt
the non-empty subset $I_d \subseteq \{1,\ldots,d\}$ of coordinates.
The cardinality of these subsets will be denoted by $a_d=\#I_d$
and we set $b_d=d-a_d$ for the number of coordinates without (anti)symmetry conditions.
Finally, for $d\in\N$ the non-increasingly ordered eigenvalues $\lambda_{d,i}=\widetilde{\lambda}_{d,\psi(i)}$, $i\in\N$,
are given by \link{eq:eigenpairs_sym} and \link{eq:eigenpairs_sym_ordered}, respectively.
They are constructed out of the squared singular values $\lambda=(\lambda_m)_{m\in\N}$
of the underlying solution operator $S_1\colon H_1\nach\G_1$.

As an immediate consequence of \link{nth_error} we see that the
initial error of approximating $S_{d,I_d}$ on the unit ball $\widetilde{\F}_d=\B(P_{I_d}(H_d))$ is given by
\begin{gather*}
				\eps_d^{\mathrm{init}} 
				= e^\wor(0,d; P_I(H_d)) 
				= \sqrt{\lambda_{d,1}}  
				= \begin{cases}
						\sqrt{\lambda_1^d}, & \text{ if } P=\SI,\\
						\sqrt{\lambda_1^{b_d} \cdot \lambda_1 \cdot \ldots \cdot \lambda_{a_d}}, & \text{ if } P=\AI.
				\end{cases}
\end{gather*}
Clearly, we need to assume that this initial error is strictly positive
for any reasonably large $d\in\N$
because otherwise we have (strong) polynomial tractability by default.
In particular,
if the number of antisymmetric coordinates $a_d$
grows with the dimension then this condition implies that
the whole sequence of univariate eigenvalues $\lambda$
need to be strictly positive.
Moreover, similar to the entire tensor product problems studied in \autoref{sect:tensor_complexity},
we always assume that $\lambda_2>0$ in order to avoid triviality.
Consequently, we have $\#\M_1 \geq 2$.

Now we are ready to conclude a first general condition which is 
necessary for (strong) polynomial tractability of both symmetric and antisymmetric problems as long as we deal with the
absolute error criterion.
It is independent of the concrete choice of the (anti)symmetry conditions we impose.
\begin{lemma}[General necessary conditions, absolute errors]\label{prop_general}
					Let $P\in\{\SI,\AI\}$ and consider $S_I=(S_{d,I_d})_{d\in\N}$ as defined above,
					where $I_d$ is arbitrarily fixed for every $d\in\N$.
					Then the fact that $S_I$ is polynomially tractable
					with the constants $C,p>0$ and $q\geq 0$ implies that 
					$\lambda=(\lambda_m)_{m\in\N} \in \l_\tau$ for all $\tau > p/2$.
					Moreover, for	any such $\tau$ and all $d\in\N$ the following estimate holds:
					\begin{gather*}
								\frac{1}{(\lambda_{d,1})^\tau} \sum_{\bm{k}\in\nabla_d} \left(\widetilde{\lambda}_{d,\bm{k}}\right)^\tau 
								\leq (1+C)\, d^q
									+ C^{2\tau/p} \, \zeta\!\left(\frac{2\tau}{p}\right) \left( \frac{d^{2q/p}}{\lambda_{d,1}} \right)^\tau.
					\end{gather*}
\end{lemma}
\begin{proof}
From \autoref{Thm:General_Tract_abs} we know that for any $\tau > p/2$ and $r=2q/p$
polynomial tractability yields
\begin{gather}\label{sup_condition3}
				\sup_{d\in\N} \frac{1}{d^r} \left( \sum_{i=f(d)}^\infty \left(\lambda_{d, i}\right)^\tau \right)^{1/\tau} <\infty,
\end{gather}
where the function $f\colon \N \nach \N$ is given by $f(d)=\ceil{(1+C)\,d^q}$.
This particularly implies that 
the sum in the brackets converges for every fixed $d\in\N$.
Therefore, especially for $d=1$ the tail series 
$\sum_{i=f(1)}^\infty (\lambda_{1,i})^\tau = \sum_{m= \ceil{1+C}}^\infty (\lambda_{m})^\tau$
needs to be finite which is possible only if $\lambda=(\lambda_m)_{m\in\N} \in \l_\tau$.

So, let us turn to the second assertion.
Obviously \link{sup_condition3} implies the existence of some constant $C_1>0$
such that
\begin{gather*}
				\sum_{i=f(d)}^\infty \left(\lambda_{d,i}\right)^\tau 
				\leq C_1 d^{r\tau} \quad \text{for all} \quad d\in\N.
\end{gather*}
Indeed, \autoref{Thm:General_Tract_abs} yields that we can take 
$C_1 = C^{2\tau/p} \zeta(2\tau/p)$.
Due to the ordering of $(\lambda_{d,i})_{i\in\N}$
the rest of the sum can also be bounded easily for any $d\in\N$ by
\begin{gather*}
				\sum_{i=1}^{f(d)-1} \left(\lambda_{d,i}\right)^\tau 
				\leq \left(\lambda_{d,1}\right)^\tau \cdot (f(d)-1).
\end{gather*}
Since $\sum_{\bm{k}\in\nabla_d} (\widetilde{\lambda}_{d,\bm{k}})^\tau = \sum_{i=1}^\infty (\lambda_{d,i})^\tau$, 
it remains to show that $f(d)-1 \leq (1 + C) d^{q}$
for every $d\in\N$ which is also obvious due to the definition of $f$.
\end{proof}

\subsection{Symmetric problems (absolute errors)}
Apart from the general assertion $\lambda \in \l_\tau$, 
we focus our attention on further necessary conditions for (strong) polynomial tractability
in the symmetric setting.
The following proposition yields a slight improvement compared to the corresponding assertion
stated in \cite{W12b} which can be obtained without using essential new ideas.

\begin{prop}[Necessary conditions, symmetric case]
				Let $S_I=(S_{d,I_d})_{d\in\N}$ be the problem considered in \autoref{prop_general} and set $P=\SI$.
				Moreover, assume $\lambda_1 \geq 1$.
				\begin{itemize}
								\item If $S_I$ is polynomially tractable then $b_d \in \0(\ln d)$, as $d\nach\infty$.
								\item If $S_I$ is strongly polynomially tractable then $b_d \in \0(1)$, as $d\nach\infty$, and $\lambda_1=1>\lambda_2$.
				\end{itemize}
\end{prop}
\begin{proof}
Assume $\lambda_1\geq 1$ and let $\tau$ be given by \autoref{prop_general}.
Then, independent of the amount of symmetry conditions,
we have $\lambda_{d,1}=\lambda_1^d \geq 1$ and
there exist absolute constants $r\geq 0$ and $C>1$ such that
\begin{gather}\label{stp_estimate}
			\frac{1}{(\lambda_1)^{\tau d}} \sum_{\bm{k}\in\nabla_d} \left(\widetilde{\lambda}_{d,\bm{k}}\right)^\tau
			\leq C \, d^r, \quad d\in\N,
\end{gather}
due to \autoref{prop_general}.
In the case of strong polynomial tractability we even have $r=0$.
For $d\geq 2$ we use the product structure of 
$\widetilde{\lambda}_{d,\bm{k}}$, $\bm{k}\in\nabla_d$, 
provided by \link{eq:eigenpairs_sym}.
That is, we split the sum
\wrt the coordinates with and without symmetry conditions. 
Hence, we conclude
\begin{gather}\label{splitting}
			\sum_{\bm{k}=(\bm{h},\bm{j})\in\nabla_d} \widetilde{\lambda}_{d,\bm{k}}^\tau 
			= \sum_{\bm{j}\in (\M_1)^{b_d}} \widetilde{\lambda}_{b_d,\bm{j}}^\tau 
					\sum_{\substack{\bm{h}\in (\M_1)^{a_d},\\h_1\leq\ldots\leq h_{a_d}}} \widetilde{\lambda}_{a_d,\bm{h}}^\tau
			= \left( \sum_{m=1}^{\# \M_1} \lambda_m^\tau \right)^{b_d} 
					\sum_{\substack{\bm{h}\in\M_{a_d},\\h_1\leq\ldots\leq h_{a_d}}} \widetilde{\lambda}_{a_d,\bm{h}}^\tau
\end{gather}
for $d=a_d+b_d\geq 2$
which leads to
\begin{gather*}
			\left( \sum_{m=1}^{\# \M_1} \left(\frac{\lambda_m}{\lambda_1}\right)^\tau \right)^{b_d} 
					\sum_{\substack{\bm{h}\in\M_{a_d},\\h_1\leq\ldots\leq h_{a_d}}} \prod_{l=1}^{a_d} 
							\left( \frac{\lambda_{h_l}}{\lambda_1}\right)^\tau
			\leq C \, d^r.
\end{gather*}
In any case the second sum in the above inequality is bounded from below by $1$.
Thus, using $\#\M_1 \geq 2$ we conclude that 
$(1+\lambda_2^\tau/\lambda_1^\tau)^{b_d} \leq \left( \sum_{m=1}^{\#\M_1} \lambda_m^\tau/\lambda_1^\tau \right)^{b_d}$
needs to be polynomially bounded from above.
Since we always assume $\lambda_2>0$ this leads to the claimed bounds on $b_d$.

It remains to show the assertions on the two largest univariate eigenvalues
in the case of strong polynomial tractability.
To this end, assume for a moment that $\lambda_1>1$.
Then, because of $\lambda_2>0$, there need to exist some $K\in\N_0$ such that
$\lambda_2 \geq (1/\lambda_1)^K$.
Now it is easy to see that 
(independent of the number of symmetry conditions) 
there are at least $1+\floor{d/(K+1)}$ different $\bm{k}\in\nabla_d$ 
such that $\widetilde{\lambda}_{d,\bm{k}}\geq 1$. 
Namely, for $l=0,\ldots,\floor{d/(K+1)}$ we can take the first $d-l$ coordinates of $\bm{k}\in\nabla_d$
equal to one. To the remaining coordinates we assign the value two
and obtain
\begin{gather*}
		\widetilde{\lambda}_{d,\bm{k}}
		= \lambda_1^{d-l}\lambda_2^{l}
		\geq \lambda_1^{Kl}\lambda_2^{l} \geq 1.
\end{gather*}
In other words, we have $\lambda_{d,1+\floor{d/(K+1)}} \geq 1$.
On the other hand, strong polynomial tractability implies 
$\sum_{i=\ceil{1+C}}^\infty \lambda_{d,i}^\tau \leq C_1$ 
for some absolute constants $\tau,C,C_1>0$ and all $d\in\N$; 
see~\link{sup_condition3}.
Consequently, for every $d\geq d_0= (2+C)(K+1)$ we obtain
$1+\floor{d/(K+1)}\geq \ceil{1+C}$ and thus
\begin{align*}
			C_1 
			&\geq \sum_{i=\ceil{1+C}}^\infty \lambda_{d,i}^\tau 
			\geq \sum_{i=\ceil{1+C}}^{1+\floor{d/(K+1)}} \lambda_{d,i}^\tau \\
			&\geq \lambda_{d,1+\floor{d/(K+1)}}^\tau (2+\floor{d/(K+1)}-\ceil{1+C}) \\
			&\geq \frac{d}{K+1}-(1+C).
\end{align*}
Obviously this is a contradiction and we conclude $\lambda_1=1$.
Finally, we need to show that we necessarily have $\lambda_2<1$.
Assuming that $\lambda_1=\lambda_2=1$ leads to $K=0$ in the discussion above 
and hence we obtain the same contradiction as before.
Therefore the proof is complete.
\end{proof}

Note in passing that
independent of the number of symmetry conditions 
the information complexity $n^{\mathrm{wor}}_{\mathrm{abs}}(\epsilon,d;S_{d,I_d}\colon\B(\SI_{I_d}(H_d))\nach\G_d)$ 
needs to grow at least linearly in $d$
if we assume $\lambda_1 \geq 1$ and $\lambda_2>0$.

We continue the analysis of $I$-symmetric problems with respect to the absolute error criterion
by proving that the stated necessary conditions are also sufficient for (strong) polynomial tractability.
For this purpose we need a rather technical preliminary lemma.
For the convenience of the reader we include a full proof 
that uses only elementary induction arguments. 

\begin{lemma}\label{lemma_symNEW}
				Let $(\mu_m)_{m\in\N}$ be a non-increasing sequence of 
				non-negative real numbers with $\mu_1>0$
				and set $\mu_{s,\bm{k}}=\prod_{l=1}^s \mu_{k_l}$ for $\bm{k}\in\N^s$ and $s\in\N$.\\
				Then, for all $V\in\N_0$ and every $d\in\N$, it holds
				\begin{gather}\label{estimate_VNEW}
								\sum_{\substack{\bm{k}\in\N^d,\\1\leq k_1\leq\ldots\leq k_d}} \!\!\mu_{d,\bm{k}}
								\leq \left(\mu_1\right)^d \, d^V \left( 1 + V + \sum_{L=1}^d \left(\mu_1\right)^{-L} \!\!
										 \sum_{ \substack{ \bm{j^{(L)}}\in\N^L,\\V+2\leq j_1^{(L)}\leq\ldots\leq j_L^{(L)} } } \!\! \mu_{L,\bm{j^{(L)}}} \right).
				\end{gather}
\end{lemma}

\begin{proof}
\emph{Step 1}. By induction on $s$ we first prove that for every fixed $m\in\N$
\begin{gather}\label{case_mNEW}
				\sum_{\substack{\bm{k}\in\N^s,\\m\leq k_1\leq\ldots\leq k_s}} \!\!\mu_{s,\bm{k}}
				= (\mu_m)^s + \sum_{l=1}^s (\mu_m)^{s-l} \!\!\sum_{ \substack{ \bm{j^{(l)}}\in\N^l,\\m+1\leq j_1^{(l)}\leq\ldots\leq j_l^{(l)} } } \!\! \mu_{l,\bm{j^{(l)}}}
				\quad \text{for all} \quad s\in\N.
\end{gather}
Easy calculations show that this holds at least for the initial step $s=1$.
Therefore, assume the assertion~\link{case_mNEW} to be true for some $s\in\N$.
Then
\begin{align*}
				\sum_{\substack{\bm{k}\in\N^{s+1},\\m\leq k_1\leq\ldots\leq k_{s+1}}} \!\!\mu_{s+1,\bm{k}}
				&= \sum_{k_1=m}^\infty \mu_{k_1} \!\!\sum_{\substack{\bm{h}\in\N^{s},\\k_1\leq h_1\leq\ldots\leq h_{s}}} \!\!\mu_{s,\bm{h}}	\\
				&= \mu_m \!\!\sum_{\substack{\bm{h}\in\N^{s},\\m\leq h_1\leq\ldots\leq h_{s}}} \!\! \mu_{s,\bm{h}} + \!\!\sum_{\substack{\bm{k}\in\N^{s+1},\\m+1\leq k_1\leq\ldots\leq k_{s+1}}} \!\!\mu_{s+1,\bm{k}}.
\end{align*}
Now, by inserting the induction hypothesis for the first sum and 
renaming $\bm{k}$ to~$\bm{j^{(s+1)}}$ in the remaining sum, we conclude that 
$\sum_{\substack{\bm{k}\in\N^{s+1},\\m\leq k_1\leq\ldots\leq k_{s+1}}} \mu_{s+1,\bm{k}}$ 
equals
\begin{gather*}
				(\mu_m)^{s+1} 
				+ \sum_{l=1}^s (\mu_m)^{s+1-l} \!\!\sum_{ \substack{ \bm{j^{(l)}}\in\N^l,\\m+1\leq j_1^{(l)}\leq\ldots\leq j_l^{(l)} } } \!\!\mu_{l,\bm{j^{(l)}}} 
				+ \!\!\sum_{\substack{\bm{j^{(s+1)}}\in\N^{s+1},\\m+1\leq j^{(s+1)}_1\leq\ldots\leq j^{(s+1)}_{s+1}}} \!\!\mu_{s+1,\bm{j^{(s+1)}}}.
\end{gather*}
Hence \link{case_mNEW} also holds for $s+1$ and the induction is complete.

\emph{Step 2}. Here we prove \link{estimate_VNEW} via another induction on $V\in\N_0$.
Therefore, let $d\in\N$ be arbitrarily fixed.
The initial step, $V=0$, corresponds to \link{case_mNEW} for $s=d$ and $m=1$.
Thus assume \link{estimate_VNEW} to be valid for some fixed $V\in\N_0$.
Then, by using \link{case_mNEW} for $s=L$ and $m=V+2$, we see that the right-hand side of \link{estimate_VNEW} equals 
\begin{align*}
					(\mu_1)^d \, d^V \left( 1 + V + \sum_{L=1}^d (\mu_1)^{-L} \left( (\mu_{V+2})^L + \sum_{l=1}^L (\mu_{V+2})^{L-l} \!\! \sum_{ \substack{ \bm{j^{(l)}}\in\N^l,\\(V+2)+1\leq j_1^{(l)}\leq\ldots\leq j_l^{(l)} } } \!\!\mu_{l,\bm{j^{(l)}}} \right) \right).
\end{align*}
Now we estimate $1+V$ by $d\,(1+V)$, take advantage of the non-increasing ordering of $(\mu_m)_{m\in\N}$, and extend the inner sum from $L$ to $d$
in order to obtain
\begin{align*}
					\sum_{\substack{\bm{k}\in\N^d,\\1\leq k_1\leq\ldots\leq k_d}} \!\!\mu_{d,\bm{k}}
					\leq (\mu_1)^{d} \, d^{V+1} \left( 1 + (V + 1) + \sum_{l=1}^d (\mu_1)^{-l} \!\!\sum_{ \substack{ \bm{j^{(l)}}\in\N^l,\\(V+1)+2\leq j_1^{(l)}\leq\ldots\leq j_l^{(l)} } } \!\! \mu_{l,\bm{j^{(l)}}} \right).
\end{align*}
Since this estimate corresponds to \link{estimate_VNEW} for $V+1$ the claim is proven.
\end{proof}

Now the sufficient conditions read as follows.

\begin{prop}[Sufficient conditions, symmetric case]\label{prop_suf_sym}
				Let $P=\SI$, assume $S_I$ to be the problem considered in \autoref{prop_general}, and
				let $\lambda=(\lambda_m)_{m\in\N}\in \l_{\tau_0}$ for some $\tau_0\in(0,\infty)$.
				\begin{itemize}
								\item If $\lambda_1<1$ then $S_I$ is strongly polynomially tractable.
								\item If $\lambda_1=1>\lambda_2$ and $b_d \in \0(1)$ then $S_I$ is strongly polynomially tractable.
								\item If $\lambda_1=1$ and $b_d \in \0(\ln d)$, as $d\nach\infty$, then $S_I$ is polynomially tractable.
				\end{itemize}
\end{prop}

\begin{proof}
\emph{Step 1}. 
We start the proof by exploiting the property $\lambda\in\l_{\tau_0}$;
namely we use the ordering of $(\lambda_m)_{m\in\N}$ to conclude that
\begin{gather*}
				m \, \lambda_m^{\tau_0} 
				\leq \lambda_1^{\tau_0} + \ldots + \lambda_m^{\tau_0} 
				< \sum_{i=1}^\infty \lambda_i^{\tau_0} = \norm{\lambda \sep \l_{\tau_0}}^{\tau_0} < \infty
				\quad \text{for any} \quad m\in\N.
\end{gather*}
Hence, there exists some $C_{\tau_0}>0$ such that 
$\lambda_m$ is bounded from above by 
$C_{\tau_0} \cdot m^{-r}$ for every $r\leq 1/\tau_{0}$.
Therefore there is some index such that for every larger $m\in\N$ we have
$\lambda_m<1$. We denote the smallest of these indices by $m_0$.
Similar to the calculations of Novak and Wo{\'z}niakowski~\cite[p. 180]{NW08} this leads to
\begin{gather*}
				\sum_{m=m_0}^\infty \!\!\lambda_m^\tau 
				\leq (p+1) \lambda_{m_0}^\tau + C_{\tau_0}^\tau \int_{m_0+p}^\infty \!\!x^{-\tau r} \dlambda^1(x) 
				= (p+1) \lambda_{m_0}^\tau + \frac{C_{\tau_0}^\tau}{\tau r - 1} \frac{1}{(m_0+p)^{\tau r -1}}
\end{gather*}
for every $p\in\N_0$ and all $\tau$ such that $\tau r > 1$.
In particular, with $r=1/\tau_0$ we obtain for all $\tau>\tau_0$ and any $p\in\N_0$ the estimate
\begin{gather*}
				\sum_{m=m_0}^\infty (\lambda_m)^\tau 
				\leq (p+1) \,(\lambda_{m_0})^\tau + \frac{1/\tau}{1/\tau_0 - 1/\tau} \left( \frac{C_{\tau_0}^{1/(1/\tau_0 - 1/\tau)}}{m_0+p} \right)^{\tau (1/\tau_0-1/\tau)}.
\end{gather*}
Note that for a given $\delta>0$ there exists some constant $\tau_1\geq \tau_0$ 
such that for all $\tau > \tau_1$ it is $1/(1/\tau_0 - 1/\tau) \in (\tau_0,\tau_0+\delta)$.
Hence, if $p\in\N_0$ is sufficiently large then we conclude that for all $\tau > \tau_1$
\begin{align*}
				\sum_{m=m_0}^\infty (\lambda_m)^\tau 
				\leq (p+1) \,(\lambda_{m_0})^\tau + \frac{\tau_0+\delta}{\tau_1} \left( \frac{C_1}{m_0+p} \right)^{\tau/(\tau_0+\delta)},
\end{align*}
where we set $C_1 = \max{1,C_{\tau_0}^{\tau_0+\delta}}<m_0+p$.
Finally, since $\lambda_{m_0}<1$, both the summands tend to zero as $\tau$ approaches infinity.
In particular, there need to exist some $\tau > \tau_1\geq \tau_0$ such that
\begin{gather*}
				\sum_{m=m_0}^\infty (\lambda_m)^\tau \leq \frac{1}{2}.
\end{gather*}

\emph{Step 2}. 
Now all the stated assertions can be seen using the second point of 
\autoref{Thm:General_Tract_abs}.
Indeed, for polynomial tractability it is sufficient to show that
\begin{gather}\label{sum_pol}
				\sum_{\bm{k}\in\nabla_d} \left( \widetilde{\lambda}_{d,\bm{k}} \right)^\tau 
				= \sum_{i=1}^\infty \left( \lambda_{d,i} \right)^\tau
				\leq C \,d^{r \tau}\quad \text{for all} \quad d\in\N
\end{gather}
and some $C,\tau > 0$, as well as some $r\geq 0$.
If this even holds for $r=0$ we obtain strong polynomial tractability.

In the case $\lambda_1<1$ we can estimate the sum on the left of \link{sum_pol}
from above by $( \sum_{m=1}^\infty \lambda_m^\tau )^d$ 
since clearly $\nabla_d\subseteq \M_d\subseteq\N^d$.
Using Step 1 with $m_0=1$ we conclude that 
$\sum_{\bm{k}\in\nabla_d} (\widetilde{\lambda}_{d,\bm{k}})^\tau \leq 2^{-d}$
for some large $\tau > \tau_0$. 
Hence the problem is strongly polynomially tractable in this case.

For the proof of the remaining points we assume that $\lambda_1=1$.
In any case we have 
\begin{gather*}
			\sum_{k\in\nabla_1} \left( \widetilde{\lambda}_{1,k} \right)^\tau 
			\leq \sum_{m=1}^\infty (\lambda_m)^{\tau_0} 
			= \norm{\lambda \sep \l_{\tau_0}}^{\tau_0} < \infty
\end{gather*}			
for all $\tau\geq \tau_0$ because of $\lambda \in \l_{\tau_0}$. 
Therefore we can assume $d\geq 2$ in the following.
Recall that we can split the first sum in \link{sum_pol} 
\wrt the coordinates with and without symmetry conditions.
That is, for $d=a_d+b_d\geq 2$ we use \link{splitting}.

If $\lambda_2<1$ and $b_d$ is universally bounded then
the first factor in this splitting can be bounded by a constant and the second factor 
can be estimated using \autoref{lemma_symNEW} with $V=0$, $d$ replaced by $a_d$ 
and $\mu$ replaced by $\lambda^\tau$.\footnote{Observe that this choice particularly implies that $\mu_m=0$ for any $m>\#\M_1$.}
Consequently, for any $\tau\geq \tau_0$,
\begin{gather}\label{symfactor}
				\!\sum_{\substack{\bm{h}\in \M_{a_d},\\h_1\leq\ldots\leq h_{a_d}}} \!\!\! \left(\widetilde{\lambda}_{a_d, \bm{h}}\right)^\tau 
				\leq 1 + \sum_{L=1}^{a_d} \!\sum_{ \substack{ \bm{j^{(L)}}\in\N^L,\\2\leq j_1^{(L)}\leq\ldots\leq j_L^{(L)} } } \!\!\! \left(\widetilde{\lambda}_{L,\bm{j^{(L)}}}\right)^\tau
				\leq 1 + \sum_{L=1}^{a_d} \left( \sum_{m=2}^\infty (\lambda_{m}^\tau) \right)^L.
\end{gather}
Now, with the help of Step 1 and the properties of geometric series, 
we see that if~$\tau$ is large enough then \link{symfactor}
can be estimated further by $1 + \sum_{L=1}^\infty 2^{-L} = 2$.
In summary also $\sum_{\bm{k}\in\nabla_d} (\widetilde{\lambda}_{d,\bm{k}})^\tau$ 
is universally bounded in this case 
and therefore the problem~$S_I$ is strongly polynomially tractable.

To prove the last point we argue in the same manner.
Here the assumption $b_d \in \0(\ln d)$, as $d\nach\infty$, yields that the first factor in the splitting \link{splitting}
is polynomially bounded in $d$.
For the second factor we again apply \autoref{lemma_symNEW}, but in this case we set
$V=m_0-2$, where $m_0$ denotes the first index $m\in\N$ such that $\lambda_{m}<1$.
Keep in mind that this index is at least two because of $\lambda_1=1$. 
On the other hand, it needs to be finite,
since $\lambda \in \l_{\tau_0}$.
Therefore, due to the same arguments as above, the second factor in the splitting \link{splitting} is polynomially bounded in $d$, too.
All in all, this proves \link{sum_pol} and thus $S_I$ is polynomially tractable in this case.
\end{proof}

We summarize the results obtained for $I$-symmetric tensor product problems $S_I=(S_{d,I_d})_{d\in\N}$ in the following
theorem.

\begin{theorem}[Polynomial tractability of sym. problems, absolute errors]\label{thm:tract_sym_abs}
			Let\\ $S_1 \colon H_1 \nach \G_1$ denote a compact linear operator 
			between two Hilbert spaces and let $\lambda=(\lambda_m)_{m\in \N}$ be the sequence 
			of eigenvalues of $W_1={S_1}^{\!\dagger} S_1$ \wrt a non-increasing ordering. 
			Moreover, for $d>1$ let $\leer \neq I_d \subseteq\{1,\ldots,d\}$ and
			assume $S_I=(S_{d,I_d})_{d\in\N}$ to be the linear tensor product problem 
			$S=(S_d)_{d\in\N}$ restricted to the 
			$I_d$-symmetric subspaces $\SI_{I_d}(H_d)$ 
			of the $d$-fold tensor product spaces $H_d$. 
			Consider the worst case setting with respect to the absolute error criterion
			and let $\lambda_2>0$.
			Then $S_I$ is strongly polynomially tractable 
			if and only if $\lambda \in \l_\tau$ for some $\tau\in(0,\infty)$ and
			\begin{itemize}
						\item $\lambda_1<1$, or
						\item $1=\lambda_1>\lambda_2$ and $(d-\#I_d) \in \0(1)$, as $d\nach\infty$.
			\end{itemize}
			Moreover, provided that $\lambda_1 \leq 1$ the problem is polynomially tractable 
			if and only if $\lambda \in \l_\tau$ for some $\tau\in(0,\infty)$ and
			\begin{itemize}
						\item $\lambda_1<1$, or
						\item $\lambda_1=1$ and $(d-\#I_d) \in \0(\ln d)$, as $d\nach\infty$.
			\end{itemize}
\end{theorem}

Note that we do not have sufficient conditions for polynomial
tractability in the case when $\lambda_1>1$.
We only know that $(d-\#I_d) \in \0(\ln d)$, as $d\nach\infty$, is necessary in this situation.
Anyway, we completely characterized strong polynomial tractability of symmetric problems.
In this respect we improved the results known from \cite{W12b}.
Moreover, we have shown that the stated results also hold for finite-dimensional
and for non-separable source spaces $H_1$.

Before we turn to the complexity of antisymmetric problems 
we briefly focus on the normalized error criterion 
for the $I$-symmetric setting in the next subsection.

\subsection{Symmetric problems (normalized errors)}
Due to \link{n_wor_norm} and \link{eq:eigenpairs_sym} the information complexity
of $I$-symmetric problems $S_I=(S_{d,I_d})_{d\in\N}$ in the worst case setting
\wrt the normalized error criterion is given by
\begin{gather*}
			n_\no^\wor(\epsilon', d; \SI_{I_d}(H_d)) 
			= \# \left\{ \bm{k}\in\nabla_d \sep \frac{\widetilde{\lambda}_{d,\bm{k}}}{\lambda_{d,1}} = \prod_{l=1}^d \left( \frac{\lambda_{k_l}}{\lambda_1} \right) > (\epsilon')^2 \right\}
\end{gather*}
for $\epsilon' \in (0,1)$ and $d\in\N$,
since we have $(\epsilon_d^{\rm init})^2=\lambda_{d,1}=\lambda_1^d$ 
for any kind of symmetric problem. 
In contrast, for the absolute error criterion~\link{n_wor} yields that
$n_\ab^\wor(\epsilon, d; \SI_{I_d}(H_d)) = \# \left\{ \bm{k}\in\nabla_d \sep \widetilde{\lambda}_{d,\bm{k}}=\prod_{l=1}^d \lambda_{k_l} > \epsilon^2 \right\}$,
where $\epsilon>0$ and $d\in\N$.
Hence, using the ideas in the proof of \autoref{thm:unweightedtensor_norm} 
it suffices to study a scaled tensor product problem $T_d\colon \SI_{I_d}(H_d)\nach \G_d$
\wrt the absolute error criterion in order to obtain tractability results
for $S_I$ in the normalized situation.
To this end, recall that the squared singular values of $T_1$ equal $\mu = (\mu_m)_{m\in\N}$ with
$\mu_m = \lambda_m / \lambda_1$ such that we always have $\mu_1=1$. 
Furthermore, we obviously have $\mu \in \l_\tau$ if and only if $\lambda \in \l_\tau$.
This leads to the following theorem.

\begin{theorem}[Polynomial tractability of symmetric problems, normalized errors]
			Consider the situation of \autoref{thm:tract_sym_abs}.
			We study the the worst case setting with respect to the normalized error criterion.
			Then $S_I=(S_{d,I_d})_{d\in\N}$ is strongly polynomially tractable if and only if 
			\begin{gather*}
						\lambda \in \l_\tau \text{ for some } \tau\in(0,\infty) 
						\quad \text{and} \quad 
						\lambda_1>\lambda_2
						\quad \text{and} \quad 
						(d-\#I_d) \in \0(1), \text{ as } d\nach\infty.
			\end{gather*}
			Moreover, the problem $S_I$ is polynomially tractable if and only if 
			\begin{gather*}
						\lambda \in \l_\tau \text{ for some } \tau\in(0,\infty) 
						\quad \text{and} \quad 
						(d-\#I_d) \in \0(\ln d), \text{ as } d\nach\infty.
			\end{gather*}
\end{theorem}

\subsection{Antisymmetric problems (absolute errors)}\label{sect:antisym_abs}
We start this subsection with sufficient conditions
for (strong) polynomial tractability which slightly improve the results stated in~\cite[Proposition~5]{W12b}.

\begin{prop}[Sufficient conditions, antisymmetric case]\label{prop_suf_asy}
				Let $P=\AI$, sup\-pose $S_I=(S_{d,I_d})_{d\in\N}$ to be the problem considered in \autoref{prop_general}, 
				and let $\lambda=(\lambda_m)_{m\in\N}\in \l_{\tau_0}$ for some $\tau_0\in(0,\infty)$.
				\begin{itemize}
								\item If $\lambda_1<1$ then $S_I$ is strongly polynomially tractable,
											independent of the number of antisymmetry conditions.
								\item If $\lambda_1 \geq 1$	and if there exist constants 
											$\tau \geq \tau_0$, $d_0\in\N$, as well as $C\geq1$, and $q\geq0$ such that 
											for the number of antisymmetric coordinates	$a_d$ in dimension~$d$ 
											it holds that
											\begin{gather}\label{suf_condition}
														\frac{\ln{(a_d!)}}{d} + \frac{\ln{(C\,d^q)}}{d} \geq \ln(\norm{\lambda \sep \l_\tau}^\tau) \quad \text{for all} \quad d\geq d_0
											\end{gather}
											then the problem $S_I$ is polynomially tractable.
											If this even holds for $q=0$ then we obtain strong polynomial tractability.
				\end{itemize}
\end{prop}

\begin{proof}
Like for the symmetric setting, the proof of these 
sufficient conditions is based on
the second point of \autoref{Thm:General_Tract_abs}.
We show that under the given assumptions for some $\tau \geq \tau_0$ the whole sum of the eigenvalues
\begin{gather}\label{sum_abs}
				\sum_{\bm{k}\in\nabla_d} \left(\widetilde{\lambda}_{d,\bm{k}}\right)^\tau
				= \sum_{i=1}^\infty (\lambda_{d,i})^\tau
\end{gather}
is universally bounded, or polynomially bounded in $d$, respectively.
Note that since we deal with the case $P=\AI$ now, 
the set $\nabla_d$ is given by the second line in \link{def_nabla}.
Moreover observe that for $d=1$ there is no antisymmetry condition at all.
That is, we have $\nabla_1 = \M_1\subseteq\N$ and the sums in~\link{sum_abs}
equal $\norm{ \lambda \sep \l_\tau}^\tau \leq \norm{ \lambda \sep \l_{\tau_0}}^\tau$
in this case.
Therefore, due to the hypothesis $\lambda \in \l_{\tau_{0}}$, the term for $d=1$ is finite.

Hence, let $d\geq 2$ be arbitrarily fixed.
Without loss of generality we may reorder the set of coordinates such that 
$I_d = \{i_1,\ldots,i_{a_d}\}=\{1,\ldots,a_d\}$.
That means, we assume partial antisymmetry with respect to the first $a_d$ coordinates. 
For $s\in\N$ with $s\geq d$ let us define cubes of 
multi-indices 
\begin{gather*}
				Q_{d,s} = \{1,\ldots,s\}^d.
\end{gather*}
Furthermore, let $U_{a_d,s}= \{ \bm{j} \in Q_{a_d,s} \sep j_1<j_2<\ldots<j_{a_d} \}$
denote the $a_d$-dimensional projection of $Q_{d,s}$ which reflects the assumed antisymmetry conditions.
With this notation we obtain
\begin{gather*}
				\sum_{\bm{k} \in \nabla_d} \left(\widetilde{\lambda}_{d,\bm{k}}\right)^\tau 
				= \lim_{s \nach \infty} \sum_{\bm{k} \in \nabla_d \cap Q_{d,s}} \left(\widetilde{\lambda}_{d,\bm{k}}\right)^\tau,
\end{gather*}
where the set of multi-indices under consideration 
$\nabla_d \cap Q_{d,s}$ can be represented as a subset of $U_{a_d,s} \times Q_{b_d,s}$.
We will assume $b_d=d-a_d>0$ in what follows to ensure this splitting to be non-trivial.
Because of the product structure of $\widetilde{\lambda}_{d,\bm{k}}$, $\bm{k}\in\nabla_d$, this implies
\begin{gather}\label{eq:cube_est}
				\sum_{\bm{k}=(\bm{j},\bm{i}) \in \nabla_d \cap Q_{d,s}} \left(\widetilde{\lambda}_{d,\bm{k}}\right)^\tau 
				\leq \left(\sum_{\bm{j} \in U_{{a_d},s}} \prod_{l=1}^{a_d} \lambda_{j_l}^\tau \right) \left( \sum_{\bm{i} \in Q_{b_d,s}} \prod_{l=1}^{b_d} \lambda_{i_l}^\tau\right).
\end{gather}
Since the sequence $\lambda=(\lambda_m)_{m\in\N}$ is an element of $\l_{\tau_0}\hookrightarrow \l_\tau$ 
we can easily estimate the second factor for every $s\geq d$ from above by
\begin{gather}\label{est}
				\sum_{\bm{i} \in Q_{b_d,s}} \prod_{l=1}^{b_d} \lambda_{i_l}^\tau 
				= \prod_{l=1}^{b_d} \sum_{m=1}^s \lambda_{m}^\tau 
				= \left( \sum_{m=1}^s \lambda_{m}^\tau \right)^{b_d} 
				\leq \left( \sum_{m=1}^\infty \lambda_{m}^\tau \right)^{1/\tau \cdot b_d \cdot \tau } 
				= \norm{\lambda \sep \l_\tau}^{b_d \cdot \tau}.
\end{gather}
To handle the first term we need an additional argument.
Note that due to the structure of $U_{{a_d},s}$ we have
\begin{gather*}
				\sum_{\bm{j} \in Q_{{a_d},s}} \prod_{l=1}^{a_d} \lambda_{j_l}^\tau 
				= \sum_{\substack{\bm{j} \in Q_{{a_d},s}\\\exists k,m: \, j_k=j_m}} \prod_{l=1}^{a_d} \lambda_{j_l}^\tau 
						+ a_d! \sum_{\bm{j} \in U_{{a_d},s}} \prod_{l=1}^{a_d} \lambda_{j_l}^\tau
				\geq a_d! \sum_{\bm{j} \in U_{{a_d},s}} \prod_{l=1}^{a_d} \lambda_{j_l}^\tau.
\end{gather*}
Consequently, using the same arguments as in \link{est}, this
yields the upper bound $\norm{\lambda \sep \l_\tau}^{a_d \cdot \tau} / (a_{d}!)$ 
for the first factor in \link{eq:cube_est}.
Once again this bound does not depend on $s\geq d$.
Hence, due to $d=a_d+b_d$, we conclude that
\begin{gather*}
				\sum_{\bm{k} \in \nabla_d} \left(\widetilde{\lambda}_{d,\bm{k}}\right)^\tau 
				= \lim_{s \nach \infty} \sum_{\bm{k} \in \nabla_d \cap Q_{d,s}} \left(\widetilde{\lambda}_{d,\bm{k}}\right)^\tau
				\leq \frac{1}{a_d!} \norm{\lambda \sep \l_\tau}^{\tau d}
				\quad \text{for every} \quad d\in\N
\end{gather*}
and any choice of $\AI_{I_d}$.
Of course, for every $d < d_0$ this upper bound is trivially 
less than an absolute constant.
Thus, to prove the second assertion of this \autoref{prop_suf_asy} it is enough to show that
\begin{gather*}
			\frac{1}{a_d!} \norm{\lambda \sep \l_\tau}^{\tau d} 
			\leq C\, d^q
			\quad \text{for all} \quad d\geq d_0,
\end{gather*} 
as well as for some $C\geq1$ and some $q\geq 0$.
But this is equivalent to our hypothesis stated in \link{suf_condition}.
Hence the condition \link{suf_condition} implies (strong) polynomial tractability of $S_I$,
independently of the value of $\lambda_1$.

Note that now it suffices to show that $\lambda_1<1$ already yields \link{suf_condition}
with $q=0$ and $C=1$ in order to complete the proof.
To see this, observe that (due to Step 1 in the proof of \autoref{prop_suf_sym}) 
we know that there 
exists some $\tau > \tau_0$ such that
$\norm{\lambda \sep \l_\tau}^\tau = \sum_{m=1}^\infty \lambda_m^\tau$ is strictly less than $1$.
Thus the right-hand side of \link{suf_condition} is negative in this case,
whereas the left-hand side is non-negative for every choice of~$a_d$.
\end{proof}

Let us briefly comment the latter result.
Clearly, for any $q\geq0$ the term $\ln (C\, d^q)/d$ in~\link{suf_condition} tends to zero as $d$ approaches infinity.
Hence there is not much difference in the stated sufficient condition for strong polynomial and for polynomial tractability.
Moreover, we need to mention that \autoref{Thm:General_Tract_abs}
allows us to omit the largest $f(d)-1$ eigenvalues
$\lambda_{d,i}$, where $f(d)$ may grow polynomially in $(\epsilon_d^{\rm init})^{-1}$
with $d$, but we did not use this fact in the above proof.

The next example investigates how fast $a_d$ needs to grow with the dimension~$d$ in order to fulfill the condition~\link{suf_condition}.

\begin{example}\label{ex:gamma}
For any $d\in\N$ and some $\gamma > 0$ let
\begin{gather}\label{eq:def_ad}
			a_d = \ceil{\frac{d}{\ln d^\gamma}}.
\end{gather}
Then Stirling's formula provides that $a_d \ln(a_d/e) \leq \ln(a_d!) < \ln(a_d)-a_d + a_d \ln(a_d)$ if $d$ (and hence also $a_d$) is sufficiently large.
Consequently,
\begin{gather*}	
				\frac{\ln(a_d!)}{d}
				\geq \frac{a_d \, \ln(a_d/e)}{d} 
				\geq \frac{1}{\gamma} \cdot \frac{\ln \left( d \cdot \frac{1}{e\gamma\ln d} \right)}{\ln d} 
				= \frac{1}{\gamma} \left( 1 -  \frac{\ln \left( e\gamma\ln d \right)}{\ln d} \right)
				\nearrow
				\frac{1}{\gamma},
\end{gather*}
as $d\nach\infty$. 
On the other hand, we have $a_d/d \leq 1/(\gamma \ln d)+1/d$ and thus
\begin{gather*}
				\frac{\ln(a_d!)}{d}
				< \frac{\ln(a_d)-a_d}{d} 
						+ \left( \frac{1}{\gamma \ln d} + \frac{1}{d} \right) \ln(a_d)
				= \frac{2 \ln (a_d) - a_d}{d} 
						+ \frac{1}{\gamma} \cdot \frac{\ln(a_d)}{\ln d} \leq \frac{1}{\gamma}.
\end{gather*}
So we see that $\gamma$ in \link{eq:def_ad} needs to be strictly 
smaller than $\ln^{-1}(\norm{\lambda\sep\l_\tau}^\tau)$ in order to fulfill \link{suf_condition}
with $q=0$.
In particular, it follows that assumptions like $a_d=\ceil{d^\beta}$
with $\beta<1$ are not sufficient to conclude tractability using 
the second point of \autoref{prop_suf_asy}.\hfill$\square$
\end{example}

Now we turn to necessary conditions.
As in the symmetric setting \autoref{prop_general} yields that $\lambda\in\l_\tau$ is needed
for polynomial tractability.
In addition, we will see that we need a condition similar to~\link{suf_condition},
particularly if we deal with slowly decreasing eigenvalues~$\lambda$.

\begin{prop}[Necessary conditions, antisymmetric case]\label{prop_nec_asy}
			Let $P = \AI$ and assume $S_I=(S_{d,I_d})_{d\in\N}$ to denote the problem considered
			in \autoref{prop_general}.
			Furthermore, let $S_I$ be polynomially tractable
			with the constants $C,p>0$ and $q\geq 0$.
			Then, for $d$ tending to infinity, the initial error 
			$\epsilon_d^{\mathrm{init}}$ tends to zero faster than the inverse 
			of any polynomial.
			Moreover, $\lambda= (\lambda_m)_{m\in\N} \in \l_\tau$ for every $\tau > p/2$ 
			and there exists some $d^* \in \N$, as well as $C_2\geq1$, such that
			\begin{gather}\label{bound_a}
						\frac{1}{d} \sum_{m=1}^{a_d} \ln \!\left( \frac{\norm{\lambda \sep \l_\tau}^\tau}{\lambda_m^\tau} \right)
								+ \frac{\ln(C_2 \, d^{2q\tau/p})}{d}
						\geq
						\ln\left(\norm{\lambda \sep \l_\tau}^\tau\right)
						\quad \text{for all} \quad d\geq d^*.
			\end{gather}
			Thus we either have $\lambda_1<1$, or $\lim_{d\nach \infty} a_d = \infty$.
\end{prop}

\begin{proof}
\emph{Step 1}. 
For the whole proof assume $\tau > p/2$ to be fixed. 
Then \autoref{prop_general} shows that $\lambda \in \l_\tau$.
Like in \link{splitting} for the symmetric case, we can split the sum of the eigenvalues
such that for all $d\in\N$
\begin{gather*}
			\sum_{\bm{k}\in\nabla_d} \left(\widetilde{\lambda}_{d,\bm{k}}\right)^\tau 
			= \left( \sum_{m=1}^{\#\M_1} \lambda_m^\tau \right)^{b_d} \sum_{\substack{\bm{j}\in\M_{a_d},\\j_1<\ldots<j_{a_d}}} \!\!\left(\widetilde{\lambda}_{a_d,\bm{j}}\right)^\tau 
			\geq \norm{\lambda \sep \l_\tau}^{\tau b_d} \cdot \lambda_1^\tau \cdot \ldots \cdot \lambda_{a_d}^\tau.
\end{gather*}
Hence \autoref{prop_general} together with the fact that $\lambda_{d,1} = \lambda_1^{b_d} \cdot \lambda_1 \cdot \ldots \cdot \lambda_{a_d}$ gives
\begin{gather}\label{asy_est}
			\left(\frac{\norm{\lambda \sep \l_\tau}^\tau}{\lambda_1^\tau} \right)^{b_d}
			\leq (1 + C)\,d^q
						+ C^{2\tau/p} \,\zeta\!\left(\frac{2\tau}{p}\right) \left( \frac{d^{2q/p}}{\lambda_{d,1}} \right)^\tau.
\end{gather}
In what follows we will use this inequality to conclude all the stated assertions.

\emph{Step 2}.
Here we prove the limit property for 
the initial error $\epsilon_d^{\mathrm{init}}=\sqrt{\lambda_{d,1}}$, \ie we need to show that for every fixed polynomial $\P > 0$
\begin{gather}\label{zero_limit}
			\lambda_{d,1} \, \P(d) \longrightarrow 0, \quad \text{as} \quad d \nach \infty.
\end{gather}
Since $\lambda_{d,1}\leq \lambda_1^{b_d} \cdot \lambda_1^{a_d}=\lambda_1^d$ 
we can restrict ourselves to the non-trivial case $\lambda_1 \geq 1$ in the following.
Assume that there exists a subsequence $(d_k)_{k\in\N}$ of natural numbers, 
as well as some constant $C_0>0$, such that $\lambda_{d_k,1}\, \P(d_k)$ is bounded from below by~$C_0$
for every $k\in\N$.
Then for every $d=d_k$ the right-hand side of~\link{asy_est} is bounded from above by some other polynomial $\P_1(d_k)>0$. 
On the other hand, due to the general condition $\lambda_2 > 0$, 
the term $\norm{\lambda \sep \l_\tau}^\tau / \lambda_1^\tau$ is strictly larger than one.
Thus it follows that there exists some $C_1>0$ such that
\begin{gather*}
			b_{d_k} \leq C_1 \ln(d_k) \quad \text{for every} \quad k\in\N.
\end{gather*}
Therefore we obtain that $a_{d_k} = d_k - b_{d_k} \nach \infty$, as $k\nach \infty$.
Moreover, the assumed boundedness of $\lambda_{d_k,1} \, \P(d_k)$ leads to
\begin{gather*}
			C_0 \,\P(d_k)^{-1} 
			\leq \lambda_{d_k,1} 
			\leq \lambda_1^{C_1 \ln(d_k)} \cdot \lambda_1 \cdot \ldots \cdot\lambda_{a_{d_k}} 
			= d_k^{C_1 \ln(\lambda_1)} \cdot \lambda_1 \cdot \ldots \cdot\lambda_{a_{d_k}}
\end{gather*}
since $\lambda_1 \geq 1$.
In the first step of the proof of \autoref{prop_suf_sym} we saw that
 $\lambda \in \l_\tau$ yields the existence 
of some $C_\tau>0$ such that $\lambda_m \leq C_\tau m^{-1/\tau}$
for every $m\in\N$. 
Indeed, this holds for $C_\tau=\norm{\lambda \sep \l_\tau}>1$.
Hence 
$\lambda_1^\tau \cdot \ldots \cdot \lambda_{a_{d_k}}^\tau \leq C_\tau^{\tau a_{d_k}} (a_{d_k}!)^{-1}$ 
which gives
\begin{gather*}
			\left( \frac{a_{d_k}}{e}\right)^{a_{d_k}} 
			\leq a_{d_k}! \leq (C_\tau^\tau)^{a_{d_k}} \,\P_2(d_k) 
			\quad \text{for all} \quad k\in\N
\end{gather*}
and some other polynomial $\P_2 > 0$.
If $k$ is sufficiently large then we conclude that
\begin{gather*}
			a_{d_k} \leq a_{d_k} \ln \!\left( \frac{a_{d_k}}{e\,C_\tau^\tau} \right) \leq \ln(\P_2(d_k)),
\end{gather*}
since $a_{d_k} \nach \infty$ implies $a_{d_k}/(e\, C_\tau^\tau) \geq e$ for $k\geq k_0$.
Therefore the number of antisymmetric coordinates $a_d$ needs to be
logarithmically bounded from above for every $d$ out of the sequence $(d_k)_{k\geq k_0}$.
Because also $b_{d_k}$ was found to be logarithmically bounded this is a contradiction
to the fact $d_k = a_{d_k} + b_{d_k}$.
Consequently, the hypothesis $\lambda_{d_k,1} \P(d_k) \geq C_0 > 0$ can not be true
for any subsequence $(d_k)_k$.
In other words, it holds \link{zero_limit}.

\emph{Step 3}.
Next we show \link{bound_a}. 
From the former step we know 
that there needs to exist some $d^*\in\N$ such that
$1/\lambda_{d,1} \geq 1$ for all $d \geq d^*$.
Hence, \link{asy_est} together with $\tau > p/2$ implies
\begin{gather*}
			\left(\frac{\norm{\lambda \sep \l_\tau}^\tau}{\lambda_1^\tau} \right)^{b_d}
			\leq C_2 \left( \frac{d^{2q/p}}{\lambda_{d,1}} \right)^\tau 
			= \frac{C_2 \, d^{2q\tau/p}}{\lambda_1^{\tau b_d} \cdot \lambda_1^\tau \cdot \ldots \cdot \lambda_{a_d}^\tau}
			\quad \text{for} \quad d\geq d^*,
\end{gather*}
where we set $C_2 = 1 + C + C^{2\tau/p} \zeta(2\tau/p)$.
Therefore we obtain
\begin{gather*}
			C_2 \, d^{2q\tau/p} \prod_{k=1}^{a_d} \frac{\norm{\lambda \sep \l_\tau}^\tau}{\lambda_k^\tau}
			\geq \norm{\lambda\sep \l_\tau}^{\tau d} 
\end{gather*}
for all $d\geq d^*$, which is equivalent to the claimed estimate~\link{bound_a}.

\textit{Step 4}.
It remains to show that $\lambda_1 \geq 1$ implies that
$\lim_{d\nach \infty} a_d$ is infinite.
To this end, note that every summand in \link{bound_a} is strictly positive.
If we assume for a moment the existence of a subsequence $(d_k)_{k\in\N}$
such that $a_{d_k}$ is bounded for every $k\in\N$ then the left-hand side of
\link{bound_a} is less than some positive constant divided by $d_k$.
Hence it tends to zero if $k$ approaches infinity.
On the other hand, the right-hand side of \link{bound_a} 
is strictly larger than some positive constant, because of $\lambda_1 \geq 1$ and $\lambda_2>0$. 
This contradiction completes the proof.
\end{proof}

As mentioned before there are examples such that the sufficient condition
\link{suf_condition} from \autoref{prop_suf_asy} is also necessary 
(up to some constant factor) in order to conclude polynomial tractability
in the antisymmetric setting.
Now we are ready to give such an example.

\begin{example}
Consider the situation of \autoref{prop_general} for $P=\AI$
and assume the problem $S_I$ to be polynomially tractable.
In addition, for a fixed $\tau\in(0,\infty)$, let 
$\lambda=(\lambda_m)_{m\in\N} \in \l_\tau$ be given such that $\lambda_1 \geq 1$
and assume the existence of some $m_0 \in \N$ such that
\begin{gather}\label{assumption}
				\lambda_m \geq \frac{\norm{\lambda \sep \l_\tau}}{m^{\alpha / \tau}}
				\quad \text{for all} \quad m>m_0 \quad \text{and some} \quad \alpha>1.
\end{gather}
Then we claim that there exist constants $\bar{d}\in\N$, $C \geq 1$, and $r \geq 0$ such that
\begin{gather}\label{claimed_bound}
				\alpha \cdot \frac{\ln{(a_d!)}}{d} + \frac{\ln{(C\,d^r)}}{d} 
				\geq \ln(\norm{\lambda \sep \l_\tau}^\tau) 
				\quad \text{for all} \quad d\geq \bar{d}.
\end{gather}

Recall that due to \autoref{prop_suf_asy}, 
for the amount of antisymmetry $a_d$, 
it was sufficient to assume \link{claimed_bound} with $\alpha=1$
in order to conclude (strong) polynomial tractability; see \link{suf_condition}.
Moreover keep in mind that we know from \autoref{ex:gamma}
that $\ln(a_d!)/d$ tends to $1/\gamma$ if we assume $a_d$ to be given by \link{eq:def_ad}.
Hence in the present example we have strong polynomial tractability if $\gamma<\ln^{-1}(\norm{\lambda\sep\l_\tau}^\tau)$,
whereas the problem is polynomially intractable if $\gamma> \alpha/\ln(\norm{\lambda\sep\l_\tau}^\tau)$.

Before we prove the assertion it might be useful to give a concrete example where
\link{assumption} holds true.
Therefore set $\lambda_m=1/m^2$, $\tau=m_0=1$, and $\alpha=3$.
Then it is easy to check that $\norm{\lambda \sep \l_\tau}=\zeta(2)=\pi^2/6$ 
and we obviously have $\lambda_1=1$.

To see that \link{claimed_bound} holds true 
we can use \autoref{prop_nec_asy} and, in particular, inequality~\link{bound_a}.
Since $\lambda_1 \geq 1$ we know that $\lim_{d} a_d = \infty$, \ie $a_d > m_0$
for every $d$ larger than some $d_1\in\N$.
Furthermore, note that \link{assumption} is equivalent to
\begin{gather*}
				\ln \!\left( \frac{\norm{\lambda \sep \l_\tau}^\tau}{\lambda_m^\tau} \right) 
				\leq \alpha \ln(m)
				\quad \text{for all} \quad m>m_0.
\end{gather*}
Hence if $d\geq d_1$ then we can estimate the sum in \link{bound_a} from above by
\begin{gather*}
				\frac{1}{d} \sum_{m=1}^{a_d} \ln \!\left( \frac{\norm{\lambda \sep \l_\tau}^\tau}{\lambda_m^\tau} \right)
				\leq \frac{m_0}{d}  \cdot \ln \!\left( \frac{\norm{\lambda \sep \l_\tau}^\tau}{\lambda_{m_0}^\tau} \right) + \frac{\alpha}{d} \sum_{m=m_0+1}^{a_d} \ln(m) 
				\leq \frac{C_\lambda}{d} + \alpha \cdot \frac{\ln (a_d!)}{d}.
\end{gather*}
Obviously, for $d$ larger than some $d_2\in\N$ the term $C_\lambda + \ln(C_2 \, d^{2q\tau/p})$ is less than
$\ln(C\, d^r)$, where $C\geq 1$ and $r\geq 0$. 
Here $r=0$ if and only if $q=0$ in \link{prop_nec_asy}, \ie
if the problem is strongly polynomially tractable.
Consequently we can conclude \link{claimed_bound} from \link{bound_a} by choosing $\bar{d}=\max{d_1,d_2,d^*}$.
\hfill$\square$
\end{example}

Although there remains a small gap between the necessary 
and the sufficient conditions for the absolute error criterion, 
the most important cases of antisymmetric
tensor product problems are covered by our results.
Let us summarize the main facts.

\begin{theorem}[Tractability of antisymmetric problems, absolute errors]\label{thm_asy_abs}
			Let\\ $S_1 \colon H_1 \nach \G_1$ denote a compact linear operator 
			between two Hilbert spaces and let $\lambda=(\lambda_m)_{m\in \N}$ be the sequence 
			of eigenvalues of $W_1={S_1}^{\!\dagger} S_1$ \wrt a non-increasing ordering. 
			Moreover, for $d>1$ let $\leer \neq I_d \subseteq\{1,\ldots,d\}$ and
			assume $S_I=(S_{d,I_d})_{d\in\N}$ to be the linear tensor product problem 
			$S=(S_d)_{d\in\N}$ restricted to the 
			$I_d$-antisymmetric subspaces $\AI_{I_d}(H_d)$ 
			of the $d$-fold tensor product spaces $H_d$. 
			Consider the worst case setting with respect to the absolute error criterion
			and let $\lambda_2>0$.
			Then for the case $\lambda_1 < 1$ the following statements are equivalent:
			\begin{itemize}
						\item $S_I$ is strongly polynomially tractable.
						\item $S_I$ is polynomially tractable.
						\item There exists a constant $\tau \in (0,\infty)$ such that $\lambda \in \l_\tau$.
			\end{itemize}
			Moreover, the same equivalences hold true if $\lambda_1\geq 1$ and $\#I_d$ grows linearly with the dimension $d$.
\end{theorem}

At this point we mention that for the case of fully antisymmetric problems, 
\ie for $\#I_d=a_d=d$, an explicit formula for the information complexity
\wrt the absolute error criterion is known.
Furthermore, simple examples can be constructed which show that we cannot expect
the same nice tractability behavior if we deal with normalized errors.
For further details the interested reader is referred to \cite[Proposition~8]{W11}.

\subsection{Antisymmetric problems (normalized errors)}
Up to now every complexity assertion in this chapter 
was mainly based on \autoref{Thm:General_Tract_abs} which dealt with the 
general situation of arbitrary compact
linear operators between Hilbert spaces and with the absolute error criterion.
While investigating tractability properties of $I$-symmetric 
problems with respect to the normalized error criterion, 
we were able to use assertions from the absolute error setting.
Since for $I$-antisymmetric problems 
the structure of the initial error is more complicated,
this approach will not work again.
Therefore we recall \autoref{Thm:General_Tract_norm} 
as a replacement of \autoref{Thm:General_Tract_abs} 
for the normalized setting.
This in hand, we can give the following necessary conditions for (strong)
polynomial tractability.

\begin{prop}[Necessary conditions, antisymmetric case]
			Let $S_I=(S_{d,I_d})_{d\in\N}$ denote an $I$-antisymmetric problem
			as defined at the beginning of \autoref{sect:complexity_antisym}
			and consider the worst case setting \wrt to normalized errors.
			Then the fact that~$S_I$ is polynomially tractable
			with the constants $C,p>0$ and $q\geq 0$
			implies that $\lambda=(\lambda_m)_{m\in\N} \in \l_\tau$ for all $\tau>p/2$.
			Moreover, for $d$ tending to infinity, 
			$\epsilon_d^{\mathrm{init}}$ tends to zero faster 
			than the inverse of any polynomial and $b_d\in\0(\ln d)$, as $d\nach\infty$.
			Thus we have $\lim_{d\nach \infty} a_d/d= 1$.
			In addition, if $S_I$ is strongly polynomially tractable 
			then $b_d \in \0(1)$, as $d\nach\infty$.
\end{prop}
\begin{proof}
From \autoref{Thm:General_Tract_norm} it follows that there 
is some $C_1>0$ such that
\begin{gather}\label{eq:est_normed_eigenvalues}
			\frac{1}{(\lambda_{d,1})^\tau} \sum_{\bm{k}\in\nabla_d} \left(\widetilde{\lambda}_{d,\bm{k}}\right)^\tau 
			= \sum_{i=1}^\infty \left( \frac{\lambda_{d,i}}{\lambda_{d,1}} \right)^\tau 
			\leq C_1 d^{2\tau q/p}
			\quad \text{ for every } \quad d\in\N
\end{gather}
and all $\tau > p/2$.
Once more the index set $\nabla_d$ is given as in~\link{def_nabla}.
Indeed, \autoref{Thm:General_Tract_norm} yields that it is sufficient to take
$C_1 = 2\,(1+C)^{2\tau/p} \, \zeta(2\tau/p)$.
As in the proof of \autoref{prop_general} it suffices 
to consider the case $d=1$ in \link{eq:est_normed_eigenvalues}
to see that $\lambda\in\l_\tau$ is necessary for polynomial tractability.
Moreover, like with the arguments of Step 1 in the proof of \autoref{prop_nec_asy},
it follows that
\begin{gather}\label{est_norm}
			\left( \frac{\norm{\lambda \sep \l_\tau}^\tau}{\lambda_1^\tau} \right)^{b_d}
			\leq C_1 \,d^{2\tau q/p}, 	\quad d\in\N,
\end{gather}
since $\lambda_{d,1}=\lambda_1^{b_d}\cdot \lambda_1 \cdot \ldots \cdot \lambda_{a_d}$.
Due to the general assertion $\lambda_2>0$ we have $\norm{\lambda \sep \l_\tau}^\tau > \lambda_1^\tau$ 
and thus polynomial tractability of $S_I$ implies the bound $b_d \leq C_2 \ln(d)$ for some $C_2 \geq 0$,
\ie $b_d \in \0(\ln d)$, as $d\nach\infty$.
Therefore we obviously have
\begin{gather*}
			1\geq \frac{a_d}{d} 
			= 1- \frac{b_d}{\ln d} \cdot \frac{\ln d}{d} 
			\geq 1-C_2 \cdot \frac{\ln d}{d} \longrightarrow 1, 
			\quad d\nach \infty.
\end{gather*}
The proof that strong polynomial tractability leads to $b_d \in \0(1)$, as $d\nach\infty$,
can be obtained using \link{est_norm} with the same arguments as before and $q=0$.
Finally we need to show the assertion concerning $\epsilon_d^{\rm init}$.
Here we refer to Step 2 in the proof of \autoref{prop_nec_asy}.
\end{proof}

\section{Applications}\label{sect:applications}
This last section of the present chapter is devoted to
applications of the theory developed previously.
In \autoref{sect:toy_ex} we follow the lines of the introduction of \cite{W12b} 
and illustrate the power of imposing additional (anti)symmetry conditions to linear tensor product problems
by using simple toy examples.
Afterwards, in \autoref{sect:wave}, 
we focus our attention to more advanced problems which we are faced with in practice.
There we briefly introduce wavefunctions 
and show how our results allow it to handle the approximation problem
for such classes of functions.
\subsection{Toy examples}\label{sect:toy_ex}
The aim of the following simple examples is to show that exploiting an a priori knowledge
about (anti)symmetries of a given tensor product problem can
help to obtain tractability, but it does not make the problem trivial
in general. 

Let $S=(S_d \colon H_d \nach \G_d)_{d\in\N}$
denote a tensor product problem between Hilbert spaces.
Remember that due to \autoref{sect:TensorBasics} for complexity studies it suffices to specify 
the singular values of the univariate operator $S_1$.
To simplify the presentation we slightly abuse the notation
and denote the information complexity of the entire problem $S$ by
$n^{\mathrm{ent}}(\epsilon,d)$.
We want to compare this quantity with the respective information complexities
of the restriction of $S$ to the fully symmetric and the fully antisymmetric subspaces of $(H_d)_{d\in\N}$.
These numbers will be denoted by $n^{\mathrm{sym}}(\eps,d)$ and $n^{\mathrm{asy}}(\eps,d)$, respectively.

Clearly, our results yield that in any case 
(as long as we deal with the worst case setting and the absolute error criterion)
\begin{gather*}
			n^{\rm asy}(\epsilon,d) 
			\leq n^{\rm sym}(\epsilon,d) 
			\leq n^{\rm ent}(\epsilon,d)
			\quad \text{for every} \quad \epsilon>0 \quad \text{and all} \quad d\in\N,
\end{gather*}
where for $d=1$ the terms coincide,
since then we do not claim any (anti)symmetry.
To see that additional (anti)symmetry conditions may reduce the 
information complexity dramatically consider the 
following three examples.

\begin{example}
Let us have a look at the simple case of a linear 
operator $S_1$ with singular values $\sigma$ such that $\lambda_1=\lambda_2=1$ 
and $\lambda_j=0$ for $j\geq 3$.
Then the information complexity of the entire tensor product problem
can be shown to be
\begin{gather*}
			n^{\rm ent}(\epsilon,d) = 2^d 
			\quad \text{for all} \quad d\in\N \quad \text{and} \quad \epsilon < 1.
\end{gather*}
Hence the problem suffers from the curse of dimensionality and is therefore 
intractable.

On the other hand, our results show that in the fully symmetric setting we
have polynomial tractability, because
\begin{gather*}
			n^{\rm sym}(\epsilon,d)=d+1 
			\quad \text{for all} \quad d\in\N \quad \text{and} \quad \epsilon < 1.
\end{gather*}

Moreover, it can be proved that in this case the complexity of the fully antisymmetric
problem decreases with increasing dimension $d$ and, finally, 
the problem even gets trivial.
In detail, we have
\begin{gather*}
			n^{\rm asy}(\epsilon,d) = \max{3-d,0} 
			\quad \text{for all} \quad d\in\N \quad \text{and} \quad \epsilon < 1
\end{gather*}
which yields strong polynomial tractability.
\hfill$\square$
\end{example}

\begin{example}
Next let us consider a more challenging problem, where 
$\lambda_1=\lambda_2=\ldots=\lambda_m=1$ and 
$\lambda_j=0$ for every $j>m\geq2$. 
For $m=2$ this obviously coincides with the example studied above, 
but letting $m$ increase may tell us more about the structure
of (anti)symmetric tensor product problems.
In this situation it is easy to check that
for every $d\in\N$ and all $\epsilon < 1$
\begin{gather*}
			n^{\rm ent}(\epsilon,d)	= m^d 
			\quad \text{and} \quad 
			n^{\rm asy}(\epsilon,d) = \begin{cases}
																		\binom{m}{d}, & d \leq m,\\
																		0, 						& d > m.
																\end{cases}
\end{gather*}
Since $\binom{m}{d}\geq 2^{d-1}$ for $d\leq\floor{m/2}$,
this means that for large $m$ the complexity in the antisymmetric case 
increases exponentially fast with $d$ up to a certain maximum.
Beyond this point it falls back to zero.

The information complexity in the symmetric setting is much harder to
calculate for this case. 
However, it can be seen that we have polynomial tractability, but
$n^{\rm sym}(\epsilon,d)$ needs to grow at least linearly with $d$
such that the symmetric problem cannot be strongly polynomially tractable,
whereas this holds in the antisymmetric setting. 
The entire problem again suffers from the curse of dimensionality.
\hfill$\square$
\end{example}

\begin{example}
For a last illustrating example consider the case $\lambda_1=1$
and $\lambda_{j+1}=j^{-\beta}$ for some $\beta\geq0$ and all $j\in\N$.
That means, we have the two largest singular values 
$\sigma_1=\sigma_2$ of $S_1$ equal to one.
The remaining series decays like the inverse of some polynomial.
If $\beta=0$ then the operator~$S_1$ is not compact, 
since the sequence $\lambda=(\lambda_m)_{m\in\N}$ does not tend to zero; 
hence all the information complexities are infinite in this case.

For $\beta>0$, any $\delta>0$, and some $C>0$ we have
\begin{gather*}
			n^{\rm ent}(\epsilon,d)\geq 2^{d}, 
			\quad n^{\rm sym}(\epsilon,d) \geq d+1,
			\quad \text{and} \quad n^{\rm asy}(\epsilon,d)\leq C \epsilon^{-(2/\beta+\delta)},
\end{gather*}
for all $\epsilon<1$, as well as every $d\in\N$.
Thus, again for the entire problem we observe the curse, whereas the
antisymmetric problem is strongly polynomially tractable.
Once more, the symmetric problem can be shown to be polynomially tractable.
Note that in this example the antisymmetric case is not trivial,
because all $\lambda_j$ are strictly positive.
If we replace $j^{-\beta}$ by $\log^{-1}(j+1)$ in this example we obtain
(polynomial) intractability even in the antisymmetric setting.
\hfill$\square$
\end{example}

\subsection{Wavefunctions}\label{sect:wave}
During the few last decades there has been considerable interest in finding approximations 
of so-called \emph{wavefunctions}, e.g., solutions of the electronic Schr\"{o}dinger equation. 
Due to the \emph{Pauli principle} of quantum physics only functions with certain
(anti)symmetry properties are of physical interest. 
For a more detailed view see, 
e.g, Hamaekers~\cite{H09}, Yserentant~\cite{Y10}, or Zeiser~\cite{Z10}.
Furthermore, for a comprehensive introduction to the topic, 
as well as a historical survey, 
we refer the reader to Hunziker and Sigal~\cite{HS00}
and Reed and Simon~\cite{RS78}.

In particular, the notion of multiple partial antisymmetry \wrt 
two sets of coordinates is useful for describing wavefunctions~$\Psi$.
In computational chemistry such functions 
occur as models which describe quantum states of certain physical $d$-particle systems. 
Formally, these functions depend on $d$ blocks of variables $\bm{y_i}=(\bm{x^{(i)}},s^{(i)})$, 
for $i=1,\ldots,d$, which represent the spacial coordinates 
$\bm{x^{(i)}}=(x_1^{(i)},x_2^{(i)},x_3^{(i)})\in\R^3$ and 
certain additional intrinsic parameters 
$s^{(i)} \in C$ of each particle $\bm{y_i}$ within the system.
Hence, rearranging the arguments such that 
$\bm{x}=(\bm{x^{(1)}},\ldots,\bm{x^{(d)}})$ and $\bm{s}=(s^{(1)},\ldots,s^{(d)})$ yields that
\begin{gather*}
				\Psi \colon (\R^{3})^{d} \times C^d \nach \R, \quad (\bm{x},\bm{s}) \mapsto \Psi(\bm{x},\bm{s}).
\end{gather*}
In the case of systems of electrons one of the most important parameters 
is called \emph{spin} and it can take only two values, i.e., 
$s^{(i)}\in C=\{-\frac{1}{2}, + \frac{1}{2}\}$. 
Due to the Pauli principle 
the only wavefunctions $\Psi$ 
that are physically admissible are those which are antisymmetric 
in the sense that for $I\subseteq\{1,\ldots,d\}$ and $I^c=\{1,\ldots,d\}\setminus I$
\begin{gather*}
				\Psi(\bm{\pi(x)},\bm{\pi(s)}) = (-1)^{\abs{\pi}} \Psi(\bm{x},\bm{s}) \quad \text{for all} \quad \pi \in \S_I \cup \S_{I^c}.
\end{gather*}
Thus $\Psi$ changes its sign if we replace any particles $\bm{y_i}$ and $\bm{y_j}$ 
by each other which possess the same spin, \ie $s^{(i)}=s^{(j)}$.
So the set of particles, and therefore also the set of spacial coordinates,
naturally split into two groups $I_+$ and $I_-$.
In detail, for wavefunctions of $d$ particles $\bm{y_i}$
we can (without loss of generality) assume
that the first $\#I_+$ indices $i$ belong to the group of positive spin, 
whereas the rest of the particles possess negative spin, 
\ie $I_+=\{1,\ldots,\#I_+\}$ and $I_-= I_+^c=\{\#I_+ + 1,\ldots, d\}$.

In physics it is well-known that some problems, e.g., the electronic Schr\"{o}dinger equation,
which involve (general) wavefunctions can be reduced to a bunch of similar problems,
where each of them only acts on functions $\Psi_{\bm{s}}$ 
out of a certain Hilbert space $\F_d = \F_d(\bm{s})$.
That is,
\begin{gather*}
				\Psi_{\bm{s}} =\Psi(\bm{\cdot},\bm{s}) \in \F_d \subset \{f \colon (\R^{3})^{d} \nach \R\}
\end{gather*}
with a given fixed spin configuration $\bm{s}\in C^d$. 
Of course every possible spin configuration~$\bm{s}$ 
corresponds to exactly one choice 
$I_+\subseteq\{1,\ldots,d\}$ of indices. 
Moreover, it is known that $\F_d$ is a Hilbert space which
possesses a tensor product structure.
Therefore we can model wavefunctions 
as elements of certain classes of smoothness, 
e.g., $\F_d \subset H_d = H_1\otimes\ldots\otimes H_1 = W_2^{(1,\ldots,1)}((\R^3)^d)$, 
as Yserentant~\cite{Y10} recently did,
and incorporate spin properties 
by using projections of the type 
$\AI = \AI_{I_+} \circ \AI_{I_-}$, 
as defined in \autoref{sect:antisym_HFS}.
In particular, \autoref{lemma_basis} then yields that
\begin{gather*}
				\F_d = \AI(H_d) = \AI_{I_+}(H_{\#I_+}) \otimes \AI_{I_-}(H_{\#I_-})
\end{gather*}
and the system of all
\begin{gather*}
				\overline{\xi}_{\bm{k}} 
				= \sqrt{\#\S_{I_+} \cdot \#\S_{I_-}} \cdot \AI(e_{\bm{k}}), \quad \bm{k} \in \overline{\nabla}_d,
\end{gather*}
with
\begin{gather*}
				\overline{\nabla}_d 
				= \left\{ \bm{k}=(\bm{i},\bm{j}) \in \N^{\#I_+}\!\times\N^{\#I_-} \sep i_1 < i_2 < \ldots < i_{\#I_+} \text{ and } j_{1} < \ldots < j_{\#I_-} \right\}
\end{gather*}
builds an orthonormal basis of $\F_d = \AI(H_d)$, where
the set $\{e_{\bm{m}} \sep \bm{m} \in \N^d \}$ 
is once again assumed to be an orthonormal 
tensor product basis of $H_d=H_1\otimes \ldots \otimes H_1$ 
constructed with the help of
$\{e_m \sep m\in\N\}$, an arbitrary orthonormal basis of $H_1$. 

Note that in the former sections the underlying 
Hilbert space $H_1$ always consists of 
univariate functions. 
In contrast, wavefunctions of one particle depend on at least three (spacial) variables, 
but we want to stress the point that this is just a formal issue.
Anyway, our approach radically decreases the degrees of freedom 
and improves the solvability of certain problems $S=(S_d)_{d\in\N}$ 
like the approximation problem, \ie $S_d=\id\colon H_d \nach \G_d$ for every $d\in\N$,
considered in connection with the electronic Schr\"{o}dinger equation.

\autoref{theo:opt_sym_algo} provides an algorithm which is
optimal for the $\G_d$-approxima\-tion of 
$d$-particle wavefunctions in $\F_d$
with respect to all linear algorithms 
that use at most $n$ continuous linear functionals.
Therefore we only need to choose the right ONB 
$\{e_m=\phi_m \sep m\in\N\}$ of $H_1$ which coincides 
with the eigenfunctions of the 
univariate operator $W_1={S_1}^{\!\dagger} S_1$.
Moreover, the error can be calculated 
exactly in terms of the eigenvalues 
$\lambda = (\lambda_m)_{m\in\N}$ of $W_1$.

Furthermore it is possible to prove 
a modification of \autoref{thm_asy_abs} 
for problems dealing with wavefunctions.
In fact, for the mentioned approximation problem 
polynomial tractability as well as strong polynomial tractability are
equivalent to the fact that the sequence $\lambda$ 
of the squared singular values
of the univariate problem belong to some $\l_\tau$-space
if we consider the absolute error criterion.
The reason is that all the assertions in \autoref{sect:antisym_abs}
can be easily extended to the multiple partially antisymmetric case.
In detail, if we denote the number of 
antisymmetric coordinates~$\bm{x^{(i)}}$ within
each antisymmetry group $I_d^{(m)}\subseteq\{1,\ldots,d\}$ 
by $a_{d,m}$ with $m=1,\ldots,M$
then the constraint $a_d + b_d = d$ extends to 
\begin{gather*}
		a_{d,1}+\ldots+a_{d,M}+b_d=d.
\end{gather*}
Here $b_d$ again denotes the number of coordinates 
without any antisymmetry condition.
In conclusion, the sufficient condition \link{suf_condition} in
\autoref{prop_suf_asy} transfers to
\begin{gather*}
			\frac{1}{d} \sum_{m=1}^M \ln(a_{d,m}!) \geq \norm{\lambda \sep \l_\tau}^\tau, 
			\quad \text{for all} \quad d\geq d_0,
\end{gather*}
which is always satisfied in the case of wavefunctions, 
since then $M=2$ and the cardinality $a_{d,m}$ 
of at least one of the groups of the same spin needs to grow linearly with the dimension~$d$.

  \cleardoubleplainpage
	\backmatter
	\kopffussBib
	\cleardoublepage
\phantomsection
\addcontentsline{toc}{chapter}{Bibliography}
\bibliographystyle{MyBibStyle}


\end{document}